\documentclass{amsart}
\usepackage{stmaryrd}
\usepackage{enumerate}
\usepackage{mathdots}
\usepackage{booktabs}
\usepackage{float}
\usepackage{cite}
\usepackage[linktocpage]{hyperref}
\usepackage{mathtools}
\usepackage{tabularray}
\usepackage{changepage}
\usepackage{multicol}

\SetSymbolFont{stmry}{bold}{U}{stmry}{m}{n}

\renewenvironment{description}
  {\list{}{\labelwidth10pt\leftmargin10pt\parsep5pt
    }}
  {\endlist}

\makeatletter
\g@addto@macro\bfseries{\boldmath}
\makeatother

\allowdisplaybreaks

\hypersetup{unicode=false, pdftoolbar=true, pdfmenubar=true,
    pdffitwindow=false, pdfstartview={FitH}, pdfnewwindow=true,
    colorlinks=false, linkcolor=blue, citecolor=green}

\newtheorem{theorem}{Theorem}[section]
\newtheorem{proposition}[theorem]{Proposition}
\newtheorem{corollary}[theorem]{Corollary}
\newtheorem{example}[theorem]{Example}
\newtheorem{lemma}[theorem]{Lemma}
\newtheorem{conventions}[theorem]{Conventions}
\theoremstyle{definition}
\newtheorem{definition}[theorem]{Definition}
\numberwithin{equation}{section}

\begin{document}

\title[Nilpotency indices]{Nilpotency indices for quantum Schubert cell
algebras}

\author{Garrett Johnson}

\address{Department of Mathematics and Physics \\North Carolina Central
University\\Durham, NC 27707\\USA}

\email{gjohns62@nccu.edu}

\author{Hayk Melikyan}

\address{Department of Mathematics and Physics \\North Carolina Central
University\\Durham, NC 27707\\USA}

\email{gmelikian@nccu.edu}



\thanks{The authors were supported by NSF grant DMS-1900823.}

\begin{abstract}

    We study quantum analogs of $\operatorname{ad}$-nilpotency and Engel
    identities in quantum Schubert cell algebras ${\mathcal U}_q^+[w]$. For
    each pair of Lusztig root vectors, $X_\mu$ and $X_\lambda$, in ${\mathcal
    U}_q^+[w]$, where $w$ belongs to a finite Weyl group $W$ and $\mu$ precedes
    $\lambda$ with respect to a convex order on the roots in $\Delta_+ \cap
    w(\Delta_-)$, we find the smallest natural number $k$, called the
    nilpotency index, so that $\left(\operatorname{ad}_q X_\mu \right)^k$ sends
    $X_\lambda$ to $0$, where $\operatorname{ad}_q X_\mu$ is the $q$-adjoint
    map.

    We start by observing that every pair of Lusztig root vectors can be
    naturally associated to a triple $(u, i, j)$, where $u \in W$ and $i$ and
    $j$ are indices such that $\ell(s_i u s_j) = \ell(u) - 2$. In light of
    this, we define an equivalence relation, based upon the weak left and weak
    right Bruhat orders, on the set of such triples. We show this equivalence
    relation respects nilpotency indices, and that each equivalence class
    contains an element of the form $(v, r, s)$, where $v$ is either (1) a
    bigrassmannian element satisfying a certain orthogonality condition, or (2)
    the longest element of some subgroup of $W$ generated by two simple
    reflections. For each such $(v, r, s)$, we compute the associated
    nilpotency index.

\end{abstract}

\maketitle

\setcounter{tocdepth}{1}

{\let\bfseries\mdseries \tableofcontents}

\section{Introduction}

Nilpotency is a widely studied property throughout algebra. It is a key concept
needed to better understand general algebraic objects.  In the context of Lie
algebras, one of the most important and classically well known results involves
nilpotency, namely Engel's theorem: \textit{A finite dimensional Lie algebra
$\mathfrak{g}$ is nilpotent if and only if $\operatorname{ad}(X)$ is nilpotent
for all $X\in \mathfrak{g}$.} With the aim of having a criterion for
nilpotency, variants of Engel's theorem exist for other classes of algebras
\cite{Formanek, Stitzinger1, Stitzinger2}, as well as for infinite-dimensional
Lie algebras, notably nil Lie algebras of bounded index.  In particular,
Kostrikin \cite{Kostrikin1} proved every finitely generated Lie algebra over a
field of characteristic $p$ which satisfies the \textit{Engel identity}
\[
    (\operatorname{ad} X)^k = 0,
\]
where $k \leq p$ or $p = 0$, is nilpotent. Using connections between nilpotent
groups and Lie algebras satisfying the Engel identity, Zel'manov
\cite{Zelmanov1, Zelmanov2} settled the restricted Burnside problem, which at
the time had been an open problem in group theory for roughly 90 years.

\subsection{Main topics}

In this paper, our focus is on the structure of the quantum Schubert cell
algebras ${\mathcal U}_q^+[w]$, particularly on quantum analogs of
$\operatorname{ad}$-nilpotency and Engel identities for these algebras.

We aim to find the smallest $k \in \mathbb{N}$ (called the \textit{nilpotency
index}) such that
\[
    \left(\operatorname{ad}_qX_\mu\right)^k (X_\lambda) = 0,
\]
where $X_\mu$ and $X_\lambda$ are Lusztig root vectors in a quantum Schubert
cell algebra, and $\operatorname{ad}_q$ is the $q$-adjoint map.  Interestingly,
bigrassmannian Weyl group elements play an important role in this analysis.

We also introduce algorithms for finding explicit presentations of quantum
Schubert cell algebras in terms of generators and relations, which seems to be
lacking in the literature.

\subsection{Background on quantum Schubert cells}

Quantum Schubert cell algebras were defined by Lusztig \cite{L1, L} and De
Concini, Kac, and Procesi \cite{DKP}.  They have gained much interest in recent
years, and have appeared in several contexts, including cluster algebras
\cite{GLS, GY2}, Hopf algebras (coideal subalgebras and coinvariants) \cite{HK,
HS, JJ}, crystal basis theory \cite{Lusztig}, and ring theory \cite{Y}.  There
is a quantum Schubert cell algebra for every element $w$ in the Weyl group of a
symmetrizable Kac-Moody Lie algebra $\mathfrak{g}$. The algebra associated to
$w$, denoted ${\mathcal U}_q^+[w]$, is a $q$-deformation of the universal
enveloping algebra of the nilpotent Lie algebra $\mathfrak{n}_w :=
\mathfrak{n}^+ \cap w(\mathfrak{n}^-)$, where $\mathfrak{n}^\pm$ are the
nilradicals of a pair of opposite Borel subalgebras of $\mathfrak{g}$. The set
of roots of $\mathfrak{n}_w$ will be denoted by $\Delta_w$. The roots in
$\Delta_w$ will be referred to as the \textit{roots of $w$}. These are
precisely the positive roots that get sent to negative roots by the action of
$w^{-1}$.

The algebra ${\mathcal U}_q^+[w]$ is a subalgebra of the positive part of
${\mathcal U}_q(\mathfrak{g})$, and is generated by a collection of elements
$X_\beta$, called \textit{Lusztig root vectors}, indexed by the roots $\beta
\in \Delta_w$. With respect to the grading on ${\mathcal U}_q(\mathfrak{g})$ by
the root lattice $Q$, ${\mathcal U}_q^+[w]$ is a graded subalgebra,
\[
    {\mathcal U}_q^+[w] = \oplus_{\lambda \in Q} \left({\mathcal
    U}_q^+[w]\right)_\lambda,    \hspace{10mm} \left( {\mathcal U}_q^+[w]
    \right)_\lambda = {\mathcal U}_q^+[w] \cap \left( {\mathcal
    U}_q(\mathfrak{g}) \right)_\lambda.
\]
With this, the Lusztig root vector $X_\beta$ has degree $\beta$.

We note that the $X_\beta$'s are constructed using Lusztig's braid symmetries
of ${\mathcal U}_q(\mathfrak{g})$, and, as elements of ${\mathcal
U}_q(\mathfrak{g})$, depend on a chosen convex order on $\Delta_w$ (or,
equivalently, on a chosen reduced expression for $w$).  However, the algebra
${\mathcal U}_q^+[w]$, which is defined to be the subalgebra of ${\mathcal
U}_q(\mathfrak{g})$ generated by the $X_\beta$'s, does not depend on the
ordering.  In view of this, we will tacitly assume there is some fixed convex
order on $\Delta_w$ whenever we refer to a Lusztig root vector $X_\beta \in
{\mathcal U}_q^+[w]$. With each convex order on $\Delta_w$, say $\beta_1 <
\cdots < \beta_N$, there is a presentation of ${\mathcal U}_q^+[w]$ as an
iterated Ore extension over the base field $\mathbb{K}$,
\[
    {\mathcal U}_q^+[w] = \mathbb{K}[X_{\beta_1}][X_{\beta_2}; \sigma_2,
    \delta_2] \ldots [X_{\beta_N}; \sigma_N, \delta_N],
\]
and thus ordered monomials in the variables $X_{\beta_1},\cdots, X_{\beta_N}$
are a $\mathbb{K}$-basis of ${\mathcal U}_q^+[w]$.

\subsection*{Conventions and Notation}

Throughout, all algebras are defined over an arbitrary base field $\mathbb{K}$.
The deformation parameter $q \in \mathbb{K}$ is nonzero and not a root of
unity. As mentioned, quantum Schubert cell algebras ${\mathcal U}_q^+[w]$ are
certain subalgebras of ${\mathcal U}_q(\mathfrak{g})$; the underlying Lie type
of $\mathfrak{g}$ will be denoted by
\[
    X_n, \hspace{5mm} X \in \left\{A, B, C, D, E, F, G\right\},
\]
where $n > 1$ is the rank. The associated Weyl group will be denoted by
$W(X_n)$, or $W$ for short.  We use $\mathbf{I} := [1, n]$ to denote the index
set of $W$. Let $s_i$ ($i\in \mathbf{I}$) be the simple reflections in $W$, and
let $\Pi := \left\{\alpha_i\right\}_{i \in \mathbf{I}}$ be the set of simple
roots.

\subsection{Main results}

Before stating the main results, we need to define the $q$-adjoint map. First,
let $\langle \hspace{3pt}, \hspace{3pt} \rangle$ be the symmetric form on the
root lattice $Q$, normalized so that $\langle \alpha, \alpha \rangle = 2$ for
short roots $\alpha$. Define the \textit{$q$-commutator} $\left[ \cdot, \cdot
\right] : {\mathcal U}_q^+[w] \times {\mathcal U}_q^+[w] \to {\mathcal
U}_q^+[w]$ by the rule
\[
    [X, Y] := XY - q^{\langle \lambda, \mu \rangle} YX,
\]
for all homogeneous $X \in \left({\mathcal U}_q^+[w]\right)_\lambda$, and $Y
\in \left({\mathcal U}_q^+[w]\right)_\mu$, ($\lambda, \mu \in Q$), and extend
bilinearly. Define the \textit{$q$-adjoint map} $\operatorname{ad}_q :
{\mathcal U}_q^+[w] \to \operatorname{End}_\mathbb{K}\left({\mathcal
U}_q^+[w]\right)$ by the rule
\[
    \left(\operatorname{ad}_q X\right)(Y) := [X, Y]
\]
for all $X, Y \in {\mathcal U}_q^+[w]$.

Next, let $\theta$ be the highest root. Written as an integral combination of
simple roots, suppose $\theta = c_1 \alpha_1 + \cdots + c_n \alpha_n$.  Define
\begin{equation}
    \label{c_max, defn}
    c_{\operatorname{max}} : = \operatorname{max} \left( c_1 ,\cdots, c_n
    \right).
\end{equation}
Thus, $c_{\operatorname{max}} = 1$ in type $A_n$, $c_{\operatorname{max}} = 2$
in types $B_n$ ($n > 1$), $C_n$ ($n > 1$), and $D_n$ ($n > 3$),
$c_{\operatorname{max}} = 3$ in types $E_6$ and $G_2$, $c_{\operatorname{max}}
= 4$ in types $E_7$ and $F_4$, and $c_{\operatorname{max}} = 6$ in type $E_8$.

\begin{theorem}

    \label{intro, theorem 1}

    Let $c_{\operatorname{max}}$ be as defined in (\ref{c_max, defn}) above.
    Then for every pair of Lusztig root vectors, $X_\mu$ and $X_\lambda$ with
    $\mu < \lambda$, in a quantum Schubert cell algebra ${\mathcal U}_q^+[w]$,
    \[
        \left(\operatorname{ad}_q X_\mu \right)^{c_{\operatorname{max}} + 1}
        (X_\lambda) = 0.
    \]

\end{theorem}

\noindent Our results are actually stronger than the statement in Theorem
(\ref{intro, theorem 1}) because, in proving this, we effectively provide an
algorithm that takes a pair of Lusztig root vectors, $X_\mu$ and $X_\lambda$
with $\mu < \lambda$, and returns the smallest $k \in \mathbb{N}$ such that
$\left(\operatorname{ad}_q X_\mu \right)^k (X_\lambda) = 0$.

The motivation for obtaining the result in Theorem (\ref{intro, theorem 1})
came from our work in \cite{JM}, where we investigated automorphisms of
${\mathcal U}_q^+[w]$. If one considers the algebra ${\mathcal Y}_q^+[w]$
generated by $\left\{ y_\beta : \beta \in \Delta_w \right\}$ with defining
relations analogous to those in the statement in Theorem (\ref{intro, theorem
1}), except with the $X$'s replaced with $y$'s, then Theorem (\ref{intro,
theorem 1}) implies $y_\beta \mapsto X_\beta$ (for all $\beta \in \Delta_w$)
defines a surjective algebra homomorphism from ${\mathcal Y}_q^+[w]$ onto
${\mathcal U}_q^+[w]$. Hence, every automorphism of ${\mathcal U}_q^+[w]$
induces an automorphism of ${\mathcal Y}_q^+[w]$. Notice also the defining
relations of ${\mathcal Y}_q^+[w]$ resemble the $q$-Serre relations
\[
    \operatorname{ad}_q(E_i)^{1 - c_{ij}} (E_j) = 0,  \hspace{10pt} (i,j \in
    \mathbf{I} \text{ and } i \neq j),
\]
which are the defining relations for the positive part ${\mathcal
U}_q^+(\mathfrak{g})$ of ${\mathcal U}_q(\mathfrak{g})$; the automorphism
problem for ${\mathcal U}_q^+(\mathfrak{g})$ was completely settled in
\cite{Y2}.

Perhaps $\operatorname{ad}_q$-nilpotency, as stated in Theorem (\ref{intro,
theorem 1}), is not so surprising, given the connection between ${\mathcal
U}_q^+[w]$ and the Lie algebra $\mathfrak{n}_w$, known to be nilpotent.
However, the bound, $c_{max} + 1$, on the nilpotency index is intriguing.  For
comparison, for root vectors $x_\mu, x_\lambda$ in the non-$q$-deformed algebra
${\mathcal U}(\mathfrak{n}_w)$, $\left(\operatorname{ad}x_\mu\right)^{M + 1}
(x_\lambda) = 0$, where $M$ is the largest integer such that $\lambda + M\mu$
is a root. Thus, in ${\mathcal U}(\mathfrak{n}_w)$, the nilpotency index in
bounded above by the length of the longest root string ($2$ in simply-laced
types, $3$ in types $B,C,F$, and $4$ in type $G_2$).

\begin{theorem}

    \label{intro, theorem 2}

    For each Lie type $X_n$ with $n > 1$, there exists $w\in W(X_n)$ and a pair
    of Lusztig root vectors, $X_\mu$ and $X_\lambda$ with $\mu < \lambda$, in
    ${\mathcal U}_q^+[w]$ such that
    \[
        \left( \operatorname{ad}_q X_\mu \right)^{c_{\operatorname{max}}}
        (X_\lambda) \neq 0.
    \]

\end{theorem}

Moreover, we provide tables (Tables (\ref{table, small rank ABCD}) -
(\ref{table, E8})) that can be used to obtain explicit examples of Weyl group
elements $w$ and Lusztig root vectors, $X_\mu$ and $X_\lambda$, fulfilling the
conditions stated in Theorem (\ref{intro, theorem 2}).

\subsection{Outline and other results}

Section (\ref{section, QSC}) contains an overview of quantum Schubert cell
algebras and their properties most relevant for our purposes, as well as
algorithms to find explicit presentations of these algebras. In Section
(\ref{section, Gamma(W)}), we introduce the set
\[
    \Gamma(W) := \left\{ (w, i, j) \in W \times \mathbf{I} \times \mathbf{I} :
    \ell(s_i w s_j) = \ell(w) - 2 \right\}
\]
and show how any pair of Lusztig root vectors, say $X_\mu$ and $X_\lambda$ with
$\mu < \lambda$, can be associated to an element of $\Gamma(W)$. Conversely,
every $x \in \Gamma(W)$ can be associated to some pair ($X_\mu$, $X_\lambda$)
with $\mu < \lambda$. As it turns out, this correspondence gives each $x \in
\Gamma(W)$ a well-defined \textit{nilpotency index},
\[
    {\mathcal N}(x) := \operatorname{min} \left\{ k \in \mathbb{N} \cup \left\{
    \infty \right\} : \left( \operatorname{ad}_q X_\mu \right)^k (X_\lambda) =
    0 \right\}.
\]

For $(w, i, j) \in \Gamma(W)$, define $(w, i, j)^* := (w^{-1}, j, i)$. Observe
$x \mapsto x^*$ defines an involutive map on $\Gamma(W)$. We prove the
following curious result on a case-by-case basis, by Lie type $X_n$.

\begin{theorem}

    Let $X_n$ be a Lie type with $n > 1$, and let $x = (w, i, j) \in
    \Gamma(W(X_n))$. If $\langle \alpha_i, \alpha_i \rangle =\langle \alpha_j,
    \alpha_j \rangle$, then ${\mathcal N}(x) = {\mathcal N}(x^*)$.

\end{theorem}

\begin{corollary}

    Suppose $X_n$ is a simply laced Lie type.  Then ${\mathcal N}(x) =
    {\mathcal N}(x^*)$ for all $x \in \Gamma(W(X_n))$.

\end{corollary}

In Section (\ref{section, relations on Gamma(W)}), we define some partial
orders, $\xlongrightarrow{L}$ and $\xlongrightarrow{R}$, on $\Gamma(W)$, which
are analogs to the weak left (right) Bruhat order on $W$, respectively.  Let
$\xlongrightarrow{}$ be the partial order generated by $\xlongrightarrow{L}$
and $\xlongrightarrow{R}$, and let $\xleftrightarrow{\hspace{4mm}}$ be the
equivalence relation generated by $\xlongrightarrow{}$. The most important
property with each of these relations is preservation of nilpotency indices,
\begin{center}
    ($x \xlongrightarrow{L} y$ or $x \xlongrightarrow{R} y$ or $x
    \xlongrightarrow{} y$ or $x \xleftrightarrow{\hspace{4mm}} y$) implies
    ${\mathcal N}(x) = {\mathcal N}(y)$,
\end{center}
for $x, y \in \Gamma(W)$.  An element $x\in \Gamma(W)$ is a \textit{minimal
element} if $x \xlongrightarrow{} y$ implies $x = y$.

Before stating our main result regarding minimal elements, we need to recall a
definition: $w \in W$ is \textit{bigrassmannian} if there is only one simple
reflection $s_i$ such that $\ell(s_iw) < \ell(w)$ and only one simple
reflection $s_j$ such that $\ell(ws_j) < \ell(w)$.  Bigrassmannian elements
have been studied in a variety of contexts, including Bruhat orders of Coxeter
groups, essential sets, Schubert calculus, and representation theory \cite{EH,
EL, GK, KMM, KMM2, Ko1, Ko2, Las Sch, R, RWY}.  In a forthcoming paper, we will
investigate connections between essential sets, defined in \cite{F}, and
nilpotency indices in quantum Schubert cell algebras.

\begin{theorem}

    \label{minimal elements, thm}

    The element $(w, i, j) \in \Gamma(W)$ is a minimal element if and only if
    either (1) $w$ is bigrassmannian with $\ell(w) > 1$, or (2) $w$ is the
    longest element in a subgroup of $W$ generated by two simple reflections.

\end{theorem}

The proof of Theorem (\ref{minimal elements, thm}) is the culmination of the
theorems and lemmas in Section (\ref{minimal elements}), which provide, in
effect, an algorithm to construct a chain
\[
    x \xlongrightarrow{} \cdots \xlongrightarrow{} y,
\]
starting at any arbitrary $x \in \Gamma(W)$ and ending at a minimal element
$y$.

In light of this, we turn our attention to bigrassmannian elements $w$. For our
purposes, it suffices to assume $w$ has full support because the support of a
bigrassmannian $w$ can be identified with a connected subgraph of the
underlying Dynkin diagram. Viewing $w$ as an element of the Weyl group
associated to this subgraph, it has full support.  Section (\ref{section,
bigrassmannian elements}) contains an explicit description of the
bigrassmannian elements of full support in non-exceptional Lie types, as well
as the number of them in the exceptional types.

In Section (\ref{section, orthogonality condition}), we introduce an
orthogonality condition on bigrassmannian elements.  For a bigrassmannian $w
\in W$ and a root $\beta$, let $S_{w, \beta}$ be the set of simple roots in
$\Pi \cap w\Pi$ orthogonal to $\beta$,
\[
    S_{w, \beta} := \left\{ \alpha \in \Pi \cap w\Pi : \langle \alpha, \beta
    \rangle = 0 \right\}.
\]
Suppose $s_i$ and $s_j$ are the unique simple reflections such that $\ell(s_iw)
< \ell(w)$ and $\ell(ws_j) < \ell(w)$, respectively. We say $w$ satisfies the
\textit{orthogonality condition} if
\[
    S_{w, \alpha_i} \subseteq S_{w, w(\alpha_j)} \text{ and } S_{w^{-1},
    \alpha_j} \subseteq S_{w^{-1}, w^{-1}(\alpha_i)}.
\]
The set of such bigrassmannian elements will be denoted by $BiGr_\perp(X_n)$.

\begin{theorem}

    Suppose $(w, i, j) \in \Gamma(W(X_n))$ with $w$ bigrassmannian.  Then there
    exists $(v, r, s) \in \Gamma(W(X_n))$ such that $v \in BiGr_\perp(X_n)$ and
    $(w, i, j) \xleftrightarrow{\hspace{4mm}} (v, r, s)$.

\end{theorem}

The proof of this statement describes how to obtain such an element $(v, r, s)$
from a given $(w, i, j)$. Observe that if a bigrassmannian element $w$ doesn't
satisfy the orthogonality condition, this means, by definition, there is a
simple root
\[
    \alpha \in \left(S_{w, \alpha_i} \backslash S_{w, w(\alpha_j)}\right) \cup
    \left(S_{w^{-1}, \alpha_j} \backslash S_{w^{-1}, w^{-1}(\alpha_i)}\right).
\]
As it turns out, this simple root $\alpha$ plays an important role in showing
$(w, i, j)$ is equivalent (under $\xleftrightarrow{\hspace{4mm}}$) to an
element $(w^\prime, i^\prime, j^\prime)$ with $\ell(w^\prime) < \ell(w)$
(see Theorem \ref{theorem, second stage reduction}).

In Sections (\ref{section, nil-index, small rank}) and (\ref{section,
nil-index, general}), we compute the nilpotency index ${\mathcal N}(x)$
associated to each $x \in \Gamma(W(X_n))$ of the form $(v, r, s)$, where $v \in
BiGr_\perp(X_n)$ and has full support, for Lie types $A_n$, $B_n$, $C_n$, and
$D_n$. We do the same in Section (\ref{section, nilpotency indices,
exceptional}), but for the exceptional types. The calculations involve applying
commutation relations, and as such, are straightforward.  However, for some of
the exceptional types, particularly $E_7$ and $E_8$, these commutation
relations can be quite lengthy. For this reason, we use a computer to help
automate some of these calculations.  The results are summarized in data
tables: Tables (\ref{table, small rank ABCD}) and (\ref{table, general cases
ABCD}) for types $A_n$, $B_n$, $C_n$, and $D_n$, and Tables (\ref{table, G2}) -
(\ref{table, E8}) for the exceptional types.

\section{Quantum Schubert cells}

\label{section, QSC}

\subsection{Preliminaries}

Let $\mathfrak{g}$ be a finite dimensional complex simple Lie algebra of rank
$n$. Define the index set $\mathbf{I} := [1,n]$, and let $\Pi = \{\alpha_i\}_{i
\in \mathbf{I}}$ be a set of simple roots of $\mathfrak{g}$ with respect to a
fixed Cartan subalgebra $\mathfrak{h} \subset \mathfrak{g}$ such that the
labelling of the simple roots agrees with the labelling in \cite[Section
11.4]{Humphreys}.  The root system of $\mathfrak{g}$ will be denoted by
$\Delta$, and the sets of positive and negative roots will be denoted by
$\Delta_+$ and $\Delta_-$, respectively. The triangular decomposition of
$\mathfrak{g}$ will be denoted by
\[
    \mathfrak{g} = \mathfrak{n}^- \oplus \mathfrak{h} \oplus \mathfrak{n}^+,
\]
where
\[
    \mathfrak{n}^{\pm} := \bigoplus_{\alpha \in \Delta_\pm}
    \mathfrak{g}_\alpha, \hspace{10mm} \mathfrak{g}_\alpha := \left\{ x \in
    \mathfrak{g} \mid [h,x] = \alpha(h)x \text{ for all }h \in \mathfrak{h}
    \right\}.
\]

In many cases, we may wish to refer to the Lie type of $\mathfrak{g}$, and in
those situations it will be denoted by
\[
    X_n, \hspace{5mm} X \in \left\{A,B,C,D,E,F,G \right\}.
\]
Let $Q = \oplus_{i \in \mathbf{I}} \mathbb{Z}\alpha_i$ be the root lattice of
$\mathfrak{g}$, and let $\langle \,\, ,\, \rangle : Q \times Q \to \mathbb{Z}$
be a symmetric nondegenerate ad-invariant $\mathbb{Z}$-bilinear form,
normalized so that $\langle \alpha, \alpha \rangle = 2$ for short roots
$\alpha$. The length of a root $\beta$ will be denoted
\[
    \| \beta \| := \sqrt{ \langle \beta, \beta \rangle}.
\]
For each $i \in \mathbf{I}$, define the simple coroot $\alpha_i^\vee := 2
\alpha_i / \langle \alpha_i, \alpha_i \rangle$, and let
\[
    c_{ij} := \langle \alpha_i^\vee, \alpha_j \rangle = \frac{2 \langle
    \alpha_i , \alpha_j \rangle}{\langle \alpha_i, \alpha_i \rangle} \in \mathbb{Z},
    \hspace{5mm} (i,j \in \mathbf{I}).
\]
The matrix $(c_{ij})_{i,j \in \mathbf{I}}$ is the associated Cartan matrix of
$\mathfrak{g}$.

\subsection{The algebra \texorpdfstring{${\mathcal U}_q(\mathfrak{g})$}{Uq(g)}}

Let $\mathbb{K}$ be a field. For a nonzero scalar $v \in \mathbb{K}$, define
\[
        \widehat{v} := v - v^{-1},
\]
and for $k \in \mathbb{N}$, define
\[
    [k]_v := \frac{v^k - v^{-k}}{v - v^{-1}}, \hspace{10mm}
    [k]_v! := [k]_v [k - 1]_v \cdots [1]_v.
\]
Now fix a nonzero scalar $q \in \mathbb{K}$ that is not a root of unity, and
for $i \in \mathbf{I}$, let $q_i := q^{\langle \alpha_i, \alpha_i \rangle/ 2}$.
The quantized universal enveloping algebra ${\mathcal U}_q(\mathfrak{g})$ is an
associative $\mathbb{K}$-algebra with \textit{Chevalley generators}
\[
    K_\mu, E_i, F_i, \hspace{5mm} (\mu \in Q, i \in \mathbf{I}).
\]

Before stating the defining relations of ${\mathcal U}_q(\mathfrak{g})$, we
need to introduce some notation. There is a $Q$-gradation on the algebra
${\mathcal U}_q(\mathfrak{g})$,
\[
    {\mathcal U}_q(\mathfrak{g}) = \bigoplus_{\lambda \in Q} {\mathcal
    U}_q(\mathfrak{g})_\lambda, \hspace{5mm} {\mathcal U}_q(\mathfrak{g})_\mu
    := \left\{ u\in {\mathcal U}_q(\mathfrak{g}) : K_\lambda u = q^{\langle
    \lambda, \mu \rangle}uK_\lambda \text{ for all } \lambda \in Q \right\}.
\]
With respect to this grading, the Chevalley generators are homogeneous
elements. In particular,
\[
    E_i \in {\mathcal U}_q(\mathfrak{g})_{\alpha_i}, \hspace{5mm} F_i \in
    {\mathcal U}_q(\mathfrak{g})_{-\alpha_i}, \hspace{5mm} K_\mu \in {\mathcal
    U}_q(\mathfrak{g})_0, \hspace{10mm} (i\in \mathbf{I}, \mu \in Q).
\]
Next, we introduce the $q$-commutator.
\begin{definition}

    \label{def, q-commutators}

    For all $\mu, \eta \in Q$ and homogeneous elements $x \in {\mathcal
    U}_q(\mathfrak{g})_{\mu}$ and $y \in {\mathcal U}_q(\mathfrak{g})_{\eta}$,
    define the \textit{$q$-commutator}
    \[
        [x, y] := x y - q^{\langle \mu, \eta \rangle} y x.
    \]

\end{definition}

\noindent We will frequently apply the two identities of Proposition
(\ref{Jacobi}) below. Each of them is straightforward to verify from
Definition (\ref{def, q-commutators}).

\begin{proposition}

    \label{Jacobi}

    For all $x \in \left({\mathcal U}_q(\mathfrak{g})\right)_{\lambda}$, $y \in
    \left({\mathcal U}_q(\mathfrak{g})\right)_\eta$, $z \in \left({\mathcal
    U}_q(\mathfrak{g})\right)_\mu$,

    \begin{equation}
        \label{q-Jacobi identity}
        [x, [y, z]] = [[x, y], z] - q^{-\langle \eta, \mu\rangle}[[x, z], y] -
        \left(q^{\langle \eta, \mu\rangle} - q^{-\langle \eta,
        \mu\rangle}\right) [x, z]y,
    \end{equation}
    and
    \begin{equation}
        \label{q-Leibniz identity}
            [x, yz] = [x, y]z + q^{\langle \lambda, \eta\rangle}y[x,z].
    \end{equation}

\end{proposition}

The identities (\ref{q-Jacobi identity}) and (\ref{q-Leibniz identity}) will be
referred to as the \textit{$q$-Jacobi} and \textit{$q$-Leibniz} identities,
respectively.

\begin{definition}

    For all $\mu \in Q$ and $x \in {\mathcal U}_q(\mathfrak{g})_\mu$, let
    $\operatorname{ad}_qx:{\mathcal U}_q(\mathfrak{g}) \to {\mathcal
    U}_q(\mathfrak{g})$ be the (unique) linear map such that
    \[
        \left(\operatorname{ad}_q x\right)(y) = [x, y],
    \]
    for every homogeneous element $y \in {\mathcal U}_q(\mathfrak{g})$.

\end{definition}

The defining relations of ${\mathcal U}_q(\mathfrak{g})$ are
\begin{align}
    &K_0 = 1,
    \\
    &K_\mu K_\lambda = K_{\lambda + \mu} \hspace{5mm} (\lambda, \mu \in Q),
    \\
    &K_\mu E_i = q^{\langle \mu,  \alpha_i \rangle} E_i K_\mu \hspace{5mm} (\mu
    \in Q, i \in \mathbf{I}),
    \\
    &K_\mu F_i = q^{-\langle \mu, \alpha_i \rangle} F_i K_\mu \hspace{5mm} (\mu
    \in Q, i \in \mathbf{I}),
    \\
    &E_iF_j - F_jE_i = \delta_{ij} \frac{K_{\alpha_i} - K_{-\alpha_i}}{q_i
        - q_i^{-1}} \hspace{5mm} (i, j \in \mathbf{I}),
    \\
    &\label{q-Serre 1}(\operatorname{ad}_qE_i)^{1 - c_{ij}} (E_j) = 0
    \hspace{5mm} (i, j \in \mathbf{I} \text{ with } i \neq j),
    \\
    &\label{q-Serre 2}(\operatorname{ad}_qF_i)^{1 - c_{ij}} (F_j) = 0
    \hspace{5mm} (i, j \in \mathbf{I} \text{ with } i \neq j).
\end{align}
Relations (\ref{q-Serre 1}) and (\ref{q-Serre 2}) are called the
\textit{$q$-Serre relations}.

There is an algebra automorphism $\omega: {\mathcal U}_q(\mathfrak{g}) \to
{\mathcal U}_q(\mathfrak{g})$ such that
\begin{equation}
    \label{def, omega}
    \omega(E_i) = F_i, \hspace{5mm} \omega(F_i) = E_i, \hspace{5mm}
    \omega(K_\mu) = K_{-\mu}, \hspace{5mm} (i \in \mathbf{I}, \mu \in Q).
\end{equation}
The algebra ${\mathcal U}_q(\mathfrak{g})$ has a
triangular decomposition,
\[
    {\mathcal U}_q(\mathfrak{g}) \cong {\mathcal U}_q(\mathfrak{n}^-) \otimes
    {\mathcal U}_q(\mathfrak{h}) \otimes {\mathcal U}_q(\mathfrak{n}^+),
\]
where ${\mathcal U}_q(\mathfrak{n}^-)$, ${\mathcal U}_q(\mathfrak{h})$, and
${\mathcal U}_q(\mathfrak{n}^+)$ are the subalgebras of ${\mathcal
U}_q(\mathfrak{g})$ generated by the $F$'s, $K$'s, and $E$'s respectively.

\subsection{Lusztig symmetries of \texorpdfstring{${\mathcal
U}_q(\mathfrak{g})$}{Uq(g)}}

We will denote the simple reflections in the Weyl group $W$ of $\mathfrak{g}$
by
\[
    s_i, \hspace{5mm} (i \in \mathbf{I}).
\]
The corresponding generators of the braid group ${\mathcal B}_{\mathfrak{g}}$
of $\mathfrak{g}$ will be denoted by
\[
    T_i, \hspace{5mm} (i \in \mathbf{I}).
\]
In order to indicate the underlying Lie type, we may in some circumstances
denote the Weyl group $W$ by
\[
    W(X_n), \hspace{5mm} X \in \left\{A, B, C, D, E, F, G \right\}.
\]
In \cite[Section 37.1.3]{L}, Lusztig defines an action of the braid group
${\mathcal B}_{\mathfrak{g}}$ via algebra automorphisms on ${\mathcal
U}_q(\mathfrak{g})$.  In fact, Lusztig defines the symmetries $T_{i,1}^\prime$,
$T_{i,-1}^{\prime}$, $T_{i,1}^{\prime\prime}$, and $T_{i,-1}^{\prime\prime}$.
By \cite[Proposition 37.1.2]{L}, these are automorphisms of ${\mathcal
U}_q(\mathfrak{g})$, while by \cite[Theorem 39.4.3]{L} they satisfy the braid
relations.  For short, we will adopt the abbreviation $T_i :=
T^{\prime\prime}_{i,1}$. Lusztig's symmetries are given by the formulas
\[
    \begin{split}
        &T_i(K_\mu) = K_{s_i(\mu)},
        \\
        &T_i(E_j) =
        \begin{cases}
            - F_i K_{\alpha_i}, &
            (i = j),
            \\
            \left(\operatorname{ad}_q E_i \right)^{(-c_{ij})} (E_j), &
            (i \neq j),
            \\
        \end{cases}
        \\
        &T_i(F_j) =
        \begin{cases}
            - K_{-\alpha_i} E_i, &
            (i = j),
            \\
            \left(-q_i\right)^{-c_{ij}} \left(\operatorname{ad}_q
            F_i\right)^{(-c_{ij})}(F_j), &
            (i \neq j),
        \end{cases}
    \end{split}
\]
\noindent where, for a nonnegative integer $k$,
\[
    \left(\operatorname{ad}_q E_i \right)^{(k)} := \frac{ \left(
            \operatorname{ad}_q E_i\right)^k}{[k]_{q_i}!},
            \hspace{10mm}
    \left(\operatorname{ad}_q F_i \right)^{(k)} := \frac{ \left(
            \operatorname{ad}_q F_i\right)^k}{[k]_{q_i}!}.
\]
If $w \in W$ has a reduced expression $w = s_{i_1} \cdots s_{i_N} \in W$, we
write
\[
    T_w = T_{i_1} T_{i_2} \cdots T_{i_N}.
\]
A key property of the braid symmetries is given in the following theorem
(see e.g. \cite[Proposition 8.20]{Jantzen}).

\begin{theorem}

    \label{Jantzen, 8.20}

    If $w\in W$ and $w(\alpha_i) = \alpha_j$, then $T_w(E_i) = E_j$.

\end{theorem}

We claim that
\begin{equation}
    \label{orthogonality of Tw}
    T_w \left( [ x, y ] \right) = [T_w(x), T_w(y)], \hspace{5mm} \left(\mu, \nu
    \in Q, x \in \left({\mathcal U}_q(\mathfrak{g}) \right)_\mu, y \in \left(
    {\mathcal U}_q(\mathfrak{g}) \right)_\nu\right).
\end{equation}
Since $T_w$ is an algebra automorphism of ${\mathcal U}_q(\mathfrak{g})$,
\[
    T_w \left( [ x, y ] \right) = T_w (xy - q^{\langle \mu, \nu \rangle} yx) =
    T_w(x) T_w(y) - q^{\langle \mu, \nu \rangle} T_w(y) T_w(x),
\]
and (\ref{orthogonality of Tw}) follows from the fact that
\begin{equation}
    \label{Tw, homogeneous}
    T_w(x) \in {\mathcal U}_q(\mathfrak{g})_{w(\mu)}, \hspace{5mm} (\mu \in Q
    \text{ and } x \in {\mathcal U}_q(\mathfrak{g})_\mu),
\end{equation}
(and $T_w(y) \in {\mathcal U}_q(\mathfrak{g})_{w(\nu)}$) (see e.g.
\cite[Section 8.18]{Jantzen}), and $\langle w(\mu), w(\nu) \rangle = \langle
\mu, \nu \rangle$.

\subsection{Quantum Schubert cell algebras \texorpdfstring{${\mathcal
U}_q^\pm[w]$}{Uq[w]}}

For a reduced expression,
\[
    w = s_{i_1} \cdot s_{i_2} \cdots s_{i_N} \in W,
\]
define the roots
\[
    \beta_1 = \alpha_{i_1}, \beta_2=s_{i_1}\alpha_{i_2},...,\beta_N =
    s_{i_1}\cdots s_{i_{N-1}}\alpha_{i_N},
\]
and root vectors
\begin{equation}
    \label{root vectors}
    X_{\beta_1} = E_{i_1}, X_{\beta_2} = T_{s_{i_1}} E_{i_2},..., X_{\beta_N} =
    T_{s_{i_1}} \cdots T_{s_{i_{N-1}}} E_{i_N}.
\end{equation}

We will denote the set of \textit{radical roots} (or \textit{roots of $w$}) by
\[
    \Delta_w := \left\{ \beta_1, \dots, \beta_N\right\}.
\]
These roots are precisely the positive roots that get sent to negative roots by
the action of $w^{-1}$. The subalgebra of ${\mathcal U}_q(\mathfrak{g})$
generated by the root vectors $X_{\beta_1},\dots, X_{\beta_N}$ is contained in
the positive part ${\mathcal U}_q(\mathfrak{n}^+)$ (see e.g. \cite[Proposition
8.20]{Jantzen}).  This subalgebra will be denoted by ${\mathcal U}_q^+[w]$,
\[
    {\mathcal U}_q^+[w] := \langle X_{\beta_1}, \dots X_{\beta_N}
    \rangle \subseteq {\mathcal U}_q (\mathfrak{n}^+).
\]
Analogously, the subalgebra of ${\mathcal U}_q(\mathfrak{g})$ generated by the
negative root vectors
\[
    X_{-\beta_1} = F_{i_1}, X_{-\beta_2} = T_{s_{i_1}} F_{i_2},...,
    X_{-\beta_N} = T_{s_{i_1}} \cdots T_{s_{i_{N-1}}} F_{i_N}.
\]
will be denoted ${\mathcal U}_q^-[w]$. It is a subalgebra of the negative part
${\mathcal U}_q(\mathfrak{n}^-)$.

As it turns out, ${\mathcal U}_q^-[w]$ is isomorphic to ${\mathcal U}_q^+[w]$.
In particular, let $\varpi_i$ ($i \in \mathbf{I}$) be the fundamental weights.
From the identity
\[
    T_w \left(\omega \left( u \right) \right) = \left(\prod_{i \in \mathbf{I}}
    \left( -q \right)^{\langle w \mu - \mu, \varpi_i \rangle} \right) \omega
    \left( T_w \left(u \right) \right), \hspace{5mm} (\mu \in Q, u \in
    \left({\mathcal U}_q(\mathfrak{g} \right)_{\mu},
\]
(see e.g.  \cite[Eqn. 8.18.5]{Jantzen}), it follows that restricting $\omega$
to ${\mathcal U}_q^+[w]$ induces an algebra isomorphism $\omega: {\mathcal
U}_q^+[w] \xlongrightarrow{\cong} {\mathcal U}_q^-[w]$.

De Concini, Kac, and Procesi \cite[Proposition 2.2]{DKP} proved that the
algebras ${\mathcal U}_q^\pm[w]$ do not depend on the reduced expression for
$w$. Moreover, Berenstein and Greenstein \cite{AB} proved that
\[
    {\mathcal U}_q^{\pm}[w] = {\mathcal U}_q^\pm(\mathfrak{g}) \cap
    T_w\left({\mathcal U}_q^\mp(\mathfrak{g})\right)
\]
for any $\mathfrak{g}$ of finite type, and they also conjectured
\cite[Conjecture 5.3]{AB} this result holds for any symmetrizable Kac-Moody Lie
algebra $\mathfrak{g}$.  Their conjecture was later proven independently by
Kimura \cite[Theorem 1.1 (1)]{Kimura} and Tanisaki \cite[Proposition
2.10]{Tanisaki}.

Since ${\mathcal U}_q^+[w]$ and ${\mathcal U}_q^-[w]$ are isomorphic algebras,
we will focus only on ${\mathcal U}_q^+[w]$. The isomorphism $\omega: {\mathcal
U}_q^+[w] \xlongrightarrow{\cong} {\mathcal U}_q^-[w]$ be can applied to
translate every result to ${\mathcal U}_q^-[w]$.

Every quantum Schubert cell ${\mathcal U}_q^+[w]$ has a PBW basis
\[
    X_{\beta_1}^{m_1}\cdots X_{\beta_N}^{m_N}, \hspace{.4cm} m_1,...,m_N \in
    \mathbb{Z}_{\geq 0},
\]
of ordered monomials \cite[Proposition 40.2.1]{L}. Moreover, they have
presentations as quantum nilpotent algebras (also called
\textit{Cauchon-Goodearl-Letzter extensions}),
\[
    {\mathcal U}_q^+[w] = \mathbb{K}[X_{\beta_1}][X_{\beta_2}; \sigma_2,
    \delta_2] \cdots [X_{\beta_N}; \sigma_N, \delta_N],
\]
(see \cite[Lemma 2.1]{GeigerYakimov}).

For $1 < i < j \leq N$, define the
\textit{interval subalgebra}
\begin{equation}
    \label{definition, interval subalgebra}
    {\mathbf U}_{[i, j]} := \langle X_{\beta_i},X_{\beta_{i + 1}},\dots,
    X_{\beta_j} \rangle \subseteq {\mathcal U}_q^+[w]
\end{equation}
as the subalgebra generated by $X_{\beta_i}, X_{\beta_{i + 1}},\dots,
X_{\beta_j}$. Ordered monomials
\[
    X_{\beta_i}^{m_i} \cdots X_{\beta_j}^{m_j}, \hspace{5mm} m_i,\dots, m_j \in
    \mathbb{Z}_{\geq 0}
\]
form a basis of $\mathbf{U}_{[i, j]}$.  The Levendorskii-Soibelmann
straightening rule \cite[Prop. 5.5.2]{LS} tells us that for all $1 \leq i < j
\leq N$,
\begin{equation}
    \label{L-S straightening}
    [X_{\beta_i}, X_{\beta_j}] \in \mathbf{U}_{[i + 1, j - 1]} \cap {\mathcal
    U}_q(\mathfrak{g})_{\beta_i + \beta_j}.
\end{equation}
As a simple consequence of the straightening rule, we have the following
result, which we will use extensively.

\begin{proposition}

    \label{LS corollary}

    Let ${\mathcal U}_q^+[w]$ be a quantum Schubert cell algebra with Lusztig
    root vectors $X_{\beta_1},\dots, X_{\beta_N}$.  Suppose $1 \leq i < j \leq
    N$. If there fails to exist a nonnegative integral combination of roots in
    $\left\{\beta_{i+1},\dots, \beta_{j-1}\right\}$ that sum to $\beta_i +
    \beta_j$, then $[X_{\beta_i}, X_{\beta_j}] = 0$.

\end{proposition}

Applying the straightening formula (\ref{L-S straightening}) in conjunction
with the $q$-Leibniz identity gives us the following more general result.

\begin{proposition}

    \label{bound on nil-index}

    Let ${\mathcal U}_q^+[w]$ be a quantum Schubert cell algebra with Lusztig
    root vectors $X_{\beta_1},\dots, X_{\beta_N}$.  Suppose $1 \leq i < j \leq
    N$.

    \begin{enumerate}

        \item Then for all $k\in \mathbb{N}$,
            \[
                \left( \operatorname{ad}_q(X_{\beta_i}) \right)^k
                \left(X_{\beta_j}\right) \in \mathbf{U}_{[i + 1, j - 1]} \cap
                \left({\mathcal U}_q^+[w]\right)_{k\beta_i + \beta_j}.
            \]

        \item
            \label{bound on nil-index, 2}
            If, for some fixed $k \in \mathbb{N}$, there fails to exist a
            nonnegative integral combination of roots in $\left\{ \beta_{i +
            1},\dots, \beta_{j - 1} \right\}$ that sum to $k\beta_i + \beta_j$,
            then
            \[
                \left( \operatorname{ad}_q(X_{\beta_i}) \right)^k
                \left(X_{\beta_j}\right) = 0.
            \]

    \end{enumerate}

\end{proposition}

We note that part (\ref{bound on nil-index, 2}) of Proposition (\ref{bound on
nil-index}) can be used to find an upper bound on the nilpotency index of
$\operatorname{ad}_q(X_{\beta_i})$ acting on $X_{\beta_j}$.

\begin{example}

    \normalfont

    Consider the reduced expression
    \[
        w = s_4 s_5 s_6 s_7 s_8 s_2 s_3 s_4 s_5 s_6 s_7 s_1 s_3 s_4 s_5 s_6 s_2
        s_4 s_5 s_3 s_4 s_1 s_2 s_3 s_4 s_5 s_6 s_7 s_8 \in W(E_8)
    \]
    in the type $E_8$ Weyl group. Here, $\ell(w) = 29$. Let $X_{\beta_1},
    \cdots, X_{\beta_{29}}$ be the Lusztig root vectors in ${\mathcal
    U}_q^+[w]$, corresponding to this reduced expression of $w$. From part
    (\ref{bound on nil-index, 2}) of Proposition (\ref{bound on nil-index})
    above, we can conclude $\left(\operatorname{ad}_q (X_{\beta_1}) \right)^6
    (X_{\beta_{29}}) = 0$, and thus $\left(\operatorname{ad}_q (X_{\beta_1})
    \right)^k (X_{\beta_{29}}) = 0$ for all $k \geq 6$.

    Also, there exists a nonnegative integral combination of roots in
    $\left\{ \beta_2, \cdots, \beta_{28} \right\}$ that sum to $5 \beta_1 +
    \beta_{29}$, namely
    \[
        \beta_2 + \beta_3 + \beta_4 + \beta_5 + 3 \beta_6 + 2 \beta_7 + 2
        \beta_{12},
    \]
    and hence no conclusion can be made, based solely on part (\ref{bound on
    nil-index, 2}) of Proposition (\ref{bound on nil-index}), as to whether or
    not $\left(\operatorname{ad}_q (X_{\beta_1}) \right)^5 (X_{\beta_{29}})$
    equals $0$.

    As it turns out, we will show in Section (\ref{section, E8}) that
    $\left(\operatorname{ad}_q (X_{\beta_1}) \right)^2 (X_{\beta_{29}}) = 0$
    and $\left(\operatorname{ad}_q (X_{\beta_1}) \right) (X_{\beta_{29}}) \neq
    0$. In other words, the nilpotency index of
    $\operatorname{ad}_q(X_{\beta_1})$ acting on $X_{\beta_{29}}$ is $2$. In
    fact, $w = \eta_{759}^{-1}$, where $\eta_1,\eta_2,\cdots \in W(E_8)$ are
    certain Weyl group elements defined in Appendix (\ref{appendix, elements
    E_n}).

\end{example}

For our purposes, the following well-known result (Proposition (\ref{as nested
E})) will be very useful for determining how to explicitly write the Lusztig
root vectors $X_\beta$ as elements of ${\mathcal U}_q(\mathfrak{n}^+)$. Our proof
provides a recursive algorithm for writing $X_\beta$ in terms of the Chevalley
generators $E_i$ ($i \in \mathbf{I}$).

For $i_1,\cdots, i_m \in \mathbf{I}$, define the \textit{nested $q$-commutator}
\begin{equation}
    \label{definition, nested q-commutator}
    \mathbf{E}_{i_1,\ldots,i_m} := [[[ \cdots [E_{i_1}, E_{i_2}],
    E_{i_3}],\cdots,], E_{i_m}]  \in {\mathcal U}_q(\mathfrak{n}^+).
\end{equation}

\begin{proposition}

    \label{as nested E}

    Let ${\mathcal E}$ be the smallest subset of ${\mathcal
    U}_q(\mathfrak{n}^+)$ such that

    \begin{enumerate}

        \item $E_i \in {\mathcal E}$ for all $i \in \mathbf{I}$,

        \item if $x, y \in {\mathcal E}$, then $[x, y] \in {\mathcal E}$,

        \item if $x \in {\mathcal E}$ and the Lie type is $B_n$, $C_n$, or
            $F_4$, then $\frac{1}{[2]_q} x \in {\mathcal E}$, and

        \item if $x \in {\mathcal E}$ and the Lie type is $G_2$, then
            $\frac{1}{[2]_q} x \in {\mathcal E}$ and $\frac{1}{[3]_q!} x \in
            {\mathcal E}$.

    \end{enumerate}

    Then $T_w(E_j) \in {\mathcal E}$ for all $w \in W$ and $j \in \mathbf{I}$
    such that $w (\alpha_j) > 0$.

\end{proposition}

\begin{proof}

    We prove this by inducting on the length of $w$. First, if $\ell(w) = 0$,
    then $w$ is the identity element and $T_w(E_j) = E_j$.

    Now assume $\ell(w) > 0$.  Recall, an index $k\in \mathbf{I}$ is called a
    right descent of a Weyl group element $x$ if $\ell(xs_k) < \ell(x)$. The
    set of right descents of $x$ will be denoted ${\mathcal D}_R(x)$.  Since
    $j$ is not a right descent of $w$, this means $ws_j$ is not the identity
    element, and thus $ws_j$ has at least one right descent.

    Suppose first that $w s_j$ has more than one right descent. This means
    there exists some $p \in \mathbf{I}$ with $p \neq j$ such that $s_j \leq_L
    u \leq_L w s_j$, where $u$ is the longest element of the subgroup of $W$
    generated by $s_j$ and $s_p$, and $\leq_L$ is the weak left Bruhat order.
    This means there exists a reduced expression for $u$ that ends in $s_j$,
    and $\ell(u) + \ell(ws_ju^{-1}) = \ell(ws_j)$.  For the purposes here, we
    may consider any such index $p$ such that this condition is met.  However,
    to construct a systematic algorithm, we choose $p$ to be the smallest such
    index.  With this, $u s_j (\alpha_j)$ is a simple root, say $\alpha_k$,
    $\ell(w s_j u) < \ell(w)$, and $T_w(E_j) = T_{w s_j u} (E_k)$.  Observe
    also that $k \not\in {\mathcal D}_R(w s_j u)$.

    Now suppose $w s_j$ has only one right descent, namely $j$, and $\ell(w) >
    0$. Therefore $w$ has at least one right descent. Again, for the purposes
    here, we may consider any such right descent, but to be systematic, we
    focus on the smallest right descent of $w$, say $k$. There are three cases
    to consider: $\langle \alpha_k^\vee, \alpha_j \rangle = -1$, $\langle
    \alpha_k^\vee, \alpha_j \rangle = -2$, or $\langle \alpha_k^\vee, \alpha_j
    \rangle = -3$.

    If $\langle \alpha_k^\vee, \alpha_j \rangle = -1$, then, provided $j
    \not\in {\mathcal D}_R(ws_k)$, we write
    \[
        T_w(E_j) = T_{w s_k} T_{s_k} (E_j) = T_{w s_k} \left([E_k, E_j]\right)
        = [T_{w s_k} (E_k), T_{w s_k} (E_j)].
    \]
    In this situation, neither $j$ nor $k$ are right descents of $ws_k$, and
    $\ell(ws_k) < \ell(w)$.  However, if $j \in {\mathcal D}_R(ws_k)$, then the
    order of $s_js_k$ must be at least $4$. We write
    \[
        T_w(E_j) = T_{ws_ks_j} T_{s_js_k}(E_j) = T_{ws_ks_j} (\mathbf{E}_{j,k})
        = [T_{ws_ks_j}(E_j), T_{ws_ks_j}(E_k)]
    \]
    whenever $s_js_k$ has order $4$. In this case, $j,k \not\in {\mathcal
    D}_R(ws_ks_j)$, and $\ell(ws_js_k) < \ell(w)$. However, if $s_js_k$ has
    order $6$, then the underlying Lie type is $G_2$, $\alpha_k$ is the long
    simple root, and $\alpha_j$ is the short simple root. Here, there are $4$
    possibilities for $w$ (where $\ell(w) > 0$, $k \in {\mathcal D}_R(w)$, and
    $ws_j$ has only one right descent). In each of these 4 cases, we can, up to
    a scalar multiple, write $T_w(E_j)$ as nested $q$-commutators of Chevalley
    generators. In particular,
    \begin{align*}
        &T_{s_k}(E_j) = \mathbf{E}_{k,j}, \hspace{5mm} T_{s_j s_k}
        (E_j) = \frac{1}{[2]_q} [E_j, \mathbf{E}_{j,k}],
        \\
        &T_{s_k s_j s_k} (E_j) = \frac{1}{[2]_q} \mathbf{E}_{k,j,j},
        \hspace{5mm}T_{s_j s_k s_j s_k}(E_j) = \mathbf{E}_{j,k}.
    \end{align*}

    On the other hand, if $\langle \alpha_k^\vee, \alpha_j \rangle = -2$, then,
    provided $j \not\in {\mathcal D}_R(w s_k)$, we have
    \[
        T_w(E_j) = T_{ws_k} T_{s_k}(E_j) = \frac{1}{[2]_q} [T_{w s_k} (E_k),
        [T_{w s_k} (E_k), T_{w s_k} (E_j)]].
    \]
    Recall, $k$ is a right descent of $w$. Thus $k \not\in {\mathcal
    D}_R(ws_k)$ and $\ell(ws_k) < \ell(w)$. To handle the case when $j \in
    {\mathcal D}_R(ws_k)$, we write
    \[
        \begin{split}
            T_w(E_j) &= T_{ws_ks_j} T_{s_js_k}(E_j) = \frac{1}{[2]_q}
            T_{ws_ks_j} (\mathbf{E}_{j,k,k})
            \\
            &= \frac{1}{[2]_q} [[ T_{ws_ks_j}(E_j), T_{ws_ks_j} (E_k)],
            T_{ws_ks_j}(E_k)].
        \end{split}
    \]
    In this case, $\ell(ws_ks_j) < \ell(w)$ and $j,k \not\in {\mathcal
    D}_R(ws_ks_j)$.

    Finally, if $\langle \alpha_k^\vee, \alpha_j \rangle = -3$, then the
    underlying Lie type is $G_2$, $\alpha_k$ is the short simple root, and
    $\alpha_j$ is the long simple root. As in the previous case involving type
    $G_2$, there are 4 possibilities for $w$ (where $\ell(w) > 0$, $j \not\in
    {\mathcal D}_R(w)$, and $ws_j$ has only one right descent), and in each of
    these 4 cases, we can write $T_w(E_j)$ as  nested $q$-commutators of the
    Chevalley generators, up to a scalar multiple. In particular,
    \begin{align*}
        &T_{s_k}(E_j) = \frac{1}{[3]_q!} [E_k, [E_k, \mathbf{E}_{k,j}]],
        \hspace{5mm} T_{s_j s_k} (E_j) = \frac{1}{[3]_q!} [\mathbf{E}_{j,k},
        \mathbf{E}_{j,k,k}],
        \\
        &T_{s_k s_j s_k} (E_j) = \frac{1}{[3]_q!} [[E_k, \mathbf{E}_{k,j}],
        \mathbf{E}_{k,j}],
        \hspace{5mm} T_{s_j s_k s_j s_k}(E_j) = \frac{1}{[3]_q!}
        \mathbf{E}_{j,k,k,k}.
    \end{align*}

\end{proof}

\subsection{Obtaining an explicit presentation of \texorpdfstring{${\mathcal
U}_q^+[w]$}{Uq^+[w]}}

\label{section, L(i,j,r,s) and R(i,j,r,s)}

We describe two algorithms, \texttt{L(i,j,r,s)} and \texttt{R(i,j,r,s)}, that
can be used, in many cases, to find explicit presentations of quantum Schubert
cell algebras.  While these algorithms apply to any underlying Lie type $X_n$,
we will use them later (in Sections (\ref{section, F4}) and (\ref{section,
E678}) for those cases when the Lie types are $F_4$, $E_6$, $E_7$, and $E_8$ to
find explicit presentations of ${\mathcal U}_q^+[w_0]$, where $w_0$ is the
longest element of the Weyl group.  Once an explicit presentation is found, it
is a routine exercise in applying $q$-commutator straightening relations to
calculate nilpotency indices.

Consider a pair of Lusztig root vectors, $X_{\beta_i}$ and $X_{\beta_j}$ with
$i < j$, in a quantum Schubert cell algebra ${\mathcal U}_q^+[w]$. An equation
of the form
\[
    [X_{\beta_i}, X_{\beta_j}] = f(X_{\beta_1},\dots, X_{\beta_{\ell(w)}}),
\]
where the right hand side $f(X_{\beta_1},\dots, X_{\beta_{\ell(w)}})$ is a
$\mathbb{K}$-linear combination of ordered monomials in $X_{\beta_1},\dots,
X_{\beta_{\ell(w)}}$, will be called a \textit{$q$-commutator relation}.  We
will informally say the \textit{$q$-commutator relation involving
$[X_{\beta_i}, X_{\beta_j}]$ has been determined (or established, obtained,
found, et cetera.)} if it is known how to write $[X_{\beta_i}, X_{\beta_j}]$
explicitly as a linear combination of ordered monomials.  Recall the set of
ordered monomials is a basis of ${\mathcal U}_q^+[w]$.  Thus, determining the
$q$-commutator relations among all pairs of Lusztig roots vectors gives an
explicit presentation of ${\mathcal U}_q^+[w]$.

The algorithms \texttt{L(i,j,r,s)} and \texttt{R(i,j,r,s)} can each be used to
determine the $q$-commutator relation involving $[X_{\beta_i}, X_{\beta_j}]$.
Applying the first step in the algorithm \texttt{L(i,j,r,s)} or
\texttt{R(i,j,r,s)} requires that there exists a relation of the form
$X_{\beta_i} = c[X_{\beta_r}, X_{\beta_s}]$ or $X_{\beta_j} = c[X_{\beta_r},
X_{\beta_s}]$, respectively, for some nonzero $c \in \mathbb{K}$ and $r < s$.
This means certain $q$-commutator relations must first be established before
determining others.  However, once the $q$-commutator relation involving
$[X_{\beta_i}, X_{\beta_j}]$ has been obtained, we may use it to find others.

The algorithm \texttt{R(i,j,r,s)} begins by replacing $X_{\beta_j}$ with a
scalar multiple of a $q$-commutator $[X_{\beta_r}, X_{\beta_s}]$ (if such a
relation exists and has already been discovered), and then proceeds by applying
the $q$-Jacobi identity.  For instance, if $X_{\beta_j} = c[X_{\beta_r},
X_{\beta_s}]$ for some nonzero scalar $c\in \mathbb{K}$ and $r$ and $s$ with $r
< s$, then
\[
    [X_{\beta_i}, X_{\beta_j}] = c \left([[X_{\beta_i}, X_{\beta_r}],
    X_{\beta_s}] - q^{-\langle \beta_r, \beta_s \rangle} [[X_{\beta_i},
    X_{\beta_s}], X_{\beta_r}] - \widehat{q}_{r, s} [X_{\beta_i}, X_{\beta_s}]
    X_{\beta_r} \right),
\]
where $\widehat{q}_{r,s} := q^{\langle \beta_r, \beta_s \rangle} - q^{- \langle
\beta_r, \beta_s \rangle}$.

Next, if it is already known how to write the $q$-commutators $[X_{\beta_i},
X_{\beta_r}]$ and $[X_{\beta_i}, X_{\beta_s}]$ as linear combinations of
ordered monomials, say $[X_{\beta_i}, X_{\beta_r}] = \sum b_{\bf k} x^{\bf k}$
and $[X_{\beta_i}, X_{\beta_s}] = \sum c_{\bf k} x^{\bf k}$ ($b_{\bf k}, c_{\bf
k} \in \mathbb{K}$ and each $x^{\bf k}$ is an ordered monomial), then
\[
    [X_{\beta_i}, X_{\beta_j}] = c\sum_{\bf k} \left( b_{\bf k} [x^{\bf k},
    X_{\beta_s}] - q^{-\langle \beta_r, \beta_s \rangle} c_{\bf k} [x^{\bf k},
    X_{\beta_r}] +
    \widehat{q}_{r, s} c_{\bf k} x^{\bf k} X_{\beta_r} \right).
\]
Finally, use known $q$-commutator relations to reorder the variables involved
in each of the monomials $x^{\bf k} X_{\beta_s}$, $X_{\beta_s} x^{\bf k}$,
$x^{\bf k} X_{\beta_r}$, and $X_{\beta_r} x^{\bf k}$.  For short, these steps
(which depend on $i$, $j$, $r$, and $s$) will be denoted by
$\texttt{R(i,j,r,s)}$. In summary, we have the following three-step algorithm.

\vspace{2mm}

\begin{tblr}{width=0.93\textwidth, colspec={Q[l] X[l]}}
    \hline
    \SetCell[c=2]{l}{{{$\texttt{R(i,j,r,s)}$ (to determine the $q$-commutator
    relation involving $[X_{\beta_i}, X_{\beta_j}]$)}}}
    \\
    \hline
    \textit{Step 1:} &Identify an existing (i.e. known) relation of the form
    $X_{\beta_j} = c[X_{\beta_r}, X_{\beta_s}]$, where $c\in \mathbb{K}$ is
    nonzero, and make the substitution to get $[X_{\beta_i}, X_{\beta_j}] =
    c[X_{\beta_i}, [X_{\beta_r}, X_{\beta_s}]]$.
    \\
    \textit{Step 2:} &Apply the $q$-Jacobi identity to write $[X_{\beta_i},
    [X_{\beta_r}, X_{\beta_s}]]$ as a linear combination of $[[X_{\beta_i},
    X_{\beta_r}], X_{\beta_s}]$, $[[X_{\beta_i}, X_{\beta_s}], X_{\beta_r}]$,
    and $[X_{\beta_i}, X_{\beta_s}] X_{\beta_r}$.
    \\
    \textit{Step 3:} &Replace each of the $q$-commutators,
    $[X_{\beta_i}, X_{\beta_r}]$ and $[X_{\beta_i}, X_{\beta_s}]$, with a
    linear combination of ordered monomials $\sum c_{\bf k} x^{\bf k}$, then
    use other known relations, if necessary, to reorder the variables involved
    in each monomial $x^{\bf k}X_{\beta_r}$, $x^{\bf k}X_{\beta_s}$,
    $X_{\beta_r} x^{\bf k}$, and $X_{\beta_s} x^{\bf k}$.
    \\
    \hline
\end{tblr}

\noindent Rather than replace $X_{\beta_j}$ with a scalar multiple of a
$q$-commutator, we can, in some cases, do this with $X_{\beta_i}$. This gives
us a slightly different algorithm that will be denoted by
$\texttt{L(i,j,r,s)}$.

\vspace{4mm}

\begin{tblr}{width=0.93\textwidth, colspec={Q[l] X[l]}}
    \hline
    \SetCell[c=2]{l}{{{$\texttt{L(i,j,r,s)}$ (to determine the $q$-commutator
    relation involving $[X_{\beta_i}, X_{\beta_j}]$)}}}
    \\
    \hline
    \textit{Step 1:} &Identify an existing (i.e. known) relation of the form
    $X_{\beta_i} = c[X_{\beta_r}, X_{\beta_s}]$, where $c\in \mathbb{K}$ is
    nonzero, and make the substitution to get $[X_{\beta_i}, X_{\beta_j}] =
    c[[X_{\beta_r}, X_{\beta_s}], X_{\beta_j}]$.
    \\
    \textit{Step 2:} &Apply the $q$-Jacobi identity to write $[[X_{\beta_r},
    X_{\beta_s}], X_{\beta_j}]$ as a linear combination of $[X_{\beta_r},
    [X_{\beta_s}, X_{\beta_j}]]$, $[[X_{\beta_r}, X_{\beta_j}], X_{\beta_s}]$,
    and $[X_{\beta_r}, X_{\beta_j}] X_{\beta_s}$.
    \\
    \textit{Step 3:} &Replace each of the $q$-commutators, $[X_{\beta_s},
    X_{\beta_j}]$ and $[X_{\beta_r}, X_{\beta_j}]$, with a linear combination
    of ordered monomials $\sum c_{\bf k} x^{\bf k}$, then use other known
    relations, if necessary, to reorder the variables involved in each monomial
    $x^{\bf k}X_{\beta_r}$, $x^{\bf k}X_{\beta_s}$, $X_{\beta_r} x^{\bf k}$,
    and $X_{\beta_s} x^{\bf k}$.
    \\
    \hline
\end{tblr}

\vspace{4mm}

As mentioned, in order to apply the algorithms $\texttt{R(i,j,r,s)}$ and
$\texttt{L(i,j,r,s)}$, there needs to exist $q$-commutator relations of the
form $X_{\beta_j} = c[X_{\beta_r}, X_{\beta_s}]$ and $X_{\beta_i} =
c[X_{\beta_r}, X_{\beta_s}]$, respectively, where $c \in \mathbb{K}$ is a
nonzero scalar and $r < s$.  However, whenever there are sufficiently many
relations of this type, the algorithms $\texttt{R(i,j,r,s)}$ and
$\texttt{L(i,j,r,s)}$ can be used to explicitly write the defining relations
(i.e. the $q$-commutation relations among all pairs of Lusztig root vectors) in
a quantum Schubert cell algebra ${\mathcal U}_q^+[w]$.

\begin{example}

    \normalfont

    Consider the reduced expression $w = s_3s_2s_4s_3s_4s_1s_2s_3s_2 \in
    W(B_4)$. Let $x_1,\cdots, x_9$ be the corresponding Lusztig root vectors in
    ${\mathcal U}_q^+[w]$. We will illustrate how to use the algorithm
    $\texttt{R(1,5,3,9)}$ to determine the $q$-commutation relation involving
    $[x_1, x_5]$.

    The first step in $\texttt{R(1,5,3,9)}$ says to use a relation of the form
    $[x_3, x_9] = c x_5$, where $c$ is some nonzero scalar.  We need to assume
    it has already been established that, in fact, $[x_3, x_9] = x_5$. With
    this at hand, we make a substitution to get $[x_1, x_5] = [x_1, [x_3,
    x_9]]$. Next, we move on to Step 2 of $\texttt{R(1,5,3,9)}$, which tells us
    to use the $q$-Jacobi identity. This gives us
    \[
        [x_1, x_5] = [[x_1, x_3], x_9] - q^2[[x_1, x_9], x_3] + [2]_q
        \widehat{q} [x_1, x_9] x_3.
    \]
    Now we apply Step 3, which tells us to first write each of the
    $q$-commutators $[x_1, x_3]$ and $[x_1, x_9]$ as a linear combination of
    ordered monomials. For this, we need to assume that it has already been
    established that $[x_1, x_9] = x_2$. From Proposition (\ref{LS corollary}),
    $[x_1, x_3] = 0$. Hence,
    \[
        [x_1, x_5] = -q^2[x_2, x_3] + [2]_q \widehat{q} x_2 x_3.
    \]
    The last part of Step 3 tells us to reorder the variables involved in the
    monomials $x_2 x_3$ and $x_3x_2$. From Proposition (\ref{LS corollary}),
    $[x_2, x_3] = 0$, and thus $[x_1, x_5] = [2]_q \widehat{q} x_2 x_3$.

\end{example}

\section{The set \texorpdfstring{$\Gamma(W)$}{Gamma(W)}}

\label{section, Gamma(W)}

Suppose $w = s_{i_1} \cdots s_{i_N} \in W$ is a reduced expression, and let
$X_{\beta_1},\ldots, X_{\beta_N}$ be the corresponding Lusztig root vectors in
the quantum Schubert cell algebra ${\mathcal U}_q^+[w]$. For each pair $j, k$
with $1 \leq j < k \leq N$, recall our main goal is to find the smallest
natural number $p$, which we refer to as the (local) \textit{nilpotency index},
so that
\[
    \left(\operatorname{ad}_q(X_{\beta_j})\right)^p (X_{\beta_k}) = 0.
\]
This leads us to consider the $q$-commutator $[X_{\beta_j}, X_{\beta_k}]$. If
it equals $0$, the nilpotency index is $1$. Otherwise, we must consider the
nested $q$-commutators $[X_{\beta_j}, [X_{\beta_j}, X_{\beta_k}]]$,
$[X_{\beta_j}, [X_{\beta_j}, [X_{\beta_j}, X_{\beta_k}]]]$, and so on.

Our strategy to compute nilpotency indices is to apply the Lusztig symmetries
of ${\mathcal U}_q(\mathfrak{g})$ to each of $X_{\beta_j}$ and
$X_{\beta_k}$ in such a way to simplify the computations involved.

As it turns out, to a pair of Lusztig root vectors, $X_{\beta_j}$ and
$X_{\beta_k}$ with $j < k$, the data necessary to compute a nilpotency index is
the triple
\begin{equation}
    \label{w, i, j}
    (s_{i_j}\cdots s_{i_k}, i_j, i_k) \in W \times \mathbf{I} \times
    \mathbf{I}.
\end{equation}
To illustrate, we get from (\ref{orthogonality of Tw}),
\[
    [X_{\beta_j}, X_{\beta_k}] \!=\! [T_{s_{i_1} \cdots s_{i_{j-1}}} (E_{i_j}),
    T_{s_{i_1} \cdots s_{i_{k - 1}}} (E_{i_k})] \!=\! T_{s_{i_1} \cdots s_{i_{j
    - 1}}} \!\left( [E_{i_j}, T_{s_{i_j} \cdots s_{i_{k - 1}}} (E_{i_k}) ]
    \right).
\]
Thus, computing the nilpotency index is equivalent to finding the smallest
natural number $p$ such that
\[
    \left( \operatorname{ad}_q(E_{i_j}) \right)^p \left(T_{s_{i_j} \cdots
    s_{i_{k - 1}}} (E_{i _k})\right) = 0.
\]
In (\ref{w, i, j}) above, the Weyl group element $s_{i_j} \cdots s_{i_k}$
always has length greater than $1$ and it has a reduced expression that begins
with $s_{i_j}$ and ends with $s_{i_k}$.  Conversely, to any triple $(w, i, j)
\in W \times \mathbf{I} \times \mathbf{I}$, such that $\ell(w) > 1$ and such
that $w$ and has a reduced expression beginning with $s_i$ and ending with
$s_j$ (equivalently, this means $\ell(s_iws_j) = \ell(w) - 2$), we can
associate to $(w, i, j)$ a pair of Lusztig root vectors in some quantum
Schubert cell algebra. In particular, $(w, i, j)$ is associated to the pair of
Lusztig root vectors, $X_{\alpha_i}$ and $X_{ws_j(\alpha_j)}$, in the quantum
Schubert cell algebra ${\mathcal U}_q^+[w]$. In view of this, for $w \in W$,
define
\[
    \begin{split}
        &{\mathcal D}_L(w) : = \left\{ k \in \mathbf{I} : \ell(s_kw) < \ell(w)
        \right\}, \hspace{5mm}
        {\mathcal D}_R(w) : = \left\{ k \in \mathbf{I} : \ell(ws_k) < \ell(w)
        \right\}.
    \end{split}
\]
The elements in the sets ${\mathcal D}_L(w)$ and ${\mathcal D}_R(w)$ are called
\textit{left descents} and \textit{right descents} of $w$, respectively.
Define
\begin{equation}
    \label{definition, Gamma W}
    \begin{split}
        \Gamma(W) := &\left\{ (w, i, j) \in W \times \mathbf{I} \times
        \mathbf{I}: j \in {\mathcal D}_R(w) \text{ and } i \in {\mathcal D}_L(w
        s_j) \right\}
        \\
        = &\left\{ (w, i, j) \in W \times \mathbf{I} \times \mathbf{I} :
        \ell(s_iws_j) = \ell(w) - 2 \right\}.
    \end{split}
\end{equation}
The set $\Gamma(W)$ contains all possible triples $(w, i, j)$, as in (\ref{w,
i, j}), that can be associated to a pair of Lusztig root vectors in a quantum
Schubert cell algebra.

From the discussion above, the problem of computing a
nilpotency index reduces to the problem of finding the smallest $p\in
\mathbb{N}$ so that $\left(\operatorname{ad}_q(E_i)\right)^p (T_{ws_j}(E_j)) =
0$, where $(w, i, j) \in \Gamma(W)$.  This motivates the following definition.

\begin{definition}

    Let $(w, i, j) \in \Gamma(W)$. Define the \textit{nilpotency index} of $(w,
    i, j)$ as
    \begin{equation}
        \label{definition, nil-index}
        {\mathcal N}(w, i, j) := \operatorname{min}\left\{p \in \mathbb{N} \cup
        \left\{ \infty \right\}: \left( \operatorname{ad}_q(E_i) \right)^p
        \left(T_{ws_j}(E_j) \right) = 0 \right\}.
    \end{equation}

\end{definition}

For $w \in W$, we denote the \textit{support} of $w$ by
\[
    \operatorname{supp}(w) := \left\{ i \in \mathbf{I} \mid s_i \leq w \text{
    with respect to the Bruhat order} \right\}.
\]

\noindent
\begin{minipage}{\textwidth}
\begin{theorem}

    \label{theorem, nil-index, 1}

    Suppose $(w, i, j) \in \Gamma(W)$.

    \begin{enumerate}

        \item If $i \not\in \operatorname{supp}(s_i w)$ or $j \not\in
            \operatorname{supp}(w s_j)$, then ${\mathcal N}(w, i, j) = 1$.

        \item If $i = j$ and $i \not\in \operatorname{supp}(s_i w s_i)$, then
            ${\mathcal N}(w, i, i) = 2$.

    \end{enumerate}

\end{theorem}
\end{minipage}

\begin{proof}

    In proving this theorem, we invoke the results in \cite[Proposition
    2.6]{GY1}, which characterize those generators $x_i$ of a CGL extension
    $\mathbb{K}[x_1][x_2; \sigma_2, \delta_2] \cdots [x_N; \sigma_N, \delta_N]$
    that are prime. We apply these results to quantum Schubert cell algebras
    ${\mathcal U}_q^+[w]$ to determine exactly which Lusztig root vectors
    $X_\beta$ are prime. Begin with a reduced expression $s_{i_1} \cdots
    s_{i_N}$ of $w$. As usual, let $X_{\beta_1},\cdots, X_{\beta_N}$ be the
    corresponding Lusztig root vectors of ${\mathcal U}_q^+[w]$.  The results
    of \cite[Proposition 2.6]{GY1} together with \cite[(9.24)]{GY2} tell us
    that a Lusztig root vector $X_{\beta_k}$ is prime if and only if the index
    $i_k$ appears only once among $i_1,\cdots, i_N$. Furthermore,
    \cite[Proposition 2.6(c)]{GY1} implies that a Lusztig root vector
    $X_{\beta_k}$ is prime if and only if $[X_{\beta_k}, X_{\beta_\ell}] = 0$
    and $[X_{\beta_p}, X_{\beta_k}] = 0$ for all $1 \leq p < k < \ell \leq N$.

    To prove (1), we first suppose $(w, i, j) \in \Gamma(W)$ and either $i
    \not\in \operatorname{supp}(s_i w)$ or $j \notin \operatorname{supp}(w
    s_j)$. Hence, there exists a reduced expression $s_{i_1} \cdots s_{i_N}$ of
    $w$ with $i_1 = i$, $i_N = j$, and such that either the index $i \in
    \mathbf{I}$ or the index $j\in \mathbf{I}$ only appears once in the list
    $i_1,\cdots, i_N$. From the discussion in the previous paragraph, the
    $q$-commutator $[X_{\beta_1}, X_{\beta_N}]$ equals $0$. Therefore,
    ${\mathcal N}(w, i, j) = 1$.

    Next, to prove (2), assume $(w, i, i) \in \Gamma(W)$ and $i \not\in
    \operatorname{supp}(s_i w s_i)$. Hence, if $s_{i_1} \cdots s_{i_N}$ is a
    reduced expression of $w$ such that $i_1 = i_N = i$, then $i \not\in
    \left\{i_2,\cdots, i_{N-1} \right\}$. Therefore, $[X_{\beta_1},
    X_{\beta_p}] = 0$ for $1 < p < N$ because $X_{\beta_1}$ is a prime element
    in the subalgebra ${\mathcal U}_q^+[ws_i] \subseteq {\mathcal U}_q^+[w]$.
    However, since $X_{\beta_1}$ is not a prime element in ${\mathcal
    U}_q^+[w]$, we have $[X_{\beta_1}, X_{\beta_N}] \neq 0$. Hence, ${\mathcal
    N}(w, i, i) > 1$. The Levendorskii-Soibelmann straightening formulas
    (\ref{L-S straightening}) tell us that the $q$-commutator $[X_{\beta_1},
    X_{\beta_N}]$ can be written as a linear combination of monomials in the
    variables $X_{\beta_2},\cdots, X_{\beta_{N-1}}$. However, since
    $[X_{\beta_1}, X_{\beta_p}] = 0$ for all $1 < p < N$, then by applying the
    $q$-Leibniz identity (\ref{q-Leibniz identity}), we obtain $[X_{\beta_1},
    [X_{\beta_1}, X_{\beta_N}]] = 0$. Thus, ${\mathcal N}(w, i, i) = 2$.

\end{proof}

The next theorem gives the cardinality of $\Gamma(W)$. From this result, we
obtain an approximation formula, $\left| \Gamma(W) \right| \approx
\frac{n^2}{4} \left| W \right|$, where $n$ is the rank of the underlying Lie
type $X_n$.

\begin{theorem}

    \label{Gamma W, cardinality}

    Let $\mathfrak{g}$ be a Lie algebra of type $X_n$, and let $W$ be the Weyl
    group.  Let $n_s$ ($n_\ell$) be the number of short (long) simple roots,
    and let $\left| \Delta_s \right|$ ($\left| \Delta_\ell \right|$) be the
    number of short (long) roots. Then
    \[
        \frac{\left| \Gamma(W) \right|}{\left| W \right|} =
        \begin{cases}
            \frac{n^2}{4} - \frac{n_s^2}{2 \left| \Delta_s \right|} -
            \frac{n_\ell^2}{2 \left| \Delta_\ell \right|},
            &
            (\text{if $X_n$ is not simply laced}),
            \\
            n^2 \left( \frac{1}{4} - \frac{1}{2 \left| \Delta \right|} \right),
            &
            (\text{if $X_n$ is simply laced}).
        \end{cases}
    \]

\end{theorem}

\begin{proof}

    For fixed $i, j \in \mathbf{I}$, define $W_{i,j} := \left\{ w \in W :
    \ell(s_i w s_j) = \ell(w) - 2 \right\}$.  With this, we have
    \[
        \left| \Gamma (W) \right| = \sum_{i,j \in \mathbf{I}} \left| W_{i, j}
        \right|,
    \]
    and thus the statement follows by summing the cardinalities $\left| W_{i,
    j} \right|$ over all $i, j \in \mathbf{I}$.

    Define $K_j := \left| \left\{ \beta \in \Delta: \langle \beta, \beta
    \rangle = \langle \alpha_j, \alpha_j \rangle \right\} \right|$. In other
    words, $K_j$ equals the number of roots having the same length as
    $\alpha_j$. We will show $\left| W_{i,j} \right| = \frac{1}{4} \left| W
    \right| - \frac{1}{2K_j} \left| W \right|$ whenever the simple roots
    $\alpha_i$ and $\alpha_j$ have the same length, whereas $\left| W_{i,j}
    \right| = \frac{1}{4} \left| W \right|$ if $\alpha_i$ and $\alpha_j$ have
    different lengths.

    To determine the cardinality of $W_{i, j}$, consider first the action of
    the Klein $4$-group $G = \left\{1, u, v, uv \right\}$ on $W$, where $G$
    acts by the rules $u.w = s_iw$ and $v.w = ws_j$ for all $w\in W$.  Orbits
    of $W$ under the action of $G$ either have the form $\left\{ w, s_i w, w
    s_j, s_i w s_j \right\}$ or $\left\{ w, s_i w \right\}$.  The latter case
    occurs when $s_i w = w s_j$, and in this case, $uv \in G$ acts trivially on
    each element in the orbit $\left\{ w, s_i w\right\}$.  Hence, each element
    $x$ in such an orbit has the property that $\ell(s_i x s_j) = \ell(x)$.
    However, in an orbit of the form $\left\{ w, s_i w, w s_j, s_i w s_j
    \right\}$, there is exactly one element $x$ such that $\ell(s_i x s_j) =
    \ell(x) - 2$. Therefore, $\left| W_{i, j} \right|$ equals the number the
    orbits of size $4$ of $W$ under the action of $G$.

    Observe that the orbit of an element $w$ has the form $\left\{ w, s_i w
    \right\}$ if and only if $s_i w = w s_j$, and this situation occurs
    precisely when $w (\alpha_j) \in \left\{ \alpha_i, -\alpha_i \right\}$.
    Since there are $K_j$ roots in the orbit of $\alpha_j$ under the action of
    $W$, there are $\left| W \right| / K_j$ Weyl group elements that act
    trivially on $\alpha_j$.  Thus, if $\alpha_i$ and $\alpha_j$ have the same
    length, there are $2\left| W \right| / K_j$ Weyl group elements $w$ such
    that $w (\alpha_j) \in \left\{ \alpha_i, -\alpha_i \right\}$, whereas if
    $\alpha_i$ and $\alpha_j$ do not have the same length, there are no such
    elements $w$.  Therefore, there are $\left| W \right| / K_j$ orbits of $W$
    of size $2$ under the action of $G$ when $\alpha_i$ and $\alpha_j$ have the
    same length. However, if $\alpha_i$ and $\alpha_j$ do not have the same
    length, there are no orbits of size $2$.  Hence, there remains $\left| W
    \right|/4 - \left| W \right| / (2K_j)$ orbits of size $4$ of $W$ under the
    action of $G$, provided $\alpha_i$ and $\alpha_j$ have the same length, and
    $\left| W \right| / 4$ such orbits otherwise.

\end{proof}

\subsection{Some relations on \texorpdfstring{$\Gamma(W)$}{Gamma(W)}}

\label{section, relations on Gamma(W)}

Next, we define some relations on $\Gamma(W)$ and show each of them preserves
nilpotency indices. That is to say, if $x, x^\prime \in \Gamma(W)$ are related
to each other under any one of these relations, ${\mathcal N}(x) = {\mathcal
N}(x^\prime)$.

Before defining these relations, we recall the \textit{weak left Bruhat order}
$\leq_L$ on $W$, which is defined by the rule
\[
    v \leq_L w \iff \ell(v) + \ell(w v^{-1}) = \ell(w) \hspace{5mm}(v, w \in
    W).
\]
Observe that $v \leq_L w$ if and only if there exists a reduced expression for
$w$ such that the final substring is a reduced expression for $v$. Analogously,
the \textit{weak right Bruhat order} $\leq_R$ on $W$ is defined by the rule
\[
    v \leq_R w \iff \ell(v) + \ell(v^{-1} w) = \ell(w) \hspace{5mm}(v, w \in
    W).
\]

We define relations, $\xlongrightarrow{L}$ and $\xlongrightarrow{R}$, on
$\Gamma(W)$ that resemble the weak left and weak right Bruhat orders,
respectively.

\begin{definition}

    Define relations $\xlongrightarrow{\hspace{1mm} L \hspace{1mm}}$ and
    $\xlongrightarrow{\hspace{1mm} R \hspace{1mm}}$, on $\Gamma(W)$ by the
    rules
    \begin{equation}
        \label{definition, L ordering}
        (w, i, j) \xlongrightarrow{\hspace{1mm} L \hspace{1mm}} (u, i^\prime,
        j^\prime) \iff u \leq_L w, w^{-1}(\alpha_i) = u^{-1} (\alpha_{i^\prime}),
        \text{ and } j = j^\prime
    \end{equation}
    and
    \begin{equation}
        \label{definition, R ordering}
        (w, i, j) \xlongrightarrow{\hspace{1mm} R \hspace{1mm}} (u, i^\prime,
        j^\prime) \iff u \leq_R w, w(\alpha_j) = u(\alpha_{j^\prime}), \text{ and }
        i = i^\prime. \hspace{18pt}
    \end{equation}

\end{definition}
From the definitions, it can be shown that each of $\xlongrightarrow{L}$ and
$\xlongrightarrow{R}$ is a partial order.

\begin{proposition}

    The relations $\xlongrightarrow{L}$ and $\xlongrightarrow{R}$ are partial
    orders on $\Gamma(W)$.

\end{proposition}

For each index $k \in \mathbf{I}$, we also define relations, $\xlongrightarrow{(k,
L)}$ and $\xlongrightarrow{(k, R)}$, on $\Gamma(W)$. These are not partial orders.

\begin{definition}

    For each $k \in \mathbf{I}$, define relations,
    $\xlongrightarrow{\hspace{1mm} (k, L) \hspace{1mm}}$ and
    $\xlongrightarrow{\hspace{1mm} (k, R) \hspace{1mm} }$, on $\Gamma(W)$ by
    the rules
    \begin{equation}
        \label{definition, (k, L) reduction}
        (w, i, j) \xlongrightarrow{\hspace{1mm} (k, L) \hspace{1mm}} (u, i^\prime,
        j^\prime) \iff \begin{tabular}{l} $\ell(u) < \ell(w)$, $(ws_k,i, j)
        \xlongrightarrow{\hspace{1mm} R \hspace{1mm}} (w, i, j)$,\\
        $(ws_k, i, j) \xlongrightarrow{\hspace{1mm} L \hspace{1mm}} (u, i^\prime,
        j^\prime)$, \text{ and } $u \leq_L w$,\end{tabular}
    \end{equation}
    and
    \begin{equation}
        \label{definition, (k, R) reduction}
        (w, i, j) \xlongrightarrow{\hspace{1mm} (k, R) \hspace{1mm}} (u, i^\prime,
        j^\prime) \iff \begin{tabular}{l} $\ell(u) < \ell(w)$, $(s_kw,i, j)
        \xlongrightarrow{\hspace{1mm} L \hspace{1mm}} (w, i, j)$,\\
        $(s_kw, i, j) \xlongrightarrow{\hspace{1mm} R \hspace{1mm}} (u, i^\prime,
        j^\prime)$, \text{ and } $u \leq_R w$.\end{tabular}
    \end{equation}

\end{definition}

We note that the union of $\xlongrightarrow{L}$ and $\xlongrightarrow{R}$ is a
reflexive and anti-symmetric relation on $\Gamma(W)$, and taking the transitive
closure of this union breaks neither reflexivity nor anti-symmetry. Hence,
there is a partial order generated by $\xlongrightarrow{L}$ and
$\xlongrightarrow{R}$. Secondly, observe that the transitive closure of the
symmetric closure of a partial order is an equivalence relation.

\begin{definition}

    \label{definition, equivalence relation}

    Let $\xlongrightarrow{}$ be the partial order generated by
    $\xlongrightarrow{L}$ and $\xlongrightarrow{R}$, and let
    $\xleftrightarrow{\hspace{4mm}}$ be the equivalence relation generated by
    $\xlongrightarrow{}$.

\end{definition}

From the discussion above, we have the following proposition.

\begin{proposition}

    $ $

    \begin{enumerate}

        \item For $x, y \in \Gamma(W)$, $x \xlongrightarrow{} y$ if and only if
            there exist finitely many elements $x_1,\cdots, x_k \in \Gamma(W)$
            such that $x = x_1$, $y = x_k$, and for each $i \in [1, k - 1]$,
            either (1) $x_i \xlongrightarrow{L} x_{i + 1}$, or (2) $x_i
            \xlongrightarrow{R} x_{i + 1}$.

        \item For $x, y \in \Gamma(W)$, $x \xleftrightarrow{\hspace{4mm}} y$ if
            and only if there exist finitely many elements $x_1, x_2,\cdots,
            x_k \in \Gamma(W)$ such that $x = x_1$, $y = x_k$, and for each $i
            \in [1, k - 1]$, either (1) $x_i \xlongrightarrow{L} x_{i + 1}$,
            (2) $x_{i + 1} \xlongrightarrow{L} x_i$, (3) $x_i
            \xlongrightarrow{R} x_{i + 1}$, or (4) $x_{i + 1}
            \xlongrightarrow{R} x_i$.

    \end{enumerate}

\end{proposition}

The following theorem gives a key property regarding these relations and
nilpotency indices. In particular, nilpotency indices are preserved under each
of these relations.

\begin{theorem}

    \label{nilpotency indices, preserved}

    Suppose $x, y \in \Gamma(W)$ such that at least one of the following hold:

    \begin{enumerate}

        \item $x \xlongrightarrow{L} y$ or $x \xlongrightarrow{R} y$,

        \item $x \xlongrightarrow{(k, L)} y$ or $x \xlongrightarrow{(k, R)} y$
            for some $k \in \mathbf{I}$,

        \item $x \xlongrightarrow{} y$ or $x \xleftrightarrow{\hspace{4mm}} y$.

    \end{enumerate}

    Then ${\mathcal N}(x) = {\mathcal N}(y)$.

\end{theorem}

\begin{proof}

    From the definition of the relations $\xlongrightarrow{(k, L)}$,
    $\xlongrightarrow{(k, R)}$, $\xlongrightarrow{}$, and
    $\xleftrightarrow{\hspace{4mm}}$. It suffices to prove that ${\mathcal
    N}(x) = {\mathcal N}(y)$ if $x \xlongrightarrow{L} y$ or $x
    \xlongrightarrow{R} y$.

    Assume $x \xlongrightarrow{L} y$. Suppose $x = (w, i, j)$ for some $w \in
    W$ and $i, j \in \mathbf{I}$. Thus, $y = (v, i^\prime, j)$ for some $v\in
    W$ and $i^\prime \in \mathbf{I}$. We will show that the Lusztig symmetry
    $T_{w v^{-1}}$ sends the Chevalley generator $E_{i^\prime}$ to $E_i$, and
    sends $T_{vs_j}(E_j)$ to $T_{ws_j}(E_j)$. By the definition of $x
    \xlongrightarrow{L} y$, we have $v \leq_L w$ and $w^{-1}(\alpha_i) =
    v^{-1}(\alpha_{i^\prime})$.  Therefore, $w v^{-1}(\alpha_{i^\prime}) =
    \alpha_i$, and thus, by Theorem (\ref{Jantzen, 8.20}), $T_{w v^{-1}}
    (E_{i^\prime}) = E_i$. Since $j$ is a right descent of both $w$ and $v$,
    then
    \[
        \ell(w v^{-1}) + \ell(v s_j) = \left(\ell(w) - \ell(v)\right) +
        \left(\ell(v) - 1\right) = \ell(w) - 1 = \ell(w s_j).
    \]
    Therefore, $T_{wv^{-1}} T_{v s_j} = T_{ws_j}$ and we obtain $T_{w v^{-1}}
    \left(T_{v s_j}(E_j) \right) = T_{w s_j}(E_j)$.  Since $T_{w v^{-1}}$ is an
    algebra automorphism of ${\mathcal U}_q (\mathfrak{g})$, and $T_{w
    v^{-1}}(u) \in \left( {\mathcal U}_q(\mathfrak{g})
    \right)_{wv^{-1}(\beta)}$ for all homogeneous elements $u \in \left(
    {\mathcal U}_q(\mathfrak{g}) \right)_\beta$, we obtain
    \[
        \left(\operatorname{ad}_q (E_i) \right)^p \left(T_{ws_j}(E_j) \right) =
        T_{wv^{-1}} \left( \left( \operatorname{ad}_q(E_{i^\prime}) \right)^p
        \left(T_{vs_j}(E_j) \right) \right)
    \]
    for all $p \in \mathbb{N}$. Therefore ${\mathcal N}(x) = {\mathcal N}(y)$.

    Now assume $x \xlongrightarrow{R} y$. As before, suppose $x = (w, i, j)$
    for some $w \in W$ and $i, j \in \mathbf{I}$. Thus $y = (v, i, j^\prime)$
    for some $v \in W$ and $j^\prime \in \mathbf{I}$. In this setting, $v
    \leq_R w$ and $w(\alpha_j) = v(\alpha_{j^\prime})$, and thus $v^{-1} w
    (\alpha_j) = \alpha_{j^\prime}$. From Theorem (\ref{Jantzen, 8.20}),
    $T_{v^{-1}w}(E_j) = E_{j^\prime}$. However, using $v \leq_R w$ in
    conjunction with the fact that $j$ and $j^\prime$ are right descents of $w$
    and $v$, respectively, we obtain
    \[
        \ell(v s_{j^\prime}) + \ell(v^{-1} w) = \left(\ell(v) - 1\right) +
        \left(\ell(w) - \ell(v) \right) = \ell(w) - 1 = \ell(ws_j).
    \]
    From $v^{-1 }w (\alpha_j) = \alpha_{j^\prime}$ we have $s_{j^\prime} v^{-1}
    w = v^{-1} w s_j$, and thus $(v s_{j^\prime}) \cdot (v^{-1} w) = ws_j$.
    Hence, $T_{ws_j} = T_{v s_{j^\prime}} T_{v^{-1} w}$. Therefore, $T_{ws_j}
    (E_j) = T_{v s_{j^\prime}} T_{v^{-1} w} (E_j) = T_{v
    s_{j^\prime}}(E_{j^\prime})$.  This implies
    \[
        \left(\operatorname{ad}_q (E_i) \right)^p \left(T_{ws_j}(E_j) \right) =
        T_{wv^{-1}} \left( \left( \operatorname{ad}_q(E_{i^\prime}) \right)^p
        \left(T_{vs_j}(E_j) \right) \right)
    \]
    for all $p \in \mathbb{N}$. Thus, ${\mathcal N}(x) = {\mathcal N}(y)$.

\end{proof}

\noindent Observe that if $(w, i, j)  \xlongrightarrow{*} (u, i^\prime,
j^\prime)$, where $*$ is any one of the symbols $L$, $R$, $(k, L)$, $(k, R)$
(for $k \in \mathbf{I}$), or blank, then $\ell(u) \leq \ell(w)$. In view of
this, we make the following definitions.

\begin{definition}

    $ $

    \begin{enumerate}

        \item An element $(x, y) \in \Gamma(W) \times \Gamma(W)$ belonging to
            any of the relations $\xlongrightarrow{L}$, $\xlongrightarrow{R}$,
            $\xlongrightarrow{(k, L)}$, $\xlongrightarrow{(k, R)}$ (for $k \in
            \mathbf{I}$), or $\xlongrightarrow{}$ will be called a
            \textit{reduction}.

        \item If there is a reduction $x \xlongrightarrow{*} y$ or a sequence
            of reductions,
            \[
                x \xlongrightarrow{*} \cdots \xlongrightarrow{*} y,
            \]
            where the $*$ symbol is $L$, $R$, $(k, L)$, $(k, R)$ (for $k \in
            \mathbf{I}$), or blank, and may vary from one reduction to the
            next, we say $x$ \textit{reduces} to $y$.

        \item An element $x \in \Gamma(W)$ will be called a \textit{minimal
            element} of the partially ordered set $(\Gamma(W),
            \xlongrightarrow{*})$ (where the $*$ symbol is $L$, $R$, or blank)
            if $x \xlongrightarrow{*} y$ implies $x = y$.

        \end{enumerate}

\end{definition}

\subsection{Covering relations for the posets \texorpdfstring{$(\Gamma(W),
    \xlongrightarrow{L})$}{(Gamma(W), -L->)} and \texorpdfstring{$(\Gamma(W),
        \xlongrightarrow{R})$}{(Gamma(W), -R->)}}

In this section, we give an explicit description of the covering relations in
the posets $(\Gamma(W), \xlongrightarrow{L})$ and $(\Gamma(W),
\xlongrightarrow{R})$. This will be beneficial for a couple of reasons.  First,
when proving whether or not a given property is preserved under these
relations, it suffices to consider only covering relations. For instance, if
$f: \Gamma(W) \to S$ is a function from $\Gamma(W)$ to a set $S$, and we wish
to prove $f(x) = f(y)$ whenever $x \xlongrightarrow{L} y$ (or $x
\xlongrightarrow{R} y$), it suffices to consider only those situations when $x
\xlongrightarrow{L} y$ (or $x \xlongrightarrow{R} y$) is a covering relation.
Secondly, covering relations can be used to iteratively construct the connected
component of the Hasse diagram containing any given element $x\in \Gamma(W)$.
Begin by finding all covers of $x$ as well as all elements $y$ that $x$ covers,
then iteratively repeat this process until no new elements are encountered.  Of
relevance to us is that nilpotency indices are preserved under these relations,
and thus ${\mathcal N}(x) = {\mathcal N}(y)$ whenever $x, y \in \Gamma(W)$
belong to the same connected component of the Hasse diagram.

We recall some basic definitions related to posets. As usual, for a poset
$({\mathcal P}, \geq)$ and a pair of elements $x, y \in {\mathcal P}$, we write
$x > y$ whenever $x\geq y$ and $x \neq y$. Recall, a \textit{covering relation}
$x \geq y$ is a relation such that $x > y$ and there fails to exist $z \in
{\mathcal P}$ such that $x > z > y$. In this setting, we say $x$
\textit{covers} $y$, or that $y$ \textit{is covered by} $x$.

\begin{definition}

    \label{definition, s_ab}

    For a pair of distinct indices $a, b \in \mathbf{I}$, let $w_0(a, b)$ be
    the longest element of the subgroup of $W$ generated by $s_a$ and $s_b$.

\end{definition}

\begin{proposition}

    Let $(w, i, j) \in \Gamma(W)$.

    \begin{enumerate}

        \item If there exists $t \in \mathbf{I}$ with $t \neq i$ such that
            $w_0(i, t) \leq_R ws_j$, then
            \begin{equation}
            \label{definition, elementary reduction, left}
                (w, i, j) \xlongrightarrow{L} (w_0(i, t) \cdot s_i \cdot w,
                i^\prime, j),
            \end{equation}
            where $i^\prime = t$ if the order of $s_i s_t$ is $3$, and
            $i^\prime = i$ otherwise.

        \item If there exists $t \in \mathbf{I}$ with $t \neq j$ such that
            $w_0(j, t) \leq_L s_iw$, then
            \begin{equation}
            \label{definition, elementary reduction, right}
                (w, i, j) \xlongrightarrow{R} (w \cdot s_j \cdot w_0(j, t), i,
                j^\prime),
            \end{equation}
            where $j^\prime = t$ if the order of $s_j s_t$ is $3$, and
            $j^\prime = j$ otherwise.

    \end{enumerate}

\end{proposition}

\begin{proof}

    We will prove (2). The proof of (1) is similar.  Suppose $(w, i, j) \in
    \Gamma(W)$ and $w_0(j, t) \leq_L s_iw$ for some $t \neq j$. Let $j^\prime =
    t$ if the order of $s_js_t$ is $3$, or $j^\prime = j$ otherwise. For
    brevity, define $u := ws_j \cdot w_0(j, t)$.

    First observe $s_j w_0(j, t) (\alpha_{j^\prime}) = \alpha_j$.  Hence, $u
    (\alpha_{j^\prime}) = w (\alpha_j)$.  Next, we will prove $(u, i, j^\prime)
    \in \Gamma(W)$ by showing $j^\prime$ is right descent of $u$, and $i$ is a
    left descent of $us_{j^\prime}$.  Since $j$ is a right descent of $w$, $u
    (\alpha_{j^\prime}) = w(\alpha_j) < 0$.  Therefore $j^\prime$ is a right
    descent of $u$.  Using the identity $u s_{j^\prime} = w \cdot w_0(j, t)$
    together with $w_0(j, t) \leq_L s_iw \leq_L w$ gives us
    $\ell(us_{j^\prime}) = \ell(w) - \ell(w_0(j, t))$.  Since $w_0(j, t)$ is an
    involution, $\left(s_ius_{j^\prime}\right) \cdot w_0(j, t) = s_iw$. Thus,
    $\ell(s_i u s_{j^\prime}) + \ell(w_0(j, t)) = \ell(s_iw) = \ell(w) - 1$.
    Therefore, $\ell(s_i u s_{j^\prime}) = \ell(us_{j^\prime}) - 1$.
    Equivalently, $i$ is a left descent of $u s_{j^\prime}$.  Finally, from the
    identities $u \cdot w_0(j, t) s_j = w$ and $\ell(u) + \ell(w_0(j, t) s_j) =
    \ell(u) + \ell(w_0(j, t)) - 1 = \ell(w)$, it follows that $u \leq_R w$.

\end{proof}

\begin{definition}

    The reductions in (\ref{definition, elementary reduction, left}) and
    (\ref{definition, elementary reduction, right}) will be referred to as
    \textit{elementary reductions}.

\end{definition}

We are now ready to prove that every reduction is either an elementary
reduction or a composition of elementary reductions.  That is to say, for every
reduction of the form $(w, i, j) \xlongrightarrow{L} (v, k, j)$ (or of the form
$(w, i, j) \xlongrightarrow{R} (v, i, k)$), there exist $(w_1, i_1, j), \cdots,
(w_N, i_N, j) \in \Gamma(W)$ (or $(w_1, i, j_1), \ldots, (w_N, i, j_N) \in
\Gamma(W)$) such that each reduction in the composition of reductions
\begin{equation}
    \label{sequence of elementary reductions, left}
    (w, i, j) = (w_1, i_1, j) \xlongrightarrow{L} (w_2, i_2, j)
    \xlongrightarrow{L} \cdots \xlongrightarrow{L} (w_N, i_N, j) = (v,
    k, j).
\end{equation}
or
\begin{equation}
    \label{sequence of elementary reductions, right}
    (w, i, j) = (w_1, i, j_1) \xlongrightarrow{R} (w_2, i, j_2)
    \xlongrightarrow{R} \cdots \xlongrightarrow{R} (w_N, i, j_N) = (v, i, k),
\end{equation}
respectively, is an elementary reduction.

\begin{theorem}

    \label{theorem, to elementary reductions}

    Every reduction, $x \xlongrightarrow{L} y$ or $x \xlongrightarrow{R} y$, is
    either an elementary reduction or a composition of elementary reductions as
    in (\ref{sequence of elementary reductions, left}) or (\ref{sequence of
    elementary reductions, right}), respectively.

\end{theorem}

\begin{proof}

    We will prove that every reduction of the form $x \xlongrightarrow{R} y$ is
    either an elementary reduction or a composition of elementary reductions.
    The analogous statement for reductions of the form $x \xlongrightarrow{L}
    y$ can be proved similarly.

    To begin, suppose $(w, i, j) \xlongrightarrow{R} (v, i, k)$ for some pair
    of elements, $(w, i, j)$ and $(v, i, k)$, in $\Gamma(W)$ with $w \neq v$.
    Hence $v \leq_R w$ and $w(\alpha_j) = v(\alpha_k)$.  Define $u := v^{-1}
    w$. Thus, $\ell(u) = \ell(w) - \ell(v) > 0$ and $u(\alpha_j) = \alpha_k$.
    Observe that $j$ is not a right descent of $u$ because $u(\alpha_j) > 0$.
    Let $s_{i_1} \cdots s_{i_N}$ be a reduced expression for $u$. Since $j$ is
    not a right descent of $u$, then $s_{i_1} \dots s_{i_N} s_j$ is a reduced
    expression for $us_j$. Using $u(\alpha_j) = \alpha_k$, we have $us_j = s_k
    u$.  Therefore, there exists a sequence of braid relations that can be
    performed on the reduced expression $s_{i_1} \cdots s_{i_N} s_j$ to
    transform it into $s_k s_{i_1} \cdots s_{i_N}$. In this sequence of braid
    relations, there must exist at least one that involves the rightmost $s_j$
    (appearing in the reduced expression $s_{i_1} \cdots s_{i_N} s_j$).
    Equivalently, this means there exists some $t \in \mathbf{I}$ with $t \neq
    j$ so that $w_0(j, t) \leq_L us_j$, and thus $w_0(j, t) s_j \leq_L u \leq_L w$.
    Therefore, $w \cdot s_j \cdot w_0(j, t) \leq_R w$. Next, let $j^\prime \in
    \mathbf{I}$ be defined by the rule $\alpha_j = s_j w_0(j, t)
    (\alpha_{j^\prime})$. Thus, $j^\prime \in \left\{j ,t \right\}$ and we have
    $w(\alpha_j) = w \cdot s_j \cdot w_0(j, t) (\alpha_{j^\prime})$.  Hence, $(w,
    i, j) \xlongrightarrow{R} (w_2, i, j_2)$, where $w_2 := w \cdot s_j \cdot
    w_0(j, t)$ and $j_2 := j^\prime$.

    Observe that $v^{-1} w_2 \leq_R u$. Hence, $v \leq_R w_2$. This observation
    in conjunction with $w(\alpha_j) = w_2(\alpha_{j_2}) = v(\alpha_k)$ means
    $(w_2, i, j_2) \xlongrightarrow{R} (v, i, k)$. Thus, we have a composition
    of reductions
    \[
        (w, i, j) \xlongrightarrow{R} (w_2, i, j_2) \xlongrightarrow{R} (v, i,
        k),
    \]
    where the first reduction $(w, i, j) \xlongrightarrow{R} (w_2, i, j_2)$ is
    an elementary reduction. If $w_2 \neq v$, we can repeat this entire process
    to construct yet another elementary reduction $(w_2, i, j_2)
    \xlongrightarrow{R} (w_3, i, j_3)$, giving us
    \[
        (w, i, j) \xlongrightarrow{R} (w_2, i, j_2) \xlongrightarrow{R} (w_3,
        i, j_3) \xlongrightarrow{R} (v, i, k).
    \]
    Repeating this process will eventually produce a composition of elementary
    reductions, as in (\ref{sequence of elementary reductions, right}), because
    $\ell(w) > \ell(w_2) > \ell(w_3) > \cdots$.

\end{proof}

\noindent As a consequence of Theorem (\ref{theorem, to elementary
reductions}), elementary reductions coincide with covering relations.

\begin{corollary}

    The elementary reductions (\ref{definition, elementary reduction, left})
    and (\ref{definition, elementary reduction, right}) are the covering
    relations of the posets $(\Gamma(W), \xlongrightarrow{\hspace{2mm} L
    \hspace{2mm}} )$ and $(\Gamma(W), \xlongrightarrow{\hspace{2mm} R
    \hspace{2mm}})$, respectively.

\end{corollary}

\subsection{Miscellaneous}

For $(w, i, j) \in \Gamma(W)$, define
\[
    (w, i, j)^* := (w^{-1}, j, i).
\]
It is a simple observation that the map given by $x \mapsto x^*$ is an
involution on $\Gamma(W)$.

\begin{definition}

    An element $x\in \Gamma(W)$ is \textit{self-dual} if $x = x^*$.

\end{definition}

\begin{proposition}

    \label{L, R, equivalent, inverse}

    Suppose $x, y \in \Gamma(W)$ and $k \in \mathbf{I}$.  Then

    \begin{enumerate}

        \label{proposition, a}

        \item $x \xlongrightarrow{L} y$ if and only if $x^* \xlongrightarrow{R}
            y^*$, and

        \item $x \xlongrightarrow{\hspace{1mm} (k, L) \hspace{1mm}} y$ if and
            only if $x^* \xlongrightarrow{\hspace{1mm} (k, R) \hspace{1mm}}
            y^*$.

        \item \label{proposition, a, 3} $x \xleftrightarrow{\hspace{4mm}} y$ if
            and only if $x^*\xleftrightarrow{\hspace{4mm}} y^*$.

    \end{enumerate}

\end{proposition}

\begin{proof}

    This follows from the definition of the relations $\xlongrightarrow{L}$,
    $\xlongrightarrow{R}$, $\xlongrightarrow{\hspace{1mm} (k, L)
    \hspace{1mm}}$, $\xlongrightarrow{\hspace{1mm} (k, R) \hspace{1mm}}$, and
    $\xleftrightarrow{\hspace{4mm}}$.

\end{proof}

\begin{corollary}

    \label{corollary, a}

    Suppose $x\in \Gamma(W)$ is self-dual and $x \xleftrightarrow{\hspace{4mm}}
    y$ for some $y \in \Gamma(W)$. Then $y \xleftrightarrow{\hspace{4mm}} y^*$.

\end{corollary}

\begin{definition}

    For an element $(w, i, j) \in \Gamma(W)$, define the triple
    \[
        \chi(w, i, j) := \left(\langle \alpha_i, \alpha_i \rangle, \langle
        \alpha_j, \alpha_j \rangle, \langle \alpha_i, ws_j(\alpha_j)
        \rangle\right),
    \]
    where $\langle \,\,,\,\, \rangle$ is the symmetric bilinear form on the
    root lattice $Q$, normalized such that $\langle \alpha, \alpha \rangle = 2$
    for short roots $\alpha$.

\end{definition}

\begin{proposition}

    Suppose $x, y \in \Gamma(W)$ and $x \xleftrightarrow{\hspace{4mm}} y$. Then
    $\chi(x) = \chi(y)$.

\end{proposition}

\begin{proof}

    Let ${\mathcal G}$ be the group of automorphisms of ${\mathcal
    U}_q(\mathfrak{g})$ generated by the Lusztig symmetries $T_w$ ($w \in W$).
    Suppose $x = (u, i, j)$ and $y = (v, k, p)$. The proof of Theorem
    (\ref{nilpotency indices, preserved}) shows that if $x
    \xleftrightarrow{\hspace{4mm}} y$, there exists $\phi \in {\mathcal G}$
    such that $\phi(E_i) = E_k$ and $\phi(T_{us_j}(E_j)) = T_{v s_p} (E_p)$.
    Since $\phi \in {\mathcal G}$, it follows from (\ref{Tw, homogeneous}) that
    there exists $z \in W$ so that $\phi \left( {\mathcal
    U}_q(\mathfrak{g})_\lambda \right) \subseteq {\mathcal
    U}_q(\mathfrak{g})_{z(\lambda)}$ for all $\lambda \in Q$.  Therefore,
    $z(\alpha_i) = \alpha_k$ and $zus_j(\alpha_j) = vs_p(\alpha_p)$.  Hence,
    $\alpha_i$ and $\alpha_k$ lie in the same $W$-orbit, and thus $\alpha_i$
    and $\alpha_k$ have the same length. Similarly, $\alpha_j$ and $\alpha_p$
    have the same length. We also have
    \[
        \langle \alpha_i, us_j(\alpha_j) \rangle = \langle z (\alpha_i), z u
        s_j (\alpha_j) \rangle = \langle \alpha_k, vs_p (\alpha_p) \rangle.
    \]

\end{proof}

\section{Minimal elements of the poset \texorpdfstring{$(\Gamma(W),
    \xlongrightarrow{})$}{(Gamma(W), -->)}}

    \label{minimal elements}

In this section, we classify the minimal elements of the poset $(\Gamma(W),
\xlongrightarrow{})$.  We remark first that if $w \in W$ is a Weyl group
element and there is only one tuple $(i, j) \in \mathbf{I} \times \mathbf{I}$
such that $(w, i, j) \in \Gamma(W)$, we write
\begin{equation}
    \label{def, llbracket}
    \llbracket w \rrbracket
\end{equation}
as an abbreviation for $(w, i, j)$.  This shorthand is typically used in those
situations when $w \in W$ is \textit{bigrassmannian}, i.e.  $\lvert {\mathcal
D}_L(w) \rvert = \lvert {\mathcal D}_R(w) \rvert = 1$.

Recall Definition (\ref{definition, s_ab}): for distinct $a, b \in \mathbf{I}$,
$w_0(a, b)$ is the longest element of the subgroup of $W$ generated by $s_a$
and $s_b$. The main result of this section is given in the following theorem.

\begin{theorem}

    \label{reduce by L and R}

    Let $x \in \Gamma(W)$. Then $x$ is a minimal element of the partially
    ordered set $(\Gamma(W), \xlongrightarrow{})$ if and only if either

    \begin{enumerate}

        \item  $x = \llbracket u \rrbracket$ for some bigrassmannian $u \in W$,
            or

        \item $x = (w_0(p, k), p, p^\prime) $ for some $p, k \in \mathbf{I}$
            with $p \neq k$ (where $p^\prime = p$ if the order of $s_ps_k$ is
            $3$, and $p^\prime = k$ otherwise).

    \end{enumerate}
    \noindent If case (2) holds, ${\mathcal N}(x) = 1 - \langle \alpha_k,
    \alpha_p^\vee \rangle \in [1, 4]$.

\end{theorem}

Before proving Theorem (\ref{reduce by L and R}), we classify the minimal
elements of the posets $(\Gamma(W), \xlongrightarrow{L})$ and $(\Gamma(W),
\xlongrightarrow{R})$.

\subsection{Minimal elements of the posets \texorpdfstring{$(\Gamma(W),
    \xlongrightarrow{L})$}{(Gamma(W), -L->)} and \texorpdfstring{$(\Gamma(W),
        \xlongrightarrow{R})$}{(Gamma(W), -R->)}}

We will show that any given element $(w, i, j) \in \Gamma(W)$ can be reduced by
a sequence of reductions
\begin{equation*}
    (w, i, j) \xlongrightarrow{L} \cdots \xlongrightarrow{L} (v, i^\prime, j)
    \xlongrightarrow{R} \cdots \xlongrightarrow{R} (u, i^\prime, j^\prime)
\end{equation*}
such that $vs_j$ has only one left descent and $s_{i^\prime}u$ has only one
right descent.  Such a Weyl group element $v$ has at most two left descents.
Similarly, $u$ has at most two right descents. We prove in Lemma
(\ref{reduction, two right descents}) that if $u$ has two right descents, then
$v$ has two left descents and $u = w_0(a, b)$ for some $a, b \in \mathbf{I}$. On
the other hand, we show that if $v$ has only one left descent, then $u$ is
bigrassmannian (Lemma \ref{lemma, right descent reduction, 2}).

The following lemma is a well-known property of Weyl groups (see e.g.
\cite[Eqn. 2.60]{Knapp}), which we will tacitly use throughout this section.

\begin{lemma}

    \label{wj=kw}

    For $w\in W$ and $j, k \in \mathbf{I}$, $ws_j = s_kw$ if and only if
    $w(\alpha_j) \in \left\{-\alpha_k, \alpha_k\right\}$.

\end{lemma}

\noindent The following lemma gives a characterization of the elements
$w_0(a, b)$.

\begin{lemma}

    \label{lemma, characterization of w_0(a, b)}

    A Weyl group element $w \in W$ is of the form $w_0(i, j)$, for some
    distinct indices $i, j \in \mathbf{I}$, if and only if there exist $a, b
    \in \mathbf{I}$ with $a \neq b$ such that ${\mathcal D}_L(w) = {\mathcal
    D}_R(w) = \left\{a, b\right\}$ and each of $w s_a$ and $w s_b$ is
    bigrassmannian. In this case, $w = w_0(a, b)$.

\end{lemma}

\begin{proof}

    Suppose $w \in W$ such that ${\mathcal D}_{L}(w) = {\mathcal D}_R(w) =
    \left\{a, b \right\}$ and each of $w s_a$ and $w s_b$ is bigrassmannian.
    Since $a$ and $b$ are the only right descents of $w$, $w_0(a, b) \leq_L w$.
    Similarly, since $a$ and $b$ are the only left descents of $w$, $w_0(a, b)
    \leq_R w$.  For short, define $u := w_0(a, b) w$.  Since $u \leq_L w$, every
    right descent of $u$ is also a right descent of $w$. Thus ${\mathcal
    D}_R(u) \subseteq \left\{a, b \right\}$.  Suppose, to reach a
    contradiction, that $a\in {\mathcal D}_R(u)$. Thus, $\ell(w_0(a, b)) +
    \ell(us_a) = \ell(w_0(a,b)) + \ell(u) - 1 = \ell(w) - 1 = \ell(ws_a)$.  This
    implies $w_0(a, b) \leq_R ws_a$, and thus $a, b \in {\mathcal D}_L(ws_a)$.
    However, this contradicts that $ws_a$ is bigrassmannian.  Thus, $a \not\in
    {\mathcal D}_R(u)$. By swapping the roles of $a$ and $b$, we similarly
    obtain $b \not\in {\mathcal D}_R(u)$. Hence $u$ is the identity element.
    Therefore $w = w_0(a, b)$.

    Now suppose $w = w_0(a, b)$ for some distinct pair of indices $a, b \in
    \mathbf{I}$. It is a simple observation that $a$ and $b$ are the only left
    descents and right descents of $w$, and that $ws_a$ and $ws_b$ are
    bigrassmannian.

\end{proof}

The next theorem gives us the minimal elements of $(\Gamma(W),
\xlongrightarrow{L}$). It says, in effect, $(w, i, j) \in \Gamma(W)$ is a
minimal element if and only if $ws_j$ has only one left descent.

\begin{theorem}

    \label{theorem, left descent reduction}

    Let $(w, i, j) \in \Gamma(W)$. Then the following are equivalent.

    \begin{enumerate}

        \item There exists a reduction $(w, i, j) \xlongrightarrow{L} (v,
            i^\prime, j)$ for some $v\in V$ and $i^\prime \in \mathbf{I}$ such
            that $\ell(v) < \ell(w)$.

        \item The Weyl group element $ws_j$ has more than one left descent.

     \end{enumerate}

\end{theorem}

\begin{proof}

    We begin by showing statement (1) implies statement (2). Suppose $(w, i,
    j)$ and $(v, i^\prime, j)$ are as in statement (1). Since $v \leq_L w$,
    there exists $u\in W$ so that $uv = w$ and $\ell(u) + \ell(v) = \ell(w)$.
    Since $j$ is a right descent of $w$ and $v$, we get $\ell(u) + \ell(vs_j) =
    \ell(u) + \ell(v) - 1 = \ell(w) - 1 = \ell(w s_j)$. Since also $ws_j =
    uvs_j$, then $u \leq_R ws_j$.  Since $w^{-1}(\alpha_i) =
    v^{-1}(\alpha_{i^\prime})$, we obtain $u^{-1}(\alpha_i) = v w^{-1}
    (\alpha_i) = \alpha_{i^\prime} > 0$. Hence, $i$ is not a right descent of
    $u^{-1}$. Equivalently, $i$ is not a left descent of $u$. However, since $u
    \leq_R ws_j$, every left descent of $u$ is also a left descent of $ws_j$.
    Since $\ell(u) = \ell(w) - \ell(v) > 0$, $u$ must have some left descent
    $t$.  Hence $t,i \in {\mathcal D}_L(ws_j)$.

    Next we show that statement (2) implies statement (1). Since $ws_j$ has
    more than one left descent, there exists some $t \in \mathbf{I}$ with $t
    \neq i$ such that $w_0(i, t) \leq_R ws_j$ (recall the definition of $w_0(i,
    t)$ in (\ref{definition, s_ab})). For short, let $u = w_0(i, t)$.  Put
    $i^\prime = t$ wherever $s_i s_t$ has order $3$, or put $i^\prime = i$
    otherwise.  Define the Weyl group element $v := s_{i^\prime} u w$.  We will
    prove that there is a reduction $(w, i, j) \xlongrightarrow{L} (v,
    i^\prime, j)$ and that $\ell(v) < \ell(w)$.  In other words, we will show
    that $j \in {\mathcal D}_R(v)$, $i^\prime \in {\mathcal D}_L(vs_j)$, $v
    \leq_L w$, $\ell(v) < \ell(w)$, and $v^{-1}(\alpha_{i^\prime}) =
    w^{-1}(\alpha_i)$.

    We will first prove that $i^\prime \in {\mathcal D}_L(vs_j)$.  Since $u
    \leq_R ws_j$, $\ell(u) + \ell(u^{-1}ws_j) = \ell(ws_j)$.  However, $u =
    u^{-1}$. Hence, $\ell(uws_j) + \ell(u) = \ell(ws_j)$.  Also, since
    $i^\prime \in {\mathcal D}_R(u)$, we have $\ell(us_{i^\prime}) = \ell(u) -
    1$.  With these observations at hand, we prove now that $i^\prime \in
    {\mathcal D}_L(vs_j)$. This is equivalent to proving $i^\prime \not\in
    {\mathcal D}_L(uws_j)$. To reach a contradiction, suppose $i^\prime$ is a
    left descent of $uws_j$.  Hence, we obtain
    \[
        \ell(us_{i^\prime}) + \ell(vs_j) = \left(\ell(u) - 1 \right) +
        \left(\ell(uws_j) - 1\right) = \ell(ws_j) - 2 < \ell(ws_j) =
        \ell(us_{i^\prime}vs_j).
    \]
    That is to say, we conclude that the sum of the lengths of the two Weyl
    group elements, $us_{i^\prime}$ and $vs_j$, is strictly less than the
    length of their product $us_{i^\prime}vs_j$, which is absurd.

    Next, we prove $v \leq_L w$ and $\ell(v) < \ell(w)$.  We remark first that
    one can also prove $i^\prime \in {\mathcal D}_L(v)$ by replacing $vs_j$
    with $v$, and replacing $ws_j$ with $w$, everywhere in the proof above that
    $i^\prime \in {\mathcal D}_L(vs_j)$. Using $i^\prime \in {\mathcal D}_L(v)$
    in conjunction with $uw = s_{i^\prime} v$ and $u \leq_R w$ gives us
    $\ell(v) - 1 = \ell(s_{i^\prime} v) = \ell(uw) = \ell(w) - \ell(u)$.  Since
    $w = us_{i^\prime}v$ and
    \[
        \ell(w) = \ell(u) + \ell(v) - 1 = \ell(us_{i^\prime}) + \ell(v),
    \]
    it follows that $v \leq_L w$. Furthermore, we have a length formula for
    $v$, namely $\ell(v) = \ell(w) - \ell(u) + 1$. However, $\ell(u) > 1$.
    Hence, $\ell(v) < \ell(w)$.

    Next, we prove $j \in {\mathcal D}_R(v)$. Observe again $\ell(v) = \ell(w)
    - \ell(u) + 1$.  Using $vs_j = s_{i^\prime} uws_j$, $i^\prime \not\in
    {\mathcal D}_L(uws_j)$,  $u \leq_L ws_j$, and $j \in {\mathcal D}_R(w)$
    gives us
    \[
        \ell(vs_j) = \ell(s_{i^\prime}uws_j) = \ell(uws_j) + 1 = \ell(ws_j) -
        \ell(u) + 1 = \ell(w) - \ell(u).
    \]
    Hence $\ell(vs_j) = \ell(v) - 1$. Equivalently, $j \in {\mathcal D}_R(v)$.

    Lastly, we prove $w^{-1}(\alpha_i) = v^{-1}(\alpha_{i^\prime})$.  Since
    $(us_{i^\prime}) s_{i^\prime} = s_i (u s_{i^\prime})$, then $us_{i^\prime}
    (\alpha_{i^\prime})$ is equal to either $\alpha_i$ or $-\alpha_i$ by Lemma
    (\ref{wj=kw}).  However, $i^\prime$ is not a right descent of $u
    s_{i^\prime}$, and thus $u s_{i^\prime} (\alpha_{i^\prime})$ is a positive
    root. Hence, $us_{i^\prime} (\alpha_{i^\prime}) = \alpha_i$ and we conclude
    that $v^{-1} (\alpha_{i^\prime}) = w^{-1} u s_{i^\prime}
    (\alpha_{i^\prime}) = w^{-1} (\alpha_i)$.

\end{proof}

\noindent Invoking Proposition (\ref{L, R, equivalent, inverse}) gives us an
analogous result.

\begin{theorem}

    \label{theorem, right descent reduction, 1}

    Let $(w, i, j) \in \Gamma(W)$. Then the following are equivalent.

    \begin{enumerate}

        \item There exists a reduction $(w, i, j) \xlongrightarrow{R} (v,
            i, j^\prime)$ for some $v\in V$ and $j^\prime \in \mathbf{I}$ such
            that $\ell(v) < \ell(w)$.

        \item The Weyl group element $s_iw$ has more than one right descent.

     \end{enumerate}

\end{theorem}

\subsection{Proof of Theorem (\ref{reduce by L and R})}

Theorems (\ref{theorem, left descent reduction}) and (\ref{theorem, right
descent reduction, 1}) together imply that, starting with an arbitrary element
$(w, i, j) \in \Gamma(W)$, reductions can repeatedly be applied to obtain a
sequence of reductions
\begin{equation}
    \label{sequence of reductions, 1}
    (w, i, j) \xlongrightarrow{L} \cdots \xlongrightarrow{L} (v, i^\prime, j)
    \xlongrightarrow{R} \cdots \xlongrightarrow{R} (u, i^\prime, j^\prime)
\end{equation}
such that $vs_j$ has only one left descent and $s_{i^\prime}u$ has only one
right descent.  Such a Weyl group element $v$ can have at most two left
descents. Similarly, $u$ has at most two right descents. We need the following
lemma, which focuses on the case when $u$ has two right descents.

\begin{lemma}

    \label{reduction, two right descents}

    Suppose
    \[
        (v, i, j) \xlongrightarrow{R} \cdots \xlongrightarrow{R} (u, i, m),
    \]
    where $vs_j$ has only one left descent and $s_i u$ has only one right
    descent. Suppose also $u$ has (exactly) two right descents. Then there
    exists $p \in \mathbf{I}$ with $p \neq i$ such that

    \begin{enumerate}

        \item ${\mathcal D}_L(v) = \left\{i, p\right\}$, and

        \item $u = w_0(i, p)$.

    \end{enumerate}

    In this setting, $m = i$ if the order of $s_i s_p$ is $3$, and $m = p$
    otherwise.

\end{lemma}

\begin{proof}

    Let ${\mathcal D}_R(u) = \left\{m, k \right\}$.  Hence, a braid relation
    involving the simple reflections $s_{m}$ and $s_k$ can be performed in the
    rightmost positions of some reduced expression of $u$. Therefore $w_0(m, k)
    \leq_L u$. Since $m$ is the only right descent of $s_i u$, $s_i u = u s_k$
    and thus $i$ is a left descent of $u$.

    Since $us_k$ is not the identity element, $us_k$ must have at least one
    left descent, say $p \in \mathbf{I}$, with $p \neq i$. Thus $p$ is also a
    left descent of $u$.  However, since $u \leq_R v$, every left descent of
    $u$ is also a left descent of $v$.  Since $v$ has at most two left
    descents, $\left\{ i, p \right\} = {\mathcal D}_L(u) = {\mathcal D}_L(v)$.

    Since $vs_j$ has only one left descent, namely $i$, then $vs_j(\alpha_j) =
    \alpha_p$. Therefore, by Lemma (\ref{wj=kw}), $v s_j = s_p v$.  From the
    sequence of reductions
    \[
        (v, i, j) \xlongrightarrow{R} \cdots \xlongrightarrow{R} (u, i, m),
    \]
    it follows that $v(\alpha_j) = u(\alpha_{m})$.  Therefore, $v^{-1} u s_{m}
    = s_j v^{-1} u$. Thus,
    \[
        s_p u = s_p v \cdot v^{-1} u = v s_j \cdot v^{-1} u = v \cdot v^{-1} u
        s_{m} = u s_{m}.
    \]

    Next, let $y$ be the Weyl group element $s_p s_i$, $s_p s_i s_p$, $s_p s_i
    s_p s_i$, or $s_p s_i s_p s_i s_p s_i$ whenever the order of $s_{m} s_k$ is
    $2$, $3$, $4$, or $6$, respectively.  Analogously, let $z$ be $s_k s_{m}$,
    $s_{m} s_k s_{m}$, $s_k s_{m} s_k s_{m}$, or $s_k s_{m} s_k s_{m} s_k
    s_{m}$ when $s_i s_p$ has order $2$, $3$, $4$, or $6$, respectively.  From
    the identities $u s_{m} = s_p u$ and $u s_k = s_i u$, we obtain $u \cdot
    w_0(m, k) = y u$ and $w_0(i, p) u = u z$.  Since $w_0(m, k) \leq_L u$, then
    by the definition of the weak Bruhat order, $\ell(u \cdot w_0(m, k)) =
    \ell(u) - \ell(w_0(m, k))$. Hence,
    \[
        \ell (u) \leq \ell(y^{-1}) + \ell(yu) = \ell(y) + \ell(u \cdot w_0(m,
        k)) = \ell(u) + \ell(y) - \ell(w_0(m, k)).
    \]
    Therefore $\ell(y) \geq \ell(w_0(m, k))$ and thus the order of $s_{m} s_k$
    does not exceed the order of $s_i s_p$.  Similarly,
    \[
        \ell (u) \leq \ell(uz) + \ell(z^{-1}) = \ell(w_0(i, p) u) + \ell(z) =
        \ell(u) + \ell(z) - \ell(w_0(i, p))
    \]
    and we deduce that the order of $s_i s_p$ does not exceed the order of
    $s_{m} s_k$.  Therefore, $s_i s_p$ and $s_{m} s_k$ have the same order.
    Hence $y = w_0(i, p)$ and $z = w_0(m, k)$.

    Since $u \cdot w_0(m, k) \leq_R u$, every left descent of $u \cdot w_0(m,
    k)$ is a left descent of $u$. Thus, ${\mathcal D}_L\left( u \cdot w_0(m, k)
    \right) \subseteq \left\{ i, p \right\}$. However, since $u \cdot w_0(m, k)
    = w_0(i, p) u$, then neither $i$ nor $p$ are left descents of $u \cdot
    w_0(m, k)$.  Hence, $w_0(i, p) u$ ($= u \cdot w_0(m, k)$) lacks a left
    descent and thus $w_0(i, p) u$ is the identity element.  Equivalently, $u =
    w_0(i, p)$.

\end{proof}

The next lemma can be applied to handle the case when the Weyl group element
$v$ appearing in the chain of reductions (\ref{sequence of reductions, 1}) has
only one left descent. In this situation, there exists a sequence of reductions
\[
    (v, i^\prime, j) \xlongrightarrow{R} \cdots \xlongrightarrow{R} (u,
    i^\prime, j^\prime)
\]
with $u$ bigrassmannian.

\begin{lemma}

    \label{lemma, right descent reduction, 2}

    Let $(w, i, j) \in \Gamma(W)$, and suppose $i$ is the only left descent of
    $w$. Then the following are equivalent.

    \begin{enumerate}

        \item There exists a reduction $(w, i, j) \xlongrightarrow{R} (v, i,
            j^\prime)$ for some $v \in W$ and $j^\prime \in \mathbf{I}$ such
            that $\ell(v) < \ell(w)$ and $\left\{i \right\} = {\mathcal
            D}_L(v)$.

        \item The Weyl group element $w$ has more than one right descent.

    \end{enumerate}

\end{lemma}

\begin{proof}

    We begin by showing that statement (1) implies statement (2). Suppose $(w,
    i, j)$ and $(v, i, j^\prime)$ satisfy the conditions in statement (1).  In
    particular, we have $v \leq_R w$, and $v(\alpha_{j^\prime}) = w(\alpha_j)$,
    and $j \in {\mathcal D}_R(w)$.  Since $v \leq _R w$, there exists $u \in W$
    such that $vu = w$ and $\ell(u) = \ell(w) - \ell(v) > 0$.  We have
    $v(\alpha_{j^\prime}) = w(\alpha_j) = vu(\alpha_j)$.  Thus, $u(\alpha_j) =
    \alpha_{j^\prime} > 0$. Therefore, $j$ is not a right descent of $u$.
    However, since $u$ is not the identity element, then $u$ must have some
    right descent $t \in \mathbf{I}$ with $t \neq j$. Since $u \leq_L w$, then
    every right descent of $u$ is also a right descent of $w$. Thus, $j$ and
    $t$ are right descents of $w$.

    Next, we prove that statement (2) implies statement (1). Suppose now that
    $w$ has more than one right descent. Hence, there must exist some $t\in
    \mathbf{I}$ with $t \neq j$ such that $w_0(j, t) \leq_L w$ For short, let
    $u = w_0(j, t)$. Let $j^\prime \in \mathbf{I}$ be uniquely determined by
    the condition $s_{j^\prime} u = u s_{j}$ (i.e. $j^\prime = t$ if the order
    of $s_t s_j$ is $3$, and $j^\prime = j$ otherwise), and define $v :=
    wu^{-1}s_{j^\prime}$.

    Now we will show $(v, i, j^\prime) \in \Gamma(W)$ and that the reduction
    $(w, i, j) \xlongrightarrow{R} (v, i, j^\prime)$ holds. That is to say, we
    will show (1) $j^\prime$ is a right descent of $v$, (2) $i$ is a left
    descent of $vs_{j^\prime}$, (3) $v \leq_R w$, and (4) $v(\alpha_{j^\prime})
    = w(\alpha_j)$. Along the way, we also prove $\ell(v) < \ell(w)$.

    Observe first that $u(\alpha_j) = - \alpha_{j^\prime}$. Thus,
    $v(\alpha_{j^\prime}) = wu^{-1}s_{j^\prime}(\alpha_{j^\prime}) =
    w(\alpha_j) < 0$.  Hence $j^\prime$ is a right descent of $v$.
    Equivalently, $j^\prime$ is not a right descent of $vs_{j^\prime}$. Since
    $vs_{j^\prime} = wu^{-1}$, then $\ell(wu^{-1}s_{j^\prime}) = \ell(wu^{-1})
    + 1$. Hence, $\ell(v) = \ell(w u^{-1}) + 1$.  We also have $\ell(wu^{-1}) =
    \ell(w) - \ell(u)$ because $u \leq_L w$.  Therefore, we obtain the identity
    $\ell(v) = \ell(w) - \ell(u) + 1 = \ell(w) - \ell(us_j)$.  Since
    $\ell(us_j) > 0$, then $\ell(v) < \ell(w)$.  Furthermore, since $w =
    vus_j$, it follows that $v \leq_R w$.

    As it has been established already that $j^\prime$ is a right descent of
    $v$, then to prove $(v, i, j^\prime) \in \Gamma(W)$, it remains to show
    that $i$ is a left descent of $vs_{j^\prime}$. Since $w = vs_{j^\prime} u$
    and $\ell(vs_{j^{\prime}}) = \ell(v) - 1 =  \ell(w) - \ell(us_j) - 1 =
    \ell(w) - \ell(u)$, then $vs_{j^\prime} \leq_R w$. Thus, every left descent
    of $vs_{j^\prime}$ is a left descent of $w$. However, since $i$ is the only
    left descent of $w$, then either $vs_{j^{\prime}}$ lacks a left descent
    (equivalently $w = u$), or $i$ is the only left descent of $vs_{j^\prime}$.
    However, it is not possible that $w = u$ because it is assumed $w$ has only
    one left descent, yet $u$ has two descents. Thus, $i$ is a left descent of
    $vs_{j^\prime}$. Hence $(v, i, j^\prime) \in \Gamma(W)$.

    We conclude by proving $\left\{ i \right\} = {\mathcal D}_L(v)$. Since
    $v\leq_R w$, then either $\left\{ i \right\} = {\mathcal D}_L(v)$, or $v =
    \operatorname{id}$. However since we have already established that $v$ has
    a right descent, namely $j^\prime$, then $v \neq \operatorname{id}$.
    Therefore, $\left\{ i \right\} = {\mathcal D}_L(v)$.

\end{proof}

By invoking Proposition (\ref{L, R, equivalent, inverse}), the proof of Lemma
(\ref{lemma, right descent reduction, 2}) can be modified accordingly to give
us the following.

\begin{lemma}

    \label{lemma, left descent reduction, 2}

    Let $(w, i, j) \in \Gamma(W)$, and suppose $j$ is the only right descent of
    $w$. Then the following are equivalent.

    \begin{enumerate}

        \item There exists a reduction $(w, i, j) \xlongrightarrow{L} (v,
            i^\prime, j)$ for some $v \in W$ and $i^\prime \in \mathbf{I}$ such
            that $\ell(v) < \ell(w)$ and $\left\{j \right\} = {\mathcal
            D}_R(v)$.

        \item The Weyl group element $w$ has more than one left descent.

    \end{enumerate}

\end{lemma}

Now we are ready to prove Theorem (\ref{reduce by L and R}).

\begin{proof}[Proof of Theorem (\ref{reduce by L and R})]

    Assume $x = \llbracket u \rrbracket$, for some bigrassmannian $u \in W$.
    Suppose $\left\{i\right\} = {\mathcal D}_L(u)$ and $\left\{j\right\} =
    {\mathcal D}_R(u)$. Since $us_j \leq_R u$, every left descent of $us_j$ is
    also a left descent of $u$.  Therefore, $us_j$ has only one left descent.
    From Theorem (\ref{theorem, left descent reduction}), $x$ is a minimal
    element of $(\Gamma(W), \xlongrightarrow{L})$. Similarly, since $s_iu
    \leq_L u$, every right descent of $s_i u$ is a right descent of $u$, and
    thus $s_iu$ has only one right descent. By Theorem (\ref{theorem, right
    descent reduction, 1}), $x$ is a minimal element of $(\Gamma(W),
    \xlongrightarrow{R})$.  Recall, the partial order $\xlongrightarrow{}$ is
    defined to be the transitive closure of the union of $\xlongrightarrow{L}$
    and $\xlongrightarrow{R}$. Therefore, $x$ is also a minimal element of
    $(\Gamma(W), \xlongrightarrow{}$).

    Next, assume $x = (w_0(p, k), p, p^\prime)$, where $p, k, p^\prime \in
    \mathbf{I}$ are as described in the statement of Theorem (\ref{reduce by L
    and R}). From Lemma (\ref{lemma, characterization of w_0(a, b)}), $w_0(p,
    k) s_{p^\prime}$ is bigrassmannian. Therefore, by Theorem (\ref{theorem,
    left descent reduction}), $x$ is a minimal element of $(\Gamma(W),
    \xlongrightarrow{L})$. Furthermore, $s_p w_0(p, k)$ is bigrassmannian, and
    Theorem (\ref{theorem, right descent reduction, 1}) tells us $x$ is also a
    minimal element of $(\Gamma(W), \xlongrightarrow{R})$. Thus, $x$ is a
    minimal element of $(\Gamma(W), \xlongrightarrow{})$.

    Conversely, suppose now $x = (w, i, j)$ is a minimal element of
    $(\Gamma(W), \xlongrightarrow{}$). Therefore, $x$ is a minimal element of
    each of the partially ordered sets ($\Gamma(W), \xlongrightarrow{L})$ and
    $(\Gamma(W), \xlongrightarrow{R})$. Theorems (\ref{theorem, left descent
    reduction}) and (\ref{theorem, right descent reduction, 1}) imply $ws_j$
    has only one left descent and $s_iw$ has only one right descent. Hence $w$
    has at most two right descents.  If $w$ has only one right descent, Lemma
    (\ref{lemma, left descent reduction, 2}) implies $w$ is bigrassmannian.
    On the other hand, if $w$ has two right descents, then we can apply the
    results of Lemma (\ref{reduction, two right descents}) to conclude $w =
    w_0(i, p)$ for some $p \in \mathbf{I}$ with $p \neq i$.

    Finally, let $p, k, p^\prime \in \mathbf{I}$ again be as in the statement
    of this theorem. Since $w_0(p, k) s_{p^\prime} (\alpha_{p^\prime}) =
    \alpha_k$, then $T_{w_0(p, k) s_{p^\prime}} (E_{p^\prime})) = E_k$, and the
    $q$-Serre relations imply
    \[
        \begin{split}
            {\mathcal N}(x) &= \operatorname{min} \left\{ m \in \mathbb{N} :
            \left( \operatorname{ad}_q(E_p) \right)^m (T_{w_0(p, k)
            s_{p^\prime}} (E_{p^\prime})) = 0 \right\}
            \\
            &= \operatorname{min} \left\{ m \in \mathbb{N} :
            \left( \operatorname{ad}_q(E_p) \right)^m (E_k) = 0 \right\}
            \\
            &= 1 - \langle \alpha_k, \alpha_p \rangle.
        \end{split}
    \]

\end{proof}

\section{Bigrassmannian elements}

\label{section, bigrassmannian elements}

In Section (\ref{minimal elements}), we found a general characterization of the
minimal elements of the poset $(\Gamma(W), \xlongrightarrow{})$. In particular,
the minimal elements are those of the form $(w, i, j)$ with $w$ bigrassmannian
or $w = w_0(a, b)$ for some $a, b \in \mathbf{I}$. Our aim now is to obtain a
more explicit description of these elements. This motivates us to turn our
attention to bigrassmannian elements.

\subsection{The support of a bigrassmannian element}

\label{section: rank reduction}

Let $W$ be the Weyl group associated to a finite dimensional complex simple Lie
algebra $\mathfrak{g}$, and let $\mathbf{I}$ be the index set of
$\mathfrak{g}$. Recall, for a Weyl group element $w \in W$, we denote the
\textit{support} of $w$ by
\[
    \operatorname{supp}(w) := \left\{ i \in \mathbf{I} \mid s_i \leq w \text{
    with respect to the Bruhat order} \right\}.
\]
Let ${\mathcal G}(w)$ be the subgraph of the Dynkin diagram containing only
those vertices in $\operatorname{supp}(w)$, and containing only those edges
with endpoints at these particular vertices.

\begin{proposition}

    \label{graph, connected}

    If $w\in W$ has only one left or right descent, then ${\mathcal G}(w)$ is
    connected.

\end{proposition}

\begin{proof}

    Suppose ${\mathcal G}(w)$ is not connected. Hence, there exists a partition
    of $\mathbf{I}$, say $\mathbf{I} = \mathbf{I}_1 \cup \mathbf{I}_2$, with
    $\mathbf{I}_1$ and $\mathbf{I}_2$ nonempty such that for every pair of
    indices, $i$ and $j$, with $i \in \mathbf{I}_1$ and $j \in \mathbf{I}_2$,
    the vertices associated to $i$ and $j$ belong to different connected
    components. For every $i\in \mathbf{I}_1$ and $j \in \mathbf{I}_2$, the
    associated simple reflections $s_i$ and $s_j$ commute.  Therefore $w$ can
    be written as $w = w_1 \cdot w_2$ such that $w_1 \in \langle s_k : k \in
    \mathbf{I}_1 \rangle$, $w_2 \in \langle s_k : k \in \mathbf{I}_2 \rangle$,
    and $\ell(w_1) + \ell(w_2) = \ell(w)$. Since $w = w_1 \cdot w_2 = w_2 \cdot
    w_1$ and $w_1$ and $w_2$ are not the identity element, $w$ has more than
    one left descent and more than one right descent.

\end{proof}

In view of Proposition (\ref{graph, connected}), for every bigrassmannian
element $w \in W$, the graph ${\mathcal G}(w)$ is isomorphic to the Dynkin
diagram associated to a finite-dimensional complex simple Lie algebra
$\mathfrak{g}^\prime$ of rank $ \lvert \operatorname{supp}(w) \rvert$. Since
the Weyl group $W^\prime$ of $\mathfrak{g}^\prime$ is isomorphic to the
subgroup of $W$ generated by the simple reflections $s_i$, for $i \in
\operatorname{supp}(w)$, we can naturally view $w$ as an element of $W^\prime$.
In this scenario, we can consider relabelling the indices $i$ on the simple
reflections $s_i$ ($i \in \operatorname{supp}(w)$) to agree with the standard
labellings given to the simple reflections in $W^\prime$, as in \cite[Section
11.4]{Humphreys}. With this relabelling, $w$, as an element of $W^\prime$, has
full support. Formally, a relabelling of the indices is a bijection
\begin{equation}
    \label{definition, sigma}
    \sigma: \operatorname{supp}(w) \to \mathbf{I}^\prime,
\end{equation}
where $\mathbf{I}^\prime$ is the index set of $\mathfrak{g}^\prime$, that
preserves the symmetric form on the respective root lattices. In other words,
for each pair of indices $i, j
\in \operatorname{supp}(w) \subseteq \mathbf{I}$,
\begin{equation}
    \label{sigma, definition}
    \langle \alpha_i, \alpha_j^\vee \rangle_{\mathfrak{g}} = \langle
    \alpha_{\sigma(i)}, \alpha_{\sigma(j)}^\vee \rangle_{\mathfrak{g}^\prime},
\end{equation}
where the simple root $\alpha_i$, simple coroot $\alpha_j^\vee$, and symmetric
form appearing on the left-hand side are associated to the Lie algebra
$\mathfrak{g}$, whereas $\alpha_{\sigma(i)}$, $\alpha_{\sigma(j)}^\vee$, and
the symmetric form on the right-hand side correspond to $\mathfrak{g}^\prime$.
The function $\sigma$ defines a group isomorphism
\begin{equation}
    \label{sigma, group isomorphism}
    f_\sigma: \langle s_i : i \in \operatorname{supp}(w) \rangle
    \xlongrightarrow{\cong} W^\prime
\end{equation}
from the subgroup of $W$ generated by the simple reflections $s_i$ ($i \in
\operatorname{supp}(w)$) to $W^\prime$, where $f_\sigma(s_i) =
s_{\sigma(i)}$ for all $i \in \operatorname{supp}(w)$.

\begin{proposition}

    \label{prop, sigma}

    Suppose $(w, i, j) \in \Gamma(W)$.  Let $\sigma: \operatorname{supp}(w) \to
    \mathbf{I}^\prime$ be a relabelling as defined in (\ref{definition,
    sigma}), and let $f_\sigma$ be the induced group isomorphism as in
    (\ref{sigma, group isomorphism}).  Then
    \[
        {\mathcal N}(w, i, j) = {\mathcal N}(f_\sigma(w), \sigma(i),
        \sigma(j)).
    \]

\end{proposition}

\begin{proof}

    This follows from (\ref{sigma, definition}) and the definition of Lusztig
    symmetries.

\end{proof}

\subsection{A classification of bigrassmannian elements of full support}

By Proposition (\ref{graph, connected}), we may, without loss of generality,
classify bigrassmannian elements by their support. Recall, the support of a
bigrassmannian element $w$ can be identified with a connected subgraph of the
underlying Dynkin diagram.  This subgraph can also be identified with a Lie
type, say $X_n$, and with this identification, the bigrassmannian element $w$,
viewed as an element of the Weyl group of type $X_n$, has full support.  With
this, we will let $X_n$ denote the underlying Lie type. Define
$BiGr^{\circ}(X_n)$ as the set of bigrassmannian Weyl group elements having
full support,
\begin{equation}
    \label{definition, X}
    BiGr^{\circ}(X_n) := \left\{ w \in W(X_n) : | {\mathcal D}_L(w) | = |
    {\mathcal D}_R(w) | = 1 \text{ and } \operatorname{supp}(w) = \mathbf{I}
    \right\}
\end{equation}
In this section, we describe $BiGr^{\circ}(X_n)$ explicitly, for $X_n$ a
non-exceptional type, by listing all of its elements, and we also determine the
cardinality of $BiGr^\circ(X_n)$ for all Lie types.

We use the standard realizations of the Weyl groups of types $A_{n - 1}$,
$B_n$, $C_n$, and $D_n$ as groups of $n \times n$ permutation matrices or
signed permutation matrices and identify each simple reflection $s_i$ with a
signed permutation. Using one-line notation for permutations, we make the
identifications
\[
    s_i \longleftrightarrow [1, \dots, i - 1, i + 1, i, i + 2, \dots, n] \text{
    for } i < n,
\]
and
\[
    s_n \longleftrightarrow \begin{cases} [1,\dots,n-1, -n], &(\text{type }
    BC_n), \\ [1,\dots, n-2, -n, 1-n], &(\text{type } D_n). \end{cases}
\]
With this, we are able to associate a signed permutation to each Weyl group
element $w$ (in types $ABCD$). We will do this throughout in all that follows.
We also tacitly identify Weyl group elements $w$ with signed permutation
matrices.

As it turns out, each element $w \in BiGr^{\circ}(X_n)$ can be realized as a
(signed) permutation matrix in a certain block matrix form. First, for every
$k\in\mathbb{N}$ and $\epsilon \in \left\{-1, 1\right\}$, define the $k\times
k$ matrices
\[
    J_k :=
    \begin{pmatrix}
        0 & 0 & \cdots & 0 & 0 & -1 \\
        0 & 0 & \cdots & 0 & -1 & 0 \\
        0 & 0 & \cdots & -1 & 0 & 0 \\
        \vdots & \vdots & \iddots & \vdots & \vdots & \vdots \\
        0 & -1 & \cdots & 0 & 0 & 0 \\
        -1 & 0 & \cdots & 0 & 0 & 0
    \end{pmatrix}, \hspace{5mm}
    \operatorname{diag}_k(\epsilon) :=
    \begin{pmatrix}
        1 & 0 & 0 & \cdots & 0 & 0 \\
        0 & 1 & 0 & \cdots & 0 & 0 \\
        0 & 0 & 1 & \cdots & 0 & 0 \\
        \vdots & \vdots & \vdots & \ddots & \vdots & \vdots \\
        0 & 0 & 0 & \cdots & 1 & 0 \\
        0 & 0 & 0 & \cdots & 0 & \epsilon
    \end{pmatrix},
\]
and let $I_k$ be the $k \times k$ identity matrix.  For non-negative integers
$\ell, i, j, k, m \in \mathbb{Z}_{\geq 0}$ such that $\ell + i + j + k + m =
n$, define the $n \times n$ matrix
\begin{equation}
    \label{definition, ABCn, w_lijkm}
    w_{\ell,i,j,k,m} : =
    \begin{pmatrix}
        I_\ell & 0 & 0 & 0 & 0\\
        0 & 0 & 0 & I_k & 0 \\
        0 & 0 & J_j & 0 & 0 \\
        0 & I_i & 0 & 0 & 0 \\
        0 & 0 & 0 & 0 & I_m
    \end{pmatrix}.
\end{equation}
Each signed permutation matrix $w_{\ell, i, j, k, m}$ belongs to the Weyl group
of type $BC_n$, whereas $w_{\ell, i, j, k, m}$ belongs to the Weyl group of
type $A_{n - 1}$ only if $j = 0$.  The elements $w_{\ell, i, j, k, m}^\pm$
defined below are used to handle the type $D_n$ case, where the corresponding
Weyl group contains only those signed permutation matrices having an even
number of matrix entries equal to $1$,
\begin{equation}
    \label{definition, Dn, w_lijkm}
    w_{\ell,i,j,k,m}^{\pm} : = \operatorname{diag}_n (\pm 1) \cdot w_{\ell, i,
    j, k, m} \cdot \operatorname{diag}_n(\pm(-1)^j).
\end{equation}

\begin{proposition}

    \label{X, prop}

    Let $w_{\ell, i, j, k, m}$ and $w_{\ell, i, j, k, m}^{\pm}$ be the Weyl
    group elements as defined in (\ref{definition, ABCn, w_lijkm}) and
    (\ref{definition, Dn, w_lijkm}).

    \begin{enumerate}\itemsep=5mm

        \item For $n \geq 2$,
            $BiGr^{\circ}(A_{n-1}) = \left\{ w_{0,i,0,k,0} : i,k > 0 \text{ and
            } i + k = n \right\}$.

        \item For $n \geq 2$,
            $BiGr^{\circ}(BC_n) = \left\{ w_{0,i,j,k,m} : i + j + k + m = n
            \text{ and } j > 0 \right\}$.

        \item For $n \geq 4$,
            $BiGr^{\circ}(D_n) = \left\{ w_{0,i,j,k,m}^\pm :
            \begin{tabular}{l}
                $i + j + k + m = n$, $j > 0$, and\\ $\delta_{1,m} \delta_{0,ik}
                = \delta_{0,m} \delta_{0,ik} \delta_{1,j} = 0$
            \end{tabular}
            \right\}$.

    \end{enumerate}

\end{proposition}

\begin{proof}

    The length of a Weyl group element $w$ coincides with the number of
    positive roots that get sent to negative roots by $w$.  Viewing $W$ as a
    set of signed permutation matrices, we can determine the length of any Weyl
    group element $w$ by considering the relative positions of pairs of nonzero
    entries in the matrix realization of $w$.  For instance, in types $BCD_n$,
    two inversions are introduced every time there is a matrix entry equal to
    $-1$ in the $(i, j)$ position, and a nonzero matrix entry in the $(k,
    \ell)$ position with $i < k$ and $j < \ell$.  In types $A_{n - 1}$ and
    $BCD_n$, one additional inversion is introduced every time there are
    nonzero entries in positions $(i, j)$ and $(k, \ell)$ with $k < i$ and $j <
    \ell$.  Lastly, in type $BC_n$, an additional inversion occurs for each
    matrix entry equal to $-1$ .

    To determine which $i \in \mathbf{I}$ with $i < n$ are left descents or
    right descents, we can study how the length of $w$ changes upon swapping
    adjacent rows or adjacent columns, respectively. For all $i < n$, $i$ is a
    left descent of $w$ if and only if there exist integers $j$ and $k$ with
    $1\leq j < k \leq n$ such that either (1) there is $-1$ in the $(i, j)$
    position, and there is a nonzero entry in the $(i + 1, k)$ position, or (2)
    there is a $1$ in the $(i + 1, j)$ position, and there is a nonzero entry
    in the $(i, k)$ position.

    For types $BC_n$, $n$ is a left descent of $w$ if and only if the nonzero
    entry in the last row is $-1$, whereas in type $D_n$, $n$ is a left descent
    if and only if there exist integers $j$ and $k$ with $1 \leq j < k \leq n$
    such that either (1) there is $-1$ in position $(n - 1, j)$ and a nonzero
    entry in position $(n, k)$, or (2) there is $-1$ in position $(n, j)$ and a
    nonzero entry in position $(n - 1, k)$.

    Since $i\in \mathbf{I}$ is a right descent of $w$ if and only if $i$ is a
    left descent of $w^{-1}$ and the matrix representation of $w$ is
    orthogonal, this means the analogous conditions describing whether or not
    $i$ is a right descent can be determined by considering the transpose of
    $w$.

    In view of this, $w$ has only one left descent and only one right descent
    if and only if $w$ is of the form $w_{\ell, i, j, k, m}$ (or $w_{\ell, i,
    j, k, m}^{\pm}$ in type $D_n$), where $\ell + i + j + k + m = n$. In type
    $A_{n-1}$, $j = 0$ because the Weyl group in this situation is a set of
    permutation matrices, rather than a set of signed permutation matrices.
    That is to say, $1$ is the only nonzero entry in the matrix representation
    of every $w$.

    We have some restrictions on $\ell$, $i$, $j$, $k$, and $m$ in order for
    $w$ to have full support.  For instance, in type $A_{n-1}$, $n\in
    \operatorname{supp}(w)$ if and only if $m = 0$. In types $A_{n-1}$ and
    $BCD_n$, $1 \in \operatorname{supp}(w)$ if and only if $\ell = 0$, and for
    types $BCD_n$, $n \in \operatorname{supp}(w)$ if and only if $j > 0$.

    The condition $\delta_{1,m} \delta_{0,ik} = 0$ in type $D_n$ is necessary
    to prevent $n - 1$ and $n$ from both being left or right descents of $w$.
    In particular, if $m = 1$ and $i = 0$, then $n-1$ and $n$ are left descents
    of $w$, whereas if $m = 1$ and $k = 0$, $n-1$ and $n$ are both right
    descents.  Finally, the condition $\delta_{0,m} \delta_{0,ik} \delta_{1,j}
    = 0$ is needed in type $D_n$ to ensure that $n - 1 \in
    \operatorname{supp}(w)$.

\end{proof}

The next proposition tells us the cardinalities of $BiGr^{\circ}(X_n)$. We
remark that these can be computed using results of Geck and Kim \cite{GK},
where they found the cardinalities of $BiGr(X_n)$. As the rank $n$ increases,
the set of bigrassmannian elements $BiGr(X_n)$ becomes quite small compared to
the size of the Weyl group $W$. This means the left and right reduction
processes will reduce arbitrary elements $(w, i, j) \in \Gamma(W)$ to a
relatively few number of cases.

\begin{proposition}

    \label{X_n, cardinalities}

    Let $BiGr^{\circ}(X_n)$ be as defined in (\ref{definition, X}). Then

\begin{enumerate}\itemsep=5mm

        \item $\displaystyle{|BiGr^{\circ}(A_n)| = n}$ (for $n \geq 1$),

        \item $\displaystyle{|BiGr^{\circ}(BC_n)| =
            \frac{n(n+1)(n+2)}{6}}$ (for $n \geq 2$),

        \item $\displaystyle{\left|BiGr^{\circ}(D_n)\right| =
            \frac{(n-2)(n^2 + 8n - 15)}{6}}$ (for $n \geq 4$),

        \item (Exceptional Lie types)

            \vspace{1mm}

            $\left|BiGr^{\circ}(E_6)\right| = 119$,
            $\left|BiGr^{\circ}(E_7)\right| = 641$,
            $\left|BiGr^{\circ}(E_8)\right| = 7406$,

            \vspace{2mm}

            $\left|BiGr^{\circ}(F_4)\right| = 76$,
            $\left|BiGr^{\circ}(G_2)\right| = 8$.

    \end{enumerate}

\end{proposition}

\begin{proof}

    Referring to the description of $BiGr^{\circ}(A_n)$ given in Proposition
    (\ref{X, prop}), we have $BiGr^{\circ}(A_n) = \left\{ w_{0, 1, 0, n, 0},
    w_{0, 2, 0, n - 1, 0},\dots, w_{0, n, 0, 1, 0} \right\}$. Hence $\lvert
    BiGr^{\circ}(A_n) \rvert = n$.

    Next we determine the cardinalities $\lvert BiGr^{\circ}(BC_n) \rvert$ and
    $\lvert BiGr^{\circ}(D_n) \rvert$, but first we introduce some notation
    that will be used throughout this part of the proof. Define the set of
    tuples
    \[
        \mathbb{T}_n := \left\{ (i, j, k, m) \in \left( \mathbb{Z}_{\geq 0}
        \right)^4 : i + j + k + m = n \right\}.
    \]
    For each nonnegative integer $i\in \mathbb{Z}_{\geq 0}$, define
    $\delta_{i,*} := 1$ and $\delta_{i,+} := 1 - \delta_{i, 0}$, where
    $\delta_{i, 0}$ is the Kronecker delta, and for $z_1, z_2, z_3, z_4 \in
    \left\{ 0, 1, *, + \right\}$, let $\lvert z_1, z_2, z_3, z_4\rvert$ denote
    the cardinality of the set
    \[
        \left\{ (i,j,k,m) \in \mathbb{T}_n : \delta_{i, z_1} \delta_{j, z_2}
        \delta_{k, z_3} \delta_{m, z_4} = 1 \right\}.
    \]
    We write the binomial coefficients $B(a, b) := a! / ((a-b)! \cdot b!)$ for
    nonnegative integers $a, b \in \mathbb{Z}_{\geq 0}$ with $a \geq b$.

    The cardinality of $BiGr^{\circ}(BC_n)$ can be computed by subtracting the
    number of elements in $\mathbb{T}_n$ of the form $(i, 0, k, m)$ from
    $\lvert \mathbb{T}_n \rvert$.  Therefore,
    \[
        \lvert BiGr^{\circ}(BC_n) \rvert = \lvert *,*,*,*\rvert - \lvert
        *,0,*,*\rvert = B(n + 3, n) - B(n + 2, n) = B(n + 2, 3).
    \]
    In order to compute the cardinality of $BiGr^{\circ}(D_n)$, we begin by
    partitioning the set $BiGr^{\circ}(D_n)$ into a disjoint union of the sets
    $S_+ \!:=\!  \left\{ w_{0, i, j, k, m}^\pm \!\!\in BiGr^{\circ}(D_n) : m >
    0\right\}$ and $S_0 \!:=\! \left\{ w_{0, i, j, k, m}^\pm \!\!\in
    BiGr^{\circ}(D_n) : m = 0 \right\}$.  If $m > 0$, $w_{0, i, j, k, m}^+ =
    w_{0, i, j, k, m}^-$.  Hence, we can choose to write the elements in $S_+$
    using, say, only the $+$ symbol.  With this, $S_+$ can be characterized as
    the set containing the elements in $BiGr^{\circ}(D_n)$ of the form $w_{0,
    i, j, k , m}^+$, where $m > 0$.  Choosing to only use the $+$ symbol means
    there is now a canonical way to write any given element of $S_+$, and from
    the conditions on $i,j,k,m$ listed in the description of
    $BiGr^{\circ}(D_n)$ in Proposition (\ref{X, prop}), we conclude that
    $\lvert S_+ \rvert$ equals the cardinality of the set
    \[
        \widehat{S} = \left\{(i, j, k, m) \in \mathbb{T}_n : j > 0, m > 0,
        \delta_{1,m} \delta_{0, ik} = 0 \right\}.
    \]
    The set $\widehat{S}$ is the disjoint union of the sets $\widehat{S}_1:=
    \left\{(i,j,k,m) \in \mathbb{T}_n : j > 0, m > 1 \right\}$ and
    $\widehat{S}_2 := \left\{ (i,j,k,1) \in \mathbb{T}_n : i,j,k > 0\right\}$.
    By the inclusion-exclusion principle, we obtain $\lvert \widehat{S}_1
    \rvert = \lvert *,*,*,*\rvert - \lvert *,0,*,*\rvert - \lvert *,*,*,0\rvert
    + \lvert *,0,*,0\rvert - \lvert *,*,*,1\rvert + \lvert *,0,*,1\rvert$.
    Thus, $ \lvert \widehat{S}_1 \rvert = B(n + 3, n) - 2B(n + 2, n) + B(n + 1,
    n) - B(n + 1, n - 1) + B(n, n - 1)$. This simplifies to $\lvert
    \widehat{S}_1 \rvert = B(n, 3)$. Using similar arguments, it can be shown
    that $\lvert \widehat{S}_2 \rvert = \lvert +,+,+,1\rvert = B(n - 2, 2)$.
    Hence, $\lvert S_+ \rvert = B(n, 3) + B(n - 2, 2)$. Now we aim to compute
    $\lvert S_0 \rvert$.  Observe first that $w_{\ell, 1, j, k, 0}^\pm =
    w_{\ell, 0, j + 1, k, 0}^\mp$ and $w_{\ell, i, j, 1, 0}^\pm = w_{\ell, i, j
    + 1, 0, 0}^\pm$.  Thus, every element in $S_0$ can be written in such a way
    that $i > 0$ and $k > 0$.  The condition $\delta_{0, m} \delta_{0,ik}
    \delta_{1, j} = 0$ from Proposition \ref{X, prop} is automatically
    satisfied. Thus, $S_0$ is the set of elements $w_{0, i, j, k, 0}^\pm$ with
    $(i, j, k, 0) \in \mathbb{T}_n$ and $i, j, k > 0$.  Furthermore, any given
    element in $S_0$ has a unique way to write it in the form $w_{0, i, j, k,
    0}^\pm$ (with $i, j, k, > 0$). Hence, we have $\lvert S_0 \rvert = 2\lvert
    +, +, +, 0\rvert  = 2B(n - 1, 2)$. Finally, we compute $\lvert
    BiGr^{\circ}(D_n) \rvert$ by summing the cardinalities $\lvert S_+ \rvert$
    and $\lvert S_0 \rvert$.  This gives us $\lvert BiGr^{\circ}(D_n) \rvert =
    \frac{1}{6}(n - 2)(n^2 + 8n - 15)$.

    We will focus now on the exceptional types. First, $\lvert
    BiGr^{\circ}(G_2) \rvert = 8$ because every element of the Weyl group
    except the identity, the simple reflections, and the longest element are
    bigrassmannian elements of full support. We can find the remaining
    cardinalities $\lvert BiGr^{\circ}(E_6) \lvert$, $\lvert BiGr^{\circ}(E_7)
    \lvert$, $\lvert BiGr^{\circ}(E_8) \lvert$, and $\lvert BiGr^{\circ}(F_4)
    \lvert$ via a brute force approach, inspecting each element of the Weyl
    group one at a time, or by invoking the results of Geck and Kim found in
    \cite{GK}.  With the aid of a computer, the brute force approach can be
    completed in a relatively short amount of time except for type $E_8$, where
    the size of the Weyl group $W$ is too large to exhaustively study each
    individual element in a reasonable timespan. In this setting, one can
    instead search the poset $(W, \leq_L)$ breadth first, starting with the
    minimal element (i.e. the identity element). Whenever a new element $w$ is
    encountered in the search, determine first if it belongs to
    $BiGr^{\circ}(E_8)$, then either (1) append the covers of $w$ (with respect
    to the ordering $\leq_L$) to the search queue if $w$ has fewer than two
    right descents, or (2) append nothing to the search queue if $w$ has more
    than one right descent. This means that Weyl group elements greater than
    elements $w$ having more than one right descent will not be considered in
    this search. However we need not consider these elements anyway, because
    every element greater than such $w$ also has more than one right descent.
    This search strategy, in effect, significantly reduces the total number of
    Weyl group elements to inspect.

\end{proof}

\section{An orthogonality condition}

\label{section, orthogonality condition}

We begin this section by proving that every element in $\Gamma(W)$ of the form
$(w, i, j)$ with $w$ bigrassmannian can be reduced by a sequence of reductions
\[
    (w, i, j)  \xlongrightarrow{\hspace{2mm} * \hspace{2mm}} \cdots
    \xlongrightarrow{\hspace{2mm} * \hspace{2mm}} (v, i^\prime, j^\prime),
\]
using the relations $\xlongrightarrow{(k, L)}$ and $\xlongrightarrow{(k, R)}$
with $k \in \mathbf{I}$, to an element $(v, i^\prime, j^\prime)$ such that $v$
is bigrassmannian and also satisfies a certain \textit{orthogonality condition}
(which will be defined in Proposition (\ref{proposition, orthogonality
condition})).  The set of bigrassmannian elements satisfying this particular
orthogonality condition will be denoted
\[
    BiGr_\perp(X_n),
\]
where $X_n$ is the underlying Lie type.  The results of this section,
specifically Proposition (\ref{proposition, orthogonality condition}), imply
that elements of $\Gamma(W)$ of the form $(v, i^\prime, j^\prime)$ with $v \in
BiGr_\perp(X_n)$ cannot be further reduced under any of the relations
$\xlongrightarrow{\hspace{2mm} L \hspace{2mm}}$, $\xlongrightarrow{\hspace{2mm}
R \hspace{2mm}}$, $\xlongrightarrow{\hspace{2mm} (k, L) \hspace{2mm}}$,  or
$\xlongrightarrow{\hspace{2mm} (k, R) \hspace{2mm}}$ for any $k \in
\mathbf{I}$.

Recall we classify bigrassmannian elements by their support. The support of a
bigrassmannian element $w$ can be identified with a connected subgraph of the
underlying Dynkin diagram (see Proposition (\ref{graph, connected})).  Viewing
$w$ as an element of the Weyl group associated to this subgraph, it has full
support. With this, we define
\[
    BiGr_\perp^\circ(X_n) := \left\{ w \in BiGr_\perp(X_n) :
    \operatorname{supp}(w) = \mathbf{I} \right\}
\]
(i.e. the set of bigrassmannian elements of full support and satisfying the
orthogonality condition).

We give an explicit description of $BiGr_\perp^\circ(X_n)$ by listing all of
its elements. Proposition (\ref{Yn, description}) covers types $ABCD$, whereas
we forgo giving an explicit description for the exceptional types here, but
rather include this in Sections (\ref{section, G2}) and (\ref{section, F4}) for
types $G_2$ and $F_4$, respectively, and as an appendix (Appendix
(\ref{appendix, elements E_n})) for types $E_6$, $E_7$, and $E_8$.

As it turns out, $BiGr_\perp^\circ(X_n)$ has at most $5$ elements whenever
$X_n$ is of type $ABCD$. Moreover, $BiGr_\perp^\circ(A_n)$ is empty for $n > 3$
and $\lvert BiGr_\perp^\circ(BCD_n) \rvert = 4$ for $n > 4$.  Conveniently,
this means the problem of computing nilpotency indices is significantly
simplified by reducing to a small number of cases.

\noindent
\begin{minipage}{\textwidth}
\begin{theorem}

    \label{theorem, second stage reduction}

    Let $w\in BiGr(X_n)$, and assume $(w, i, j) \in \Gamma(W)$.

    \begin{enumerate}

        \item

            \label{theorem, second stage reduction, 1}

                Suppose there exists $k\in \mathbf{I} \backslash \left\{ i
                \right\}$ so that $s_ks_i = s_is_k$ and $s_kw = ws_p$ for some
                $p\in \mathbf{I}$ with $s_j s_p \neq s_p s_j$.  Then there
                exists $v \in W$ and $j^\prime \in \mathbf{I}$ such that
                \[
                    (w, i, j) \xlongrightarrow{\hspace{2mm} (k, R)
                    \hspace{2mm}} (v, i, j^\prime).
                \]

        \item

            \label{theorem, second stage reduction, 2}

            Suppose there exists $k\in \mathbf{I} \backslash \left\{ j
            \right\}$ so that $s_ks_j = s_js_k$ and $s_pw = ws_k$ for some
            $p\in \mathbf{I}$ with $s_i s_p \neq s_p s_i$. Then there exists $v
            \in W$ and $i^\prime \in \mathbf{I}$ such that
            \[
                (w, i, j) \xlongrightarrow{\hspace{2mm} (k, L) \hspace{2mm}}
                (v, i^\prime, j).
            \]

    \end{enumerate}

\end{theorem}
\end{minipage}

\begin{proof}

    We will prove (1). The proof of (2) is similar.  We begin by showing that
    $(s_kw, i , j) \in \Gamma(W)$, and that we have the reduction $(s_k, i, j)
    \xlongrightarrow{L} (w, i, j)$. Observe that since $i$ is the only left
    descent of $w$, and $k \neq i$, then $\ell(s_kw) = \ell(ws_p) = \ell(w) +
    1$.  Hence $w \leq_L s_kw$. This implies that every right descent of $w$ is
    also a right descent of $s_kw$. In particular, $j \in {\mathcal
    D}_R(s_kw)$.  Thus, $\ell(s_kws_j) = \ell(s_kw) - 1 = \ell(w)$.  Since $(w,
    i, j) \in \Gamma(W)$, this means that $j \in {\mathcal D}_R(w)$ and $i\in
    {\mathcal D}_L(ws_j)$.  Hence, we obtain $\ell(s_iws_j) = \ell(ws_j) - 1 =
    \ell(w) - 2$.  Furthermore, we have
    \[
        s_kws_j = s_ks_i \cdot s_iws_j = s_is_k \cdot s_iws_j.
    \]
    Thus, $\ell(s_kws_j) = \ell(s_iws_j) + \ell(s_is_k)$. Hence, $s_i \leq_R
    s_is_k \leq_R s_kws_j$. Therefore, $i \in {\mathcal D}_L(s_kws_j)$, and
    $(s_kw, i , j) \in \Gamma(W)$. Since $s_is_k = s_ks_i$, then it follows
    that $w \cdot (s_kw)^{-1}(\alpha_i) = s_k(\alpha_i) = \alpha_i$.  We
    established earlier that $w \leq_L s_kw$.  Hence, the reduction $(s_kw, i,
    j) \xlongrightarrow{L} (w, i, j)$ holds.

    For the remainder of this proof, let $x$ be the Weyl group element $s_p$,
    $s_js_p$, or $s_js_ps_js_p$ whenever $s_ps_j$ has order $3$, $4$, or $6$,
    respectively.  Put $y\in W$ equal to $s_js_p$, $s_ps_js_p$, or
    $s_ps_js_ps_js_p$ in these respective situations.  Let $j^\prime \in
    \mathbf{I}$ be uniquely determined by the condition $s_{j^\prime}y = ys_j$
    (i.e. $j^\prime = p$ whenever $s_js_p$ has order 3, and $j^\prime = j$ if
    $s_js_p$ has order $4$ or $6$), and define $v := w s_j x^{-1}
    s_{j^\prime}$.

    It remains to show that $(v, i, j^\prime) \in \Gamma(W)$, $v \leq_R w$,
    $\ell(v) < \ell(w)$, and that we have the reduction $(s_kw, i, j)
    \xlongrightarrow{R} (v, i, j^\prime)$.

    We proceed by proving next that $\ell(v) < \ell(w)$.  Since $p$ and $j$ are
    the only right descents of $s_kw$, it follows that $x \leq_L ws_j$.  We
    have $v s_{j^\prime} \cdot x = ws_j$. Hence, $\ell(v s_{j^\prime}) +
    \ell(x) = \ell(ws_j) = \ell(w) - 1$.  Therefore $\ell(v s_{j^\prime}) <
    \ell(w) - 1$. Hence $\ell(v) < \ell(w)$.

    We find it useful to next show that $v s_{j^\prime} \leq_R w$. We also
    establish a couple of helpful identities involving lengths of various Weyl
    group elements. Observe that $xs_js_p = ys_j = s_{j^\prime}y$.  Thus, we
    have
    \begin{equation}
        \label{eqn 1}
        s_k w = w s_j x^{-1} \cdot x \cdot s_j s_p = w s_j x^{-1} \cdot y s_j =
        w s_j x^{-1} \cdot s_{j^\prime} y = v \cdot y.
    \end{equation}
    Hence, $w = v s_{j^\prime} \cdot y s_j s_p$. Since $\ell(ys_js_p) = \ell(y)
    = \ell(x) + 1$, we have
    \begin{equation}
        \label{eqn 2}
        \ell(v s_{j^\prime}) + \ell(ys_js_p) = \ell(v s_{j^\prime} ) + \ell(x)
        + 1 = \ell(v s_{j^\prime}) + \ell(y) = \ell(w).
    \end{equation}
    Thus, $v s_{j^\prime} \leq_R w$.

    We prove next that $(v, i, j^\prime) \in \Gamma(W)$. That is to say, we
    will show $j^\prime \in {\mathcal D}_R(v)$ and $i \in {\mathcal D}_L(v
    s_{j^\prime})$. In fact, we will prove the stronger result that $\left\{ i
    \right\} = {\mathcal D}_L(v s_{j^\prime})$.  We begin first by proving that
    $j^\prime \in {\mathcal D}_R(v)$. We have already established that
    $\ell(s_k w) = \ell(w ) + 1$. Since $\ell(v s_{j^\prime}) = \ell(w) -
    \ell(y)$ (from (\ref{eqn 2}) above), then either $\ell(v) = \ell(w) -
    \ell(y) - 1$, or $\ell(v) = \ell(w) - \ell(y) + 1$.  However, we have $s_kw
    = v \cdot y$ (from (\ref{eqn 1}) above). Thus $\ell(v \cdot y) = \ell(s_kw)
    = \ell(w) + 1$.  Since $\ell(w) + 1 = \ell(v y) \leq \ell(v) + \ell(y)$,
    then $\ell(v) + \ell(y) = \ell(w) + 1$.  Thus $\ell(v s_{j^\prime}) =
    \ell(v) - 1$. That is to say, $j^\prime \in {\mathcal D}_R(v)$.  Now we
    prove that $\left\{ i \right\} = {\mathcal D}_L(v s_{j^\prime})$.  Since $v
    s_{j^\prime} \leq_R w$, every left descent of $v s_{j^\prime}$ is a left
    descent of $w$.  However, since $i$ is the only left descent of $w$, then
    either $v s_{j^\prime}$ lacks a left descent, or $\left\{i \right\} =
    {\mathcal D}_L(v s_{j^\prime})$.  To reach a contradiction, suppose $v
    s_{j^\prime}$ lacks a left descent. Thus, $v s_{j^\prime}$ is the identity
    element.  In this situation, $w = xs_j$ (because $v = w s_j x^{-1}
    s_{j^\prime}$).  Hence, $w$ is equal to $s_p s_j$, $s_j s_p s_j$, or $s_j
    s_p s_j s_p s_j$ whenever the order of $s_j s_p$ is $3$, $4$, or $6$,
    respectively. In all cases, $j^\prime$ is the only left descent of $w$
    ($=xs_j$). That is to say, $j^\prime = i$. We have $w s_p = s_jw$ when the
    order of $s_j s_p$ is $3$, whereas $w s_p = s_p w$ if the order of $s_j
    s_p$ is $4$ or $6$.  Hence, $k = j$ if $s_j s_p$ has order $3$, and $k = p$
    if $s_j s_p$ has order $4$ or $6$. This contradicts the hypothesis that
    $s_i s_k = s_k s_i$. Hence $\left\{i \right\} = {\mathcal D}_L(v
    s_{j^\prime})$. This concludes the proof that $(v, i, j^\prime) \in
    \Gamma(W)$.

    Next, we verify that the reduction $(s_k w, i, j) \xlongrightarrow{R} (v,
    i, j^\prime)$ holds. This means we will show that $v \leq_R s_k w$ and $s_k
    w(\alpha_j) = v(\alpha_{j^\prime})$.  We have already established that
    $j^\prime \in {\mathcal D}_R(v)$. Hence, $\ell(vs_{j^\prime}) = \ell(v) -
    1$. We have also observed that $\ell(s_k w) = \ell(w) + 1$. These
    identities together with (\ref{eqn 2}) imply $\ell(s_k w) = \ell(v) +
    \ell(y)$. Since $s_k w = v y$, $v \leq_R s_k w$.  Additionally, we have
    $s_kw(\alpha_j) = v y (\alpha_j) = v(\alpha_{j^\prime})$.  Hence, we have
    the reduction $(s_kw, i, j) \xlongrightarrow{R} (v, i, j^\prime)$.

    To prove that $v \leq _R w$, we observe first that, from (\ref{eqn 1}), we
    have $w = v y s_p$. Using (\ref{eqn 2}), together with the identities
    $\ell(v s_{j^\prime}) = \ell(v) - 1$ and $\ell(ys_p) = \ell(y) - 1$, we
    conclude $\ell(v) + \ell(ys_p) = \ell(v s_{j^\prime}) + \ell(y) = \ell(w)$.
    Thus, $v \leq_R w$.

\end{proof}

The following proposition gives a characterization of those $(w, i, j) \in
\Gamma(W)$ with $w \in BiGr(X_n)$ not involved in any reduction of the form
$(w, i, j) \xlongrightarrow{\hspace{2mm} (k, R) \hspace{2mm}} (v, i, j^\prime)$
or $(w, i, j) \xlongrightarrow{\hspace{2mm} (k, L) \hspace{2mm}} (v, i^\prime,
j)$ for any $k\in \mathbf{I}$.

\begin{proposition}

    \label{proposition, orthogonality condition}

    Suppose $w \in BiGr(X_n)$ and $(w, i, j) \in \Gamma(W)$. The
    following are equivalent.

    \begin{enumerate}

        \item (Orthogonality condition) For every simple root $\alpha \in \Pi$,
            \[
                (\langle \alpha, \alpha_j \rangle = 0 \text{ and } w(\alpha)
                \in \Pi) \text{ implies } \langle w(\alpha), \alpha_i \rangle =
                0,
            \]
            and
            \[
                (\langle \alpha, \alpha_i \rangle = 0 \text{ and }
                w^{-1}(\alpha) \in \Pi) \text{ implies } \langle
                w^{-1}(\alpha), \alpha_j \rangle = 0.
            \]

        \item The element $(w, i, j) \in \Gamma(W)$ is not involved in any
            reduction of the form $(w, i, j) \xlongrightarrow{\hspace{2mm} (k,
            R) \hspace{2mm}} (v, i^\prime, j^\prime)$ or $(w, i, j)
            \xlongrightarrow{\hspace{2mm} (k, L) \hspace{2mm}} (v, i^\prime,
            j^\prime)$ for any $k\in \mathbf{I}$.

    \end{enumerate}

\end{proposition}

\begin{proof}

    Using Lemma (\ref{wj=kw}), the statements in Theorem (\ref{theorem, second
    stage reduction}) can be rephrased so that they state, in effect, that (2)
    implies (1).

    Next we prove (1) implies (2). Suppose there exists an index $k \in
    \mathbf{I}$ such that $(w, i, j) \in \Gamma(W)$ is involved in a reduction
    of the form $(w, i, j) \xlongrightarrow{\hspace{2mm} (k, R) \hspace{2mm}}
    (v, i^\prime, j^\prime)$ or $(w, i, j) \xlongrightarrow{\hspace{2mm} (k, L)
    \hspace{2mm}} (v, i^\prime, j^\prime)$. Let us assume $(w, i, j)
    \xlongrightarrow{\hspace{2mm} (k, L) \hspace{2mm}} (v, i^\prime,
    j^\prime)$. We will show the simple root $\alpha_k$ violates the conditions
    in statement (1). Specifically, we will prove $\langle \alpha_k , \alpha_j
    \rangle = 0$, $w (\alpha_k) \in \Pi$, and $\langle w(\alpha_k), \alpha_i
    \rangle \neq 0$. If we were to assume instead that $(w, i, j)$ is involved
    in a reduction of the form $(w, i, j) \xlongrightarrow{\hspace{2mm} (k, R)
    \hspace{2mm}} (v, i^\prime, j^\prime)$, we can proof in a similar manner
    that $\langle \alpha_k, \alpha_i \rangle = 0$, $w^{-1}(\alpha_k) \in \Pi$,
    and $\langle w^{-1}(\alpha_k, \alpha_j \rangle \neq 0$.

    With this, suppose $(w, i, j) \xlongrightarrow{\hspace{2mm} (k, L)
    \hspace{2mm}} (v, i^\prime, j^\prime)$. By definition (\ref{definition, (k,
    L) reduction}), $\ell(v) < \ell(w)$, $v \leq_L w$, $(ws_k, i, j)
    \xlongrightarrow{\hspace{2mm} R \hspace{2mm}} (w, i, j)$, and $(ws_k, i, j)
    \xlongrightarrow{\hspace{2mm} L \hspace{2mm}} (v, i^\prime, j^\prime)$.
    Thus, $j = j^\prime$, $w \leq_R w s_k$, $ws_k(\alpha_j) = w(\alpha_j)$, $v
    \leq_L ws_k$, and $v(\alpha_{i^\prime}) = ws_k(\alpha_i)$.  Hence, we have
    $s_k (\alpha_j) = \alpha_j$. Therefore, $\langle \alpha_k, \alpha_j \rangle
    = 0$ and $s_j s_k = s_k s_j$. By Theorem (\ref{theorem, left descent
    reduction}), $ws_ks_j$ has more than one left descent. Since $i$ is the
    only left descent of $w$, this implies that at least one of $w(\alpha_k)$
    or $ws_k(\alpha_j)$ is a simple root. However, since $j$ is a right descent
    of $w$, then $ws_k(\alpha_j) = w(\alpha_j) < 0$. In particular,
    $ws_k(\alpha_j)$ is not a simple root.  Thus $w(\alpha_k)$ is a simple
    root, say $\alpha_p$.  Therefore, $s_p w = ws_k$.

    It remains to prove that $\langle \alpha_p, \alpha_i \rangle \neq 0$.
    Observe that $i$ and $p$ are the only left descents of $ws_k$.  To reach a
    contradiction, suppose $\langle \alpha_p, \alpha_i \rangle = 0$.
    Therefore, it follows that $s_p s_i = s_i s_p$, and in this setting,
    $(ws_k, i, j) \xlongrightarrow{\hspace{2mm} L \hspace{2mm}} x$ only if $x =
    (ws_k, i, j)$ or $x = (w, i, j)$.  Since $w$ has only one left descent,
    $(w, i, j)$ cannot be further reduced under $\xlongrightarrow{L}$. This
    contradicts $(ws_k, i, j) \xlongrightarrow{\hspace{2mm} L \hspace{2mm}} (v,
    i^\prime, j^\prime)$ for some $v\in W$ with $\ell(v) < \ell(w)$.

\end{proof}

Recall the definitions (\ref{definition, ABCn, w_lijkm}) and (\ref{definition,
Dn, w_lijkm}) of $w_{\ell,i,j,k,m}$ and $w_{\ell,i,j,k,m}^\pm$.

\begin{proposition}

    \label{second stage reductions, ABCD}

    $ $

    \begin{enumerate}

        \item We have the following reductions in type $A_n$ ($n > 1$):

            \begin{enumerate}

                \item $\llbracket w_{0, i, 0, k, 0} \rrbracket \,
                    \xlongrightarrow{ \hspace{3mm} (1, R)\hspace{3mm}}
                    \llbracket w_{1, i, 0, k - 1, 0} \rrbracket$ \hspace{12pt}
                    $(k > 2)$,

                \item $\llbracket w_{0, i, 0, k, 0} \rrbracket \,
                    \xlongrightarrow{ \hspace{3mm} (n, R) \hspace{2.5mm}}
                    \llbracket w_{0, i - 1, 0, k, 1} \rrbracket$ \hspace{12pt}
                    $(i > 2)$.

            \end{enumerate}

        \vspace{4mm}

        \item We have the following reductions in type $BC_n$ ($n > 1$):

            \begin{enumerate}

                \item $\llbracket w_{0, i, j, k, m} \rrbracket
                    \xlongrightarrow{ \hspace{2.8mm} (1, R)\hspace{3mm}}
                    \llbracket w_{1, i, j, k - 1, m} \rrbracket$ \hspace{11pt}
                    $(k > 1)$,

                \item $\llbracket w_{0, i, j, 0, m} \rrbracket
                    \xlongrightarrow{ \hspace{3mm} (1, R) \hspace{3mm}}
                    \llbracket w_{1, i, j - 1, 0, m} \rrbracket$ \hspace{11pt}
                    $(j > 2)$,

                \item $\llbracket w_{0, i, j, 0, 1} \rrbracket \hspace{1.1mm}
                    \xlongrightarrow{ \hspace{2.6mm} (n, R) \hspace{3mm}}
                    \llbracket w_{0, i + 1, 0, j, 0} \rrbracket$ \hspace{15pt}
                    $(i > 0)$,

                \item $\llbracket w_{0, i, j, 0, m} \rrbracket \!
                    \xlongrightarrow{ (i + j + 1, R)} \llbracket w_{0, i + 1,
                    j, 0, m-1} \rrbracket$ \hspace{80000sp} $(m > 1$ and $i >
                    0)$,

                \item $\llbracket w_{0,i,j,k,m} \rrbracket \xlongrightarrow{
                    \hspace{1.2mm} (k + 1, R) \hspace{1mm}} \llbracket w_{0, i,
                    j - 1, k + 1, m} \rrbracket$ \hspace{30000sp} $(k > 0$ and
                    $j > 2)$.

            \end{enumerate}

        \vspace{4mm}

        \item We have the following reductions in Type $D_n$ ($n > 3$):

            \begin{enumerate}

                \item $\llbracket w_{0, i, j, k, m}^+ \rrbracket
                    \xlongrightarrow{\hspace{2.9mm} (1, R) \hspace{3mm}}
                    \llbracket w_{1, i, j, k - 1, m}^+ \rrbracket$
                    \hspace{12pt} $(k > 1)$,
                    \\
                    $\llbracket w_{0, i, j, k, 0}^- \rrbracket \hspace{1mm}
                    \xlongrightarrow{\hspace{2.9mm} (1, R) \hspace{3mm}}
                    \llbracket w_{1, i, j, k - 1, 0}^- \rrbracket$
                    \hspace{15pt} $(k > 1$ and $i > 1)$,

                \item $\llbracket w_{0,i,j,0,m}^+ \rrbracket
                    \xlongrightarrow{\hspace{2.9mm} (1, R) \hspace{2.9mm}}
                    \llbracket w_{1,i,j-1,0,m}^+ \rrbracket$ \hspace{12pt} $(j
                    > 2$ and $m > 1)$,

                \item $\llbracket w_{0,i,j,0,2}^+ \rrbracket \hspace{1mm}
                    \xlongrightarrow{\hspace{2.9mm} (n, R) \hspace{3mm}}
                    \llbracket w_{0, i + 2, 0, j, 0}^+ \rrbracket$
                    \hspace{15pt} $(i > 0)$,
                    \\
                    $\llbracket w_{0,i,j,0,2}^- \rrbracket \hspace{1mm}
                    \xlongrightarrow{\hspace{1.3mm} (n - 1, R) \hspace{1.2mm}}
                    \llbracket w_{0, i + 2, 0, j, 0}^- \rrbracket$
                    \hspace{15pt} $(i > 0)$,

                \item $\llbracket w_{0,i,j,0,m}^+ \rrbracket
                    \xlongrightarrow{(i + j + 1, R)} \llbracket w_{0, i + 1, j,
                    0, m - 1}^+ \rrbracket$ \hspace{1pt} $(i > 0$ and $m > 2)$,

                \item $\llbracket w_{0, i, j, k, m}^+ \rrbracket
                    \xlongrightarrow{\hspace{1.4mm} (k + 1, R) \hspace{1.4mm}}
                    \llbracket w_{0, i, j - 1, k + 1, m}^+ \rrbracket$
                    \hspace{0pt} $(j > 2$ and $0 < k < n - 3)$,
                    \\
                    $\llbracket w_{0, i, j, k, 0}^- \rrbracket \hspace{1mm}
                    \xlongrightarrow{\hspace{1.4mm} (k + 1, R) \hspace{1.4mm}}
                    \llbracket w_{0, i, j - 1, k + 1, 0}^- \rrbracket
                    \hspace{1mm}$ \hspace{30000sp} $(i > 1$ and $j >  2$ and $k
                    > 0)$.

            \end{enumerate}

    \end{enumerate}

\end{proposition}

\begin{proof}

    In general, to show a reduction $(w, r, t) \xlongrightarrow{(a, R)} (v, r,
    t^\prime)$ holds means we must verify the following four statements: (1) $w
    \leq_L s_a w$, (2) $(s_a w)^{-1}(\alpha_r) = w^{-1}(\alpha_r)$, (3) $v
    \leq_R s_a w$, and (4) $s_a w (\alpha_t) = v(\alpha_{t^\prime})$.
    Statements (1) and (2) will follow provided $a \not\in {\mathcal D}_L(w)$
    and $s_a s_r = s_r s_a$, and these two statements will be straightforward
    to verify in the context of this proposition, but we first need to know the
    left descents of the Weyl group elements $w_{\ell, i, j, k, m}$ and
    $w_{\ell, i, j, k, m}^\pm$.  We also need to identify the right descents,
    $t$ and $t^\prime$, of $w$ and $v$, respectively, to verify statement (4).
    Once the right descents $t$ and $t^\prime$ are determined, we will confirm
    that $\left(v^{-1}s_a w\right) s_t = s_{t^\prime} \left(v^{-1} s_a
    w\right)$. This would imply $v^{-1}s_aw(\alpha_t) \in \left\{\pm
    \alpha_{t^\prime} \right\}$. If also $t$ is not a right descent of
    $v^{-1}s_aw$, then $v^{-1}s_aw(\alpha_t) = \alpha_{t^\prime}$.  Lastly, we
    must determine the lengths of the Weyl group elements $w_{\ell, i, j, k,
    m}$ and $w_{\ell, i, j, k, m}^\pm$ involved in this proposition, because,
    assuming $a$ is not a left descent of $w$, statement (3) is equivalent to
    $\ell(v^{-1} s_a w) = \ell(w) - \ell(v) + 1$.  Since the Weyl group element
    $v^{-1} s_a w$ plays a key role, we explicitly calculate it in each case
    relevant to this proposition (see Table (\ref{table, proof})).

    \begin{table}[!h]
        \caption{ Proof of Proposition (\ref{second stage reductions, ABCD})
        (data)}
        \vspace{-10pt}
    \begin{tabular}{llllll}
        \hline
        & $w$ & $v$ & $a$ & $v^{-1}s_aw$ &
        \\
        \hline
        1.a & $w_{0,i,0,k,0}$ & $w_{1,i,0,k-1,0}$ & $1$ &$w_{0, i + 1, 0, 1, k
        - 2}$&
        \\
        1.b & $w_{0,i,0,k,0}$ & $w_{0,i-1,0,k,1}$ & $n$ & $w_{i - 2,1, 0, k +
        1, 0}$&
        \\
        2.a & $w_{0,i,j,k,m}$ & $w_{1,i,j,k-1,m}$ & $1$ & $w_{0, i + j + 1, 0,
        1, k + m - 2}$ &
        \\
        2.b & $w_{0,i,j,0,m}$ & $w_{1,i,j-1,0,m}$ & $1$ & $w_{0, i + j - 2, 1,
        0, m + 1}$ &
        \\
        2.c & $w_{0,i,j,0,1}$ & $w_{0,i+1,0,j,0}$ & $n$ & $w_{i, 0, j + 1, 0,
        0}$ &
        \\
        2.d & $w_{0,i,j,0,m}$ & $w_{0,i+1,j,0,m-1}$ & $i + j + 1$ & $w_{i, j +
        1, 0, 1, m - 2}$ &
        \\
        2.e & $w_{0,i,j,k,m}$ & $w_{0,i,j-1,k+1,m}$ & $k + 1$ & $w_{i + j - 2,
        0, 1, k + 1, m}$ &
        \\
        3.a & $w_{0,i,j,k,m}^+$ & $w_{1,i,j,k-1,m}^+$ & $1$ & $w_{0, i + j + 1,
        0, 1, k + m - 2}^+$ &
        \\
            & $w_{0,i,j,k,0}^-$ & $w_{1,i,j,k-1,0}^-$ & $1$ & $w_{0, i + j + 1,
            0, 1, k - 2}^-$ &
        \\
        3.b & $w_{0,i,j,0,m}^+$ & $w_{1,i,j-1,0,m}^+$ & $1$ & $w_{i + j - 2, 0,
        1, 0, m + 1}^+$  &
        \\
        3.c & $w_{0,i,j,0,2}^+$ & $w_{0,i+2,0,j,0}^+$ & $n$ & $w_{i, 0, j+2, 0,
        0}^+$ &
        \\
            & $w_{0,i,j,0,2}^-$ & $w_{0,i+2,0,j,0}^-$ & $n - 1$ & $w_{i, 0,
            j+2, 0, 0}^-$ &
        \\
        3.d & $w_{0,i,j,0,m}^+$ & $w_{0,i+1,j,0,m-1}^+$ & $i + j + 1$ & $w_{i,
        j + 1, 0, 1, m - 2}^+$ &
        \\
        3.e & $w_{0,i,j,k,m}^+$ & $w_{0,i,j-1,k+1,m}^+$ & $k + 1$ & $w_{i + j -
        2, 0, 1, k + 1, m}^+$ &
        \\
            & $w_{0,i,j,k,0}^-$ & $w_{0,i,j-1,k+1,0}^-$ & $k + 1$ & $w_{i + j -
            2, 0, 1, k + 1, 0}^-$ & \vspace{1mm}
        \\
        \hline
    \end{tabular}
    \label{table, proof}
    \end{table}

    The proof of Proposition (\ref{X, prop}) describes how to determine the
    left and right descents of $w_{\ell, i, j, k, m}$ and $w_{\ell, i, j, k,
    m}^\pm$. As it turns out, each Weyl group element involved in the statement
    of this proposition has only one left descent and only one right descent.
    In particular, the elements $w_{\ell, i, j, k, m}$ and $w_{\ell, i, j, k,
    m}^+$ have $\ell + j + k$ as the only left descent, and $\ell + i + j$ as
    the only right descent. The left and right descents of the elements of the
    form $w_{\ell, i, j, k, m}^-$ can be determined by first finding the left
    and right descents of $w_{\ell, i, j, k, m}^+$, then swapping $n - 1$ and
    $n$ everywhere. In other words, $n - 1$ is a left (right) descent of
    $w_{\ell, i, j, k, m}^-$ if and only if $n$ is a left (right) descent of
    $w_{\ell, i, j, k, m}^+$.  Analogously, $n$ is a left (right) descent of
    $w_{\ell, i, j, k, m}^-$ if and only if $n - 1$ is a left (right) descent
    of $w_{\ell, i, j, k, m}^+$.  In view of this, we can verify in each case
    that $a$ is not a left descent of $w$, and $s_a s_r = s_r s_a$, where $r$
    is the left descent of $w$. That is to say, statements (1) and (2) in the
    opening sentence of this proof can be confirmed on a case by case basis. We
    can also verify that in each case, the right descent $t$ of $w$ is not a
    right descent of $v^{-1}s_a w$.

    It remains to show that in every case, we have $\ell(v^{-1}s_aw) = \ell(w)
    - \ell(v) + 1$ and $\left(v^{-1}s_aw \right) s_t = s_{t^\prime}
    \left(v^{-1} s_a w \right)$ (where $t$ and $t^\prime$ are the right
    descents of $w$ and $v$, respectively). The first paragraph in the proof of
    Proposition (\ref{X, prop}) describes how to compute the lengths of the
    Weyl group elements $w_{\ell, i,j,k,m}$ and $w_{\ell,i,j,k,m}^\pm$.  We
    obtain $\ell\left(w_{\ell, i, j, k, m}\right) = ik + j \left( i + k +
    (j+1)/2 + 2m \right)$ and $\ell ( w_{\ell, i, j, k, m}^\pm ) = \ell \left(
    w_{\ell, i, j, k, m} \right) - j$.  With these length formulas on hand, it
    is straightforward to confirm the identity $\ell \left( v^{-1} s_a w\right)
    = \ell(w) - \ell(v) + 1$ holds in every case. We can also verify on a
    case by case basis that $\left( v^{-1} s_a w \right) s_t = s_{t^\prime}
    \left( v^{-1} s_a w \right)$ by identifying the simple reflections $s_t$
    and $s_{t^\prime}$ with their respective signed permutation matrices, and
    computing the matrix products $\left( v^{-1} s_a w\right) \cdot s_t$ and $
    s_{t^\prime} \cdot \left( v^{-1} s_a w \right)$.

\end{proof}

We conclude this section by identifying those elements $w \in
BiGr^{\circ}(X_n)$ not involved in any reduction of the form $\llbracket w
\rrbracket \xlongrightarrow{\hspace{2mm} (k, L) \hspace{2mm}} (v, i, j)$ or
$\llbracket w \rrbracket \xlongrightarrow{\hspace{2mm} (k, R) \hspace{2mm}} (v,
i, j)$ for Lie types $ABCD$.  We also determine how many such $w$ exist in the
exceptional types. Define the sets
\begin{equation}
    \label{definition, BiGr_perp}
    BiGr_{\perp}(X_n) := \left\{ w \in BiGr(X_n) :
    \begin{tabular}{l} $w$ satisfies the orthogonality \\ condition of
    Proposition (\ref{proposition, orthogonality condition}) \end{tabular}
    \right\},
\end{equation}
\begin{equation}
    \label{definition, BiGr_perp^circ}
    BiGr_{\perp}^{\circ}(X_n) := \left\{ w \in BiGr^{\circ}(X_n) :
    \begin{tabular}{l} $w$ satisfies the orthogonality \\ condition of
    Proposition (\ref{proposition, orthogonality condition}) \end{tabular}
    \right\}.
\end{equation}
By applying the results of Proposition (\ref{proposition, orthogonality
condition}), it follows that the set $BiGr_{\perp}^{\circ}(X_n)$ can
equivalently be defined as the set of elements $w \in BiGr^\circ(X_n)$ such
that there fails to exist a reduction of the form $\llbracket w \rrbracket
\xlongrightarrow{(k, L)} (v, i, j)$ or $\llbracket w \rrbracket
\xlongrightarrow{(k, R)} (v, i, j)$ for any $k\in \mathbf{I}$.

\begin{proposition}

    \label{Yn, description}

    Let $BiGr_{\perp}^{\circ}(X_n)$ be as defined in (\ref{definition,
    BiGr_perp^circ}). We have the following:

    \begin{enumerate}

        \item (Small rank cases, types $ABCD$)

            \vspace{2mm}

            $\hspace{-10pt} \bullet \hspace{4pt} BiGr_{\perp}^{\circ}(A_2) =
            \left\{ w_{0, 1, 0, 2, 0}, w_{0, 2, 0, 1, 0} \right\}$,

            \vspace{3mm}

            $\hspace{-10pt} \bullet \hspace{4pt} BiGr_{\perp}^{\circ}(A_3) =
            \left\{ w_{0, 2, 0, 2, 0} \right\}$,

            \vspace{3mm}

            $\hspace{-10pt} \bullet \hspace{4pt} BiGr_{\perp}^{\circ}(BC_2) =
            \left\{ w_{0, 0, 1, 1, 0}, w_{0, 1, 1, 0, 0}, w_{0, 0, 1, 0, 1},
            w_{0, 0, 2, 0, 0} \right\}$,

            \vspace{3mm}

            $\hspace{-10pt} \bullet \hspace{4pt} BiGr_{\perp}^{\circ}(BC_3) =
            \left\{ w_{0, 0, 2, 1, 0}, w_{0, 1, 2, 0, 0}, w_{0, 0, 1, 0, 2},
            w_{0, 0, 2, 0, 1}, w_{0, 1, 1, 1, 0} \right\}$,

            \vspace{3mm}

            $\hspace{-10pt} \bullet \hspace{4pt} BiGr_{\perp}^{\circ}(D_4) =
            \left\{ w_{0, 0, 4, 0, 0}^+, w_{0, 0, 1, 0, 3}^+, w_{0, 0, 2, 0,
            2}^+, w_{0, 1, 1, 1, 1}^+, w_{0, 1, 2, 1, 0}^+ \right\}$,

            \vspace{2mm}

        \item (General cases, types $ABCD$)

            \vspace{3mm}

            $\hspace{-10pt} \bullet \hspace{4pt} BiGr_{\perp}^{\circ}(A_n) =
            \emptyset$ for $n > 3$,

            \vspace{3mm}

            $\hspace{-10pt} \bullet \hspace{4pt} BiGr_{\perp}^{\circ}(BC_n) =
            \left\{ w_{0, 0, 1, 0, n - 1}, w_{0, 0, 2, 0, n - 2}, w_{0, 1, 1,
            1, n - 3}, w_{0, 1, 2, 1, n - 4} \right\}$

            for $n > 3$,

            \vspace{3mm}

            $\hspace{-10pt} \bullet \hspace{4pt} BiGr_{\perp}^{\circ}(D_n) =
            \left\{ w_{0, 0, 1, 0, n - 1}^+, w_{0, 0, 2, 0, n - 2}^+, w_{0, 1,
            1, 1, n - 3}^+, w_{0, 1, 2, 1, n - 4}^+ \right\}$

            for $n > 4$.

    \end{enumerate}

\end{proposition}

\begin{proof}

    Proposition (\ref{second stage reductions, ABCD}) gives several instances
    where a reduction of the form $\llbracket w \rrbracket \xlongrightarrow{(k,
    R)} \llbracket v \rrbracket$ exists, and Proposition (\ref{L, R,
    equivalent, inverse}) implies that for each of these reductions, there is
    yet another reduction $\llbracket w^{-1} \rrbracket \xlongrightarrow{(k,
    L)} \llbracket v^{-1} \rrbracket$. In our setting, each Weyl group element
    $w$ is identified with a signed permutation matrix, and under this
    identification, $w$ is an orthogonal matrix. Hence $w^{-1}$ is the matrix
    transpose of $w$. Therefore $w_{\ell,i,j,k,m}$ and $w_{\ell,k,j,i,m}$ are
    inverses of each other, whereas the inverse of $w_{\ell,i,j,k,m}^\pm$ is
    $w_{\ell,k,j,i,m}^\pm$ if $j$ is even, or $w_{\ell,k,j,i,m}^\mp$ if $j$ is
    odd.  With this at hand, one can identify precisely those elements $w \in
    BiGr^{\circ}(X_n)$ that are not involved in any of the reductions
    $\llbracket w \rrbracket \xlongrightarrow{(a, R)} \llbracket v \rrbracket$
    of Proposition (\ref{second stage reductions, ABCD}), nor in any of the
    reductions $\llbracket w^{-1} \rrbracket \xlongrightarrow{(a, L)}
    \llbracket v^{-1} \rrbracket$ obtained by applying Proposition (\ref{L, R,
    equivalent, inverse}).

    It remains to verify that each of the elements $w$ declared as belonging to
    the set $BiGr_\perp^\circ(X_n)$ according this proposition, indeed, belongs
    to $BiGr_\perp^\circ(X_n)$. This can be checked on a case by case basis.
    In each case, there fails to exists an index $k\in \mathbf{I}$ satisfying
    the conditions described in parts (\ref{theorem, second stage reduction,
    1}) or (\ref{theorem, second stage reduction, 2}) of Theorem (\ref{theorem,
    second stage reduction}).

\end{proof}

Finally, we consider the exceptional types $X_n$. We determine which elements
of $BiGr^\circ(X_n)$ belong to $BiGr_{\perp}^{\circ}(X_n)$ by considering each
$w\in BiGr^{\circ}(X_n)$ one at a time.  With $w$ fixed, we check whether any
of the indices $k\in \mathbf{I}$ meet the conditions given in Theorem
(\ref{theorem, second stage reduction}). For instance, in the most complicated
situation, namely when the Lie type is $E_8$, there are 7406 elements belonging
to $BiGr^{\circ}(E_8)$ (see Proposition (\ref{X_n, cardinalities})), and for
each of these elements, there are 8 indices $k\in \mathbf{I}$ to check. It is a
tedious, yet straightforward task. A modern desktop computer can handle all of
these computations within a matter of a few minutes.

\begin{proposition}

    \label{proposition, Y cardinalities, exceptional}

    Let $BiGr_{\perp}^{\circ}(X_n)$ be as defined in (\ref{definition,
    BiGr_perp^circ}). We have the following:
    \[
        \begin{split}
            &\lvert BiGr_{\perp}^{\circ}(E_6) \rvert = 20, \hspace{3mm} \lvert
            BiGr_{\perp}^{\circ}(E_7) \rvert = 113, \hspace{3mm} \lvert
            BiGr_{\perp}^{\circ}(E_8) \rvert = 1702, \hspace{3mm}
            \\
            &\lvert BiGr_{\perp}^{\circ}(F_4) \rvert = 34, \hspace{3mm} \lvert
            BiGr_{\perp}^{\circ}(G_2) \rvert = 8.
        \end{split}
    \]

\end{proposition}

Explicit descriptions of all elements of $BiGr_\perp^\circ(X_n)$ for the
exceptional types are given in Section (\ref{section, nilpotency indices,
exceptional}) and Appendix (\ref{appendix, elements E_n}).

\section{Nilpotency indices: The small rank cases (Types
    \texorpdfstring{$ABCD$}{ABCD})}

\label{section, nil-index, small rank}

We turn our attention now towards computing nilpotency indices ${\mathcal
N}(\llbracket w \rrbracket)$ for all $w \in BiGr_{\perp}^{\circ}(X_n)$.
Following Proposition (\ref{Yn, description}), which gives an explicit
description of $BiGr_{\perp}^{\circ}(X_n)$ for each non-exceptional Lie type
$X_n$, we divide this problem into two main parts: (1) the \textit{small rank}
cases, and (2) the \textit{general} cases. The main objective of this section
is to compute the nilpotency indices ${\mathcal N}(\llbracket w \rrbracket)$
for Weyl group elements $w \in BiGr_{\perp}^{\circ}(X_n)$ belonging to the
small rank cases.  In the following section, we will compute the nilpotency
indices associated to elements $w$ in the general cases.  The main result of
this section is summarized in Table (\ref{table, small rank ABCD}), which gives
the nilpotency index in each small rank case.

\begin{table}[H]
    \caption{Small rank cases, Types ABCD (data)}
    \begin{tabular}{ccccc}
    \hline
    Case & Lie type $X_n$ & $w \in BiGr_{\perp}^{\circ}(X_n)$
    & $\chi(\llbracket w \rrbracket)$ &
    ${\mathcal N}( \llbracket w \rrbracket)$
    \\[2pt]
    \hline
    SmallRank.A.1 & $A_2$ & $w_{0, 1, 0, 2, 0}$   & (2, 2, 1) & 1 \\
    SmallRank.A.2 & $A_2$ & $w_{0, 2, 0, 1, 0}$   & (2, 2, 1) & 1 \\
    SmallRank.A.3 & $A_3$ & $w_{0, 2, 0, 2, 0}$   & (2, 2, 0) & 2 \\
    SmallRank.B.1 & $B_2$ & $w_{0, 0, 1, 1, 0}$   & (2, 4, 2) & 1 \\
    SmallRank.B.2 & $B_2$ & $w_{0, 1, 1, 0, 0}$   & (4, 2, 2) & 1 \\
    SmallRank.B.3 & $B_2$ & $w_{0, 0, 2, 0, 0}$   & (2, 2, 0) & 2 \\
    SmallRank.B.4 & $B_3$ & $w_{0, 0, 2, 1, 0}$   & (2, 4, 0) & 3 \\
    SmallRank.B.5 & $B_3$ & $w_{0, 1, 2, 0, 0}$   & (4, 2, 0) & 2 \\
    SmallRank.C.1 & $C_2$ & $w_{0, 0, 1, 1, 0}$   & (4, 2, 2) & 1 \\
    SmallRank.C.2 & $C_2$ & $w_{0, 1, 1, 0, 0}$   & (2, 4, 2) & 1 \\
    SmallRank.C.3 & $C_2$ & $w_{0, 0, 2, 0, 0}$   & (4, 4, 0) & 2 \\
    SmallRank.C.4 & $C_3$ & $w_{0, 0, 2, 1, 0}$   & (4, 2, 0) & 2 \\
    SmallRank.C.5 & $C_3$ & $w_{0, 1, 2, 0, 0}$   & (2, 4, 0) & 3 \\
    SmallRank.D.1 & $D_4$ & $w_{0, 0, 4, 0, 0}^+$ & (2, 2, 0) & 2 \\
    SmallRank.D.2 & $D_4$ & $w_{0, 1, 2, 1, 0}^+$ & (2, 2, 0) & 2 \\
    [2pt]
    \hline
    \end{tabular}
    \label{table, small rank ABCD}
\end{table}

Our approach to finding nilpotency indices is primarily computational. In
large part, these computations involve expressing $q$-commutators $[X_\eta,
X_\mu]$ of Lusztig root vectors as linear combinations of ordered monomials.
With this, we will tacitly apply some well-known properties of Lusztig
symmetries when performing such computations. In particular, we will adopt the
following conventions.

\begin{conventions}

    $ $

    \label{computational conventions}

    \begin{itemize}

        \item Whenever we replace an element of the form $T_w(E_i)$ with a
            Chevalley generator, say $E_j$, it is implied that $w(\alpha_i) =
            \alpha_j$ (recall Theorem (\ref{Jantzen, 8.20})).

            \vspace{2mm}

        \item Whenever we replace an element of the form $T_w(E_i)$ with a
            Lusztig root vector, say $X_{\beta}$, we tacitly are using the
            definition (\ref{root vectors}) of $X_\beta$.

            \vspace{2mm}

        \item  Whenever we replace a Lusztig symmetry $T_w$ with a composition
            of Lusztig symmetries $T_{w_1} \cdots T_{w_k}$, it is assumed
            $w = w_1 \cdots w_k$ and $\ell(w) = \sum_{i = 1}^k \ell(w_i)$.

            \vspace{2mm}

        \item For homogeneous elements $u \in {\mathcal
            U}_q(\mathfrak{g})_\eta$ and $v \in {\mathcal
            U}_q(\mathfrak{g})_\mu$, whenever we replace $T_w\left( [ u, v ]
            \right)$ with $[T_w(u), T_w(v) ]$, we are invoking the
            orthogonality property $\langle w(\eta), w(\mu) \rangle = \langle
            \eta, \mu \rangle$ in conjunction with the fact that $T_w$ is an
            algebra automorphism of ${\mathcal U}_q(\mathfrak{g})$.

            \vspace{2mm}

        \item For simple roots $\alpha_i$ and $\alpha_j$ such that
            $\frac{2\langle \alpha_i, \alpha_j \rangle}{\langle \alpha_i,
            \alpha_i \rangle} = -2$ (this only occurs in Lie types $B_n$, $C_n$
            and $F_4$), we apply the identities $T_{s_i s_j}(E_i) = [E_i, E_j]$
            and $[2]_q T_{s_j s_i} (E_j) = [[E_j, E_i], E_i]$ whenever
            necessary.

            \vspace{2mm}

        \item We will, as needed, replace elements of the form $T_{s_i}(E_j)$
            with $[E_i, E_j]$ or $\frac{1}{[2]_q} [E_i, [E_i, E_j]]$ whenever
            $\frac{2 \langle \alpha_i, \alpha_j \rangle}{\langle \alpha_i,
            \alpha_i \rangle}$ equals $-1$ or $-2$, respectively.

    \end{itemize}

\end{conventions}

Following \cite{Knapp}, the simple roots for Lie types $A_{n - 1}$, $B_n$,
$C_n$, and $D_n$ will be written as $\alpha_i = e_i - e_{i + 1}$ (for $i < n$),
and $\alpha_n$ is $e_n$, $2e_n$, or $e_{n - 1} + e_n$ in types $B_n$, $C_n$,
$D_n$, respectively.  The sets of positive roots are
\begin{align*}
    &\Delta^+(A_n) = \left\{ e_i - e_j \mid 1 \leq i < j \leq n + 1 \right\},
    \\
    &\Delta^+(B_n) = \left\{ e_i \pm e_j \mid 1 \leq i < j \leq n \right\} \cup
    \left\{ e_i \mid 1 \leq i \leq n \right\},
    \\
    &\Delta^+(C_n) = \left\{ e_i \pm e_j \mid 1 \leq i < j \leq n \right\} \cup
    \left\{ 2e_i \mid 1 \leq i \leq n \right\},
    \\
    &\Delta^+(D_n) = \left\{ e_i \pm e_j \mid 1 \leq i < j \leq n \right\}.
\end{align*}
For short, the Lusztig root vectors $X_{e_i \pm e_j}$ ($i < j$), $X_{e_i}$, and
$X_{2e_i}$ of the corresponding quantum Schubert cell algebra will be denoted
by
\begin{equation}
    \label{Lusztig root vectors, ABCD}
    X_{i, \pm j}, \hspace{2mm} X_{i, 0}, \text{ and } X_{i, i},
\end{equation}
respectively.

\begin{theorem}

    \label{SmallRank, nil-index, main theorem}

    The rightmost column in Table (\ref{table, small rank ABCD}) gives the
    nilpotency indices for each small rank case.

\end{theorem}

\begin{proof}

    We prove this on a case by case basis. However, most cases are covered by
    Theorem (\ref{theorem, nil-index, 1}). In particular, part 1 of Theorem
    (\ref{theorem, nil-index, 1}) is applicable to the cases SmallRank.A.1,
    SmallRank.A.2, SmallRank.B.1, SmallRank.B.2, SmallRank.C.1, and
    SmallRank.C.2, whereas part 2 of Theorem (\ref{theorem, nil-index, 1}) can
    be applied to the cases SmallRank.A.3, SmallRank.B.3, SmallRank.C.4,
    SmallRank.D.1, and SmallRank.D.2.

    There are only four remaining cases to consider: SmallRank.B.4,
    SmallRank.B.5, SmallRank.C.4, and SmallRank.C.5. We will consider these
    four remaining cases one at a time. When performing computations throughout
    this proof, we follow the guidelines listed in Conventions
    (\ref{computational conventions}). Recall Lusztig root vectors will be
    denoted by $X_{i, \pm k}$ ($i < k$), $X_{i, 0}$, or $X_{i, i}$
    (\ref{Lusztig root vectors, ABCD}).

    \vspace{2mm}

    \noindent (SmallRank.B.4) Let $W$ be the Weyl group associated to the
    Lie algebra of type $B_3$, and let $w = w_{0, 0, 2, 1, 0} \in W$. With
    respect to the reduced expression $w = s_3 s_2 s_1 s_3 s_2$, the convex
    order on $\Delta_w$ is
    \[
        e_3 \prec e_2 + e_3 \prec e_1 + e_3 \prec e_2 \prec e_1 + e_2.
    \]
    Proposition (\ref{LS corollary}) implies
    \begin{equation}
        \label{L-S, SmallRank.B.4}
        [X_{3, 0}, X_{2, 3}] = [X_{3, 0}, X_{1, 3}] = [X_{2, 3}, X_{1, 3}] =
        [X_{2, 0}, X_{1, 2}] = 0.
    \end{equation}
    With the adopted conventions listed in Conventions (\ref{computational
    conventions}), direct computations give us
        $X_{3, 0} = E_3$,
        $X_{2, 3} = T_{s_3} (E_2) = \frac{1}{[2]_q} [E_3, [E_3, E_2]]$, and
        \[
            X_{2,0} = T_{s_3 s_2 s_1} (E_3) = T_{s_3 s_2} T_{s_1} (E_3) =
            T_{s_3 s_2} (E_3) = [E_3, E_2].
        \]
    Hence, $[X_{3, 0}, X_{2, 0}] = [2]_q X_{2, 3}$. Furthermore, we have
    \begin{align*}
        X_{1, 2} &= T_{s_3 s_2 s_1 s_3} (E_2) = T_{s_3 s_2 s_3} T_{s_1} (E_2) =
        T_{s_3 s_2 s_3} \left( [ E_1, E_2 ] \right)
        \\
        &= [T_{s_3 s_2 s_3} (E_1), T_{s_3 s_2 s_3} (E_2)] = [ T_{s_3 s_2}
        T_{s_3} (E_1), E_2] = [T_{s_3 s_2} (E_1), E_2]
        \\
        &= [X_{1, 3}, E_2].
    \end{align*}
    Using $X_{3, 0} = E_3$ and $X_{1, 2} = [X_{1, 3}, E_2]$, the $q$-Jacobi
    identity (\ref{q-Jacobi identity}) gives us
    \[
        [X_{3, 0}, X_{1, 2}] \!=\! [E_3, [X_{1, 3}, E_2]] \!=\! [ [ E_3, X_{1,
        3}], E_2] - q_1 [[E_3, E_2], X_{1, 3}] + \widehat{q_1}\, [E_3, E_2]
        X_{1, 3}.
    \]
    We replace $[E_3, E_2]$ with $X_{2, 0}$ to obtain
    \[
        [X_{3, 0}, X_{1, 2}] = [ [ X_{3, 0}, X_{1, 3}], E_2] - q_1 [X_{2, 0},
        X_{1, 3}] + \widehat{q_1}\, X_{2, 0} X_{1, 3}.
    \]
    Since $[X_{3, 0}, X_{1, 3}] = [X_{2, 0}, X_{1, 3}] = 0$ (\ref{L-S,
    SmallRank.B.4}), then
    \[
        [X_{3, 0}, X_{1, 2}] = \widehat{q_1}\, X_{2, 0} X_{1, 3} =
        \widehat{q_1}\, X_{1, 3} X_{2, 0}.
    \]
    The $q$-Leibniz identity (\ref{q-Leibniz identity}) can next be applied to
    get
    \[
        [X_{3, 0}, [X_{3, 0}, X_{1, 2}]] = \widehat{q_1}\, \left( q_1 X_{1, 3}
        [X_{3, 0}, X_{2, 0}] + [X_{3, 0}, X_{1, 3}] X_{2, 0} \right).
    \]
    Applying the identities $[X_{3, 0}, X_{2, 0}] = [2]_qX_{2, 3}$, $[X_{3, 0},
    X_{1, 3}] = 0$, and $[X_{2, 3}, X_{1, 3}] = 0$ gives us
    \[
        [X_{3, 0}, [X_{3, 0}, X_{1, 2}]] = [2]_q \widehat{q_1}\, q_1 X_{1, 3}
        X_{2, 3} = [2]_q \widehat{q_1}\, X_{2, 3} X_{1, 3}.
    \]
    Finally, since $[X_{3, 0}, X_{2, 3}] = [X_{3, 0}, X_{1, 3}] = 0$ (\ref{L-S,
    SmallRank.B.4}), then, by applying the $q$-Leibniz identity (\ref{q-Leibniz
    identity}), we have $[X_{3, 0}, X_{2,3} X_{1, 3}] = 0$.  Hence $[X_{3, 0},
    [X_{3, 0}, [X_{3, 0}, X_{1, 2}]]] = 0$. Therefore,
    ${\mathcal N}(\llbracket w \rrbracket) = 3$.

    \vspace{3mm}

    \noindent (SmallRank.B.5) Let $W$ be the Weyl group associated to the Lie
    algebra of type $B_3$, and now let $w = w_{0, 1, 2, 0, 0} \in W$. With
    respect to the reduced expression $w = s_2 s_3 s_1 s_2 s_3$, the convex
    order on $\Delta_w$ is
    \[
        e_2 - e_3 \prec e_2 \prec e_1 - e_3 \prec e_1 + e_2 \prec e_1.
    \]
    The Lusztig root vectors of the quantum Schubert cell algebra ${\mathcal
    U}_q^+[w]$ are denoted $X_{2, -3}$, $X_{2, 0}$, $X_{1, -3}$, $X_{1, 2}$,
    and $X_{1, 0}$. We will make use of the identities
    \begin{equation}
        \label{L-S, SmallRank.B.5}
        [X_{2, -3}, X_{2, 0}] = [X_{2, -3}, X_{1, -3}] = 0,
    \end{equation}
    which follow as consequences of Proposition
    (\ref{LS corollary}). We will also use the identities $X_{1,
    -3} = [X_{2, -3}, E_1]$ and $X_{1, 0} = [X_{2, 0}, E_1]$, which are
    verified as follows,
    \begin{align*}
        X_{1, -3} &= T_{s_2 s_3}(E_1) = T_{s_2} T_{s_3}(E_1) = T_{s_2}(E_1) =
        [E_2, E_1] = [X_{2, -3}, E_1],
        \\
        X_{1, 0} &= T_{s_2 s_3 s_1 s_2}(E_3) = T_{s_2 s_1} T_{s_3 s_2} (E_3) =
        T_{s_2 s_1} ([E_3, E_2]) = [T_{s_2 s_1} (E_3), T_{s_2 s_1} (E_2)]
        \\
        &= [T_{s_2} T_{s_1}(E_3), E_1] = [T_{s_2} (E_3), E_1] = [X_{2, 0},
        E_1].
    \end{align*}
    We aim to write the $q$-commutator $[X_{2, -3}, X_{1, 0}]$ as a linear
    combination of ordered monomials in the Lusztig root vectors. We begin by
    substituting $X_{1, 0}$ with $[X_{2, 0}, E_1]$ and applying the $q$-Jacobi
    identity (\ref{q-Jacobi identity}). This yields
    \[
        [X_{2, -3}, X_{1, 0}] = [[X_{2, -3}, X_{2, 0}], E_1] -q_1[[X_{2, -3},
        E_1], X_{2, 0}] + \widehat{q_1} [X_{2, -3}, E_1] X_{2, 0}.
    \]
    Using the identity $[X_{2, -3}, X_{2, 0}] = 0$ (\ref{L-S, SmallRank.B.5})
    together with $X_{1, -3} = [X_{2, -3}, E_1]$ gives us
    \[
        [X_{2, -3}, X_{1, 0}] = -q_1[X_{1, -3}, X_{2, 0}] + \widehat{q_1} X_{1,
        -3} X_{2, 0}.
    \]
    The identity $[X_{2, 0}, X_{1, -3}] = 0$ (\ref{L-S, SmallRank.B.5}) means
    $X_{2,0}$ and $X_{1, -3}$ commute with each other.  Thus, $[X_{2, -3},
    X_{1, 0}] = \widehat{q_1} X_{2, 0} X_{1, -3}$. Since the $q$-commutator
    $[X_{2, 3}, X_{1, 0}]$ is nonzero, the nilpotency index
    ${\mathcal N}(\llbracket w \rrbracket)$ is greater than $1$.
    However, we have $[X_{2, -3}, X_{2, 0}] = [X_{2, -3}, X_{1, -3}] = 0$
    (\ref{L-S, SmallRank.B.5}). Hence, by the $q$-Leibniz identity
    (\ref{q-Leibniz identity}), we obtain $[X_{2, -3}, X_{2, 0} X_{1, -3}] =
    0$. This means $[X_{2, -3}, [X_{2, -3}, X_{1, 0}]] = 0$. Hence,
    ${\mathcal N}(\llbracket w \rrbracket) = 2$.

    \vspace{3mm}

    \noindent (SmallRank.C.4) Now let $W$ be the Weyl group associated to the
    Lie algebra of type $C_3$, and let $w$ be the Weyl group element $w = w_{0,
    0, 2, 1, 0} \in W$.  Consider the reduced expression $w = s_3 s_2 s_1 s_3
    s_2$. In this setting, the convex order on $\Delta_w$ is
    \[
        2e_3 \prec e_2 + e_3 \prec e_1 + e_3 \prec 2e_2 \prec e_1 + e_2.
    \]
    Recall, the Lusztig root vectors of the corresponding quantum Schubert cell
    algebra ${\mathcal U}_q^+[w]$ will be denoted by $X_{3, 3}$, $X_{2, 3}$,
    $X_{1, 3}$, $X_{2, 2}$, and $X_{1, 2}$ (\ref{Lusztig root vectors, ABCD}).
    We will make use of the identities
    \[
        [X_{3, 3}, X_{2, 3}] = [X_{3, 3}, X_{1, 3}] = [X_{2, 3}, X_{1, 3}] = 0,
    \]
    which follow as a consequence of Proposition (\ref{LS corollary}).  We also
    use the $q$-commutator identities, $X_{2, 3} = [X_{3, 3}, E_2]$ and $X_{1,
    2} = [X_{1, 3}, E_2]$, which can be proven by direct computations.  For
    instance, following the conventions listed in Conventions
    (\ref{computational conventions}), we obtain $X_{2, 3} = T_{s_3} (E_2) =
    [E_3, E_2] = [X_{3, 3}, E_2]$ and
    \begin{align*}
        X_{1, 2} &= T_{s_3 s_2 s_1 s_3} (E_2) = T_{s_3 s_2 s_3} T_{s_1} (E_2) =
        T_{s_3 s_2 s_3} \left( [E_1, E_2] \right)
        \\
        &= [ T_{s_3 s_2 s_3} (E_1), T_{s_3 s_2 s_3} (E_2)]
        = [T_{s_3 s_2} T_{s_3} (E_1), E_2] = [T_{s_3 s_2} (E_1),
        E_2]
        \\
        &= [X_{1, 3}, E_2].
    \end{align*}
    With this, using the identity $X_{1, 2} = [X_{1, 3}, E_2]$ and applying the
    q-Jacobi identity (\ref{q-Jacobi identity}) gives
    \[
        [X_{3, 3}, X_{1, 2}] = [[X_{3, 3}, X_{1, 3}], E_2] - q[[X_{3, 3},
        E_2],X_{1, 3}] + \widehat{q}\,\, [X_{3, 3}, E_2] X_{1, 3}.
    \]
    Next, use the identities $X_{2, 3} = [X_{3, 3}, E_2]$, $[X_{3, 3}, X_{1,
    3}] = 0$, and $[X_{2, 3}, X_{1, 3}] = 0$ to get $[X_{3, 3}, X_{1, 2}] = -
    q[X_{2, 3},X_{1, 3}] + \widehat{q}\,\,  X_{2, 3} X_{1, 3} = \widehat{q}\,\,
    X_{2, 3} X_{1, 3}$. At this point, we have determined how to write the
    $q$-commutator $[X_{3, 3}, X_{1, 2}]$ as a linear combination of ordered
    monomials, namely $[X_{3, 3}, X_{1, 2}] = \widehat{q}\,\, X_{2, 3} X_{1,
    3}$. In particular, it is nonzero. Hence, ${\mathcal N}\left( \llbracket w
    \rrbracket \right) > 1$.

    However, since $[X_{3, 3}, X_{2, 3}] = 0$ and $[X_{3, 3}, X_{1, 3}] = 0$,
    the $q$-Leibniz identity (\ref{q-Leibniz identity}) can be applied to show
    that $[X_{3, 3}, X_{2, 3} X_{1, 3}] = 0$.  Hence, $[X_{3, 3}, [X_{3, 3},
    X_{1, 2}]] = 0$. Therefore, ${\mathcal N} \left( \llbracket w \rrbracket
    \right) = 2$.

    \vspace{3mm}

    \noindent (SmallRank.C.5) Let $W$ be the Weyl group associated to the Lie
    algebra of type $C_3$, and now let $w = w_{0, 1, 2, 0, 0} \in W$. With
    respect to the reduced expression $w = s_2 s_3 s_1 s_2 s_3$, the convex
    order on $\Delta_w$ is
    \[
        e_2 - e_3 \prec 2e_2 \prec e_1 - e_3 \prec e_1 + e_2 \prec 2e_1.
    \]
    We will need the identities
    \[
        [X_{2, -3}, X_{2, 2}] = [X_{2, 2}, X_{1, -3}] = [X_{1, -3}, X_{1, 2}] =
        [X_{2, -3}, X_{1, -3}] = 0,
    \]
    which follow from Proposition (\ref{LS corollary}). Following the
    conventions of Conventions (\ref{computational conventions}), $X_{2, -3} =
    E_2$, $X_{2, 2} = T_{s_2} (E_3) = \frac{1}{[2]}_q[E_2, [E_2, E_3]]$, and
    \[
        X_{1, -3} = T_{s_2 s_3}(E_1) = T_{s_2} T_{s_3} (E_1) = T_{s_2} (E_1) =
        [E_2, E_1] = [X_{2, -3}, E_1].
    \]
    Hence $X_{1, -3} = [E_2, E_1]$. We also will make use of the identities
    $X_{1, 2} = [X_{2, 2}, E_1]$ and $[2]_q X_{1, 1} = [X_{1, 2}, E_1]$, which
    are verified as follows,
    \begin{align*}
        X_{1, 2} &= T_{s_2 s_3 s_1} (E_2) = T_{s_2 s_1} T_{s_3} (E_2) = T_{s_2
        s_1} ([E_3, E_2]) = [T_{s_2 s_1} (E_3), T_{s_2 s_1}(E_2)]
        \\
        &= [T_{s_2} T_{s_1}(E_3), E_1] = [T_{s_2}(E_3), E_1] = [X_{2, 2}, E_1],
        \\
        [2]_q X_{1, 1} &= [2]_q T_{s_2 s_3 s_1 s_2}(E_3) = [2]_q T_{s_2 s_1}
        T_{s_3 s_2} (E_3) = T_{s_2 s_1} ([[E_3, E_2], E_2]),
        \\
        &= [T_{s_2 s_1} ([E_3, E_2]), T_{s_2 s_1} (E_2)] = [T_{s_2 s_1} T_{s_3}
        (E_2), E_1] = [T_{s_2 s_3 s_1} (E_2), E_1]
        \\
        &= [X_{1, 2}, E_1].
    \end{align*}

    To compute the nilpotency index ${\mathcal N}(\llbracket w \rrbracket)$
    means we need to determine the smallest natural number $p\in \mathbb{N}$ so
    that $\left( \operatorname{ad}_q(X_{2, -3} \right)^p(X_{1, 1}) = 0$. Our
    first objective is to determine whether or not the $q$-commutator $[X_{2,
    -3}, X_{1, 1}]$ equals $0$. We will do this by writing it as a linear
    combination of ordered monomials.  However, as an intermediate step, we
    first prove $[X_{2, -3}, X_{1, 2}] = [2]_q \widehat{q} X_{2, 2} X_{1, -3}$.
    We start by using the identity $X_{1, 2} = [X_{2, 2}, E_1]$ together with
    the $q$-Jacobi identity (\ref{q-Jacobi identity}) to get
    \[
        [X_{2, -3}, X_{1, 2}] = [[X_{2, -3}, X_{2, 2}], E_1] - q^2 [[X_{2, -3},
        E_1], X_{2, 2}] + [2]_q \widehat{q} [X_{2, -3}, E_1] X_{2, 2}.
    \]
    Then we invoke the identities $[X_{2, -3}, E_1] = X_{1, -3}$  and $[X_{2,
    -3}, X_{2, 2}] = 0$ to simplify this to
    \[
        [X_{2, -3}, X_{1, 2}] = - q^2 [X_{1, -3}, X_{2, 2}] + [2]_q \widehat{q}
        X_{1, -3} X_{2, 2}.
    \]
    The identity $[X_{2, 2}, X_{1, -3}] = 0$ means $X_{2, 2}$ and $X_{1, -3}$
    commute with each other. Hence, we obtain $[X_{2, -3}, X_{1, 2}] = [2]_q
    \widehat{q} X_{2, 2} X_{1, -3}$.

    Now we determine how $[X_{2, -3}, X_{1, 1}]$ can be written as a linear
    combination of ordered monomials. Using $[2]_q X_{1, 1} = [X_{1, 2}, E_1]$
    and applying the $q$-Jacobi identity (\ref{q-Jacobi identity}) gives us
    \[
        [2]_q [X_{2, -3}, X_{1, 1}] = [[X_{2, -3}, X_{1, 2}], E_1] - [[X_{2,
        -3}, E_1], X_{1, 2}].
    \]
    Since $X_{1, -3} = [X_{2, -3}, E_1]$ and $[X_{2, -3}, X_{1, 2}] = [2]_q
    \widehat{q} X_{2, 2}, X_{1, -3}$, we obtain
    \[
        [2]_q [X_{2, -3}, X_{1, 1}] = [2]_q \widehat{q} [ X_{2, 2} X_{1, -3},
        E_1] - [X_{1, -3}, X_{1, 2}].
    \]
    Applying the identity $[X_{1, -3}, X_{1, 2}] = 0$ gives us $[X_{2, -3},
    X_{1, 1}] = \widehat{q} [ X_{2, 2} X_{1, -3}, E_1]$. From the $q$-Leibniz
    identity (\ref{q-Leibniz identity}) we have
    \[
        [X_{2, -3}, X_{1, 1}] = \widehat{q} \left( X_{2, 2} [ X_{1, -3}, E_1] +
        q[X_{2, 2}, E_1] X_{1, -3} \right).
    \]
    However, this can be simplified because the $q$-Serre relations
    (\ref{q-Serre 1}) imply
    \[
        [X_{1, -3}, E_1] = [[E_2, E_1], E_1] = E_2 E_1^2 - (q + q^{-1}) E_1 E_2
        E_1 + E_1^2 E_2 = 0.
    \]
    Therefore, it follows that $[X_{2, -3}, X_{1, 1}] = \widehat{q} q [X_{2,
    2}, E_1] X_{1, -3}$. Using $[X_{2, 2}, E_1] = X_{1, 2}$ gives us $[X_{2,
    -3}, X_{1, 1}] = \widehat{q} q X_{1, 2} X_{1, -3}$.  Since the identity
    $[X_{1, -3}, X_{1, 2}] = 0$ means $X_{1, -3} X_{1, 2} = q X_{1, 2}, X_{1,
    -3}$, we have $[X_{2, -3}, X_{1, 1}] = \widehat{q} X_{1, -3} X_{1, 2}$.

    We have now determined how to write $[X_{2, -3}, X_{1, 1}]$ as a linear
    combination of ordered monomials. Since ordered monomials form a basis of
    the algebra ${\mathcal U}_q^+[w]$, it follows that $[X_{2, -3}, X_{1, 1}]$
    is nonzero. Hence ${\mathcal N}(\llbracket w \rrbracket) > 1$.  We focus
    next on writing $[X_{2, -3}, [X_{2, -3}, X_{1, 1}]]$ as a linear
    combination of ordered monomials. The $q$-Leibniz identity (\ref{q-Leibniz
    identity}) gives us
    \[
        [X_{2, -3}, [X_{2, -3}, X_{1, 1}] = \widehat{q} \left( [X_{2, -3},
        X_{1, -3}] X_{1, 2} + q X_{1, -3} [X_{2, -3}, X_{1, 2}] \right).
    \]
    However, since $[X_{2, -3}, X_{1, -3}] = 0$ and $[X_{2, -3}, X_{1, 2}] =
    [2]_q \widehat{q} X_{2, 2} X_{1, -3}$, we have
    \[
        [X_{2, -3}, [X_{2, -3}, X_{1, 1}] = [2]_q \widehat{q}^{\,2} q X_{1, -3}
        X_{2, 2} X_{1, -3} = [2]_q \widehat{q}^{\,2} q X_{2, 2} X_{1, -3}^2.
    \]
    Since $[X_{2, -3}, X_{2, 2}] = [X_{2, -3}, X_{1, -3}] = 0$, then, as a
    consequence of the $q$-Leibniz identity (\ref{q-Leibniz identity}), $[X_{2,
    -3}, X_{2, 2} X_{1, -3}^2] = 0$. Thus, $[X_{2, -3}, [X_{2, -3}, [X_{2, -3},
    X_{1, 1}]]] = 0$.  Hence ${\mathcal N} (\llbracket w \rrbracket) = 3$.

\end{proof}

\section{Nilpotency indices: The general cases (Types
    \texorpdfstring{$ABCD$}{ABCD})}

\label{section, nil-index, general}

In this section, we compute the nilpotency indices ${\mathcal N}(\llbracket w
\rrbracket)$ associated to each Weyl group element $w$ appearing among the
so-called \textit{general} cases of Proposition (\ref{Yn, description}).  Table
(\ref{table, general cases ABCD}) summarizes the data for these cases. The main
objective of this section is to prove the rightmost column of Table
(\ref{table, general cases ABCD}) gives the nilpotency index in each of these
situations.

\begin{table}[H]
    \caption{General cases, Types ABCD (data)}
    \begin{tabular}{ccccc}
        \hline
    Case & Lie type $X_n$ & $w \in BiGr_{\perp}^{\circ}(X_n)$
    & $\chi(\llbracket w \rrbracket)$ & ${\mathcal N}( \llbracket w
    \rrbracket)$
    \\[2pt]
    \hline
    B.1 & $B_n$, ($n > 1$) & $w_{0, 0, 1, 0, n-1}$   & (4, 4, 0)  & 2 \\
    B.2 & $B_n$, ($n > 2$) & $w_{0, 0, 2, 0, n-2}$   & (4 ,4, -2) & 3 \\
    B.3 & $B_n$, ($n > 2$) & $w_{0, 1, 1, 1, n-3}$   & (4, 4, 2)  & 2 \\
    B.4 & $B_n$, ($n > 3$) & $w_{0, 1, 2, 1, n-4}$   & (4, 4, 0)  & 3 \\
    C.1 & $C_n$, ($n > 1$) & $w_{0, 0, 1, 0, n-1}$   & (2, 2, 0)  & 2 \\
    C.2 & $C_n$, ($n > 2$) & $w_{0, 0, 2, 0, n-2}$   & (2, 2, -1) & 3 \\
    C.3 & $C_n$, ($n > 2$) & $w_{0, 1, 1, 1, n-3}$   & (2, 2, 1)  & 2 \\
    C.4 & $C_n$, ($n > 3$) & $w_{0, 1, 2, 1, n-4}$   & (2, 2, 0)  & 3 \\
    D.1 & $D_n$, ($n > 3$) & $w_{0, 0, 1, 0, n-1}^+$ & (2, 2, 0)  & 2\\
    D.2 & $D_n$, ($n > 3$) & $w_{0, 0, 2, 0, n-2}^+$ & (2, 2, -1) & 3\\
    D.3 & $D_n$, ($n > 3$) & $w_{0, 1, 1, 1, n-3}^+$ & (2, 2, 1)  & 2\\
    D.4 & $D_n$, ($n > 4$) & $w_{0, 1, 2, 1, n-4}^+$ & (2, 2, 0)  & $3$\\[2pt]
    \hline
    \end{tabular}
    \label{table, general cases ABCD}
\end{table}

\vspace{1mm}

\begin{theorem}

    \label{general cases, BCD13}

    Let $w \in BiGr_{\perp}^{\circ}(X_n)$ be the Weyl group element associated
    to any of the cases B.1, B.3, C.1, C.3, D.1, or D.3 listed in Table
    (\ref{table, general cases ABCD}). Then ${\mathcal N}(\llbracket w
    \rrbracket) = 2$.

\end{theorem}

\begin{proof}

    The results of Theorem (\ref{theorem, nil-index, 1}), part (2), apply. In
    each case, $i \not\in \operatorname{supp}(s_iws_i)$, where $i \in
    \mathbf{I}$ is the unique index such that $\left\{ i \right\} = {\mathcal
    D}_L(w) = {\mathcal D}_R(w)$.

\end{proof}

\subsection{Nilpotency index, Case B.2}

Let $\mathfrak{g}$ be the complex Lie algebra of type $B_n$, where $n > 2$.
Consider the Weyl group element $w := w_{0, 0, 2, 0, n - 2} \in W(B_n)$ and the
reduced expression
\begin{equation}
    \label{B.2, reduced expression}
    w = (s_2 \cdots s_{n-1})(s_1 \cdots s_{n-2}) s_n s_{n-1} s_n (s_{n-2}
    \cdots s_1)(s_{n-1} \cdots s_2).
\end{equation}
Let $\Delta_w$ be the set of roots of $w$. The reduced expression for $w$
corresponds to the following convex order on $\Delta_w$:
\[
    \begin{split}
        &e_2 - e_3 \prec \cdots \prec e_2 - e_n \prec e_1 - e_3 \prec \cdots
        \\
        &\cdots \prec e_1 - e_n \prec e_2 \prec e_1 + e_2 \prec e_1 \prec e_2
        + e_n \prec \cdots
        \\
        &\cdots \prec e_2 + e_3 \prec e_1 + e_n \prec \cdots \prec e_1 + e_3.
    \end{split}
\]
Let ${\mathcal U}_q^+[w]$ be the quantum Schubert cell algebra associated to
$w$. The Lusztig root vectors will be denoted by
\begin{equation}
    \label{B.2, Lusztig root vectors}
    X_{1, 2}, X_{i, 0}, X_{i, \pm k} \in {\mathcal U}_q^+[w], \hspace{15pt} (1
    \leq i < 3 \leq k \leq n).
\end{equation}
Since $e_{2, -3}$ and $e_{1, 3}$ are the first and last roots in $\Delta_w$
with respect to the convex order (and the corresponding Lusztig root vectors
are $X_{2, -3}$ and $X_{1, 3}$, respectively), then in order to calculate
${\mathcal N} \left( \llbracket w \rrbracket \right)$, we need to compute the
actions of the maps $\operatorname{ad}_q (X_{2, -3}) $,
$\left(\operatorname{ad}_q (X_{2, -3}) \right)^2$, $\left(\operatorname{ad}_q
(X_{2, -3}) \right)^3$,  et cetera, on the Lusztig root vector $X_{1, 3}$. We
will show that $\left(\operatorname{ad}_q (X_{2, -3}) \right)^2 (X_{1, 3}) \neq
0$ and $\left(\operatorname{ad}_q (X_{2, -3}) \right)^3 (X_{1, 3}) = 0$.  To
accomplish this, we need to first determine some of the commutation relations
in the quantum Schubert cell algebra ${\mathcal U}_q^+[w]$. In particular, we
are most interested in the commutation relations involving $X_{2, -3}$. Some
relations involving other Lusztig root vectors $X_{i, \pm k}$ will also be
applied along the way.

\begin{lemma}

    \label{lemma, B.2, 0}

    Let $X_{i, \pm k}$ be as in (\ref{B.2, Lusztig root vectors}). Then for
    $1\leq i < 3 \leq j \leq n$ and $3 \leq k \leq n$, we have the following:
    \begin{align}
        \label{B.2, L-S, 1} &[X_{2, -k}, X_{1, -j}] = [X_{2, j}, X_{1, k}] = 0
        \hspace{8pt} (j \leq k),
        \\
        \label{B.2, L-S, 2} &[X_{i, -k}, X_{i, -j}] = [X_{i, j}, X_{i, k}] = 0
        \hspace{12.5pt} (k < j),
        \\
        \label{B.2, L-S, 3} &[X_{i, -k}, X_{i, j}] = [X_{1, -k}, X_{2, j}] = 0
        \hspace{10.3pt} (j \neq k),
        \\
        \label{B.2, L-S, 4} &[X_{i, -k}, X_{2, 0}] = [X_{2, 0}, X_{2, k}] = 0.
    \end{align}

\end{lemma}

\begin{proof}

    These identities follow as a consequence of Proposition (\ref{LS
    corollary}).

\end{proof}

\begin{lemma}

    \label{lemma, B.2, 1}

    Let $X_{i, \pm k}$ be as in (\ref{B.2, Lusztig root vectors}). Then

    \begin{enumerate}

        \item \label{lemma, B.2, 1a} $X_{2, -3} = E_2$ and $X_{1, -3} = [E_2,
            E_1]$,

        \item \label{lemma, B.2, 1b} $X_{i, -k} = [X_{i, -(k - 1)}, E_{k - 1}]$
            (for $3 < k \leq n$ and $i=1, 2$),

        \item \label{lemma, B.2, 1c} $X_{i, k} = [X_{i, k + 1}, E_k]$ (for $3
            \leq k < n$ and $i = 1, 2$),

        \item \label{lemma, B.2, 1d} $X_{1, 0} = [X_{2, 0}, E_1]$,

        \item \label{lemma, B.2, 1e} $X_{1, n} = [X_{2, n}, E_1]$,

        \item \label{lemma, B.2, 1f} $X_{2, 0} = [X_{2, -n}, E_n]$,

        \item \label{lemma, B.2, 1h} $[2]_q X_{2, n} = [X_{2, 0}, E_n]$.

    \end{enumerate}

\end{lemma}

\begin{proof}

    Part (\ref{lemma, B.2, 1a}) follows by observing that $e_2 - e_3$ ($=
    \alpha_2$) is the first root with respect to the convex order on
    $\Delta_w$. Hence, by definition of Lusztig root vectors, $X_{2, -3} =
    E_2$.  Next, adopting the computational conventions described in
    Conventions (\ref{computational conventions}), we compute $X_{1, -3} =
    T_{s_2 \cdots s_{n - 1}}(E_1) = T_{s_2} T_{s_3 \cdots s_{n - 1}} (E_1) =
    T_{s_2} (E_1) = [E_2, E_1]$.

    Next we prove part (\ref{lemma, B.2, 1b}). Fix $k$ such that $3 < k \leq
    n$, and define the Weyl group elements $u := (s_2 \cdots s_{n - 1}) (s_1
    \cdots s_{k - 3})$ and $v := s_2 \cdots s_{k - 2}$. We obtain
    \[
        \begin{split}
            X_{1, -k} &= T_u(E_{k - 2}) = T_{u s_{k - 3}} T_{s_{k - 3}} (E_{k -
            2}) = T_{u s_{k - 3}} \left( [E_{k - 3}, E_{k - 2}] \right)
            \\
            &= [T_{u s_{k - 3}} (E_{k - 3}), T_{u s_{k - 3}} (E_{k - 2})]
            = [X_{1, -(k - 1)}, E_{k - 1}]
        \end{split}
    \]
    and
    \[
        \begin{split}
            X_{2, -k} &= T_v(E_{k - 1}) = T_{v s_{k - 2}} T_{s_{k - 2}} (E_{k -
            1}) = T_{v s_{k - 2}} \left( [E_{k - 2}, E_{k - 1}] \right)
            \\
            &= [T_{v s_{k - 2}} (E_{k - 2}), T_{v s_{k - 2}} (E_{k - 1})]
            = [X_{2, -(k - 1)}, E_{k - 1}].
        \end{split}
    \]
    In a similar manner we prove part (\ref{lemma, B.2, 1c}), except now fix a
    natural number $k$ such that $3 \leq k < n$, and redefine the Weyl group
    elements $u$ and $v$ from above. Specifically, define them as $u := (s_2
    \cdots s_{n - 1})(s_1 \cdots s_{n - 2}) s_n s_{n - 1} s_n (s_{n - 2} \cdots
    s_1) (s_{n - 1} \cdots s_k)$ and $v := (s_2 \cdots s_{n - 1})(s_1 \cdots
    s_{n - 2}) s_n s_{n - 1} s_n (s_{n - 2} \cdots s_{k - 1})$. A direct
    computation yields
    \[
        \begin{split}
            X_{1, k} &= T_u(E_{k - 1}) = T_{u s_k} T_{s_k} (E_{k - 1}) = T_{u
            s_k} \left([ E_k, E_{k - 1}] \right)
            \\
            &= [T_{u s_k} (E_k), T_{u s_k} (E_{k - 1})] = [X_{1, k + 1}, E_k]
        \end{split}
    \]
    and
    \[
        \begin{split}
            X_{2, k} &= T_v(E_{k - 2}) = T_{v s_{k - 1}} T_{s_{k - 1}} (E_{k -
            2}) = T_{v s_{k - 1}} \left([ E_{k - 1}, E_{k - 2}] \right)
            \\
            &= [T_{v s_{k - 1}} (E_{k - 1}), T_{v s_{k - 1}} (E_{k - 2})] =
            [X_{2, k + 1}, E_k].
        \end{split}
    \]

    Parts (\ref{lemma, B.2, 1d}), (\ref{lemma, B.2, 1e}), and (\ref{lemma, B.2,
    1f}) are each proved similarly. For a natural number $k$ less
    than $\ell(w)$, let $w^{(k)}$ be the Weyl group element $s_{i_1} \cdots
    s_{i_k}$ obtained by multiplying together the first $k$ simple reflections
    $s_{i_1}, \cdots, s_{i_k}$ in the reduced expression of $w$ given in
    (\ref{B.2, reduced expression}). We compute
    \[
        \begin{split}
            X_{1, 0} &= T_{w^{(2n - 2)}} (E_n) = T_{w^{(2n - 4)}} T_{s_n s_{n -
            1}} (E_n) = T_{w^{(2n - 4)}} ([E_n, E_{n - 1}])
            \\
            &= [T_{w^{(2n - 4)}} (E_n), T_{w^{(2n - 4)}} (E_{n - 1})]) = [X_{2,
            0}, E_1],
            \\
            X_{1, n} &= T_{w^{(3n - 3)}}(E_{n - 1}) = T_{w^{(2n - 1)}} T_{s_{n
            - 2}} T_{s_{n - 3} \cdots s_2 s_1} (E_{n - 1}) = T_{w^{(2n - 1)}}
            T_{s_{n - 2}} (E_{n - 1})
            \\
            &= T_{w^{(2n - 1)}} \left( [ E_{n - 2}, E_{n - 1}] \right) =
            [T_{w^{(2n - 1)}} (E_{n - 2}), T_{w^{(2n - 1)}} (E_{n - 1})] =
            [X_{2, n}, E_1],
            \\
            X_{2, 0} &= T_{w^{(2n - 4)}} (E_n) = T_{w^{(n - 3)}} T_{s_{n - 1}}
            T_{s_1 \cdots s_{n - 2}} (E_n) = T_{w^{(n - 3)}} T_{s_{n - 1}}
            (E_n)
            \\
            &= T_{w^{(n - 3)}} \left( [ E_{n - 1}, E_n] \right) = [T_{w^{(n -
            3)}} (E_{n - 1}), T_{w^{(n - 3)}} (E_n)] = [X_{2, - n}, E_n].
        \end{split}
    \]
    Finally, we prove part (\ref{lemma, B.2, 1h}). We have
    \[
        \begin{split}
            [2]_q X_{2, n} &= [2]_q T_{w^{(2n - 1)}} (E_{n - 2}) = [2]_q
            T_{w^{(n - 3)}} T_{s_{n - 1} s_n} T_{s_1 \cdots s_n} (E_{n - 2})
            \\
            &= [2]_q T_{w^{(n - 3)}} T_{s_{n - 1} s_n} (E_{n - 1}) = T_{w^{(n -
            3)}} \left( [[E_{n - 1}, E_n], E_n] \right)
            \\
            &= [[T_{w^{(n - 3)}} (E_{n - 1}), T_{w^{(n - 3)}} (E_n)], T_{w^{(n
            - 3)}} (E_n)] = [[ X_{2, -n}, E_n], E_n].
        \end{split}
    \]
    The $q$-commutator identity $[X_{2, -n}, E_n] = X_{2, 0}$ from part
    (\ref{lemma, B.2, 1f}) of this lemma gives us $[2]_q X_{2, n} = [X_{2, 0},
    E_n]$.

\end{proof}

As a consequence of parts (\ref{lemma, B.2, 1a}) and (\ref{lemma, B.2, 1b}) of
Lemma (\ref{lemma, B.2, 1}), we can write each Lusztig root vector $X_{i, -k}$
(for $3 \leq k \leq n$ and $i= 1, 2$) as a sequence of nested $q$-commutators
(recall the definition in (\ref{definition, nested q-commutator})).
Specifically, for $3 \leq k \leq n$, we have
\begin{align}
    \label{B.2, as nested q-commutators, a} X_{2, -k} &= \mathbf{E}_{2, \dots,
    k - 1},
    \\
    \label{B.2, as nested q-commutators, b} X_{1, -k} &= \mathbf{E}_{2, 1, 3, 4
    \dots, k - 1}.
\end{align}
Since the Chevalley generator $E_1$ commutes with all other Chevalley
generators except $E_2$, we can also write $X_{1, -k} = \mathbf{E}_{2, \ldots,
k - 1, 1}$ for $3 \leq k \leq n$.  Hence,
\begin{equation}
    \label{B.2, 3a} X_{1, -k} = [X_{2, -k}, E_1], \hspace{5mm} (\text{for } 3
    \leq k \leq n).
\end{equation}
From Lemma (\ref{lemma, B.2, 1}), part (\ref{lemma, B.2, 1d}), we can also
write $X_{1, 0}$ as a sequence of nested $q$-commutators. We have $X_{1, 0} =
\mathbf{E}_{2, \dots, n, 1}$. Using again the fact that the Chevalley generator
$E_1$ commutes with $E_3,\cdots, E_n$, we can write $X_{1, 0} = \mathbf{E}_{2,
1, 3, 4, \ldots, n}$. Hence,
\begin{equation}
    \label{B.2, 3b} X_{1, 0} = [X_{1, -n}, E_n].
\end{equation}

\begin{lemma}

    \label{lemma, B.2, 4}

    Let $X_{i, \pm k}$ be as in (\ref{B.2, Lusztig root vectors}). Then

    \begin{enumerate}

        \item \label{lemma, B.2, 4g} $[2]_q X_{1, 2} = [X_{2, 0}, X_{1, 0}]$,

        \item \label{lemma, B.2, 4i} $X_{1, 2} = [X_{1, -n}, X_{2, n}]$.

    \end{enumerate}

\end{lemma}

\begin{proof}

    We first prove part (\ref{lemma, B.2, 4g}). In the following, let $w_k$,
    for $1 \leq k \leq \ell(w)$, be the Weyl group element obtained by
    multiplying the first $k$ simple reflections in the reduced expression
    (\ref{B.2, reduced expression}). We have
    \[
        \begin{split}
            [2]_q X_{1, 2} &= [2]_q T_{w_{2n - 3}} (E_{n - 1}) = [2]_q T_{w_{2n
            - 5}} T_{s_{n - 1} s_n} (E_{n - 1})
            \\
            &= T_{w_{2n - 5}} \left([[ E_{n - 1}, E_n], E_n] \right) = T_{w_{2n
            - 5}} \left( [T_{s_{n - 1}}(E_n), T_{s_{n - 1} s_n s_{n - 1}}(E_n)]
            \right)
            \\
            & = [T_{w_{2n - 5}} T_{s_{n - 1}}(E_n), T_{w_{2n - 5}}T_{s_{n - 1}
            s_n s_{n - 1}}(E_n)] = [T_{w_{2n - 4}} (E_n), T_{w_{2n - 2}}(E_n)]
            \\
            &= [X_{2, 0}, X_{1, 0}].
        \end{split}
    \]
    To prove part (\ref{lemma, B.2, 4i}), use the identity $[2]_q X_{2, n} =
    [X_{2, 0}, E_n]$ from part (\ref{lemma, B.2, 1h}) of Lemma (\ref{lemma,
    B.2, 1}) and the $q$-Jacobi identity (\ref{q-Jacobi identity}) to get
    \[
        [2]_q [X_{1, -n}, X_{2, n}] = [X_{1, -n}, [X_{2, 0}, E_n]] = [[X_{1,
        -n}, X_{2, 0}], E_n] - [[X_{1, -n}, E_n], X_{2, 0}].
    \]
    However, since $[X_{1, -n}, X_{2, 0}] = 0$ (\ref{B.2, L-S, 4}) and $X_{1,
    0} = [X_{1, -n}, E_n]$ (\ref{B.2, 3b}), we have
    \[
        [2]_q [X_{1, -n}, X_{2, n}] = - [X_{1, 0}, X_{2, 0}] = [X_{2, 0}, X_{1,
        0}].
    \]
    Since $[X_{2, 0}, X_{1, 0}] = [2]_q X_{1, 2}$ (\ref{lemma, B.2, 4g}), we
    obtain $[X_{1, -n}, X_{2, n}] = X_{1, 2}$.

\end{proof}

\begin{lemma}

    \label{lemma, B.2, 2}

    Let $X_{i, \pm k}$ be as in (\ref{B.2, Lusztig root vectors}). Then

    \begin{enumerate}

        \item \label{lemma, B.2, 2.1} For $3 \leq k < j \leq n$,

            \begin{enumerate}

                \item \label{lemma, B.2, 2.1a} $[X_{2, -k}, X_{1, j}] =
                    \widehat{q_1}\, X_{1, -k} X_{2, j}$,

                \item \label{lemma, B.2, 2.1b} $[X_{2, -k}, X_{1, -j}] =
                    \widehat{q_1} X_{2, -j} X_{1, -k}$.

            \end{enumerate}

        \item \label{lemma, B.2, 2.2} For $3 \leq k \leq n$,

        \begin{enumerate}

            \item \label{lemma, B.2, 2.2a} $[X_{2, -k}, X_{1, 0}] =
                \widehat{q_1} X_{1, -k} X_{2, 0}$,

            \item \label{lemma, B.2, 2.2b} $[X_{2, -k}, X_{1, 2}] = q
                \widehat{q}^{\,2} X_{1, -k} X_{2, 0}^2$,

            \item \label{lemma, B.2, 2.2c} $[X_{2, -k}, X_{2, k}] = (-q_1)^{n -
                k} \frac{q \widehat{q}}{[2]_q} X_{2, 0}^2 + \widehat{q_1}
                \sum_{k < j \leq n} (-q_1)^{j - k - 1} X_{2, -j} X_{2, j}$,

            \item \label{lemma, B.2, 2.2d} $[X_{2, -k}, X_{1, k}] =
                \widehat{q_1} X_{1, -k} X_{2, k} + (-q_1)^{n - k} \left(
                \widehat{q} X_{2, 0} X_{1, 0} - \left( \widehat{q_1} + 1
                \right) X_{1, 2} \right) \\ \phantom{\hspace{25mm}}+
                \widehat{q_1} \sum_{k < j \leq n} (-q_1)^{j - k - 1} (X_{2, -j}
                X_{1, j} - \widehat{q_1} X_{1, -j} X_{2, j})$.

        \end{enumerate}

    \end{enumerate}

    (Recall $q_1 := q^{\langle \alpha_1, \alpha_1 \rangle / 2} = q^2$.)

\end{lemma}

\begin{proof}

    We prove (\ref{lemma, B.2, 2.1a}) by inducting on $j$, starting with the
    base case $j = n$.  Assume that $k < n$. Use the identity $X_{1, n} =
    [X_{2, n}, E_1]$ from part (\ref{lemma, B.2, 1e}) of Lemma (\ref{lemma,
    B.2, 1}) and the $q$-Jacobi identity (\ref{q-Jacobi identity}) to get
    \[
        [X_{2, -k}, X_{1, n}] = [[X_{2, -k}, X_{2, n}], E_1] - q_1[[X_{2, -k},
        E_1], X_{2, n}] + \widehat{q_1}\, [X_{2, -k}, E_1] X_{2, n}.
    \]
    Next, apply the identities $[X_{2, -k}, X_{2, n}] = 0$ and $[X_{2, -k},
    E_1] = X_{1, -k}$ from (\ref{B.2, L-S, 3}) and (\ref{B.2, 3a}) to obtain
    \[
        [X_{2, -k}, X_{1, n}] = - q_1[X_{1, -k}, X_{2, n}] + \widehat{q_1}\,
        X_{1, -k} X_{2, n}.
    \]
    Since $[X_{1, -k}, X_{2, n}] = 0$ (\ref{B.2, L-S, 3}), we get $[X_{2, -k},
    X_{1, n}] = \widehat{q_1}\, X_{1, -k} X_{2, n}$. This proves the statement
    in the case when $j = n$.

    Now let $k < j < n$. The $q$-commutator identity $X_{1, j} = [X_{1, j + 1},
    E_j]$ from Lemma (\ref{lemma, B.2, 1}), part (\ref{lemma, B.2, 1c})
    together with the $q$-Jacobi identity (\ref{q-Jacobi identity}) gives us
    \[
        [X_{2, -k}, X_{1, j}] = [[ X_{2, -k}, X_{1, j + 1}], E_j] -q_1[[X_{2,
        -k}, E_j], X_{1, j + 1}] + \widehat{q_1}\, [X_{2, -k}, E_j] X_{1, j +
        1}.
    \]
    Since $X_{2, -k} = \mathbf{E}_{2,\ldots, k - 1}$  (\ref{B.2, as nested
    q-commutators, a}) and the Chevalley generator $E_j$ commutes with each of
    $E_1,\cdots, E_{k - 1}$, we have $[X_{2, - k}, E_j] = 0$, and by the
    induction hypothesis, we get $[X_{2, -k}, X_{1, j}] = \widehat{q_1}\, [
    X_{1, -k} X_{2, j + 1}, E_j]$.  Invoking the $q$-Leibniz identity
    (\ref{q-Leibniz identity}) gives us
    \[
        [X_{2, -k}, X_{1, j}] = \widehat{q_1}\, \left( X_{1, -k} [X_{2, j +
        1}, E_j] + q_1^{-1}[X_{1, -k}, E_j] X_{2, j + 1} \right).
    \]
    Since $X_{1, -k} = \mathbf{E}_{2,1,3, 4, \ldots, k - 1}$ (\ref{B.2, as
    nested q-commutators, b}), it follows that $[X_{1, - k}, E_j] = 0$.
    Furthermore, since $X_{2, j} = [X_{2, j + 1}, E_j]$ (Lemma (\ref{lemma,
    B.2, 1}), part \ref{lemma, B.2, 1c}), we obtain
    \[
        [X_{2, -k}, X_{1, j}] = \widehat{q_1}\, X_{1, -k} X_{2, j}.
    \]
    This concludes a proof of (\ref{lemma, B.2, 2.1a}).

    We prove (\ref{lemma, B.2, 2.1b}) by inducting on the difference $j - k$,
    beginning with the base case $j = k + 1$.  In this situation, use the
    identity $X_{1, -(k + 1)} = [X_{1, -k)}, E_k]$ from Lemma (\ref{lemma,
    B.2, 1}), part (\ref{lemma, B.2, 1b}) in conjunction with the $q$-Jacobi
    identity (\ref{q-Jacobi identity}) to get
    \[
        [X_{2,\! -k},\! X_{1,\! -(k + 1)}] \!=\! [[X_{2,\! -k},\! X_{1,\!
        -k}],\! E_k] - q_1 [[X_{2,\! -k},\! E_k],\! X_{1,\! -k}] +
        \widehat{q_1} [X_{2,\! -k},\! E_k] X_{1,\!  -k}.
    \]
    Next, use the identities $[X_{2, -k}, X_{1, -k}] = 0$ and $[X_{2, -k}, E_k]
    = X_{2, -(k + 1)}$ from (\ref{B.2, L-S, 1}) and Lemma (\ref{lemma, B.2,
    1}), part (\ref{lemma, B.2, 1b}), to simplify the right-hand side of the
    equation above. We have
    \[
        [X_{2, -k}, X_{1, -(k + 1)}] = - q_1 [X_{2, -(k + 1)}, X_{1, -k}] +
        \widehat{q_1} X_{2, -(k + 1)} X_{1, -k}.
    \]
    Since $[X_{2, -(k + 1)}, X_{1, -k}] = 0$ (\ref{B.2, L-S, 1}), then $[X_{2,
    -k}, X_{1, -(k + 1)}] = \widehat{q_1} X_{2, -(k + 1)} X_{1, -k}$. This
    proves the statement in the case when $j = k + 1$.

    Now assume that $3 \leq k < j \leq n$ and $j - k > 1$. Applying the
    $q$-commutator identity $X_{1, -j} = [X_{1, -(j - 1)}, E_{j - 1}]$ (Lemma
    (\ref{lemma, B.2, 1}), part \ref{lemma, B.2, 1b}) in conjunction with the
    $q$-Jacobi identity (\ref{q-Jacobi identity}) produces
    \[
        \begin{split}
            [X_{2, -k}, X_{1, -j}] = &[[X_{2, -k}, X_{1, -(j - 1)}], E_{j - 1}]
            - q_1 [[X_{2, -k}, E_{j - 1}], X_{1, -(j - 1)}]
            \\
            &+ \widehat{q_1} [X_{2, -k}, E_{j - 1}] X_{1, -(j - 1)}.
        \end{split}
    \]
    Since the Chevalley generator $E_{j - 1}$ commutes with each of
    $E_1,\ldots, E_{k - 1}$ (recall that $j - k > 1$) and $X_{2, -k} =
    \mathbf{E}_{2,\ldots, k - 1}$ (\ref{B.2, as nested q-commutators, a}), it
    follows that $[X_{2, -k}, E_{j - 1}] = 0$.  Furthermore, applying the
    induction hypothesis to rewrite the right-hand side of the equation above
    gives us
    \[
        [X_{2, -k}, X_{1, -j}] = \widehat{q_1}[X_{2, -(j - 1)} X_{1, -k}, E_{j
        - 1}].
    \]
    Next, apply the $q$-Leibniz identity (\ref{q-Leibniz identity}). We get
    \[
        [X_{2, -k}, X_{1, -j}] = \widehat{q_1} \left( X_{2, -(j - 1)} [X_{1,
        -k}, E_{j - 1}] + [X_{2, -(j - 1)}, E_{j - 1}] X_{1, -k}\right).
    \]
    However, we also have $[X_{1, -k}, E_{j - 1}] = 0$. This can be verified
    using the same argument that $[X_{2, -k}, E_{j - 1}] = 0$. The only
    difference is that $X_{1 - k} = \mathbf{E}_{2, 1, 3, 4, \ldots, k - 1}$
    (\ref{B.2, as nested q-commutators, b}), whereas $X_{2, -k} =
    \mathbf{E}_{2,\ldots, k - 1}$.  Using the identity $[X_{2, -(j - 1)}, E_{j
    - 1}] = X_{2, -j}$ (Lemma (\ref{lemma, B.2, 1}), part \ref{lemma, B.2,
    1b}), we conclude that $[X_{2, -k}, X_{1, -j}] = \widehat{q_1} X_{2, -j}
    X_{1, -k}$.

    We prove (\ref{lemma, B.2, 2.2a}) by first using the identity $X_{1, 0} =
    [X_{2, 0}, E_1]$ from part (\ref{lemma, B.2, 1d}) of Lemma (\ref{lemma,
    B.2, 1}) in conjunction with the $q$-Jacobi identity (\ref{q-Jacobi
    identity}) to obtain
    \[
        [X_{2, -k}, X_{1, 0}] = [[X_{2, -k}, X_{2, 0}], E_1] - q_1 [[X_{2, -k},
        E_1], X_{2, 0}] + \widehat{q_1} [X_{2, -k}, E_1] X_{2, 0}.
    \]
    Then we invoke the identities $[X_{2, -k}, X_{2, 0}] = 0$ and $[X_{2, -k},
    E_1] = X_{1, -k}$ from (\ref{B.2, L-S, 4}) and (\ref{B.2, 3a}),
    respectively, to get $[X_{2, -k}, X_{1, 0}] = - q_1 [X_{1, -k}, X_{2, 0}] +
    \widehat{q_1} X_{1, -k} X_{2, 0}$.  Finally, since $[X_{1, -k}, X_{2, 0}] =
    0$ (\ref{B.2, L-S, 4}), we have $[X_{2, -k}, X_{1, 0}] = \widehat{q_1}
    X_{1, -k} X_{2, 0}$.

    To prove (\ref{lemma, B.2, 2.2b}), we begin by applying the identity $[2]_q
    X_{1, 2} = [X_{2, 0}, X_{1, 0}]$ from part (\ref{lemma, B.2, 4g}) of Lemma
    (\ref{lemma, B.2, 4}) and the $q$-Jacobi identity (\ref{q-Jacobi identity})
    to obtain
    \[
        [2]_q [X_{2, -k}, X_{1, 2}] = [[X_{2, -k}, X_{2, 0}], X_{1, 0}] -
        [[X_{2, -k}, X_{1, 0}], X_{2, 0}].
    \]
    The identities $[X_{2, -k}, X_{2, 0}] = 0$ and $[X_{2, -k}, X_{1, 0}] =
    \widehat{q_1} X_{1, -k} X_{2, 0}$ from (\ref{B.2, L-S, 4}) and part
    (\ref{lemma, B.2, 2.2a}) of this lemma, respectively, can be used to
    rewrite the expression on the right-hand side. With this, we obtain $[2]_q
    [X_{2, -k}, X_{1, 2}] =  - \widehat{q_1} [X_{1, -k} X_{2, 0}, X_{2, 0}]$.
    Applying the $q$-Leibniz identity (\ref{q-Leibniz identity}) gives us
    \[
        [2]_q [X_{2, -k}, X_{1, 2}] =  - \widehat{q_1} \left( X_{1, -k} [X_{2,
        0}, X_{2, 0}] + q_1 [X_{1, -k}, X_{2, 0}] X_{2, 0} \right).
    \]
    Finally, we invoke the identity $[X_{1, -k}, X_{2, 0}] = 0$ (\ref{B.2, L-S,
    4}) to get
    \[
        [2]_q [X_{2, -k}, X_{1, 2}] = - \widehat{q_1} (1 - q^2) X_{1, -k} X_{2,
        0}^2,
    \]
    which is equivalent to $[X_{2, -k}, X_{1, 2}] = q \widehat{q}^{\,2} X_{1,
    -k} X_{2, 0}^2$.

    We prove (\ref{lemma, B.2, 2.2c}) by inducting on $k$. We begin with the
    base case, namely $k = n$.  Invoking the identity $[2]_q X_{2, n} = [X_{2,
    0}, E_n]$ from Lemma (\ref{lemma, B.2, 1}), part (\ref{lemma, B.2, 1h}) and
    the $q$-Jacobi identity (\ref{q-Jacobi identity}) gives us
    \[
        [2]_q [X_{2, -n}, X_{2, n}] = [X_{2, -n}, [X_{2, 0}, E_n]] = [[X_{2,
        -n}, X_{2, 0}], E_n] - [[X_{2, -n}, E_n], X_{2, 0}].
    \]
    Next, use the identities $[X_{2, -n}, X_{2, 0}] = 0$ and $[X_{2, -n}, E_n]
    = X_{2, 0}$ from (\ref{B.2, L-S, 4}) and part (\ref{lemma, B.2, 1f}) of
    Lemma (\ref{lemma, B.2, 1}), respectively, to obtain
    \[
        [2]_q [X_{2, -n}, X_{2, n}] = -[[X_{2, -n}, E_n], X_{2, 0}] = - [X_{2,
        0}, X_{2, 0}] = -(1 - q^2) X_{2, 0}^2 = q \widehat{q}\, X_{2, 0}^2.
    \]
    This verifies the statement in the case when $k = n$. Next let $k < n$.
    Apply the identity $X_{2, k} = [X_{2, k + 1}, E_k]$ from part (\ref{lemma,
    B.2, 1c}) of Lemma (\ref{lemma, B.2, 1}) and the $q$-Jacobi identity
    (\ref{q-Jacobi identity}) to get
    \[
        [X_{2, -k}, X_{2, k}] = [[X_{2, -k}, X_{2, k + 1}], E_k] - q_1 [[X_{2,
        -k}, E_k], X_{2, k + 1}] + \widehat{q_1} [X_{2, -k}, E_k] X_{2, k + 1}.
    \]
    Finally, we can simplify the expression on the right-hand side of the
    equation above by using the identities $[X_{2, -k}, X_{2, k + 1}] = 0$ and
    $[X_{2, -k}, E_k] = X_{2, -(k + 1)}$ from (\ref{B.2, L-S, 3}) and part
    (\ref{lemma, B.2, 1b}) of Lemma (\ref{lemma, B.2, 1}), respectively. This
    gives us
    \[
        [X_{2, -k}, X_{2, k}] = - q_1 [X_{2, -(k + 1)}, X_{2, k + 1}] +
        \widehat{q_1} X_{2, -(k + 1)} X_{2, k + 1},
    \]
    and the result follows by induction.

    Next, we prove (\ref{lemma, B.2, 2.2d}) using induction. We start with the
    base case $j = n$..  Using the identity $X_{1, n} = [X_{2, n}, E_1]$ from
    part (\ref{lemma, B.2, 1e}) of Lemma (\ref{lemma, B.2, 1}) in conjunction
    with the $q$-Jacobi identity (\ref{q-Jacobi identity}) gives us
    \[
        [X_{2, - n}, X_{1, n}] = [[X_{2, -n}, X_{2, n}], E_1] - q_1[[X_{2, -n},
        E_1], X_{2, n}] + \widehat{q_1}\, [X_{2, -n}, E_1] X_{2, n}.
    \]
    Since $[X_{2, -n}, E_1] = X_{1, -n}$ and $[X_{2, -n}, X_{2, n}] = \frac{q
    \widehat{q}}{[2]_q} X_{2, 0}^2$ (from (\ref{B.2, 3a}) and part (\ref{lemma,
    B.2, 2.2c}) of this lemma, respectively), we obtain
    \[
        [X_{2, - n}, X_{1, n}] = \frac{q \widehat{q}}{[2]_q} [X_{2, 0}^2, E_1]
        - q_1[X_{1, -n}, X_{2, n}] + \widehat{q_1} X_{1, -n} X_{2, n}.
    \]
    Next we apply the $q$-Leibniz identity (\ref{q-Leibniz identity}) to the
    $q$-commutator $[X_{2, 0}^2, E_1]$ and use $[X_{1, -n}, X_{2, n}] = X_{1,
    2}$ (Lemma (\ref{lemma, B.2, 4}), part \ref{lemma, B.2, 4i}) to get
    \[
        [X_{2, - n}, X_{1, n}] = \frac{q \widehat{q}}{[2]_q} \left( X_{2, 0}
        [X_{2, 0}, E_1] + q_1^{-1} [X_{2, 0}, E_1] X_{2, 0} \right) - q_1 X_{1,
        2} + \widehat{q_1} X_{1, -n} X_{2, n}.
    \]
    Replacing each instance of $[X_{2, 0}, E_1]$ with $X_{1, 0}$ (Lemma
    (\ref{lemma, B.2, 1}), part \ref{lemma, B.2, 1d}) gives us
    \[
        [X_{2, - n}, X_{1, n}] = \frac{q \widehat{q}}{[2]_q} \left( X_{2, 0}
        X_{1, 0} + q_1^{-1} X_{1, 0} X_{2, 0} \right) - q_1 X_{1, 2} +
        \widehat{q_1} X_{1, -n} X_{2, n}.
    \]
    The $q$-commutator identity $[X_{2, 0}, X_{1, 0}] = [2]_q X_{1, 2}$ from
    Lemma (\ref{lemma, B.2, 4}), part (\ref{lemma, B.2, 4g}), means $X_{1, 0}
    X_{2, 0} = X_{2, 0} X_{1, 0} - [2]_q X_{1, 2}$. Hence, $[X_{2, -n}, X_{1,
    n}]$ can be written as a linear combination of the monomials $X_{2, 0}
    X_{1, 0}$, $X_{1, 2}$, and $X_{1, -n} X_{2, n}$. We obtain
    \[
        [X_{2, - n}, X_{1, n}] = \widehat{q} X_{2, 0} X_{1, 0} -
        \left(\widehat{q_1} + 1\right) X_{1, 2} + \widehat{q_1} X_{1, -n} X_{2,
        n}.
    \]
    This proves (\ref{lemma, B.2, 2.2d}) in the case when $j = n$.

    Now let $k < n$. Invoking the identity $X_{1, k} = [X_{1, k + 1}, E_k]$
    (Lemma (\ref{lemma, B.2, 1}), part \ref{lemma, B.2, 1c}) and the $q$-Jacobi
    identity (\ref{q-Jacobi identity}) gives us
    \[
        [X_{2, -k}, X_{1, k}] = [[X_{2, -k}, X_{1, k + 1}], E_k] - q_1 [[X_{2,
        -k}, E_k], X_{1, k + 1}] + \widehat{q_1} [X_{2, -k}, E_k] X_{1, k + 1}.
    \]
    We next use the identities found in part (\ref{lemma, B.2, 1b}) of Lemma
    (\ref{lemma, B.2, 1}) and part (\ref{lemma, B.2, 2.1a}) of this lemma to
    replace the $q$-commutators, $[X_{2, -k}, E_k]$ and $[X_{2, -k}, X_{1, k +
    1}]$, appearing in the right-hand side of the equation above with $X_{2,
    -(k + 1)}$ and $\widehat{q_1} X_{1, -k} X_{2, k + 1}$, respectively. We get
    \[
        [X_{2, -k}, X_{1, k}] = \widehat{q_1} [X_{1, -k} X_{2, k + 1}, E_k] -
        q_1 [  X_{2, -(k + 1)}, X_{1, k + 1}] + \widehat{q_1} X_{2, -(k + 1)}
        X_{1, k + 1}.
    \]
    Invoking the $q$-Leibniz identity (\ref{q-Leibniz identity}) gives us
    \[
        \begin{split}
            [X_{2, -k}, X_{1, k}] &= \widehat{q_1} \left( X_{1, -k} [X_{2, k +
            1}, E_k] + q_1^{-1} [X_{1, -k}, E_k] X_{2, k + 1} \right)
            \\
            &\hspace{5mm}- q_1 [X_{2, -(k + 1)}, X_{1, k + 1}] + \widehat{q_1}
            X_{2, -(k + 1)} X_{1, k + 1}.
        \end{split}
    \]
    Finally, we apply the identities $[X_{1, -k}, E_k] = X_{1, -(k + 1)}$ and
    $[X_{2, k + 1}, E_k] = X_{2, k}$ found in parts (\ref{lemma, B.2, 1b}) and
    (\ref{lemma, B.2, 1c}) of Lemma (\ref{lemma, B.2, 1}), respectively, to
    obtain
    \[
        \begin{split}
            [X_{2, -k}, X_{1, k}] &= \widehat{q_1} \left( X_{1, -k} X_{2,
            k} + q_1^{-1} X_{1, -(k + 1)} X_{2, k + 1} \right)
            \\
            &\hspace{5mm}- q_1 [X_{2, -(k + 1)}, X_{1, k + 1}] +
            \widehat{q_1} X_{2, -(k + 1)} X_{1, k + 1},
        \end{split}
    \]
    and the result follows from the induction hypothesis.

\end{proof}

Now we are ready to compute ${\mathcal N}(\llbracket w \rrbracket)$.

\begin{theorem}

    Let $w = w_{0, 0, 2, 0, n - 2} \in W(B_n)$, where $n > 2$.  Then ${\mathcal
    N}\left( \llbracket w \rrbracket \right) = 3$.

\end{theorem}

\begin{proof}

    From Lemma (\ref{lemma, B.2, 2}), part (\ref{lemma, B.2, 2.2d}), we have
    \[
        \label{B.2, 5}
        \begin{split}
            [X_{2, -3}, X_{1, 3}]
            &= \widehat{q_1} \left( X_{1, -3} X_{2, 3} + \sum_{3 < j \leq n}
            (-q_1)^{j - 4} (X_{2, -j} X_{1, j} - \widehat{q_1} X_{1, -j} X_{2,
            j}) \right)
            \\
            &\hspace{5mm} + (-q_1)^{n - 3} \left( \widehat{q} X_{2, 0} X_{1, 0}
            -\left( \widehat{q_1} +  1 \right) X_{1, 2} \right).
        \end{split}
    \]
    This illustrates how to write $[X_{2, -3}, X_{1, 3}]$ as a linear
    combination of ordered monomials. Since $[X_{2, -3}, X_{1, 3}]$ is nonzero,
    ${\mathcal N}(\llbracket w \rrbracket) > 1$.

    Our next objective is to write $[X_{2, -3}, [X_{2, -3}, X_{1, 3}]]$ as a
    linear combination of ordered monomials to determine whether the nilpotency
    index is equal to $2$, or greater than $2$. This involves applying
    commutation relations among the Lusztig root vectors $X_{i, \pm k}$,
    particularly those involving $X_{2, -3}$. Observe first that the identities
    in (\ref{B.2, L-S, 1}), (\ref{B.2, L-S, 2}), (\ref{B.2, L-S, 3}), and
    (\ref{B.2, L-S, 4}) tell us that $X_{2, -3}$ $q$-commutes with $X_{1, -3}$,
    $X_{2, \pm j}$ ($j > 3$), and $X_{2, 0}$.  Furthermore, parts (\ref{lemma,
    B.2, 2.1a}), (\ref{lemma, B.2, 2.1b}), (\ref{lemma, B.2, 2.2a}),
    (\ref{lemma, B.2, 2.2b}), and (\ref{lemma, B.2, 2.2c}) of Lemma
    (\ref{lemma, B.2, 2}) give us the commutation relations involving $X_{2,
    -3}$ with any of $X_{2, 3}$, $X_{1, \pm j}$ ($j > 3$), $X_{1, 0}$, and
    $X_{1, 2}$. That is to say, we have enough information to be able to write
    $[X_{2, -3}, [X_{2, -3}, X_{1, 3}]]$ as a linear combination of ordered
    monomials.  We get
    \[
        \begin{split}
            [X_{2, -3}, [X_{2, -3}, X_{1, 3}]]
            &= \left(q_1 + q_1^{-1}\right) \widehat{q_1}^2 \left( \sum_{3 < j
            \leq n} (-q_1)^{j - 4} X_{2, -j} X_{1, -3} X_{2, j} \right)
            \\
            &\hspace{5mm} +(-q_1)^{n - 3} \left(q_1 + q_1^{-1} \right)
            \widehat{q}^{\,2} q X_{1, -3} X_{2, 0}^2.
        \end{split}
    \]
    Hence ${\mathcal N}(\llbracket w \rrbracket) > 2$.

    Finally, observe that the identities in (\ref{B.2, L-S, 1}) - (\ref{B.2,
    L-S, 4}) tell us that each of the $q$-commutators $[X_{2, -3}, X_{2, \pm
    j}]$ (for $j > 3$) equals $0$, as well as $[X_{2, -3}, X_{1, -3}]$ and
    $[X_{2, -3}, X_{2, 0}]$. In other words, $X_{2, -3}$ $q$-commutes with
    every Lusztig root vector $X_{i, \pm k}$ appearing in the expression on the
    right-hand side of the equation above. Hence, as a consequence of the
    $q$-Leibniz identity (\ref{q-Leibniz identity}), $X_{2, -3}$ also
    $q$-commutes with $[X_{2, -3}, [X_{2, -3}, X_{1, 3}]]$.  Therefore, $[X_{2,
    -3}, [X_{2, -3}, [X_{2, -3}, X_{1, 3}]]]=  0$, and this proves ${\mathcal
    N}(\llbracket w \rrbracket) = 3$.

\end{proof}

\subsection{Nilpotency index, Case B.4}

Let $\mathfrak{g}$ be the complex Lie algebra of type $B_n$, where $n > 3$.
Consider the Weyl group element $w := w_{0, 1, 2, 1, n - 4} \in W(B_n)$ and the
reduced expression
\begin{equation}
    \label{B.4, reduced expression}
    w = (s_3 \cdots s_{n-1})(s_2 \cdots s_{n-2})s_ns_{n-1}s_ns_1(s_{n-2} \cdots
    s_2) (s_{n-1} \cdots s_3)
\end{equation}
Let $\Delta_w$ be the set of roots of $w$. The reduced expression for $w$
corresponds to the following convex order on $\Delta_w$:
\begin{equation}
    \label{convex order, B.4}
    \begin{split}
        &e_3 - e_4 \prec e_3 - e_5 \prec \cdots \prec e_3 - e_n \prec e_2 - e_4
        \prec e_2 - e_5 \prec \cdots
        \\
        &\cdots \prec e_2 - e_n \prec e_3 \prec e_2 + e_3 \prec e_2 \prec e_1 -
        e_4 \prec e_3 + e_n \prec e_3 - e_{n - 1} \prec \cdots
        \\
        &\cdots \prec e_3 + e_5 \prec e_1 + e_3 \prec e_2 + e_n \prec e_2 +
        e_{n - 1} \prec \cdots \prec e_2 + e_5 \prec e_1 + e_2.
    \end{split}
\end{equation}
Let ${\mathcal U}_q^+[w]$ be the quantum Schubert cell algebra associated to
$w$. The Lusztig root vectors will be denoted by
\begin{equation}
    \label{B.4, Lusztig root vectors}
    X_{m, -4}, X_{r, s}, X_{i, 0}, X_{i, \pm k} \in {\mathcal U}_q^+[w]
\end{equation}
where $1 < i < 4 < k \leq n$, $1 \leq m < 4$, and $1 \leq r < s < 4$.

The following theorem gives us the nilpotency index of $\llbracket w_{0, 1, 2,
1, n - 4} \rrbracket$ for the case when $n = 4$.

\begin{theorem}

    Let $w = w_{0, 1, 2, 1, 0} \in W(B_4)$. Then ${\mathcal N}(\llbracket w
    \rrbracket) = 3$.

\end{theorem}

\begin{proof}

    Consider the Weyl group element $u = ws_2$. With respect to the reduced
    expression $u = s_3s_2s_4s_3s_4s_1s_2s_3s_2$, let $x_1,\cdots, x_9$ be the
    Lusztig root vectors in the associated quantum Schubert cell algebra
    ${\mathcal U}_q^+[u]$. We aim to first find a presentation of ${\mathcal
    U}_q^+[u]$. Among the ${9 \choose 2}$ ($=36$) choices of pairs of Lusztig
    root vectors, $x_i$ and $x_j$ with $i < j$, Proposition (\ref{LS
    corollary}) tells us that at least $23$ of them satisfy the trivial
    $q$-commutation relation $[x_i, x_j] = 0$.  As it turns out, the remaining
    $13$ $q$-commutation relations are non-trivial. We will show they are given
    by
    \[
        \begin{split}
            &[x_1, x_4] = q \widehat{q}^{\,2} x_2 x_3^2, \hspace{10pt}
            [x_1, x_5] = \widehat{q_1} x_2 x_3, \hspace{10pt}
            [x_1, x_7] = q \widehat{q}^2 x_3^2 x_6, \hspace{10pt}
            [x_1, x_9] = x_2,
            \\
            &[x_1, x_8] = \widehat{q_1} \left(x_2 x_7 - (\widehat{q_1} +  1)
            x_4 x_6 + \widehat{q} x_3 x_5 x_6 \right), \hspace{10pt}
            [x_2, x_7] = \widehat{q_1} x_4 x_6,
            \\
            &[x_2, x_8] = q \widehat{q}^{\,2} x_5^2 x_6, \hspace{10pt}
            [x_3, x_5] = [2]_q x_4, \hspace{10pt}
            [x_3, x_8] = \widehat{q_1} x_5 x_7, \hspace{10pt}
            [x_3, x_9] = x_5,
            \\
            &[x_4, x_8] = q \widehat{q}^{\,2} x_5^2 x_7, \hspace{10pt}
            [x_4, x_9] = \frac{q \widehat{q}}{[2]_q} x_5^2, \hspace{10pt}
            [x_7, x_9] = x_8.
        \end{split}
    \]
    (Recall $q_1 := q^{\langle \alpha_1, \alpha_1 \rangle / 2} = q^2$.) Using
    the algorithm in the proof of Theorem (\ref{as nested E}), we write each
    Lusztig root vector as either a scalar multiple of a Chevalley generator,
    or as a scalar multiple of nested $q$-commutators of Chevalley generators.
    We obtain $x_1 = E_3$, $x_2 = \mathbf{E}_{3,2}$, $x_3 = \mathbf{E}_{3,4}$,
    $x_4 = \frac{1}{[2]_q} [\mathbf{E}_{3,2}, \mathbf{E}_{3,4,4}]$, $x_5 =
    \mathbf{E}_{3,4,2}$, $x_6 = \mathbf{E}_{3,2,1}$, $[2]_q x_7 =
    [\mathbf{E}_{3,2,1}, \mathbf{E}_{3,4,4}]$, $[2]_q x_8 =
    [[\mathbf{E}_{3,2,1}, \mathbf{E}_{3,4,4}], E_2]$, and $x_9 = E_2$. From
    this, we have the $q$-commutation relations $x_2 = [x_1, x_9]$, $x_5 =
    [x_3, x_9]$, and $x_8 = [x_7, x_9]$.  By a direct computation, we obtain
    \[
        [2]_q x_4 = T_{s_3s_2s_4} (E_3) = T_{s_3s_2} [E_4, \mathbf{E}_{4,3}] =
        [T_{s_3s_2}(E_4), T_{s_3s_2} T_{s_4s_3} (E_4)] = [x_3, x_5].
    \]
    Using $x_3 = \mathbf{E}_{3,4}$, $x_3 = [x_1, E_4]$ and $[x_1, x_3] = 0$,
    together with the $q$-Jacobi identity (\ref{q-Jacobi identity}), we get
    \[
        [x_1, \mathbf{E}_{3,4,4}] = [x_1, [x_3, E_4]] = [[x_1, x_3], E_4] -
        [[x_1, E_4], x_3] = -[x_3, x_3] = q \widehat{q} x_3^2.
    \]
    Analogously, applying the identities $[x_2, x_3] = 0$, $x_2 =
    \mathbf{E}_{3,2}$, and $x_5 = \mathbf{E}_{3,4,2}$ gives us
    \[
        \begin{split}
            [x_2, \mathbf{E}_{3,4,4}] &= [x_2, [x_3, E_4]]
            = [[x_2, x_3], E_4] - [[x_2, E_4], x_3]
            = - [[x_2, E_4], x_3]
            \\
            &= - [\mathbf{E}_{3,2,4}, x_3]
            = - [\mathbf{E}_{3,4,2}, x_3]
            = - [x_5, x_3]
            = [2]_q x_4.
        \end{split}
    \]
    Hence, by using $x_7 = [x_6, \mathbf{E}_{3,4,4}]$ and the $q$-Jacobi
    identity, we obtain
    \[
        [2]_q [x_1, x_7] =  [[x_1, x_6], \mathbf{E}_{3,4,4}] - q_1 [[x_1,
        \mathbf{E}_{3,4,4}], x_6] + \widehat{q_1}[x_1, \mathbf{E}_{3,4,4}] x_6.
    \]
    Since $[x_1, x_6] = 0$ and $[x_1, \mathbf{E}_{3, 4, 4}] = q \widehat{q}
    x_3^3$, the right hand side of the equation above simplifies to give us
    $[2]_q [x_1, x_7] = -q_1 q \widehat{q} [x_3^2, x_6] + \widehat{q_1} q
    \widehat{q} x_3^2 x_6$. However, since $[x_3, x_6] = 0$, then $[x_3^2, x_6]
    = 0$, and thus we conclude $[x_1, x_7] = q \widehat{q}^{\,2} x_3^2 x_6$.
    Similarly, using $x_7 = [x_6, \mathbf{E}_{3, 4, 4}]$ and the $q$-Jacobi
    identity, we obtain
    \[
        [2]_q [x_2, x_7] = [[x_2, x_6], \mathbf{E}_{3,4,4}] - q_1 [[x_2,
        \mathbf{E}_{3,4,4}], x_6] + \widehat{q_1} [x_2, \mathbf{E}_{3,4,4}]
        x_6.
    \]
    From the identities $[x_2, x_6] = 0$, $[x_4, x_6] = 0$, and $[x_2,
    \mathbf{E}_{3, 4, 4}] = [2]_q x_4$, we conclude $[x_2, x_7] = \widehat{q_1}
    x_4 x_6$.

    To find the remaining relations, do the following sequence (recall the
    definitions of \texttt{R(i,j,r,s)} and \texttt{L(i,j,r,s)} in Section
    (\ref{section, L(i,j,r,s) and R(i,j,r,s)})) : \texttt{R(1,5,3,9)},
    \texttt{R(1,4,3,5)}, \texttt{R(3,8,7,9)}, \texttt{L(4,9,3,5)},
    \texttt{R(4,8,7,9)}, \texttt{R(1,8,7,9)}, \texttt{R(2,8,7,9)}.

    With all of the $q$-commutation relations among the Lusztig root vectors at
    hand, it is a routine computation to verify that $[x_1, [x_1, x_8]] =
    \widehat{q}^{\,2}  q (q^4 - q^{-4}) x_2 x_3^2 x_6$ and $[x_1, [x_1, [x_1,
    x_8]]] = 0$. Thus, ${\mathcal N}(x_1, x_8) = 3$. Equivalently, ${\mathcal
    N}(\llbracket w \rrbracket) = 3$.

\end{proof}

Now we consider the case when $n \geq 5$. Before proving ${\mathcal
N}(\llbracket w \rrbracket) = 3$ in these cases, we find some $q$-commutation
relations among the Lusztig root vectors.

\noindent
\begin{minipage}{\textwidth}
\begin{lemma}

    \label{lemma, B.4, 1}

    Suppose $n \geq 5$. Let $ X_{m, -4}, X_{r, s}, X_{i, \pm k} \in
    {\mathcal U}_q^+[w]$ be the Lusztig root vectors as in (\ref{B.4, Lusztig
    root vectors}). Then

    \begin{enumerate}

        \item \label{lemma, B.4, 1a} $X_{1, 2} = [X_{2, 5}, \mathbf{E}_{3, 2, 1, 4}]$,

        \item \label{lemma, B.4, 1b} $X_{1, 3} = [X_{3, 5}, \mathbf{E}_{3, 2, 1, 4}]$,

        \item \label{lemma, B.4, 1c} $X_{2, n} = [X_{3, n}, E_2]$,

        \item \label{lemma, B.4, 1d} $X_{2, k} = [X_{2, k + 1}, E_k]$ (for $4 < k < n$),

        \item \label{lemma, B.4, 1e} $X_{3, k} = [X_{3, k + 1}, E_k]$ (for $4 < k < n$),

        \item \label{lemma, B.4, 1f} $X_{2, -k} = [X_{2, -(k - 1)}, E_{k - 1}]$ (for $4 < k \leq n$),

        \item \label{lemma, B.4, 1g} $X_{2, -k} = [X_{3, -k}, E_2]$ (for $4 \leq k \leq n$),

        \item \label{lemma, B.4, 1h} $X_{1, -4} = \mathbf{E}_{3, 2, 1}$.

    \end{enumerate}

\end{lemma}
\end{minipage}

\begin{proof}

    In proving parts (\ref{lemma, B.4, 1a}) and (\ref{lemma, B.4, 1b}), we
    will, for short, let $u = s_3 s_4 s_2 s_3 s_1 \in W(B_n)$ and $w = w_{0, 1,
    2, 1, n - 4} \in W(B_n)$.  For $k \in \mathbf{I} = [1, n]$, define the Weyl
    group elements ${\mathcal S}_k := s_1 \cdot (s_1s_2 \cdots s_k) \in W(B_n)$
    and $\overline{{\mathcal S}}_k := s_2 \cdot {\mathcal S}_k \in W(B_n)$.
    Adhering to the computational conventions in Conventions
    (\ref{computational conventions}), we calculate
    \[
        \begin{split}
            T_u (E_2) &= T_{us_1} T_{s_1} (E_2)
            =T_{us_1} ([E_1, E_2])
            = [T_{us_1} (E_1), T_{us_1} (E_2)]
            \\
            &= [T_{s_3} T_{s_2} T_{s_4 s_3} (E_1), E_4]
            = [T_{s_3} T_{s_2} (E_1), E_4] = [T_{s_3} ([E_2, E_1]), E_4]
            \\
            &= [[T_{s_3} (E_2), T_{s_3}(E_1)], E_4]
            = [[[E_3, E_2], E_1], E_4] = \mathbf{E}_{3, 2, 1, 4}.
        \end{split}
    \]
    Hence,
    \[
        \begin{split}
            X_{1, 2} &= T_{ws_3} (E_3)
            = T_{w s_3 s_4} T_{s_4} (E_3)
            = T_{ws_3 s_4} \left([E_4, E_3]\right)
            = [T_{ws_3 s_4} (E_4), T_{ws_3 s_4} (E_3)]
            \\
            &= [X_{2, 5}, T_u  T_{u^{-1} \cdot ws_3s_4} (E_3)]
            = [X_{2, 5}, T_u (E_2)]
            = [X_{2, 5}, \mathbf{E}_{3, 2, 1, 4}].
        \end{split}
    \]
    Similarly,
    \[
        \begin{split}
            X_{1, 3} &= T_{w \overline{ {\mathcal S}}_{n - 1} s_2} (E_2)
            = T_{w \overline{ {\mathcal S}}_{n - 1} s_2 s_3} T_{s_3} (E_2)
            = T_{w \overline{ {\mathcal S}}_{n - 1} s_2 s_3} ([E_3, E_2])
            \\
            &= [T_{w (s_3 \cdots s_{n - 1}) s_2 s_3} (E_3), T_{w \overline{
            {\mathcal S}}_{n - 1} s_2 s_3} (E_2)] = [X_{3, 5}, T_{w \overline{
            {\mathcal S}}_{n - 1} s_2 s_3} (E_2)]
            \\
            &= [X_{3, 5}, T_u T_{u^{-1} w \overline{ {\mathcal S}}_{n - 1} s_2
            s_3} (E_2)] = [X_{3, 5}, T_u (E_2)] = [X_{3, 5}, \mathbf{E}_{3, 2,
            1, 4}].
        \end{split}
    \]
    This establishes parts (\ref{lemma, B.4, 1a}) and (\ref{lemma, B.4, 1b}).
    Parts (\ref{lemma, B.4, 1c}) - (\ref{lemma, B.4, 1h}) are proved in a
    similar manner.  We compute,
    \begingroup
    \allowdisplaybreaks
        \begin{align*}
            X_{2, n} &= T_{w \overline{{\mathcal S}}_{n - 1}} (E_{n - 1}) =
            T_{w \overline{ {\mathcal S}}_{n - 1} {\mathcal S}_{n - 2}} T_{s_{n
            - 2}} T_{{\mathcal S}_{n - 3}^{-1}} (E_{n - 1}) = T_{w
            \overline{{\mathcal S}}_{n - 1}) {\mathcal S}_{n - 2}} T_{s_{n -
            2}} (E_{n - 1})
            \\
            &= T_{w \overline{ {\mathcal S}}_{n - 1} {\mathcal S}_{n - 2}}
            ([E_{n - 2}, E_{n - 1}]) = [T_{w \overline{ {\mathcal S}}_{n - 1}
            {\mathcal S}_{n - 2}} (E_{n - 2}), T_{w \overline{ {\mathcal S}}_{n
            - 1} {\mathcal S}_{n - 2}} (E_{n - 1})]
            \\
            &= [X_{3, n}, E_2],
            \\
            X_{2, k} &= T_{w \overline{ {\mathcal S}}_{k - 1}} (E_{k - 1}) =
            T_{w \overline{ {\mathcal S}}_k} T_{s_k} (E_{k - 1}) = T_{w
            \overline{ {\mathcal S}}_k} ([E_k, E_{k - 1}])
            \\
            &= [T_{w \overline{ {\mathcal S}}_k} (E_k), T_{w \overline{
            {\mathcal S}}_k} (E_{k - 1})] = [X_{2, k + 1}, E_k],
            \\
            X_{3, k} &= T_{w \overline{ {\mathcal S}}_{n - 1} {\mathcal S}_{k -
            2}} (E_{k - 2}) = T_{w \overline{ {\mathcal S}}_{n - 1}{\mathcal
            S}_{k - 1}} T_{s_{k - 1}} (E_{k - 2}) = T_{w \overline{ {\mathcal
            S}}_{n - 1}{\mathcal S}_{k - 1}} ([E_{k - 1}, E_{k - 2}])
            \\
            &= [T_{w \overline{ {\mathcal S}}_{n - 1}{\mathcal S}_{k - 1}}
            (E_{k - 1}), T_{w \overline{ {\mathcal S}}_{n - 1}{\mathcal S}_{k -
            1}} (E_{k - 2})] = [X_{3, k + 1}, E_k],
            \\
            X_{2, -k} &= T_{\overline{ {\mathcal S}}_{n - 1} {\mathcal S}_{k -
            3}}(E_{k - 2}) = T_{\overline{ {\mathcal S}}_{n - 1} {\mathcal
            S}_{k - 4}} T_{s_{k - 3}}(E_{k - 2}) = T_{\overline{ {\mathcal
            S}}_{n - 1} {\mathcal S}_{k - 4}} ([E_{k - 3}, E_{k - 2}])
            \\
            &= [T_{\overline{ {\mathcal S}}_{n - 1} {\mathcal S}_{k - 4}} (E_{k
            - 3}), T_{\overline{ {\mathcal S}}_{n - 1} {\mathcal S}_{k - 4}}
            (E_{k - 2})] = [X_{2, -(k - 1)}, E_{k - 1}],
            \\
            X_{2, -k} &= T_{\overline{ {\mathcal S}}_{n - 1} {\mathcal S}_{k -
            3}} (E_{k - 2}) = T_{\overline{ {\mathcal S}}_{k - 2}}  T_{
            {\mathcal S}_{k - 3}} T_{s_{k - 1}} T_{s_k s_{k + 1} \cdots s_{n -
            1}} (E_{k - 2})
            \\
            &= T_{\overline{ {\mathcal S}}_{k - 2}}  T_{ {\mathcal S}_{k - 3}}
            T_{s_{k - 1}} (E_{k - 2}) = T_{\overline{ {\mathcal S}}_{k - 2}}
            T_{ {\mathcal S}_{k - 3}} \left( [E_{k - 1}, E_{k - 2}] \right)
            \\
            &= [T_{\overline{ {\mathcal S}}_{k - 2}}  T_{ {\mathcal S}_{k - 3}}
            (E_{k - 1}), T_{\overline{ {\mathcal S}}_{k - 2}}  T_{ {\mathcal
            S}_{k - 3}} (E_{k - 2})]
            \\
            &= [T_{\overline{ {\mathcal S}}_{k - 2}} (E_{k - 1}), T_{\overline{
            {\mathcal S}}_{k - 2} {\mathcal S}_{k - 3}} (E_{k - 2})] = [X_{3,
            -k}, E_2],
            \\
            X_{1, -4} &= T_{\overline{ {\mathcal S}}_{n - 1} {\mathcal S}_{n -
            2} s_n s_{n - 1} s_n} (E_1)
            \\
            &= T_{s_3} T_{s_2} T_{s_2 s_3 \overline{ {\mathcal S}}_{n - 1}
            {\mathcal S}_{n - 2} s_n s_{n - 1} s_n} (E_1) = T_{s_3} T_{s_2}
            (E_1) = T_{s_3} \left( [E_2, E_1] \right)
            \\
            &= [T_{s_3} (E_2), T_{s_3} (E_1)] = [[E_3, E_2], E_1].
        \end{align*}
    \endgroup

\end{proof}

\begin{lemma}

    \label{lemma, B.4, 2}

    For $4 \leq j < k \leq n$,

    \begin{enumerate}

        \item \label{lemma, B.4, 2a} $[X_{3, -j}, X_{2, k}] = \widehat{q_1}
            X_{2, -j} X_{3, k}$,

        \item \label{lemma, B.4, 2b} $[X_{3, -j}, X_{2, -k}] = \widehat{q_1}
            X_{3, -k}, X_{2, -j}$.

    \end{enumerate}

\end{lemma}

\begin{proof}

    We prove part (\ref{lemma, B.4, 2a}). The proof of part (\ref{lemma, B.4,
    2b}) is similar.  We induct on $k$, starting with the base case $k = n$.
    Suppose $4 \leq j < n$. From part (\ref{lemma, B.4, 1c}) of Lemma
    (\ref{lemma, B.4, 1}), $X_{2, n} = [X_{3, n}, E_2]$. Hence, from the
    $q$-Jacobi identity (\ref{q-Jacobi identity}), we have
    \[
        [X_{3, -j}, X_{2, n}] = [[X_{3, -j}, X_{3, n}], E_2] - q_1[[X_{3, -j},
        E_2], X_{3,n}] + \widehat{q_1} [X_{3, -j}, E_2]X_{3, n}.
    \]
    However, from part (\ref{lemma, B.4, 1g}) of Lemma (\ref{lemma, B.4, 1}),
    $[X_{3, -j}, E_2] = X_{2, -j}$, and thus
    \[
        [X_{3, -j}, X_{2, n}] = - q_1[X_{2, -j}, X_{3,n}] + \widehat{q_1} X_{2,
        -j}X_{3, n}.
    \]
    Proposition (\ref{LS corollary}) implies $[X_{2, -j}, X_{3, n}] = 0$.
    Hence, $[X_{3, -j}, X_{2, n}] = \widehat{q_1} X_{2, -j}X_{3, n}$, and this
    concludes the prove of part (\ref{lemma, B.4, 2a}) for the case when $k =
    n$.

    Now assume $j < k < n$. From part (\ref{lemma, B.4, 1d}) of Lemma
    (\ref{lemma, B.4, 1}), $X_{2, k} = [X_{2, k + 1}, E_k]$. Thus, from the
    $q$-Jacobi identity, we get
    \[
        [X_{3, -j}, X_{2, k}] = [[X_{3, -j}, X_{2, k + 1}], E_k] - q_1[[X_{3,
        -j}, E_k], X_{2, k + 1}] + \widehat{q_1} [X_{3, -j}, E_k] X_{2, k + 1}.
    \]
    Observe next that $X_{3, -j} = T_{s_3 \cdots s_{j - 2}} (E_{j - 1}) =
    \mathbf{E}_{3,4,\ldots, j - 1}$. Since the Chevalley generator $E_k$
    commutes with each of $E_1,\cdots, E_{j - 1}$, we get $[X_{3, -j}, E_k] =
    0$, and thus $[X_{3, -j}, X_{2, k}] = [[X_{3, -j}, X_{2, k + 1}], E_k]$.
    Invoking the induction hypothesis gives us $[X_{3, -j}, X_{2, k}] =
    \widehat{q_1} [X_{2, -j} X_{3, k + 1}, E_k]$. Next, we use the $q$-Leibniz
    identity to obtain
    \[
        [X_{3, -j}, X_{2, k}] = \widehat{q_1} \left( X_{2, -j} [X_{3, k + 1},
        E_k] + q_1^{-1} [X_{2, -j}, E_k] X_{3, k + 1} \right).
    \]
    The identity $X_{2, -k} = [X_{2, -(k - 1)}, E_k]$ from part (\ref{lemma,
    B.4, 1f}) of Lemma (\ref{lemma, B.4, 1}) in conjunction with $X_{2, -4} =
    T_{s_3 \cdots s_{n - 1}} (E_2) = T_{s_3} T_{s_4 \cdots s_{n - 1}} (E_2) =
    T_{s_3} (E_2) = \mathbf{E}_{3, 2}$ tells us that the Lusztig root vector
    $X_{2, -j}$ can be written as nested $q$-commutators of Chevalley
    generators $X_{2, -j} = \mathbf{E}_{3, 2, 4, 5, \ldots, j - 1}$. Using
    again that each of $E_1, \cdots, E_{j - 1}$ commutes with $E_k$ gives us
    $[X_{2, -j}, E_k] = 0$. Thus, $[X_{3, -j}, X_{2, k}] = \widehat{q_1} X_{2,
    -j} [X_{3, k + 1}, E_k]$. From part (\ref{lemma, B.4, 1e}) of Lemma
    (\ref{lemma, B.4, 1}), $X_{3, k} = [X_{3, k + 1}, E_k]$. Hence, $[X_{3,
    -j}, X_{2, k}] = \widehat{q_1} X_{2, -j} X_{3, k}$, and this concludes a
    proof of part (\ref{lemma, B.4, 2a}).

\end{proof}

For fixed $k \in \mathbb{N}$ with $4 < k \leq n$, let $\mathbf{U}_k$ be the
subalgebra of ${\mathcal U}_q^+[w]$ generated the Lusztig root vectors
$X_\beta$ with  $e_3 - e_k \prec \beta \prec e_3 + e_k$, where $\prec$ is the
convex order (\ref{convex order, B.4}) on the roots of $w$. Equivalently,
$\mathbf{U}_k$ is the subalgebra generated by  $X_{2, 0}$, $X_{3, 0}$, $X_{2,
3}$, $X_{1, -4}$, $X_{3, \pm m}$ for $k < m \leq n$, and $X_{2, -p}$ for $4
\leq p \leq n$.

For $x, y \in {\mathcal U}_q^+[w]$, we write
\[
    x \equiv y \hspace{5pt}\left( \operatorname{mod } \mathbf{U}_k \right)
\]
if $x - y \in \mathbf{U}_k$.

\noindent
\begin{minipage}{\textwidth}
\begin{lemma}

    \label{lemma, B.4, ad^2 = 0}

    Let $X_{i, \pm k}$ be the Lusztig root vectors as in (\ref{B.4, Lusztig
    root vectors}). Then

    \begin{enumerate}

        \item \label{lemma, B.4, ad^2 = 0, part 1}
            $[X_{3, -4}, X_{1, 3}] \equiv \widehat{q_1}^{\,2} X_{3, -5} X_{1,
            -4} X_{3, 5} \left( \operatorname{mod } \mathbf{U}_5 \right)$,

        \item \label{lemma, B.4, ad^2 = 0, part 2}
            $\left( \operatorname{ad}_q(X_{3, -4})\right)^2(X_{1, 3}) = 0$.

    \end{enumerate}

    For $4 < k \leq n$,

    \begin{enumerate}

        \setcounter{enumi}{2}

        \item \label{lemma, B.4, ad^2 = 0, part 3}
            $\left(\operatorname{ad}_q(X_{3, -4})\right) ([X_{3, -k}, X_{2,
            k}]) \equiv \widehat{q_1}^{\,2}\, X_{3, -k} X_{2, -4} X_{3, k}
            \left( \operatorname{mod } \mathbf{U}_k \right)$,

        \item \label{lemma, B.4, ad^2 = 0, part 4}
            $\left( \operatorname{ad}_q (X_{3, -4}) \right)^2 ([X_{3, -k},
            X_{2, k}]) = 0$.

    \end{enumerate}

\end{lemma}
\end{minipage}

\begin{proof}

    Using the identities $X_{3, -4} = E_3$, $X_{3, -5} = \mathbf{E}_{3, 4}$,
    and $X_{1, -4} = \mathbf{E}_{3, 2, 1}$ together with $[X_{3, -4}, X_{1,
    -4}] = [X_{3, -5}, X_{1, -4}] = 0$ and the $q$-Jacobi identity
    (\ref{q-Jacobi identity}), we obtain
    \[
        [X_{3, -4}, \mathbf{E}_{3,2,1,4}] = [X_{3, -4}, [X_{1, -4}, E_4]] =
        \widehat{q_1}\,\, X_{3, -5} X_{1, -4}.
    \]
    Next, recall from part (\ref{lemma, B.4, 1b}) of Lemma (\ref{lemma, B.4,
    1}) that $X_{1, 3} = [X_{3, 5}, \mathbf{E}_{3, 2, 1, 4}]$.  Using this
    identity in conjunction with the $q$-Jacobi and $q$-Leibniz identities, we
    get
    \[
        \begin{split}
            [X_{3, -4}, X_{1, 3}] &= [X_{3, -4}, [X_{3, 5},
            \mathbf{E}_{3,2,1,4}]]
            \\
            &= \widehat{q_1}^{\,2} X_{3, -5} X_{1, -4} X_{3, 5} - q_1
            \widehat{q_1} [X_{3, -5}, X_{3, 5}] X_{1, -4}.
        \end{split}
    \]
    Next, let $4 < k \leq n$.  We get from the $q$-Jacobi identity
    (\ref{q-Jacobi identity}),
    \[
        \begin{split}
            [X_{3, -4}, [X_{3, -k}, X_{2, k}]] &= [[X_{3, -4}, X_{3, -k}],
            X_{2, k}] - q_1 [[X_{3, -4}, X_{2, k}], X_{3, -k}]
            \\
            &\hspace{10pt}+ \widehat{q_1} [X_{3, -4}, X_{2, k}] X_{3, -k}.
        \end{split}
    \]
    However, Proposition (\ref{LS corollary}) implies $[X_{3, -4}, X_{3, -k}] =
    0$, and thus
    \[
        [X_{3, -4}, [X_{3, -k}, X_{2, k}]] = - q_1 [[X_{3, -4}, X_{2, k}],
        X_{3, -k}] + \widehat{q_1} [X_{3, -4}, X_{2, k}] X_{3, -k}.
    \]
    Next, we apply the identity $[X_{3, -4}, X_{2, k}] = \widehat{q_1} X_{2,
    -4} X_{3, k}$ from part (\ref{lemma, B.4, 2a}) of Lemma (\ref{lemma, B.4,
    2}) to get
    \[
        [X_{3, -4}, [X_{3, -k}, X_{2, k}]] = - q_1 \widehat{q_1} [X_{2, -4}
        X_{3, k}, X_{3, -k}] + \widehat{q_1}^2 X_{2, -4} X_{3, k} X_{3, -k}.
    \]
    From the $q$-Leibniz identity (\ref{q-Leibniz identity}),
    \[
        \begin{split}
            [X_{3, -4}, [X_{3, -k}, X_{2, k}]] &= - q_1 \widehat{q_1}
            \left(X_{2, -4} [X_{3, k}, X_{3, -k}] + [X_{2, -4}, X_{3,
            -k}] X_{3, k} \right)
            \\
            & \hspace{10pt} + \widehat{q_1}^2 X_{2, -4} X_{3, k} X_{3, -k}.
        \end{split}
    \]
    Proposition (\ref{LS corollary}) implies $[X_{2, -4}, X_{3, -k}] = 0$.
    Therefore, the right-hand side of the equation above can be simplified to
    give us
    \[
        [X_{3, -4}, [X_{3, -k}, X_{2, k}]] = \widehat{q_1}^2 X_{3, -k} X_{2,
        -4} X_{3, k} + q_1^{-1}\widehat{q_1} X_{2, -4} [X_{3, -k}, X_{3, k}].
    \]
    The Levendorskii-Soibelmann formulas (\ref{L-S straightening}) imply that
    $[X_{3, -k}, X_{3, k}]$ can be written as a linear combination of the
    monomials $X_{3, 0}^2$ and $X_{3, -j} X_{3, j}$ with $k < j \leq n$. Hence,
    there exist scalars $a_j, b_j, c, d \in \mathbb{K}$ such that
    \[
        [X_{3, -4}, X_{1, 3}]= \widehat{q_1}^2 X_{3, -5} X_{1, -4} X_{3,
        5} + c X_{3, 0}^2 X_{1, 4} + \sum_{5 < j \leq n} a_j X_{3, -j}X_{1, -4}
        X_{3, j}
    \]
    and
    \[
        [X_{3, -4}, [X_{3, -k}, X_{2, k}]] \!=\! \widehat{q_1}^2 X_{3, -k}
        X_{2, -4} X_{3, k} + dX_{2, -4} X_{3, 0}^2 +\!\! \sum_{k < j \leq n}
        b_j X_{3, -j}X_{2, -4} X_{3, j}.
    \]
    Parts (\ref{lemma, B.4, ad^2 = 0, part 1}) and (\ref{lemma, B.4, ad^2 = 0,
    part 3}) follow as a consequence of these identities.

    To prove parts (\ref{lemma, B.4, ad^2 = 0, part 2}) and (\ref{lemma, B.4,
    ad^2 = 0, part 4}), observe that for every Lusztig root vector $X$
    appearing in the right-hand side of either of these two equations above,
    Proposition (\ref{LS corollary}) gives us $[X_{3, -4}, X] = 0$. Thus, as a
    consequence of the $q$-Leibniz identity, $\left( \operatorname{ad}_q (X_{3,
    -4}) \right)^2 (X_{1, 3}) = 0$ and $\left( \operatorname{ad}_q (X_{3, -4})
    \right)^2 ([X_{3, -k}, X_{2, k}]) = 0$.

\end{proof}

\noindent Now we are ready to prove ${\mathcal N}(\llbracket w_{0, 1, 2, 1, n -
4} \rrbracket) = 3$ for the cases when $n \geq 5$.

\begin{theorem}

    \label{theorem, B.4}

    Suppose $n \geq 5$, and let $w = w_{0, 1, 2, 1, n - 4} \in W(B_n)$. Then
    ${\mathcal N}(\llbracket w \rrbracket) = 3$.

\end{theorem}

\begin{proof}

    To prove ${\mathcal N}(\llbracket w \rrbracket) = 3$, we must show $\left(
    \operatorname{ad}_q (X_{3, -4}) \right)^2 (X_{1, 2})$ is nonzero and
    $\left( \operatorname{ad}_q (X_{3, -4}) \right)^3 (X_{1, 2}) = 0$.  We
    begin by using the identity $X_{1, 2} = [X_{2, 5}, \mathbf{E}_{3, 2, 1,
    4}]$ from part (\ref{lemma, B.4, 1a}) of Lemma (\ref{lemma, B.4, 1}) and
    the $q$-Jacobi identity (\ref{q-Jacobi identity}) to get
    \[
        \begin{split}
            [X_{3, -4}, X_{1, 2}] &= [[X_{3, -4}, X_{2, 5}],
            \mathbf{E}_{3,2,1,4}] - q_1 [[X_{3, -4}, \mathbf{E}_{3, 2, 1, 4}],
            X_{2, 5}]
            \\
            &\hspace{10pt}+ \widehat{q_1} [X_{3, -4}, \mathbf{E}_{3, 2, 1, 4}]
            X_{2, 5}
        \end{split}
    \]
    In the proof of Lemma (\ref{lemma, B.4, ad^2 = 0}), it was shown that
    $[X_{3, -4}, \mathbf{E}_{3, 2, 1, 4}] = \widehat{q_1} X_{3, -5} X_{1, -4}$.
    Using this, together with the identity $[X_{3, -4}, X_{2, 5}] =
    \widehat{q_1} X_{2, -4} X_{3, 5}$ from part (\ref{lemma, B.4, 2a}) of Lemma
    (\ref{lemma, B.4, 2}), gives us
    \[
        [X_{3, -4}, X_{1, 2}] \!=\! \widehat{q_1} [X_{2, -4} X_{3, 5},
        \mathbf{E}_{3,2,1,4}] - q_1 \widehat{q_1} [ X_{3, -5} X_{1, -4}, X_{2,
        5}] + \widehat{q_1}^2 X_{3, -5} X_{1, -4} X_{2, 5}.
    \]
    Next, apply the $q$-Leibniz identity to get
    \[
        \begin{split}
            [X_{3, -4}, X_{1, 2}] &=
            \widehat{q_1} \left(X_{2, -4} [X_{3, 5}, \mathbf{E}_{3,2,1,4}] +
            q_1^{-1} [X_{2, -4}, \mathbf{E}_{3, 2, 1, 4}] X_{3, 5} \right)
            \\
            & \hspace{10pt} - q_1 \widehat{q_1} \left(  X_{3, -5} [X_{1, -4},
            X_{2, 5}] + [X_{3, -5}, X_{2, 5}] X_{1, -4} \right)
            \\
            &\hspace{10pt}+ \widehat{q_1}^2 X_{3, -5} X_{1, -4} X_{2, 5}.
        \end{split}
    \]
    Parts (\ref{lemma, B.4, 1f}) and (\ref{lemma, B.4, 1g}) of Lemma
    (\ref{lemma, B.4, 1}), in conjunction with the fact that $X_{3, -4} = E_3$,
    give us $X_{2, -5} = \mathbf{E}_{3, 2, 4}$. Therefore, using that the
    Chevalley generators $E_1$ and $E_4$ commute, we get $\mathbf{E}_{3, 2, 1,
    4} = \mathbf{E}_{3, 2, 4, 1} = [X_{2, -5}, E_1]$. The $q$-commutator
    identities $[X_{3, 5}, \mathbf{E}_{3, 2, 1, 4}] = X_{1, 3}$ from part
    (\ref{lemma, B.4, 1b}) of Lemma (\ref{lemma, B.4, 1}) and $[X_{1, -4},
    X_{2, 5}] = 0$, which follows from Proposition (\ref{LS corollary}), imply
    \begin{align*}
        [X_{3, -4}, X_{1, 2}] &=\widehat{q_1} X_{2, -4} X_{1, 3} +
        \widehat{q_1} q_1^{-1}[X_{2, -4}, [X_{2, -5}, E_1]] X_{3, 5}
        \\
        &\hspace{5mm} - q_1 \widehat{q_1}\,\, [X_{3, -5}, X_{2, 5}] X_{1, -4} +
        \widehat{q_1}^2\,\, X_{3, -5} X_{1, -4} X_{2, 5}.
    \end{align*}
    Recall, from part (\ref{lemma, B.4, 1h}) of Lemma (\ref{lemma, B.4, 1}),
    that $X_{1, -4} = \mathbf{E}_{3, 2, 1}$. Since $X_{2, -4} = \mathbf{E}_{3,
    2}$, we have $X_{1, -4} = [X_{2, -4}, E_1]$. Thus, we obtain, from the
    $q$-Jacobi identity,
    \[
        [X_{2, -4}, [X_{2, -5}, E_1]] = [[X_{2, -4}, X_{2, -5}], E_1] - q_1
        [X_{1, -4}, X_{2, -5}] + \widehat{q_1} X_{1, -4} X_{2, -5}.
    \]
    However, from $[X_{1, -4}, X_{2, -5}] = [X_{2, -4}, X_{2, -5}] = 0$, the
    right-hand simplifies to give us $[X_{2, -4}, [X_{2, -5}, E_1]] =
    \widehat{q_1} X_{2, -5} X_{1, -4}$, and thus
    \[
        \begin{split}
            [X_{3, -4}, X_{1, 2}]
            &= \widehat{q_1} X_{2, -4} X_{1, 3} + q_1^{-1} \widehat{q_1}^{2}
            X_{2, -5} X_{1, -4} X_{3, 5}
            \\
            &\hspace{5mm}+ \widehat{q_1}^{2} X_{3, -5} X_{1, -4} X_{2, 5} -
            q_1 \widehat{q_1} [X_{3, -5}, X_{2, 5}] X_{1, -4}.
        \end{split}
    \]
    Apply $\operatorname{ad}_q(X_{3, -4})$ to both sides, use the $q$-Leibniz
    identity and the $q$-commutation relations
    \begin{equation}
        \label{B.4, 0}
        [X_{3, -4}, X_{2, -4}] = [X_{3, -4}, X_{3, 5}] = [X_{3, -4}, X_{1, -4}]
        = [X_{3, -4}, X_{3, -5}] = 0
    \end{equation}
    to simplify to obtain
    \[
        \begin{split}
            \left( \operatorname{ad}_q (X_{3, -4}) \right)^2\! (X_{1, 2})\!
            &=\! \widehat{q_1}^2 \!\left( q_1^2 X_{3, -5} X_{1, -4} [X_{3,
            -4},\! X_{2, 5}] \!+\! q_1^{-1} [X_{3, -4},\! X_{2, -5}] X_{1, -4}
            X_{3, 5} \right)
            \\
            &\hspace{10pt}+ q_1 \widehat{q_1} \left( X_{2, -4} [X_{3, -4},
            X_{1, 3}] - [X_{3, -4}, [X_{3, -5}, X_{2, 5}]] X_{1, -4} \right).
        \end{split}
    \]
    Next, apply the formulas from Lemma (\ref{lemma, B.4, 2}) and use $X_{2,
    -4} X_{1, -4} = q_1 X_{1, -4}, X_{2, -4}$ to get
    \begin{equation}
        \label{B.4, 1}
        \begin{split}
            \left( \operatorname{ad}_q (X_{3, -4}) \right)^2 (X_{1, 2})
            &= q_1 \widehat{q_1} X_{2, -4} [X_{3, -4}, X_{1, 3}]
            \\
            &\hspace{10pt}- q_1 \widehat{q_1} [X_{3, -4}, [X_{3, -5}, X_{2,
            5}]] X_{1, -4}
            \\
            &\hspace{10pt} + (q_1 + q_1^{-1}) \widehat{q_1}^3 X_{3, -5} X_{2,
            -4} X_{1, -4} X_{3, 5}.
        \end{split}
    \end{equation}
    The results of parts (\ref{lemma, B.4, ad^2 = 0, part 1}) and (\ref{lemma,
    B.4, ad^2 = 0, part 3}) of Lemma (\ref{lemma, B.4, ad^2 = 0}), together
    with the fact that $X_{1, -4}$ and $X_{2, -4}$ belong to $\mathbf{U}_5$,
    give us
    \[
        \begin{split}
            [X_{3, -4}, [X_{3, -4}, X_{1, 2}]]
            &\equiv q_1 \widehat{q_1}^3 X_{2, -4} X_{3, -5} X_{1, -4} X_{3, 5}
            \\
            &\hspace{10pt}+ (q_1 + q_1^{-1}) \widehat{q_1}^3 X_{3, -5} X_{2,
            -4} X_{1, -4} X_{3, 5}
            \\
            &\hspace{10pt}- q_1 \widehat{q_1}^3 X_{3, -5} X_{2, -4} X_{3, 5}
            X_{1, -4} \hspace{10pt} (\operatorname{mod } \mathbf{U}_5)
        \end{split}
    \]
    From Proposition (\ref{LS corollary}), $[X_{3, 5}, X_{1, -4}] = [X_{3, -5},
    X_{2, -4}] = 0$. This implies the monomials $X_{2, -4} X_{3, -5} X_{1, -4}
    X_{3, 5}$ and $X_{3, -5}  X_{2, -4} X_{3, 5} X_{1, -4}$ appearing in the
    expression on the right-hand side, are equal. Therefore,
    \[
        [X_{3, -4}, [X_{3, -4}, X_{1, 2}]] \equiv (q_1 + q_1^{-1})
        \widehat{q_1}^3 X_{3, -5} X_{2, -4} X_{1, -4} X_{3, 5} \hspace{10pt}
        (\operatorname{mod } \mathbf{U}_5).
    \]
    Hence, $[X_{3, -4}, [X_{3, -4}, X_{1, 2}]] \neq 0$, and thus ${\mathcal
    N}(\llbracket w \rrbracket) > 2$.

    Finally, apply $\operatorname{ad}_q(X_{3, -4})$ to both sides of (\ref{B.4,
    1}). The $q$-commutation relations (\ref{B.4, 0}), together with the fact
    that $\left(\operatorname{ad}_q(X_{3, -4}) \right)^2$ sends $X_{1, 3}$ and
    $[X_{3, -5}, X_{2, 5}]$ to $0$ (by parts (\ref{lemma, B.4, ad^2 = 0, part
    2}) and (\ref{lemma, B.4, ad^2 = 0, part 4}) of Lemma (\ref{lemma, B.4,
    ad^2 = 0})), imply $\left( \operatorname{ad}_q(X_{3, -4}) \right)^3 (X_{1,
    2}) = 0$.  Thus, ${\mathcal N}(\llbracket w \rrbracket) = 3$.

\end{proof}

\subsection{Nilpotency index, Case C.2}

Let $n > 2$, and consider the following reduced expression for $w_{0, 0, 2, 0,
n - 2} \in W(C_n)$:
\begin{equation}
    \label{reduced expression, C.2}
    w_{0,0,2,0,n-2} := (s_2 \cdots s_{n-1})(s_1 \cdots s_{n-2}) s_n s_{n-1} s_n
    (s_{n-2} \cdots s_1)(s_{n-1} \cdots s_2).
\end{equation}
Our aim is to prove ${\mathcal N}(\llbracket w_{0, 0, 2, 0, n - 2} \rrbracket)
= 3$.

We begin with three lemmas that give some commutation relations in the algebra
${\mathcal U}_q^+[w_{0, 0, 2, 0, n - 2}]$. The reduced expression (\ref{reduced
expression, C.2}) corresponds to the convex order on the roots of $w_{0, 0, 2,
0, n - 2}$,
\[
    \begin{split}
        &e_2 - e_3 \prec \cdots \prec e_2 - e_n \prec e_1 - e_3 \prec \cdots
        \\
        &\cdots \prec e_1 - e_n \prec 2e_2 \prec e_1 + e_2 \prec 2e_1 \prec e_2
        + e_n \prec \cdots
        \\
        &\cdots \prec e_2 + e_3 \prec e_1 + e_n \prec \cdots \prec e_1 + e_3.
    \end{split}
\]
The Lusztig root vectors in ${\mathcal U}_q^+[w_{0, 0, 2, 0, n - 2}]$ will be
denoted
\begin{equation}
    \label{C.2, Lusztig root vectors}
    X_{1, 2}, X_{i, i}, X_{i, \pm k} \in {\mathcal U}_q^+[w_{0, 0, 2, 0, n -
    2}], \hspace{15pt} (1 \leq i < 3 \leq k \leq n).
\end{equation}

First, we have a result analogous to Lemma (\ref{lemma, B.2, 0}).

\begin{lemma}

    \label{lemma, C.2, 0}

    Let $X_{1,2}$, $X_{i, i}$, and $X_{i, \pm k}$ be as in (\ref{C.2, Lusztig
    root vectors}). For $1\leq i < 3 \leq j \leq n$ and $3 \leq k \leq n$, we
    have the following:
    \begin{align}
        \label{C.2, L-S, 1} &[X_{2, -k}, X_{1, -j}] = [X_{2, j}, X_{1, k}] = 0
        \hspace{8pt} (j \leq k),
        \\
        \label{C.2, L-S, 2} &[X_{i, -k}, X_{i, -j}] = [X_{i, j}, X_{i, k}] = 0
        \hspace{12.5pt} (k < j),
        \\
        \label{C.2, L-S, 3} &[X_{i, -k}, X_{i, j}] = [X_{1, -k}, X_{2, j}] = 0
        \hspace{10.3pt} (j \neq k),
        \\
        \label{C.2, L-S, 4} &[X_{i, -k}, X_{2, 2}] = [X_{2, 2}, X_{2, k}] = 0.
    \end{align}

\end{lemma}

\begin{proof}

    These identities follow as a consequence of Proposition (\ref{LS
    corollary}).

\end{proof}

The next two lemmas give some additional relations among the Lusztig root
vectors in ${\mathcal U}_q^+[w_{0, 0, 2, 0, n - 2}]$ and the Chevalley
generators $E_1,\cdots, E_n$.

\begin{lemma}

    \label{lemma, C.2, 1}

    Let $X_{i, \pm k}$ be as in (\ref{C.2, Lusztig root vectors}). Then

    \begin{enumerate}

        \item \label{lemma, C.2, 1a} $X_{2, -3} = E_2$ and $X_{1, -3} = [E_2,
            E_1]$,

        \item \label{lemma, C.2, 1b} $X_{i, -k} = [X_{i, -(k - 1)}, E_{k - 1}]$
            (for $3 < k \leq n$ and $i=1, 2$),

        \item \label{lemma, C.2, 1c} $X_{i, k} = [X_{i, k + 1}, E_k]$ (for $3
            \leq k < n$ and $i = 1, 2$),

        \item \label{lemma, C.2, 1d} $X_{1, n} = [X_{2, n}, E_1]$,

        \item \label{lemma, C.2, 1e} $X_{1, 2} = [X_{2, 2}, E_1]$,

        \item \label{lemma, C.2, 1f} $[2]_q X_{1, 1} = [X_{1, 2}, E_1]$,

        \item \label{lemma, C.2, 1g} $X_{2, n} = [X_{2, -n}, E_n]$,

        \item \label{lemma, C.2, 1h} $[2]_q X_{2, 2} = [X_{2, -n}, X_{2, n}]$,

        \item \label{lemma, C.2, 1i} $X_{1, 2} = [X_{1, -n}, X_{2, n}]$,

        \item \label{lemma, C.2, 1j} $X_{1, -k} = [X_{2, -k}, E_1]$, (for $3
            \leq k \leq n$).

    \end{enumerate}

\end{lemma}

\begin{proof}

    Parts (\ref{lemma, C.2, 1a}) - (\ref{lemma, C.2, 1d}) can be proved using
    the same techniques in the proof of Lemma (\ref{lemma, B.2, 1}).

    Suppose $n > 2$, and let $w = w_{0, 0, 2, 0, n - 2} \in W(C_n)$. For an
    integer $k$ with $1 \leq k \leq \ell(w) = 4n - 5$, let $w^{(k)}$ be Weyl
    group element $s_{i_1} \cdots s_{i_k}$ obtained by multiplying together the
    first $k$ simple reflections $s_{i_1}, s_{i_2} ,\cdots, s_{i_k}$ in the
    reduced expression (\ref{reduced expression, C.2}) of $w$.

    We prove parts (\ref{lemma, C.2, 1e}), (\ref{lemma, C.2, 1g}), and
    (\ref{lemma, C.2, 1i}) via direct computations,
    \begingroup
    \allowdisplaybreaks
        \begin{align*}
            X_{2, n} &= T_{w^{(2n - 1)}} (E_{n - 2}) =
            T_{w^{(n - 3)}} T_{s_{n - 1} s_n} T_{s_1 \cdots s_n} (E_{n - 2})
            \\
            &= T_{w^{(n - 3)}} T_{s_{n - 1} s_n} (E_{n - 1}) = T_{w^{(n -
            3)}} \left( [E_{n - 1}, E_n] \right)
            \\
            &= [T_{w^{(n - 3)}} (E_{n - 1}), T_{w^{(n - 3)}} (E_n)] = [X_{2,
            -n}, E_n],
            \\
            \\
            X_{1, 2} &= T_{w^{(2n - 3)}} (E_{n - 1})
            = T_{w^{2n - 4}} T_{s_n} (E_{n - 1})
            = T_{w^{(2n - 4)}} ([E_n, E_{n - 1}])
            \\
            &= [T_{w^{(2n - 4)}} (E_n), T_{w^{2n - 4}} (E_{n - 1}) ]
            = [X_{2, 2}, E_1],
            \\
            \\
            X_{1, 2} &= T_{w^{(2n - 3)}} (E_{n - 1})
            = T_{w^{(2n - 5)}} T_{s_n} T_{s_{n - 2}} (E_{n - 1})
            = T_{w^{(2n - 5)}} T_{s_n} \left( [E_{n - 2},  E_{n - 1}] \right)
            \\
            &= T_{w^{(2n - 5)}} \left( [T_{s_n} (E_{n - 2}), T_{s_n} (E_{n -
            1})] \right) = T_{w^{(2n - 5)}} \left( [E_{n - 2}, T_{s_n} (E_{n -
            1})] \right)
            \\
            &= T_{w^{(2n - 5)}} \left( [E_{n - 2}, T_{s_n} T_{s_{n - 2} s_{n -
            1} s_n} (E_{n - 2})] \right)
            \\
            &= [T_{w^{(2n - 5)}} (E_{n - 2}), T_{w^{(2n - 5)}} T_{s_n} T_{s_{n
            - 2} s_{n - 1} s_n} (E_{n - 2})]
            \\
            &= [T_{w^{(2n - 5)}} (E_{n - 2}), T_{w^{(2n - 1)}} (E_{n - 2})] =
            [X_{1, -n}, X_{2, n}].
        \end{align*}
    \endgroup
    Using the identity $X_{1, 2} = [X_{2, 2}, E_1]$ (\ref{lemma, C.2, 1e}), we
    prove part (\ref{lemma, C.2, 1f}),
    \[
        \begin{split}
            [2]_q X_{1, 1} &= [2]_q T_{w^{(2n - 2)}} (E_n) = [2]_q T_{w^{(2n -
            4)}} T_{s_n s_{n - 1}} (E_n)
            \\
            &= T_{w^{(2n - 4)}} ([[E_n, E_{n -1}], E_{n - 1}])
            \\
            &= [[ T_{w^{(2n - 4)}} (E_n), T_{w^{(2n - 4)}}
            (E_{n - 1})], T_{w^{(2n - 4)}} (E_{n - 1})]
            = [[X_{2, 2}, E_1], E_1]
            \\
            &= [X_{1, 2}, E_1].
        \end{split}
    \]
    Using the identity $X_{2, n} = [X_{2, -n}, E_n]$ (\ref{lemma, C.2, 1g}), we
    prove part (\ref{lemma, C.2, 1h}),
    \[
        \begin{split}
            [2]_q X_{2, 2} &= [2]_q T_{w^{(2n - 4)}} (E_n) = [2]_q T_{w^{(n -
            3)}} T_{s_{n - 1}} T_{s_1 \cdots s_{n - 2}} (E_n)
            \\
            &= [2]_q T_{w^{(n - 3)}} T_{s_{n - 1}} (E_n) = T_{w^{(n - 3)}}
            \left( [E_{n - 1}, [E_{n - 1}, E_n]] \right)
            \\
            &= [T_{w^{(n - 3)}} (E_{n - 1}), [T_{w^{(n - 3)}} (E_{n - 1}),
            T_{w^{(n - 3)}} (E_n)]] = [X_{2, - n}, [X_{2, -n}, E_n]]
            \\
            &= [X_{2, -n}, X_{2, n}].
        \end{split}
    \]
    Finally, we prove part (\ref{lemma, C.2, 1j}) using the same techniques in
    proving (\ref{B.2, 3a}).

\end{proof}

Analogous to (\ref{B.2, as nested q-commutators, a}) and (\ref{B.2, as nested
q-commutators, b}), we have
\begin{equation}
    \label{C.2, as nested q-commutators} X_{2, -k} = \mathbf{E}_{2, \dots, k -
    1}, \hspace{10pt} X_{1, -k} = \mathbf{E}_{2, 1, 3, 4 \dots, k - 1},
    \hspace{10pt} (\text{for }3 \leq k \leq n).
\end{equation}

\begin{lemma}

    \label{lemma, C.2, 2}

    Let $X_{i, \pm k}$ be as in (\ref{C.2, Lusztig root vectors}). Then

    \begin{enumerate}

        \item \label{lemma, C.2, 2.1} For $3 \leq k < j \leq n$,

            \begin{enumerate}

                \item \label{lemma, C.2, 2.1a} $[X_{2, -k}, X_{1, j}] =
                    \widehat{q}\, X_{1, -k} X_{2, j}$,

                \item \label{lemma, C.2, 2.1b} $[X_{2, -k}, X_{1, -j}] =
                    \widehat{q} X_{2, -j} X_{1, -k}$.

            \end{enumerate}

        \item \label{lemma, C.2, 2.2} For $3 \leq k \leq n$,

        \begin{enumerate}

            \item \label{lemma, C.2, 2.2b} $[X_{2, -k}, X_{1, 2}] =
                \widehat{q}[2]_q X_{1, -k} X_{2, 2}$,

            \item \label{lemma, C.2, 2.2c} $[X_{2, -k}, X_{2, k}] = (-q)^{n -
                k} [2]_q X_{2, 2} + \widehat{q} \sum_{k < j \leq n} (-q)^{j - k
                - 1} X_{2, -j} X_{2, j}$,

            \item \label{lemma, C.2, 2.2d} $[X_{2, -k}, X_{1, k}] = \widehat{q}
                X_{1, -k} X_{2, k} - (-q)^{n - k - 1} X_{1, 2} \\
                \phantom{\hspace{25mm}}+ \widehat{q}  \sum_{k < j \leq n}
                (-q)^{j - k - 1} \left(X_{2, -j} X_{1, j} - \widehat{q} X_{1,
                -j} X_{2, j} \right)$.

        \end{enumerate}

    \end{enumerate}

\end{lemma}

\begin{proof}

    Parts (\ref{lemma, C.2, 2.1a}) and (\ref{lemma, C.2, 2.1b}) can be proved
    using the same steps in the proof of parts (\ref{lemma, B.2, 2.1a}) and
    (\ref{lemma, B.2, 2.1b}) of Lemma (\ref{lemma, B.2, 2}).  These particular
    parts of the proof use the identities in (\ref{C.2, L-S, 1}), (\ref{C.2,
    L-S, 3}), and (\ref{C.2, as nested q-commutators}), as well as parts
    (\ref{lemma, C.2, 1b}), (\ref{lemma, C.2, 1c}), \ref{lemma, C.2, 1d}), and
    (\ref{lemma, C.2, 1j}) of Lemma (\ref{lemma, C.2, 1}),

    To prove (\ref{lemma, C.2, 2.2b}), begin by using the identity $X_{1, 2} =
    [X_{1, -n}, X_{2, n}]$ from part (\ref{lemma, C.2, 1i}) of Lemma
    (\ref{lemma, C.2, 2}) and the $q$-Jacobi identity (\ref{q-Jacobi identity})
    to get
    \[
        [X_{2, -k},\! X_{1, 2}] \!=\! [[X_{2, -k},\! X_{1, -n}], X_{2, n}] - q
        [[X_{2, -k},\! X_{2, n}],\! X_{1, -n}] + \widehat{q} [X_{2, -k},\!
        X_{2, n}] X_{1, -n}.
    \]

    Assume first $k < n$. Using the trivial $q$-commutator identity $[X_{2,
    -k}, X_{2, n}] = 0$ (\ref{C.2, L-S, 3}) in conjunction with $[X_{2, -k},
    X_{1, -n}] = \widehat{q} X_{2, -n} X_{1, -k}$ (part (\ref{lemma, C.2,
    2.1b}) of this lemma) give us $[X_{2, -k}, X_{1, 2}] = \widehat{q}[X_{2,
    -n} X_{1, -k}, X_{2, n}]$.  Next, apply the $q$-Leibniz identity
    (\ref{q-Leibniz identity}) to get $[X_{2, -k}, X_{1, 2}] =
    \widehat{q}\left(X_{2, -n} [X_{1, -k}, X_{2, n}] + [X_{2, -n}, X_{2, n}]
    X_{1, -k}\right)$.  The identities $[X_{1, -k}, X_{2, n}] = 0$ and $[2]_q
    X_{2, 2} = [X_{2, -n}, X_{2, n}]$ from (\ref{C.2, L-S, 3} and part
    (\ref{lemma, C.2, 1h}) of Lemma (\ref{lemma, C.2, 1}), respectively, imply
    $[X_{2, -k}, X_{1, 2}] = \widehat{q}[2]_q X_{2, 2} X_{1, -k}$.  Finally,
    applying $[X_{1, -k}, X_{2, 2}] = 0$ (\ref{C.2, L-S, 4}) gives us $[X_{2,
    -k}, X_{1, 2}] = \widehat{q}[2]_q X_{1, -k} X_{2, 2}$, and this establishes
    (\ref{lemma, C.2, 2.2b}) in the case when $k < n$.

    Now suppose $k = n$. From above, we have
    \[
        [X_{2, -n},\! X_{1, 2}] \!=\! [[X_{2, -n},\! X_{1, -n}],\! X_{2, n}] -
        q [[X_{2, -n},\! X_{2, n}],\! X_{1, -n}] + \widehat{q} [X_{2, -n},\!
        X_{2, n}] X_{1, -n}.
    \]
    Use the identity $[2]_q X_{2, 2} = [X_{2, -n}, X_{2, n}]$ from part
    (\ref{lemma, C.2, 1h}) of Lemma (\ref{lemma, C.2, 1}) to get
    \[
        [X_{2, -n}, X_{1, 2}] = [[X_{2, -n}, X_{1, -n}], X_{2, n}] - q[2]_q
        [X_{2, 2}, X_{1, -n}] + \widehat{q} [2]_q X_{2, 2} X_{1, -n}.
    \]
    Finally, applying the identities $[X_{2, -n}, X_{1, -n}] = 0$ and $[X_{1,
    -n}, X_{2, 2}] = 0$ from (\ref{C.2, L-S, 1}) and (\ref{C.2, L-S, 4}),
    respectively, gives us $[X_{2, -n}, X_{1, 2}] = \widehat{q} [2]_q X_{1, -n}
    X_{2, 2}$.

    We prove part (\ref{lemma, C.2, 2.2c}) by induction. The base case $j = n$
    has already been established in part (\ref{lemma, C.2, 1h}) of Lemma
    \ref{lemma, C.2, 1}). The inductive step is exactly like the B.2 case
    (refer to the proof of Lemma (\ref{lemma, B.2, 2})).  Using parts
    (\ref{lemma, C.2, 1b}) and \ref{lemma, C.2, 1c} of Lemma (\ref{lemma, C.2,
    1}) as well as (\ref{C.2, L-S, 3})), we get $[X_{2, -k}, X_{2, k}] = - q
    [X_{2, -(k + 1)}, X_{2, k + 1}] + \widehat{q} X_{2, -(k + 1)} X_{2, k +
    1}$, and the result follows by induction.

    We also use induction to prove (\ref{lemma, C.2, 2.2d}). Starting with the
    base case $j = n $, we use the identities $X_{1, 2} = [X_{2, 2}, E_1]$ and
    $[2]_q X_{2, 2} = [X_{2, -n}, X_{2, n}]$ from parts (\ref{lemma, C.2, 1e})
    and (\ref{lemma, C.2, 1h}) of Lemma (\ref{lemma, C.2, 1}), respectively,
    and the $q$-Jacobi identity (\ref{q-Jacobi identity}) to obtain
    \[
        \begin{split}
            [2]_q X_{1, 2} &= [2]_q [X_{2,2}, E_1] = [[ X_{2, -n}, X_{2, n}],
            E_1]
            \\
            &= [X_{2, -n}, [X_{2, n}, E_1]] + q [[X_{2, -n}, E_1], X_{2, n}] -
            \widehat{q} [X_{2, -n}, E_1] X_{2, n}.
        \end{split}
    \]
    Next, apply the identities $X_{1, n} = [X_{2, n}, E_1]$ and $X_{1, -n} =
    [X_{2, -n}, E_1]$ from parts (\ref{lemma, C.2, 1d}) and (\ref{lemma, C.2,
    1j}) of Lemma (\ref{lemma, C.2, 1}) to get
    \[
        [2]_q X_{1, 2} = [X_{2, -n}, X_{1, n}] + q [X_{1, -n}, X_{2, n}] -
        \widehat{q} X_{1, -n} X_{2, n}.
    \]
    Since $[X_{1, -n}, X_{2, n}] = X_{1, 2}$ (part (\ref{lemma, C.2, 1i}) of
    Lemma (\ref{lemma, C.2, 1})), we have
    \[
        [2]_q X_{1, 2} = [X_{2, -n}, X_{1, n}] + q X_{1, 2} - \widehat{q} X_{1,
        -n} X_{2, n}.
    \]
    This is equivalent to $[X_{2, -n}, X_{1, n}] = \widehat{q} X_{1, -n} X_{2,
    n} + q^{-1} X_{1, 2}$, and this establishes the case when $j = n$.  Now
    assume $3 \leq k < n$.  Using the same steps in the proof of part
    (\ref{lemma, B.2, 2.2d}) of Lemma (\ref{lemma, B.2, 2}), except now use the
    identities of parts (\ref{lemma, C.2, 1b}) and (\ref{lemma, C.2, 1c}) of
    Lemma (\ref{lemma, C.2, 1}) and part (\ref{lemma, C.2, 2.1a}) of this
    lemma, we obtain
    \[
        \begin{split}
            [X_{2, -k}, X_{1, k}] &= \widehat{q} \left( X_{1, -k} X_{2,
            k} + q^{-1} X_{1, -(k + 1)} X_{2, k + 1} \right)
            \\
            &\hspace{5mm}- q [X_{2, -(k + 1)}, X_{1, k + 1}] +
            \widehat{q} X_{2, -(k + 1)} X_{1, k + 1},
        \end{split}
    \]
    and the result follows by the induction hypothesis.

\end{proof}

Now we are ready to prove ${\mathcal N}(\llbracket w_{0, 0, 2, 0, n -
2}\rrbracket] = 3$.

\begin{theorem}

    Let $w = w_{0, 0, 2, 0, n - 2} \in W(C_n)$, where $n > 2$.  Then ${\mathcal
    N}\left( \llbracket w \rrbracket \right) = 3$.

\end{theorem}

\begin{proof}

    From part (\ref{lemma, C.2, 2.2d}) of Lemma (\ref{lemma, C.2, 2}), we have
    \[
        \begin{split}
            [X_{2, -3}, X_{1, 3}] &= \widehat{q} X_{1, -3} X_{2, 3} - (-q)^{n -
            4} X_{1, 2}
            \\
            &\hspace{10pt} + \widehat{q} \sum_{3 < j \leq n} (-q)^{j - 4}
            \left(X_{2, -j} X_{1, j} - \widehat{q} X_{1, -j} X_{2, j} \right).
        \end{split}
    \]
    The identities in  (\ref{C.2, L-S, 1}), (\ref{C.2, L-S, 2}), and (\ref{C.2,
    L-S, 3}) tell us that $X_{2, -3}$ $q$-commutes with $X_{1, -3}$, $X_{2,
    -j}$ ($j > 3$), and $X_{2, j}$ ($j > 3$), respectively, whereas parts
    (\ref{lemma, C.2, 2.1a}), (\ref{lemma, C.2, 2.1b}), (\ref{lemma, C.2,
    2.2b}), and (\ref{lemma, C.2, 2.2c}) of Lemma (\ref{lemma, C.2, 2}) give
    the $q$-commutation relations involving the $q$-commutators $[X_{2, -3},
    X_{1, j}]$ ($j > 3$), $[X_{2, -3}, X_{1, -j}]$ ($j > 3$), $[X_{2, -3},
    X_{1, 2}]$, and $[X_{2, -3}, X_{2, 3}]$, respectively. Applying these
    relations gives us
    \[
        \begin{split}
            [X_{2, -3}, [X_{2, -3}, X_{1, 3}]] &= (q + q^{-1}) \,
            \widehat{q}^{\,2} \left( \sum_{3 < j \leq n} (-q)^{j - 4} X_{2, -j}
            X_{1, -3} X_{2, j} \right)
            \\
            &\hspace{5mm}+ (-q)^{n - 3} (q + q^{-1})^2 \, \widehat{q} \, X_{1,
            -3} X_{2, 2}.
        \end{split}
    \]
    Finally, by (\ref{C.2, L-S, 1}), (\ref{C.2, L-S, 2}), (\ref{C.2, L-S, 3}),
    and (\ref{C.2, L-S, 4}), each of $[X_{2, -3}, X_{1, -3}]$, $[X_{2, -3},
    X_{2, -j}]$ ($j > 3$), $[X_{2, -3}, X_{2, j}]$ ($j > 3$), and $[X_{2, -3},
    X_{2, 2}]$ equals $0$. That is to say, $X_{2, -3}$ $q$-commutes with each
    Lusztig root vector appearing on the right-hand side of the equation above,
    and thus,
    \[
        [X_{2, -3}, [X_{2, -3}, [X_{2, -3}, X_{1, 3}]]] = 0.
    \]
    Since $\left( \operatorname{ad}_q(X_{2, -3}) \right)^2 (X_{1, 3}) \neq 0$
    and $\left( \operatorname{ad}_q(X_{2, -3}) \right)^3 (X_{1, 3}) = 0$,
    ${\mathcal N}(\llbracket w \rrbracket) = 3$.

\end{proof}

\subsection{Nilpotency index, Case C.4}

Let $\mathfrak{g}$ be the complex Lie algebra of type $C_n$, where $n > 3$.
Consider the Weyl group element $w := w_{0, 1, 2, 1, n - 4} \in W(C_n)$ and the
reduced expression
\begin{equation}
    \label{C.4, reduced expression}
    w = (s_3 \cdots s_{n-1})(s_2 \cdots s_{n-2})s_ns_{n-1}s_ns_1(s_{n-2} \cdots
    s_2) (s_{n-1} \cdots s_3)
\end{equation}
Let $\Delta_w$ be the set of roots of $w$. The reduced expression for $w$
corresponds to the following convex order on $\Delta_w$:
\begin{equation}
    \label{convex order, C.4}
    \begin{split}
        &e_3 - e_4 \prec e_3 - e_5 \prec \cdots \prec e_3 - e_n \prec e_2 - e_4
        \prec e_2 - e_5 \prec \cdots
        \\
        &\cdots \prec e_2 - e_n \prec 2e_3 \prec e_2 + e_3 \prec 2e_2 \prec e_1
        - e_4 \prec e_3 + e_n \prec e_3 - e_{n - 1} \prec \cdots
        \\
        &\cdots \prec e_3 + e_5 \prec e_1 + e_3 \prec e_2 + e_n \prec e_2 +
        e_{n - 1} \prec \cdots \prec e_2 + e_5 \prec e_1 + e_2.
    \end{split}
\end{equation}
Let ${\mathcal U}_q^+[w]$ be the quantum Schubert cell algebra associated to
$w$. The Lusztig root vectors will be denoted by
\begin{equation}
    \label{C.4, Lusztig root vectors}
    X_{m, -4}, X_{r, s}, X_{i, i}, X_{i, \pm k} \in {\mathcal U}_q^+[w]
\end{equation}
where $1 < i < 4 < k \leq n$, $1 \leq m < 4$, and $1 \leq r < s < 4$.

The following theorem gives us the nilpotency index of $\llbracket w_{0, 1, 2,
1, n - 4} \rrbracket$ for the case when $n = 4$.

\begin{theorem}

    Let $w = w_{0, 1, 2, 1, 0} \in W(C_4)$. Then ${\mathcal N}(\llbracket w
    \rrbracket) = 3$.

\end{theorem}

\begin{proof}

    Consider the Weyl group element $u = ws_2 \in W(C_4)$. With respect to the
    reduced expression $u = s_3s_2s_4s_3s_4s_1s_2s_3s_2$, let $x_1,\cdots, x_9$
    be the Lusztig root vectors in the algebra ${\mathcal U}_q^+[u]$. Among the
    ${9 \choose 2}$ ($=36$) pairs, $x_i$ and $x_j$, with $i < j$, the
    Levenorskii-Soibelmann straightening formulas tell us $23$ of them satisfy
    the trivial $q$-commutation relation $[x_i, x_j] = 0$. We will show the
    $13$ remaining relations are
    \[
        \begin{split}
            &[x_1, x_4] = \widehat{q} [2]_q x_2 x_3, \hspace{10pt}
            [x_1, x_5] = \widehat{q} x_2 x_4, \hspace{10pt}
            [x_1, x_8] = \widehat{q} \left(x_2 x_7 + q^{-1} x_4 x_6 \right),
            \\
            &[x_1, x_7] = \widehat{q} [2]_q x_3 x_6, \hspace{10pt}
            [x_2, x_7] = \widehat{q} x_4 x_6, \hspace{10pt}
            [x_1, x_9] = x_2, \hspace{10pt}
            [x_3, x_9] = x_4,
            \\
            &[x_2, x_8] = \widehat{q} [2]_q x_5 x_6, \hspace{10pt}
            [x_3, x_8] = \widehat{q} x_4 x_7, \hspace{10pt}
            [x_4, x_9] = [2]_q x_5,
            \\
            &[x_4, x_8] = \widehat{q} [2]_q x_5 x_7, \hspace{10pt}
            [x_3, x_5] = \frac{q \widehat{q}}{[2]_q} x_4^2, \hspace{10pt}
            [x_7, x_9] = x_8.
        \end{split}
    \]
    Following the proof of Theorem (\ref{as nested E}), we write each Lusztig
    root vector as a multiple of a Chevalley generator, or as a multiple of
    nested $q$-commutators of Chevalley generators. This gives us $x_1 = E_3$,
    $x_2 = \mathbf{E}_{3,2}$, $[2]_q x_3 = [E_3, \mathbf{E}_{3,4}]$, $x_4 =
    [\mathbf{E}_{3,2}, \mathbf{E}_{3,4}]$, $[2]_q^2 x_5 = [[[E_3,
    \mathbf{E}_{3, 4}], E_2], E_2]$, $x_6 = \mathbf{E}_{3,2,1}$, $x_7 =
    [\mathbf{E}_{3,2,1}, \mathbf{E}_{3,4}]$, $x_8 = [[\mathbf{E}_{3,2,1},
    \mathbf{E}_{3,4}], E_2]$, and $x_9 = E_2$. Hence, $x_2 = [x_1, x_9]$ and
    $x_8 = [x_7, x_9]$.

    We compute
    \[
        \begin{split}
            x_4 &= T_{s_3 s_2 s_4} (E_3) = T_{s_3 s_2} T_{s_4} (E_3) = T_{s_3
            s_2} (\mathbf{E}_{4, 3}) = [T_{s_3 s_2} (E_4), T_{s_3 s_2} (E_3)]
            \\
            &= [T_{s_3} T_{s_2} (E_4), E_2] = [T_{s_3} (E_4), E_2] =
            \frac{1}{[2]_q} [ [E_3, \mathbf{E}_{3, 4}], E_2],
        \end{split}
    \]
    and thus $[2]_q [\mathbf{E}_{3, 2}, \mathbf{E}_{3, 4}] = [ [E_3,
    \mathbf{E}_{3, 4}], E_2]$. Therefore, we get $[2]_q x_5 = [x_4, x_9]$ and
    $x_4 = [x_3, x_9]$.

    Next, use the identity $x_7 = [\mathbf{E}_{3, 2, 1}, \mathbf{E}_{3, 4}]$
    and the $q$-Jacobi identity (\ref{q-Jacobi identity}) to get
    \[
        [x_1, x_7] = [ [ E_3, \mathbf{E}_{3, 2, 1}], \mathbf{E}_{3, 4}] - q [
        [E_3, \mathbf{E}_{3, 4}], \mathbf{E}_{3, 2, 1}] + \widehat{q}\, [E_3,
        \mathbf{E}_{3, 4}] \mathbf{E}_{3, 2, 1}.
    \]
    Making the substitutions $x_1 = E_3$, $x_6 = \mathbf{E}_{3, 2, 1}$, $[2]_q
    x_3 = [E_3, \mathbf{E}_{3, 4}]$ gives us
    \[
        [x_1, x_7] = [[x_1, x_6], \mathbf{E}_{3, 4}] - q[2]_q [x_3, x_6] +
        \widehat{q} [2]_q x_3 x_6.
    \]
    Since $[x_1, x_6] = [x_3, x_6] = 0$, we get $[x_1, x_7] = \widehat{q} [2]_q
    x_3 x_6$.

    Similarly, from the identity $x_7 = [\mathbf{E}_{3, 2, 1}, \mathbf{E}_{3,
    4}]$ and the $q$-Jacobi identity, we have
    \[
        [x_2, x_7] = [ [ \mathbf{E}_{3, 2}, \mathbf{E}_{3, 2, 1}],
        \mathbf{E}_{3, 4}] - q [ [ \mathbf{E}_{3, 2}, \mathbf{E}_{3, 4}],
        \mathbf{E}_{3, 2, 1}] + \widehat{q} [ \mathbf{E}_{3, 2}, \mathbf{E}_{3,
        4}] \mathbf{E}_{3, 2, 1}.
    \]
    Using $x_2 = \mathbf{E}_{3, 2}$, $x_6 = \mathbf{E}_{3, 2, 1}$, and $x_6 =
    \mathbf{E}_{3, 2, 1}$ gives us
    \[
        [x_2, x_7] = [ [ x_2, x_6], \mathbf{E}_{3, 4}] - q [ x_4, x_6] +
        \widehat{q} x_4 x_6.
    \]
    Since $[x_2, x_6] = [x_4, x_6] = 0$, we get $[x_2, x_7] = \widehat{q} x_4
    x_6$.

    The remaining relations can be found by applying the following sequence of
    \texttt{R(i,j,r,s)}'s (recall the definition of \texttt{R(i,j,r,s)} in
    Section (\ref{section, L(i,j,r,s) and R(i,j,r,s)})): \texttt{R(1,4,3,9)},
    \texttt{R(1,5,4,9)}, \texttt{R(3,5,4,9)}, \texttt{R(3,8,7,9)},
    \texttt{R(4,8,7,9)}, \texttt{R(1,8,7,9)}, \texttt{R(2,8,7,9)}. With the
    commutation relations among the Lusztig root vectors at hand, it is routine
    to verify that $[x_1, [x_1, x_8]] = \widehat{q}^{\,2} [2]_q^2 x_2 x_3 x_6$
    and $[x_1, [x_1, [x_1, x_8]]] = 0$. Thus, ${\mathcal N}(x_1, x_8) = 3$.
    Equivalently, ${\mathcal N}(\llbracket w \rrbracket) = 3$.

\end{proof}

Now suppose $n \geq 5$. We note that identical versions of Lemmas (\ref{lemma,
B.4, 1}), (\ref{lemma, B.4, 2}), and (\ref{lemma, B.4, ad^2 = 0}), hold in this
situation (i.e. type $C_n$). We forgo copying them here, and instead mention
that the type $C_n$ analogs can be proved using the same steps as those found
in the type $B_n$ setting. The only modification needed is in the proof of
Lemma (\ref{lemma, B.4, ad^2 = 0}); in proving the type $C_n$ version, we
simply need to replace each instance of the monomial $X_{3, 0}^2$ with $X_{3,
3}$. The following theorem is proved using steps identical to those in the
proof of Theorem (\ref{theorem, B.4}).

\begin{theorem}

    \label{theorem, C.4}

    Suppose $n \geq 5$, and let $w = w_{0, 1, 2, 1, n - 4} \in W(C_n)$. Then
    ${\mathcal N}(\llbracket w \rrbracket) = 3$.

\end{theorem}

\subsection{Nilpotency index, Case D.2}

Let $\mathfrak{g}$ be the complex Lie algebra of type $D_n$, where $n > 3$.
Consider the Weyl group element $w := w_{0, 0, 2, 0, n - 2}^+ \in W(D_n)$.  In
this section, we prove ${\mathcal N} \left( \llbracket w \rrbracket \right) =
3$. To begin, we consider the following reduced expression,
\begin{equation}
    \label{D.2, reduced expression}
    w = (s_2 \cdots s_{n - 1}) (s_1 \cdots s_{n - 2}) s_n
    (s_{n - 2} \cdots s_1) (s_{n - 1} \cdots s_2).
\end{equation}
The set of radical roots is
\[
    \Delta_w = \left\{ e_i \pm e_j : 1 \leq i \leq 2 < j \leq n \right\} \cup
    \left\{ e_1 + e_2 \right\}.
\]
The reduced expression (\ref{D.2, reduced expression}) corresponds to the
convex order on $\Delta_w$ given by
\[
    \begin{split}
        &e_2 - e_3 \prec \cdots \prec e_2 - e_n \prec e_1 - e_3 \prec \cdots
        \prec e_1 - e_n \prec e_1 + e_2 \prec e_2 + e_n \prec \cdots
        \\
        & \cdots \prec e_2 + e_3 \prec e_1 + e_n \prec \cdots \prec e_1 + e_3.
    \end{split}
\]
The quantum Schubert cell algebra associated to $w$ will be denoted by
${\mathcal U}_q[w]$.  We denote the Lusztig root vectors $X_{e_i \pm e_j}$ by
\begin{equation}
    \label{Lusztig root vectors, D.2}
    X_{i, \pm j}.
\end{equation}
Since $e_{2, -3}$ and $e_{1, 3}$ are the first and last roots with respect to
the convex order, to show ${\mathcal N} \left( \llbracket w \rrbracket \right)
= 3$ means we need to prove $\left(\operatorname{ad}_q (X_{2, -3}) \right)^2
(X_{1, 3}) \neq 0$ and $\left(\operatorname{ad}_q (X_{2, -3}) \right)^3 (X_{1,
3}) = 0$.

\begin{lemma}

    Let $X_{i, \pm j}$ be as in (\ref{Lusztig root vectors, D.2}). Then

    \begin{align}
        \label{D.2, L-S, 0} &[X_{i, -n}, X_{i, n}] = 0, &(i = 1, 2),&
        \\
        \label{D.2, L-S, 1} &[X_{2, -k}, X_{1, -j}] = [X_{2, j}, X_{1, k}] = 0,
        &(j \leq k),&
        \\
        \label{D.2, L-S, 2} &[X_{i, -k}, X_{i, -j}] = [X_{i, j}, X_{i, k}] = 0,
        &(k < j \text{ and } i = 1, 2),&
        \\
        \label{D.2, L-S, 3} &[X_{i, -k}, X_{1, 2}] = [X_{1, 2}, X_{i, k}] = 0,
        &(k > 2 \text{ and } i = 1, 2),&
        \\
        \label{D.2, L-S, 4} &[X_{i, -k}, X_{i, j}] = [X_{1, -k}, X_{2, j}] = 0,
        &(j \neq k \text{ and } i = 1, 2).&
    \end{align}

\end{lemma}

\begin{proof}

    This follows from Proposition (\ref{LS corollary}).

\end{proof}

\begin{lemma}

    Let $X_{i, \pm k}$ be the Lusztig root vectors of the quantum Schubert cell
    algebra ${\mathcal U}_q[w]$, where $w = w_{0,0,2,0,n - 2}^+$. Then we have
    the following:

    \label{lemma, D.2, 1}

    \begin{enumerate}

        \item \label{lemma, D.2, 1a} $X_{2, -3} = E_2$ and $X_{1, -3} =
            \mathbf{E}_{2, 1}$,

        \item \label{lemma, D.2, 1b} $X_{i, -k} = [X_{i, -(k-1)}, E_{k-1}]$
            \hspace{3mm}(for $3 < k \leq n$ and  $i = 1, 2$),

        \item \label{lemma, D.2, 1c} $X_{i, n} = [X_{i, -(n-1)}, E_n]$ (for $i
            = 1, 2$),

        \item \label{lemma, D.2, 1d} $X_{i, k} = [X_{i, k + 1}, E_k]$
            \hspace{3mm} (for $3 \leq k < n$ and $i = 1, 2$),

        \item \label{lemma, D.2, 1e} $X_{1, n} = [X_{2, n}, E_1]$.

    \end{enumerate}

\end{lemma}

\begin{proof}

    Parts (\ref{lemma, D.2, 1a}), (\ref{lemma, D.2, 1b}), (\ref{lemma, D.2,
    1d}), and (\ref{lemma, D.2, 1e}) can be proved in a manner similar to the
    proofs of parts (\ref{lemma, B.2, 1a}), (\ref{lemma, B.2, 1b}),
    (\ref{lemma, B.2, 1c}) and (\ref{lemma, B.2, 1e}) of Lemma (\ref{lemma,
    B.2, 1}), respectively.

    To prove part (\ref{lemma, D.2, 1c}), define $z :=(s_2 \cdots s_{n - 1})
    (s_1 \cdots s_{n - 4}) \in W(D_n)$. We obtain
    \[
        \begin{split}
            X_{1, n} &= T_{(s_2 \cdots s_{n - 1}) (s_1 \cdots s_{n - 2}) s_n
            (s_{n - 2} \cdots s_1)} (E_{n - 1}) = T_{z s_{n - 1} s_n s_{n - 3}}
            T_{(s_{n - 1} \cdots s_1) s_n} (E_{n - 1})
            \\
            &= T_{z s_{n - 1} s_n s_{n - 3}} (E_{n - 2}) = T_{z s_{n - 1} s_n}
            T_{s_{n - 3}} (E_{n - 2}) = T_{z s_{n - 1} s_n} \left( [E_{n - 3},
            E_{n - 2}] \right)
            \\
            &= [T_{z s_{n - 1} s_n} (E_{n - 3}), T_{z s_{n - 1} s_n} (E_{n -
            2})] = [T_z T_{s_{n - 1} s_n} (E_{n - 3}), E_n]
            \\
            &= [T_z (E_{n - 3}), E_n] = [X_{1, -(n - 1)}, E_n].
        \end{split}
    \]
    Similarly, we get $X_{2, n} = [X_{2, n - 1}, E_{n}]$.

\end{proof}

From Lemma (\ref{lemma, D.2, 1}), we can write each Lusztig root vector $X_{i,
\pm k}$ as a sequence of nested $q$-commutators involving Chevalley
generators.  Specifically, for $3 \leq k \leq n$,
\begin{align}
    \label{D.2, as nested q-commutators, a} &X_{2, -k} = \mathbf{E}_{2, \dots,
    k - 1},
    \\
    \label{D.2, as nested q-commutators, b} &X_{1,-k} = \mathbf{E}_{2, 1, 3, 4
    \dots, k - 1},
    \\
    \label{D.2, as nested q-commutators, c} &X_{2, k} = \mathbf{E}_{2, 3, \dots
    n - 2, n, n - 1, n - 2 \dots, k},
    \\
    \label{D.2, as nested q-commutators, d} &X_{1, k} = \mathbf{E}_{2, 1, 3,
    4,\dots, n - 2, n, n - 1, n - 2,\dots, k}.
\end{align}
In fact, the identities (\ref{D.2, as nested q-commutators, a}) - (\ref{D.2, as
nested q-commutators, d}) are equivalent to Lemma (\ref{lemma, D.2, 1}).

\begin{lemma}

    $ $

    \label{lemma, D.2, 2}

    \begin{enumerate}

        \item \label{lemma, D.2, 2.1} For all $3 \leq k < j \leq n$,

            \begin{enumerate}[(a)]

                \item \label{lemma, D.2, 2.1.a} $[X_{2, -k}, X_{1, -j}] =
                    \widehat{q} X_{2, -j} X_{1, -k}$,

                \item \label{lemma, D.2, 2.1.c} $[X_{2, -k}, X_{1, j}] =
                    \widehat{q} \,\, X_{1, -k} X_{2, j}$.

            \end{enumerate}

        \item \label{lemma, D.2, 2.3} For $3 \leq k \leq n$,

            \begin{enumerate}[(a)]

                \item \label{lemma, D.2, 2.3.a}$[X_{2,-k}, X_{1, k}] =
                    (-q)^{n+1-k} X_{1, 2} + \widehat{q} X_{1,-k} X_{2,k} \\
                    \phantom{\hspace{25mm}} +\widehat{q}\,\, \sum_{j > k}
                    (-q)^{j-k-1}   \left( X_{2,-j}X_{1,j} - \widehat{q}
                    X_{1,-j}X_{2,j} \right)$,

                \item \label{lemma, D.2, 2.3.b} $[X_{2,-k}, X_{2, k}] =
                    \widehat{q} \sum_{j > k} (-q)^{j - k - 1} X_{2,-j} X_{2,
                    j}$.

            \end{enumerate}

    \end{enumerate}

\end{lemma}

\begin{proof}

    Parts (\ref{lemma, D.2, 2.1.a}) and (\ref{lemma, D.2, 2.1.c}) can be proved
    using the same techniques found in the proof of Lemma (\ref{lemma, B.2,
    2}).

    We next prove (\ref{lemma, D.2, 2.3.a}), but, as a preliminary result,
    first establish the commutation relation $[X_{1, -(n - 1)}, X_{2, n - 1}] =
    -q X_{1, 2} + \widehat{q} X_{1, -n} X_{2, n}$. For short, define the Weyl
    group element $u := (s_2 \cdots s_{n - 1}) (s_1 \cdots s_{n - 3}) \in W$.
    We compute
    \[
        \begin{split}
            [X_{1, -n}, X_{2, n}] &= [T_u(E_{n - 2}), T_{u s_{n - 2} s_n} (E_{n
            - 2})] = T_u \left( [E_{n - 2}, T_{s_{n - 2} s_n} (E_{n - 2}) ]
            \right)
            \\
            &= T_u \left( [E_{n - 2}, E_n] \right) = T_u T_{n - 2} (E_n) = T_{u
            s_{n - 2}} (E_n) = X_{1, 2}.
        \end{split}
    \]
    Next, use part (\ref{lemma, D.2, 1d}) of Lemma (\ref{lemma, D.2, 1}) to
    write $X_{2, n - 1} = [X_{2, n}, E_{n - 1}]$, and apply the $q$-Jacobi
    identity (\ref{q-Jacobi identity}) together with the relation $[X_{1, -(n -
    1)}, E_{n - 1}] = X_{1, -n}$ from Lemma (\ref{lemma, D.2, 1}), part
    (\ref{lemma, D.2, 1b}), to get
    \[
        [X_{1, -(n-1)}, X_{2, n - 1}] = [[X_{1,-(n-1)}, X_{2,n}], E_{n-1}] -q
        [X_{1, -n}, X_{2, n}] + \widehat{q} X_{1, -n} X_{2, n}.
    \]
    The $q$-commutator identities in (\ref{D.2, L-S, 4}) imply $[X_{1,-(n-1)},
    X_{2,n}] = 0$.  Hence, this simplifies to
    \[
        [X_{1, -(n-1)}, X_{2, n - 1}] = -q [X_{1, -n}, X_{2, n}] + \widehat{q}
        X_{1, -n} X_{2, n} = -q X_{1, 2} + \widehat{q} X_{1, -n} X_{2, n}.
    \]
    We will proceed with proving (\ref{lemma, D.2, 2.3.a}) by inducting on $k$,
    beginning with the base case $k = n$.  The identities in (\ref{D.2, as
    nested q-commutators, a}) and (\ref{D.2, as nested q-commutators, c}) tell
    us how the Lusztig root vectors $X_{2, -n}$ and $X_{2, n - 1}$ can be
    respectively written as sequences of nested $q$-commutators. In particular,
    we have $X_{2, -n} = \mathbf{E}_{2,\dots, n - 1}$ and $X_{2, n - 1} =
    \mathbf{E}_{2,3,\dots, n - 2, n, n - 1}$.  However, since the Chevalley
    generators $E_{n - 1}$ and $E_n$ commute, $X_{2, n - 1}$ can also be
    written as $X_{2, n - 1} = \mathbf{E}_{2,\dots, n}$. Hence, $X_{2, n - 1} =
    [X_{2, -n}, E_n]$.  By using the $q$-commutator identity $X_{1, n} = [X_{1,
    -(n - 1)}, E_n]$ from part (\ref{lemma, D.2, 1c}) of Lemma (\ref{lemma,
    D.2, 1}), and applying the $q$-Jacobi identity (\ref{q-Jacobi identity}),
    we get
    \[
        [X_{2, -n}, X_{1, n}] = [X_{2, -n}, X_{1, -(n - 1)}], E_n] -q [X_{2, n
        - 1}, X_{1, -(n - 1)}] +\widehat{q} X_{2, n - 1} X_{1, -(n - 1)}.
    \]
    However, we have $[X_{2, -n}, X_{1, -(n-1)}] = 0$ (\ref{D.2, L-S, 1}).
    Therefore,
    \[
        [X_{2,-n}, X_{1, n}] = [X_{1, -(n-1)}, X_{2, n - 1}] = -q X_{1,2} +
        \widehat{q} X_{1, -n} X_{2, n},
    \]
    and this completes the proof of (\ref{lemma, D.2, 2.3.a}) for the case when
    $k = n $.  Next, we proceed by induction. Assume now $k < n$. Apply the
    $q$-Jacobi identity to get
    \[
        [X_{2,-k}, X_{1, k}] = [[X_{2, -k}, X_{1,k+1}], E_k] - q [X_{2,
        -(k+1)}, X_{1, k + 1}] + \widehat{q} X_{2,-(k+1)} X_{1, k + 1}.
    \]
    Since $[X_{2, -k}, X_{1,k+1}] = \widehat{q}\,\,X_{1, -k} X_{2,k+1}$ by part
    (\ref{lemma, D.2, 2.1.c}) of this lemma,
    \[
        [X_{2,-k}, X_{1, k}] = \widehat{q} \,\, [ X_{1, -k} X_{2,k+1}, E_k] - q
        [X_{2, -(k+1)}, X_{1, k + 1}] + \widehat{q} X_{2,-(k+1)} X_{1, k + 1}.
    \]
    Finally, apply the $q$-Leibniz identity (\ref{q-Leibniz identity}) to
    obtain
    \[
        \begin{split}
            [X_{2,-k}, X_{1, k}] &=  \widehat{q} \left( X_{1, -k} X_{2,k} +
            q^{-1} X_{1, -(k+1)} X_{2,k+1} \right) \\ & \hspace{5mm} -q [X_{2,
            -(k+1)}, X_{1, k + 1}] + \widehat{q} X_{2,-(k+1)} X_{1, k + 1},
        \end{split}
    \]
    and the result follows by applying the induction hypothesis.

    We prove (\ref{lemma, D.2, 2.3.b}) also by induction.  The identities in
    (\ref{D.2, L-S, 0}) imply $[X_{2, -n}, X_{2, n}] = 0$. This establishes the
    base case $k = n$.  Next, for $3 \leq k < n$, we use the identities $X_{2,
    -(k + 1)} = [X_{2, -k}, E_k]$ and  $X_{2, k} = [X_{2, k + 1} E_k]$ found in
    Lemma (\ref{lemma, D.2, 1}), parts (\ref{lemma, D.2, 1a}) and (\ref{lemma,
    D.2, 1d}), respectively, and apply the $q$-Jacobi identity to get
    \[
        [X_{2, -k}, X_{2, k}] = [ [ X_{2, -k}, X_{2, k + 1}], E_k] -q [ X_{2,
        -(k+1)}, X_{2, k + 1}] + \widehat{q} X_{2, -(k+1)} X_{2, k + 1}.
    \]
    Since $[X_{2, -k}, X_{2, k + 1}] = 0$ (\ref{D.2, L-S, 4}),
    \[
        [X_{2, -k}, X_{2, k}] = -q [ X_{2, -(k+1)}, X_{2, k + 1}] + \widehat{q}
        X_{2, -(k+1)} X_{2, k + 1},
    \]
    and the result follows.

\end{proof}

Now we are ready to compute the nilpotency index ${\mathcal N}\left( \llbracket
w_{0, 0, 2, 0, n - 2}^+ \rrbracket \right)$.

\begin{theorem}

    \label{D.2, nilpotency index, main theorem}

    Let $w = w_{0, 0, 2, 0, n - 2}^+ \in W(D_n)$, where $n > 3$.  Then ${\mathcal
    N}\left( \llbracket w \rrbracket \right) = 3$.

\end{theorem}

\begin{proof}

    The results of Lemma (\ref{lemma, D.2, 2}) describe how $X_{2, -3}$
    commutes with certain Lusztig root vectors.  We begin by applying part
    (\ref{lemma, D.2, 2.3.a}) of Lemma (\ref{lemma, D.2, 2}), which tells us
    how to write the $q$-commutator $[X_{2, -3}, X_{1, 3}]$ as a linear
    combination of ordered monomials,
    \[
        [X_{2, -3}, X_{1, 3}] = (-q)^{n - 2}X_{1, 2} + \widehat{q} X_{1, -3}
        X_{2, 3} + \widehat{q} \sum_{j > 3} (-q)^{j - 4} \left(X_{2, -j} X_{1,
        j} \!-\!  \widehat{q} X_{1, -j} X_{2, j} \right).
    \]
    We compute $[X_{2, -3}, [X_{2, -3}, X_{1, 3}]]$ by applying the $q$-Leibniz
    identity (\ref{q-Leibniz identity}) to each term on the right-hand side of
    the equation above. When using $q$-Leibniz in this setting, we also take
    advantage of the fact that $X_{2, -3}$ $q$-commutes with many of the
    Lusztig root vectors appearing in the right-hand side. In particular, the
    $q$-commutator relations in (\ref{D.2, L-S, 1}) - (\ref{D.2, L-S, 4}) tell
    us that $[X_{2, -3}, X_{1, -3}]$, $[X_{2, -3}, X_{2, -j}]$ ($j > 3$),
    $[X_{2, -3}, X_{1, 2}]$, and $[X_{2, -3}, X_{2, j}]$ ($j > 3$) are all
    equal to $0$.  Furthermore, Lemma (\ref{lemma, D.2, 2}), parts (\ref{lemma,
    D.2, 2.1.a}), (\ref{lemma, D.2, 2.1.c}), and (\ref{lemma, D.2, 2.3.b}),
    respectively, tell us how to write $[X_{2, -3}, X_{1, -j}]$ (for $j > 3$),
    $[X_{2, -3}, X_{1, j}]$ (for $j > 3$), and $[X_{2, -3}, X_{2, 3}]$ as
    linear combinations of ordered monomials. With this, we have enough
    information to be able to write $[X_{2, -3}, [X_{2, -3}, X_{1, 3}]]$ as a
    linear combination of ordered monomials. We have
    \begin{equation}
        \label{D.2, 5}
        [X_{2, -3}, [X_{2, -3}, X_{1, 3}]] = (q + q^{-1}) \widehat{q}^{\,2}
        \left( \sum_{j > 3} (-q)^{j - 4} X_{2, -j} X_{1, -3} X_{2, j} \right).
    \end{equation}
    In particular, $[X_{2, -3}, [X_{2, -3}, X_{1, 3}]]$ is nonzero. Thus,
    ${\mathcal N}\left( \llbracket w_{0, 0, 2, 0, n - 2}^+ \rrbracket \right) >
    2$. However, $X_{2, -3}$ $q$-commutes with every variable $X_{i, \pm k}$
    appearing on the right-hand side of (\ref{D.2, 5}). That is to say,
    \[
        [X_{2, -3}, X_{2, - j}] = [X_{2, -3}, X_{1, -3}] = [X_{2, -3}, X_{2,
        j}] = 0
    \]
    for all $j > 3$. Hence, by the $q$-Leibniz identity, $X_{2, -3}$ also
    $q$-commutes with the monomial $X_{2, -j} X_{1, -3} X_{2, j}$. Therefore
    $[X_{2, -3}, [X_{2, -3}, [X_{2, -3}, X_{1, 3}]]] = 0$, and thus ${\mathcal
    N}(\llbracket w \rrbracket) = 3$.

\end{proof}

\subsection{Nilpotency index, Case D.4}

Let $\mathfrak{g}$ be the complex Lie algebra of type $D_n$, where $n > 4$.
Consider the Weyl group element $w := w_{0, 1, 2, 1, n - 4}^+ \in W(D_n)$ and
the reduced expression
\begin{equation}
    \label{D.4, reduced expression}
    w = (s_3 \cdots s_{n-1})(s_2 \cdots s_{n-2})s_ns_1(s_{n-2} \cdots
    s_2) (s_{n-1} \cdots s_3)
\end{equation}
Let $\Delta_w$ be the set of roots of $w$. The reduced expression for $w$
corresponds to the following convex order on $\Delta_w$:
\begin{equation}
    \label{convex order, D.4}
    \begin{split}
        &e_3 - e_4 \prec e_3 - e_5 \prec \cdots \prec e_3 - e_n \prec e_2 - e_4
        \prec e_2 - e_5 \prec \cdots
        \\
        &\cdots \prec e_2 - e_n \prec e_2 + e_3 \prec e_1 - e_4 \prec e_3 + e_n
        \prec e_3 - e_{n - 1} \prec \cdots
        \\
        &\cdots \prec e_3 + e_5 \prec e_1 + e_3 \prec e_2 + e_n \prec e_2 +
        e_{n - 1} \prec \cdots \prec e_2 + e_5 \prec e_1 + e_2.
    \end{split}
\end{equation}
Let ${\mathcal U}_q^+[w]$ be the quantum Schubert cell algebra associated to
$w$. The Lusztig root vectors will be denoted by
\begin{equation}
    \label{D.4, Lusztig root vectors}
    X_{m, -4}, X_{r, s}, X_{i, \pm k} \in {\mathcal U}_q^+[w],
\end{equation}
where $1 < i < 4 < k \leq n$, $1 \leq m < 4$, and $1 \leq r < s < 4$.

We note that identical versions of Lemmas (\ref{lemma, B.4, 1}), (\ref{lemma,
B.4, 2}), and (\ref{lemma, B.4, ad^2 = 0}), hold in this situation (i.e. type
$D_n$). We forgo copying them here, and instead mention that the type $D_n$
analogs can be proved using the same steps as those found in the type $B_n$
setting. The only modification needed is in the proof of Lemma (\ref{lemma,
B.4, ad^2 = 0}); in proving the type $D_n$ version, we simply need to replace
each instance of the monomial $X_{3, 0}^2$ with $0$. The following theorem is
proved using steps identical to those in the proof of Theorem (\ref{theorem,
B.4}).

\begin{theorem}

    \label{theorem, D.4}

    Suppose $n > 4$, and let $w = w_{0, 1, 2, 1, n - 4}^+ \in W(D_n)$. Then
    ${\mathcal N}(\llbracket w \rrbracket) = 3$.

\end{theorem}

\section{Nilpotency indices: The exceptional types}

\label{section, nilpotency indices, exceptional}

In this section, we calculate the nilpotency indices ${\mathcal N}(\llbracket w
\rrbracket)$ associated to each $w \in BiGr_\perp^\circ(X_n)$, where $X_n$ is
an exceptional Lie type.

\subsection{Type \texorpdfstring{$G_2$}{G2}}

\label{section, G2}

The set $BiGr_{\perp}^{\circ}(G_2)$ has 8 elements. Every element of the Weyl
group $W(G_2)$, except the identity, the simple reflections, and the longest
element, belongs to $BiGr^\circ(G_2)$. Verifying the orthogonality condition
for each of the 8 elements of $BiGr^\circ(G_2)$ can be done on a case by case
basis.

We consider the reduced expression
\[
    w_0 = s_1s_2s_1s_2s_1s_2 \in W(G_2)
\]
of the longest element. The Lusztig root vectors associated to this reduced
expression will be simply denoted as $x_1,\dots,x_6$.  The defining relations
of the quantum Schubert cell algebra ${\mathcal U}_q^+[w_0]$ appear in Hu and
Wang \cite{HW}, where they index the root vectors by Lyndon words. The
correspondence between our notation and theirs is $x_1 \leftrightarrow E_1$,
$[3]_q!x_2 \leftrightarrow E_{1112}$, $[2]_qx_3 \leftrightarrow E_{112}$,
$[3]_q!x_4 \leftrightarrow E_{11212}$, $x_5 \leftrightarrow E_{12}$, $x_6
\leftrightarrow E_2$. Proposition (\ref{LS corollary}) implies $[x_i, x_{i +
1}] = 0$ for $i \in [1, 5]$. The remaining commutation relations are as follows
(see \cite[Eqns.  2.2 - 2.7 and Lemma 3.1]{HW}):
\begin{equation}
    \label{G2, relations}
    {\renewcommand{\arraystretch}{1.2}
    \text{\begin{tabular}{p{75pt}p{75pt}p{100pt}}
            $[x_1, x_3] = [3]_q x_2$,
            &
            $[x_1, x_4] = q\widehat{q} x_3^2$,
            &
            $[x_2, x_6] = \widehat{q} x_3 x_5 + \zeta x_4$,
            \\
            $[x_3, x_5] = [3]_q x_4$,
            &
            $[x_2, x_5] = q\widehat{q} x_3^2$,
            &
            $[x_2, x_4] = \eta x_3^3$,
            \\
            $[x_1, x_5] = [2]_q x_3$,
            &
            $[x_3, x_6] = q\widehat{q} x_5^2$,
            &
            $[x_4, x_6] = \eta x_5^3$,
            \\
            $[x_1, x_6] = x_5$,
            &&
    \end{tabular}}
    }
\end{equation}

\noindent where $\zeta := q^{-3} - q^{-1} - q\in \mathbb{K}$ and $\eta = q^3
\frac{\widehat{q}^2}{[3]_q} \in \mathbb{K}$.

Every $w \in BiGr_\perp^\circ(G_2)$ has a unique reduced expression, and this
reduced expression appears as a contiguous substring of $(1, 2, 1, 2, 1, 2)$.
This implies that for every $w \in BiGr_\perp^\circ(G_2)$, the quantum Schubert
cell algebra ${\mathcal U}_q^+[w]$ can be identified with an interval
subalgebra of ${\mathcal U}_q^+[w_0]$. For example, the reduced expression of
the Weyl group element $s_2 s_1 s_2 s_1$ appears as the substring obtained from
the middle four entries of $(1, 2, 1, 2, 1, 2)$. Thus, ${\mathcal
U}_q^+[s_2s_1s_2s_1]$ is isomorphic to the subalgebra of ${\mathcal
U}_q^+[w_0]$ generated by $x_2, x_3, x_4, x_5$.  Hence, the nilpotency index
${\mathcal N}(\llbracket s_2 s_1 s_2 s_1 \rrbracket)$ can be computed by
finding the smallest natural number $p$ so that
$\left(\operatorname{ad}_q(x_2)\right)^p (x_5) = 0$. By applying the
$q$-commutation relations (\ref{G2, relations}) above, one can show $[x_2,
x_5]$ is nonzero and $[x_2, [x_2, x_5]] = 0$.  Hence, ${\mathcal N}(\llbracket
s_2 s_1 s_2 s_1 \rrbracket) = 2$. Applying this technique to the remaining
elements of $BiGr_\perp^\circ(G_2)$ gives us the following theorem.

\begin{theorem}

    The rightmost column of Table (\ref{table, G2}) gives the nilpotency index
    ${\mathcal N}(\llbracket w \rrbracket)$ for each $w \in
    BiGr_{\perp}^{\circ}(G_2)$.

\end{theorem}

\begin{center}
    \begin{table}[H]
        \caption{The Lie algebra of type $G_2$ (data)}
\begin{tabular}{lcc}
    \toprule
    $w \in BiGr_{\perp}^{\circ}(G_2)$ & $\chi(\llbracket w \rrbracket)$ &
    ${\mathcal N}(\llbracket w \rrbracket)$
    \\[0.5mm]
    \hline
    $s_2s_1$            & (6,2,3) & 1\\
    $s_1s_2$            & (2,6,3) & 1\\
    $s_2s_1s_2$         & (6,6,3) & 2\\
    $s_2s_1s_2s_1$      & (6,2,0) & 2\\
    $s_1s_2s_1$         & (2,2,1) & 2\\
    $s_2s_1s_2s_1s_2$   & (6,6,-3) & 3\\
    $s_1s_2s_1s_2s_1$   & (2,2,-1) & 3\\
    $s_1s_2s_1s_2$      & (2,6,0) & 4\\
    \bottomrule
\end{tabular}
\label{table, G2}
\end{table}
\end{center}

\subsection{Type \texorpdfstring{$F_4$}{F4}}

\label{section, F4}

We turn our attention now to the exceptional Lie type $F_4$.  We aim to compute
the nilpotency index ${\mathcal N}(\llbracket w \rrbracket)$ associated to each
$w \in BiGr_\perp^\circ(F_4)$.  In order to explicitly list the elements of
$BiGr_{\perp}^{\circ}(F_4)$, first let
\[
    f: W(F_4) \to W(F_4)
\]
be the Weyl group automorphism that interchanges the simple reflections $s_1$
and $s_4$, and interchanges $s_2$ and $s_3$, and define $\kappa_1,\dots,
\kappa_{17} \in W(F_4)$ as follows,
\begin{equation}
    \label{F4 data}
    \text{\begin{tabular}{p{160pt}p{140pt}}
        $\kappa_{1} := s_2s_3s_4s_2s_3s_1s_2s_3$,
        &
        $\kappa_2 := s_2s_3 (s_1s_2s_4s_3)^2$,
        \\
        $\kappa_3 := s_2s_3 (s_1s_2s_4s_3)^3$,
        &
        $\kappa_4 := s_2s_3 (s_1s_2s_4s_3)^4$,
        \\
        $\kappa_5 := s_2s_3 s_1s_2s_4s_3 s_1s_2$,
        &
        $\kappa_6 := s_2s_3 (s_1s_2s_4s_3)^2 s_1s_2$,
        \\
        $\kappa_7 := s_2s_3 (s_1s_2s_4s_3)^3 s_1s_2$,
        &
        $\kappa_{8} := s_1s_2s_3s_4s_2s_3s_2s_1$,
        \\
        $\kappa_{9} := s_1s_2s_3s_4s_2s_1s_3s_2$,
        &
        $\kappa_{10} := s_1s_2s_3s_4s_2s_1s_3s_2s_3$,
        \\
        $\kappa_{11} := (s_1s_2s_3s_4s_2s_3s_2)^2 s_1$,
        &
        $\kappa_{12} := s_2s_3 s_1s_2s_4s_3$,
        \\
        $\kappa_{13} := s_2s_3s_4s_2s_3s_1s_2s_3s_4$,
        &
        $\kappa_{14} := s_2s_3s_4 (s_2s_1s_3)^2 s_4s_2s_3s_2$,
        \\
        $\kappa_{15} := s_2s_3s_4 (s_1s_2s_3s_2)^2 s_4s_3s_2$,
        &
        $\kappa_{16} := s_2s_3s_4s_1s_2s_3s_2s_1$,
        \\
        $\kappa_{17} := s_2s_3s_4 s_1s_2s_3s_2s_1 s_4s_3s_2$.
        &
    \end{tabular}}
\end{equation}
One can readily verify that $\kappa_1,\dots, \kappa_{17}, f(\kappa_1),\dots,
f(\kappa_{17})$ are distinct, and each belongs to $BiGr_\perp^{\circ}(F_4)$.
Moreover,
\[
    BiGr_\perp^{\circ}(F_4) = \left\{ \kappa_1,\dots, \kappa_{17},
    f(\kappa_1),\dots, f(\kappa_{17}) \right\}.
\]

The following theorem gives some insight into the structure of the equivalence
classes (under the equivalence relation $\xleftrightarrow{\hspace{4mm}}$) in
$\Gamma(W(F_4))$.

\begin{theorem}

    \label{equivalence classes, bound, F4}

    Let $\Gamma_4 := \left\{ \llbracket \kappa
    \rrbracket \in \Gamma(W(F_4)) : \kappa \in BiGr_\perp^\circ(F_4) \right\}$.

    \begin{enumerate}

        \item The $34$ elements of $\Gamma_4$ are contained in at most $30$
            distinct equivalence classes.

        \item \label{equivalence classes, bound, F4, c} At least $2$ of the
            $34$ elements of $\Gamma_4$, namely $\llbracket \kappa_1
            \rrbracket$ and $\llbracket f(\kappa_1) \rrbracket$, are equivalent
            to an element of the form $\llbracket w \rrbracket$ with $w \in
            BiGr_\perp(F_4) \backslash BiGr_\perp^\circ(F_4)$.

    \end{enumerate}

\end{theorem}

\begin{proof}

    Consider the following elements of $\Gamma(W(F_4))$,
    \[
        \begin{split}
            &X_1 := ( \kappa_{17} \cdot s_1, 2, 2), \hspace{60.7pt} Y_1 :=
            (f(\kappa_{17} \cdot s_1), 3, 3),
            \\
            &X_2 := (\kappa_{10} \cdot s_4 s_3 s_2, 1, 4), \hspace{42.4pt}
            Y_2:= (f(\kappa_{10} \cdot s_4 s_3 s_2), 4, 1),
            \\
            &X_3 := (s_4 s_1 s_2 s_3 s_4 s_2 \cdot \kappa_{10}, 2, 3),
            \hspace{15pt} Y_3 := (f(s_4 s_1 s_2 s_3 s_4 s_2 \cdot \kappa_{10}),
            3, 2).
        \end{split}
    \]
    The result follows from the reductions
    \begin{equation*}
        \text{\begin{tabular}{p{90pt}p{90pt}p{90pt}}
            $X_1 \xlongrightarrow{L} \llbracket \kappa_9 \rrbracket$,
            &$X_2 \xlongrightarrow{L} \llbracket \kappa_{13} \rrbracket$,
            &$X_3 \xlongrightarrow{L} \llbracket \kappa_1 \rrbracket$,
            \\
            $X_1 \xlongrightarrow{R} \llbracket \kappa_{16} \rrbracket$,
            &$X_2 \xlongrightarrow{R} \llbracket \kappa_{10} \rrbracket$,
            &$X_3 \xlongrightarrow{R} \llbracket s_2 s_3 \rrbracket$,
            \\
            $Y_1 \xlongrightarrow{L} \llbracket f(\kappa_9) \rrbracket$,
            &$Y_2 \xlongrightarrow{L} \llbracket f(\kappa_{13}) \rrbracket$,
            &$Y_3 \xlongrightarrow{L} \llbracket f(\kappa_1) \rrbracket$,
            \\
            $Y_1 \xlongrightarrow{R} \llbracket f(\kappa_{16}) \rrbracket$,
            &$Y_2 \xlongrightarrow{R} \llbracket f(\kappa_{10}) \rrbracket$,
            &$Y_3 \xlongrightarrow{R} \llbracket s_3 s_2 \rrbracket$.
        \end{tabular}}
    \end{equation*}

\end{proof}

The significance of part (\ref{equivalence classes, bound, F4, c}) of Theorem
(\ref{equivalence classes, bound, F4}) is to make the point that nilpotency
indices can sometimes be found using results from a smaller rank setting. For
instance, as a consequence of the reductions listed in the proof of Theorem
(\ref{equivalence classes, bound, F4}), we have $\llbracket \kappa_1 \rrbracket
\xleftrightarrow{\hspace{4mm}} \llbracket s_2 s_3 \rrbracket$ and $\llbracket
f(\kappa_1) \rrbracket \xleftrightarrow{\hspace{4mm}} \llbracket s_3 s_2
\rrbracket$.  The only simple reflections appearing in the reduced expressions
$s_2 s_3$ and $s_3 s_2$ are $s_2$ and $s_3$.  These generate the type $B_2$
Weyl group.  Hence $\llbracket s_2 s_3 \rrbracket$ and $\llbracket s_3 s_2
\rrbracket$ can each be identified with an element in $\Gamma(W(B_2))$. Here,
we identify $\llbracket s_2 s_3 \rrbracket$ with $\llbracket w_{0,1,1,0,0}
\rrbracket$, and we identify $\llbracket s_3 s_2 \rrbracket$ with $\llbracket
w_{0,0,1,1,0} \rrbracket$. Since ${\mathcal N} \left( \llbracket w_{0,1,1,0,0}
\rrbracket \right) = {\mathcal N} \left( \llbracket w_{0,0,1,1,0} \rrbracket
\right) = 1$ (see Theorem \ref{SmallRank, nil-index, main theorem}), then
${\mathcal N}(\llbracket \kappa_1 \rrbracket) = {\mathcal N}(\llbracket
f(\kappa_1) \rrbracket) = 1$ also.

Computer experiments reveal that over 92\% of the elements in $\Gamma(W(F_4))$
are equivalent (under $\xleftrightarrow{\hspace{4mm}}$) to an element $(w, i,
j)$ with $w$ lacking full support.  Specifically, computer experiments indicate
that there are 50 equivalence classes in $\Gamma(W(F_4))$.  The 34 elements in
$\Gamma_4$ are contained in 30 equivalence classes. Exactly 28 equivalence
classes fail to contain an element $(w, i, j)$ with $w$ lacking full support,
and the union of these 28 equivalence classes contains precisely 328 elements.
As $\Gamma(W(F_4))$ has cardinality 4416 (see Theorem (\ref{Gamma W,
cardinality})), these 328 elements account for approximately only $7.427$
percent of the total number of elements in $\Gamma(W(F_4))$.

We also obtain the following curious result.

\begin{theorem}

    If $x = (w, i, j) \in \Gamma(W(F_4))$ with $\| \alpha_i \| = \| \alpha_j
    \|$, then ${\mathcal N}(x) = {\mathcal N}(x^*)$ and $\chi(x) = \chi(x^*)$.

\end{theorem}

\begin{proof}

    Let $x = (w, i, j) \in \Gamma(W(F_4))$, and assume $\| \alpha_i \| = \|
    \alpha_j \|$. Theorem (\ref{reduce by L and R}) and Proposition
    (\ref{proposition, orthogonality condition}) together imply that $x$ is
    equivalent (under $\xleftrightarrow{\hspace{4mm}}$) to an element of the
    form $(u, k, \ell)$ such that either (1) $u = w_0(a, b)$ for some $a, b\in
    \mathbf{I}$, or (2) $u \in BiGr_\perp(F_4)$. Since $x
    \xleftrightarrow{\hspace{4mm}} (u, k, \ell)$, then $\chi(x) = \chi(u, k,
    \ell)$. From Proposition (\ref{proposition, a}), $x^*
    \xleftrightarrow{\hspace{4mm}} (u, k, \ell)^*$. Thus, without loss of
    generality, we may assume $x$ has the form $(u, k, \ell)$, where either $u
    \in BiGr_\perp(F_4)$ or $u = w_0(a, b)$ for some $a, b, \in \mathbf{I}$.
    In the case when $x = (w_0(a, b), k, \ell)$, we have $x^* = (w_0(a, b),
    \ell, k)$, and it is a straightforward observation that ${\mathcal N}(x) =
    {\mathcal N}(x^*)$ and $\chi(x) = \chi(x^*)$.

    It remains to consider the case when $x = (u, k, \ell)$ with $u \in
    BiGr_\perp(F_4)$ and $\| \alpha_k \| = \| \alpha_\ell \|$.  Each $u \in
    BiGr_\perp(F_4)$ can be classified by its support (recall Section
    (\ref{section: rank reduction})).  With this, we may view $u$ as belonging
    to $BiGr_\perp^\circ(X_n)$, where $X_n$ is a Lie type of rank greater than
    $1$ with associated Dynkin diagram a subgraph of the $F_4$ Dynkin diagram.
    That is to say, $X_n$ is of type $A_2$, $B_2$, $C_2$, $B_3$, $C_3$, or
    $F_4$.

    Tables (\ref{table, small rank ABCD}) and (\ref{table, general cases ABCD})
    give the nilpotency index ${\mathcal N}(\llbracket u \rrbracket)$ for each
    $u \in BiGr_\perp^\circ(X_n)$ for $X_n$ of type $ABCD$, and upon
    inspection, we observe that ${\mathcal N}(\llbracket u \rrbracket) =
    {\mathcal N}(\llbracket u \rrbracket^*)$ in all relevant cases.

    Next, we observe that $\llbracket \kappa_9 \rrbracket$, $\llbracket
    \kappa_{16} \rrbracket$, $\llbracket f(\kappa_9) \rrbracket$, and
    $\llbracket f(\kappa_{16}) \rrbracket$ are the only elements in
    $\Gamma(W(F_4))$ of the form $(u, k, \ell)$ such that $u \in
    BiGr_\perp^\circ(F_4)$, $\| \alpha_k \| = \| \alpha_\ell \|$, and $(u, k,
    \ell) \neq (u, k, \ell)^*$. However, we have $\llbracket \kappa_9
    \rrbracket = \left( \llbracket \kappa_{16} \rrbracket \right)^*$,
    $\llbracket f(\kappa_9) \rrbracket = \left( \llbracket f(\kappa_{16})
    \rrbracket \right)^*$, $\llbracket \kappa_{16} \rrbracket = \left(
    \llbracket \kappa_9 \rrbracket \right)^*$, and $\llbracket f(\kappa_{16})
    \rrbracket = \left( \llbracket f(\kappa_9) \rrbracket \right)^*$, and thus,
    from the reductions listed in the proof of Theorem (\ref{equivalence
    classes, bound, F4}), we conclude that ${\mathcal N}(u, k, \ell) =
    {\mathcal N}(u, k, \ell)^*$ also when $u \in BiGr_\perp^\circ(F_4)$ and $\|
    \alpha_k \| = \| \alpha_\ell \|$.

    To prove $\chi(x) = \chi(x^*)$ whenever $x = (w, i, j)$ with $\| \alpha_i
    \| = \| \alpha_j \|$, we may, from the above discussion, assume $x$ has the
    form $(u, k, \ell)$ such that $\| \alpha_k \| = \| \alpha_\ell \|$ and $u
    \in BiGr_\perp^\circ(X_n)$, where $X_n$ is of type $A_2$, $B_2$, $C_2$,
    $B_3$, $C_3$, or $F_4$. From Tables (\ref{table, small rank ABCD}) and
    (\ref{table, general cases ABCD}), we see $\chi (\llbracket u \rrbracket) =
    \chi(\llbracket u \rrbracket^*)$ whenever $X_n$ is of type $A_2$, $B_2$,
    $C_2$, $B_3$, or $D_3$.  Lastly, we have $\chi(\llbracket u \rrbracket ) =
    \chi(\llbracket u \rrbracket^*)$ also when $X_n$ is of type $F_4$ (for
    instance, by checking this for $u$ equal to $\kappa_9$, $\kappa_{16}$,
    $f(\kappa_9)$, or $f(\kappa_{16})$).

\end{proof}

To compute the nilpotency index ${\mathcal N}(\llbracket w \rrbracket)$
associated to each $w \in BiGr_\perp^\circ(F_4)$, we first consider
the reduced expression for the longest element $w_0(F_4) \in W(F_4)$,
\begin{equation}
    \label{definition, reduced, F4}
    w_0(F_4) = (s_1 s_2 s_3 s_4) (s_1 s_2 s_3 s_4) (s_1 s_2 s_3 s_4) (s_1 s_2
    s_3 s_4) (s_1 s_2 s_3 s_4) (s_1 s_2 s_3 s_4).
\end{equation}

The repeating pattern in this reduced expression induces some symmetries in the
corresponding quantum Schubert cell algebra ${\mathcal U}_q^+[w_0(F_4)]$.

\begin{proposition}

    \label{symmetry, F4}

    Let $\mathfrak{g}$ be the complex simple Lie algebra of type $F_4$, and let
    $x_1,\dots, x_{24} \in {\mathcal U}_q^+[w_0(F_4)]$ denote the Lusztig root
    vectors associated to the reduced expression given in (\ref{definition,
    reduced, F4}). Define $c := s_1 s_2 s_3 s_4 \in W(F_4)$.  Then the Lusztig
    symmetry
    \[
        T_c: {\mathcal U}_q(\mathfrak{g}) \to {\mathcal U}_q(\mathfrak{g})
    \]
    is an algebra automorphism of ${\mathcal U}_q(\mathfrak{g})$ that sends
    $x_k$ to $x_{k + 4}$ for all $k \in [1, 20]$.

\end{proposition}

\begin{proof}

    It was shown in \cite[Proposition 37.1.2]{L} that the Lusztig symmetries
    $T_w$ ($w \in W$) are algebra automorphisms of ${\mathcal
    U}_q(\mathfrak{g})$.  Hence, it remains to prove that $T_c$ sends $x_k$ to
    $x_{k + 4}$ for all $k \in [1, 20]$.

    Recall first the reduced expression for $w_0(F_4)$ given in
    (\ref{definition, reduced, F4}).  For short, the product of the first $k$
    simple reflections (where $k \in [1, 24]$) in this reduced expression, read
    from left to right, will be denoted by $w_0[k]$.  Thus, by definition,
    \[
        x_k := T_{w_0[k - 1]} (E_{i_k}), \hspace{5mm} (k \leq \ell(w_0(F_4)) =
        24),
    \]
    where $s_{i_k}$ is the $k$-th simple reflection in the reduced expression.
    For $j,k \in [1, 24]$ with $j \leq k$, define the Weyl group element
    \[
        w_0[j, k] := s_{i_j} \cdots s_{i_k},
    \]
    whereas if $j > k$, $w_0[j, k]$ will be the identity element.  Since $i_k =
    i_{k + 4}$ for all $k \in [1, 20]$, then for all $k \in [1, 20]$, we have
    \[
        T_c (x_k)
        =
        T_c \circ T_{w_0[k - 1]} (E_{i_k})
        =
        T_c \circ T_{w_0[5, k + 3]} (E_{i_{k + 4}})
        =
        T_{w_0[k + 3]} (E_{i_{k + 4}})
        = x_{k + 4}.
    \]

\end{proof}

Following identical steps as in the proof of Proposition (\ref{symmetry, F4}),
we obtain a more general result.

\begin{proposition}

    \label{symmetry, general}

    Fix a natural number $N \in \mathbb{N}$.  Suppose $\mathfrak{g}$ is a
    complex simple Lie algebra and $w \in W$ is an element of the Weyl group of
    $\mathfrak{g}$ such that $\ell(w^N) = N\cdot \ell(w)$. Let $w = s_{i_1}
    \cdots s_{i_t}$ be a reduced expression, and let $x_1,\cdots, x_{t\cdot N}
    \in {\mathcal U}_q^+[w^N]$ be the Lusztig root vectors associated to the
    reduced expression
    \[
        w^N = (s_{i_1} \cdots s_{i_t}) \cdots (s_{i_1} \cdots s_{i_t}).
    \]
    Then the Lusztig symmetry $T_w$ is an algebra automorphism of ${\mathcal
    U}_q(\mathfrak{g})$ that sends $x_k$ to $x_{k + t}$ for all $k \in [1,
    t\cdot N - t]$.

\end{proposition}

In view of Proposition (\ref{symmetry, F4}), we can determine how a
$q$-commutator $[x_j, x_k]$, for $4 < j < k$, is expressed as a linear
combination of ordered monomials in the variables $x_i$ by knowing first how
$[x_{j - 4}, x_{k - 4}]$ is written as such. For instance, suppose
\[
    [x_{j - 4}, x_{k - 4}] = \sum c_{\mathbf{i}} x^{\mathbf{i}},
\]
where $c_{\mathbf{i}} \in \mathbb{K}$ and $x^{\mathbf{i}}$ is an ordered
monomial in the variables $x_1, \dots, x_{24}$. In fact, the
Levendorskii-Soibelmann straightening formulas (\ref{L-S straightening}) imply
that if the coefficient $c_{\mathbf{i}}$ is nonzero, then $x^{\mathbf{i}}$ must
be an ordered monomial in the variables $x_{j - 3},\dots, x_{k - 5}$.  We apply
the map $T_c$ to both sides of this equation to write $[x_j, x_k]$ as a linear
combination of ordered monomials.

This means a presentation of the algebra ${\mathcal U}_q^+[w_0(F_4)]$ can be
obtained by determining how the $q$-commutators $[x_j, x_k]$ are written as
linear combinations of ordered monomials for all $j, k \in [1, 24]$ with $j <
k$ and $j \leq 4$. There are $86$ such $q$-commutation relations of this type.
However, $26$ of these $86$ $q$-commutators are trivial, i.e. $[x_j, x_k] = 0$,
as a consequence of Proposition (\ref{LS corollary}). Thus, there remains $60$
$q$-commutation relations to compute.  However, $6$ of these remaining $60$
relations are given in Lemma (\ref{F4, split}) below.

\begin{lemma}

    \label{F4, split}

    Let $x_1,\dots, x_{24} \in {\mathcal U}_q^+[w_0(F_4)]$ be the Lusztig root
    vectors corresponding to the reduced expression (\ref{definition, reduced,
    F4}). Then
    \[
        \begin{split}
            &x_2 = [x_1, x_{5}],
            \hspace{5mm}
            x_3 = [x_2, x_{20}],
            \hspace{5mm}
            x_4 = [x_3, x_{24}],
            \\
            &[2]_q x_9 = [x_3, x_{20}],
            \hspace{3mm}
            x_{12} = [x_9, x_{24}],
            \hspace{3mm}
            x_{16} = [x_8, x_{24}].
        \end{split}
    \]

\end{lemma}

\begin{proof}

    From Theorem (\ref{Jantzen, 8.20}), $x_1 = E_1$, $x_5 = E_2$, $x_{20} =
    E_3$, and $x_{24} = E_4$.  The proof of Proposition (\ref{as nested E})
    describes an algorithm to write each Lusztig root vector $x_i$ as nested
    $q$-commutators, up to a scalar multiple. From this, we get $x_2 =
    \mathbf{E}_{1,2}$, $x_3 = \mathbf{E}_{1,2,3}$, $x_4 =
    \mathbf{E}_{1,2,3,4}$, $x_8 = \mathbf{E}_{2,3}$, $x_9 =  \frac{1}{[2]_q}
    \mathbf{E}_{1,2,3,3}$, $x_{12} = \frac{1}{[2]_q} \mathbf{E}_{1,2,3,3,4}$,
    and $x_{16} = \mathbf{E}_{2,3,4}$. The proof follows directly from these
    identities.

\end{proof}

As it turns out, the $q$-commutator identities that are known to be of the form
$[x_j, x_k] = 0$ (with $j < k$) by Proposition (\ref{LS corollary}), together
with the $6$ identities of Lemma (\ref{F4, split}), provide enough initial data
to apply the algorithms {\tt R(i,j,r,s)} to find the $54$ remaining
$q$-commutator relations for the algebra ${\mathcal U}_q^+[w_0(F_4)]$.
Specifically, do the following sequence of {\small\tt R(i,j,k,l)}'s:

\vspace{3pt}

\begin{adjustwidth}{15pt}{0pt}

    \begin{flushleft}

        {\small\tt R(3,12,9,24)}, {\small\tt R(5,12,9,24)}, {\small\tt
        R(6,12,9,24)}, {\small\tt R(7,12,9,24)}, {\small\tt R(8,12,9,24)},
        {\small\tt R(2,9,3,20)}, {\small\tt R(2,16,8,24)}, {\small\tt
        R(3,16,8,24)}, {\small\tt R(4,16,8,24)}, {\small\tt R(6,16,8,24)},
        {\small\tt R(1,6,5,9)}, {\small\tt R(2,6,5,9)}, {\small\tt
        R(1,7,6,24)}, {\small\tt R(2,7,6,24)}, {\small\tt R(3,7,6,24)},
        {\small\tt R(1,13,7,24)}, {\small\tt R(2,13,7,24)}, {\small\tt
        R(3,13,7,24)}, {\small\tt R(1,10,9,13)}, {\small\tt R(2,10,9,13)},
        {\small\tt R(3,10,9,13)}, {\small\tt R(1,11,9,16)}, {\small\tt
        R(2,11,9,16)}, {\small\tt R(3,11,9,16)}, {\small\tt R(4,11,9,16)},
        {\small\tt R(1,15,13,20)}, {\small\tt R(2,15,13,20)}, {\small\tt
        R(3,15,13,20)}, {\small\tt R(4,15,13,20)}, {\small\tt R(1,14,11,16)},
        {\small\tt R(2,14,11,16)}, {\small\tt R(3,14,11,16)}, {\small\tt
        R(4,14,11,16)}, {\small\tt R(1,18,15,20)}, {\small\tt R(2,18,15,20)},
        {\small\tt R(3,18,15,20)}, {\small\tt R(4,18,15,20)}, {\small\tt
        R(1,19,16,20)}, {\small\tt R(2,19,16,20)}, {\small\tt R(3,19,16,20)},
        {\small\tt R(4,19,16,20)}, {\small\tt R(2,23,20,24)}, {\small\tt
        R(3,23,20,24)}, {\small\tt R(4,23,20,24)}, {\small\tt R(1,17,8,20)},
        {\small\tt R(2,17,8,20)}, {\small\tt R(3,17,8,20)}, {\small\tt
        R(4,17,8,20)}, {\small\tt R(2,21,12,24)}, {\small\tt R(3,21,12,24)},
        {\small\tt R(1,22,16,23)}, {\small\tt R(2,22,16,23)}, {\small\tt
        R(3,22,16,23)}, {\small\tt R(4,22,16,23)}.

    \end{flushleft}

\end{adjustwidth}

\begin{example}

    \normalfont

    To write $[x_{12}, x_{23}]$ as a linear combination of ordered monomials,
    we first determine how to write $[x_4, x_{15}]$ as such, then apply the map
    $T_c^2$. Recall, from Proposition (\ref{symmetry, F4}), $c$ is the Coxeter
    element $s_1 s_2 s_3 s_4$, and the Lusztig symmetry $T_c$ sends $x_k$ to
    $x_{k + 4}$ for all $k \in [1, 20]$.  In the sequence of {\tt R(i,j,r,s)}'s
    above, we locate {\tt R(4,15,13,20)}. This means we should begin by writing
    $x_{15}$ as a scalar multiple of $[x_{13}, x_{20}]$. Since {\tt
    R(5,12,9,24)} comes before {\tt R(4,15,23,20)}, we assume it is already
    known how to write $[x_5, x_{12}]$ as a linear combination of ordered
    monomials. As it turns out, $[x_5, x_{12}] = x_7$, and thus, by applying
    $(T_c)^2$ to this relation, we obtain $[x_{13}, x_{20}] = x_{15}$. Hence,
    by the $q$-Jacobi identity, we have
    \[
        [x_4, x_{15}] = [x_4, [x_{13}, x_{20}]] =  [[x_4, x_{13}], x_{20}] -
        q^2[[x_4, x_{20}], x_{13}] + \widehat{q}\, [2]_q [x_4, x_{20}] x_{13}.
    \]
    Proposition (\ref{LS corollary}) implies $[x_4, x_{13}] = 0$, and by
    applying $T_c^{-1}$ to $[x_8, x_{24}] = x_{16}$ from Lemma (\ref{F4,
    split}), we obtain the $q$-commutator relation $[x_4, x_{20}] = x_{12}$.
    Therefore,
    \[
        [x_4, x_{15}] = - q^2[x_{12}, x_{13}] + \widehat{q}\, [2]_q x_{12}
        x_{13}.
    \]
    Proposition (\ref{LS corollary}) also implies $[x_{12}, x_{13}] = 0$.
    Hence, $[x_4, x_{15}] = \widehat{q}\, [2]_q x_{12} x_{13}$.  Finally, by
    applying $T_c^2$ to both sides of this equation, we obtain $[x_{12},
    x_{23}] = \widehat{q}\, [2]_q x_{20} x_{21}$.

\end{example}

\begin{example}

    \normalfont

    With the defining relations of the algebra ${\mathcal U}_q^+[w_0]$ at hand,
    we compute nilpotency indices. For example, to show ${\mathcal
    N}(\llbracket \kappa_4 \rrbracket) = 3$, observe first that ${\mathcal
    D}_L(\kappa_4) = \left\{2\right\}$ and ${\mathcal D}_R(\kappa_4) = \left\{
    3 \right\}$. Hence, by the definition of nilpotency index, we must show
    $[E_2, [E_2, [E_2, T_{\kappa_4 \cdot s_3} (E_3)]]] = 0$ and $[E_2, [E_2,
    T_{\kappa_4 \cdot s_3} (E_3)]] \neq 0$.

    Use the algorithm in the proof of Proposition (\ref{as nested E}) to write
    $T_{\kappa_4 \cdot s_3} (E_3)$ as nested $q$-commutators of Chevalley
    generators to get $[2]_qT_{\kappa_4 \cdot s_3} (E_3) =
    \mathbf{E}_{2,3,3,4,1}$.  Next, convert each Chevalley generator to a
    Lusztig root vector. Using the Lusztig root vectors $x_1, \dots, x_{24}$
    corresponding to the reduced expression (\ref{definition, reduced, F4}),
    Theorem (\ref{Jantzen, 8.20}) implies $x_1 = E_1$, $x_5 = E_2$, $x_{20} =
    E_3$, and $x_{24} = E_4$. Hence
    \[
        T_{\kappa_4 \cdot s_3} (E_3) = \frac{1}{[2]_q} [[[[x_5, x_{20}],
        x_{20}], x_{24}], x_1].
    \]
    Apply the $q$-commutation relations among the $x_i$'s to write the
    expression above as a linear combination of ordered monomials. We obtain
    \[
        T_{\kappa_4 \cdot s_3} (E_3) = \widehat{q}\, [2]_q x_1 x_{19} - q^2
        x_{12}.
    \]
    Using the $q$-commutation relations again gives us
    \[
        [x_5, [x_5, T_{\kappa_4 \cdot s_3} (E_3)]] = \widehat{q}^{\,2} [4]_q
        \left( \widehat{q}\, [2]_q x_1 x_5 x_8 x_{16} - q^2  x_2 x_8 x_{16}
        \right)
    \]
    and $[x_5, [x_5, [x_5, T_{\kappa_4 \cdot s_3} (E_3)]]] = 0$. Hence
    ${\mathcal N}(\llbracket \kappa_4 \rrbracket) = 3$.

\end{example}

Similarly, we calculate ${\mathcal N}(\llbracket w \rrbracket)$ for each $w \in
BiGr_\perp^\circ(F_4)$. The following theorem is the main result of this
section.

\begin{theorem}

    The nilpotency index ${\mathcal N}(\llbracket w \rrbracket)$ for each $w
    \in BiGr_{\perp}^{\circ}(F_4)$ is given in Table (\ref{table, F4}).

\end{theorem}

\begin{center}
    \begin{table}[h!]
        \caption{ The Lie algebra of type $F_4$ (data)}
        \vspace{-10pt}
\begin{tabular}{lcc|lcc}
    \toprule
    $w \in BiGr_{\perp}^{\circ}(F_4)$ & $\chi( \llbracket w
    \rrbracket)$
    & ${\mathcal N}(\llbracket w \rrbracket)$
    &$w \in BiGr_{\perp}^{\circ}(F_4)$ & $\chi( \llbracket w
    \rrbracket)$
    & ${\mathcal N}(\llbracket w \rrbracket)$
    \\[0.5mm]
    \hline
    $\kappa_{1}$ & (4, 2, 2) & 1
    &
    $f(\kappa_{10})$,  $f(\kappa_{13})$ & (2, 4, 0) & 3
    \\
    $f(\kappa_{1})$ & (2, 4, 2) & 1
    &
    $\kappa_2$,  $\kappa_3$ & (4, 2, 0) & 3
    \\
    $\kappa_{10}$, $\kappa_{13}$ & (4, 2, 0) & 2
    &
    $\kappa_4$ & (4, 2, -2) & 3
    \\
    $\kappa_8$,  $\kappa_9$, $\kappa_{16}$ & (4, 4, 2) & 2
    &
    $\kappa_{11}$, $\kappa_{14}$ & (4, 4, -2) & 3
    \\
    $f(\kappa_8)$, $f(\kappa_9)$, $f(\kappa_{16})$ & (2, 2, 1) & 2
    &
    $f(\kappa_{11})$, $f(\kappa_{14})$ & (2, 2, -1) & 3
    \\
    $\kappa_{12}$ & (4, 2, 2) & 2
    &
    $\kappa_7$ & (4, 4, -2) & 4
    \\
    $f(\kappa_{12})$ & (2, 4, 2) & 3
    &
    $f(\kappa_7)$ & (2, 2, -1) & 4
    \\
    $\kappa_5$ & (4, 4, 2) & 3
    &
    $\kappa_6$ & (4, 4, 0) & 4
    \\
    $f(\kappa_5)$ & (2, 2, 1) & 3
    &
    $f(\kappa_6)$ & (2, 2, 0) & 4
    \\
    $\kappa_{15}$, $\kappa_{17}$ & (4, 4, 0) & 3
    &
    $f(\kappa_2)$, $f(\kappa_3)$ & (2, 4, 0) & 5
    \\
    $f(\kappa_{15})$,  $f(\kappa_{17})$ & (2, 2, 0) & 3
    &
    $f(\kappa_4)$ & (2, 4, -2) & 5
    \\
    \bottomrule
\end{tabular}
\label{table, F4}
\end{table}
\end{center}

\vspace{10pt}

\subsection{Types \texorpdfstring{$E_6$, $E_7$, and $E_8$}{E6, E7, and E8}}

\label{section, E678}

The same techniques used to compute nilpotency indices in the type $F_4$ case
can also be applied to the $E_6$, $E_7$, and $E_8$ Lie types.

In Appendix (\ref{appendix, elements E_n}), we define specific Weyl group
elements
\begin{align}
    \label{E, BiGr}
    & \nu_1,\dots, \nu_{15} \in W(E_6),
    && \zeta_1,\dots, \zeta_{76} \in W(E_7),
    && \eta_1,\dots, \eta_{962} \in W(E_8).
\end{align}
belonging to the Weyl groups of types $E_6$, $E_7$, and $E_8$, respectively.
Each of these elements is bigrassmannian, has full support, and satisfies the
orthogonality condition of Proposition (\ref{proposition, orthogonality
condition}). Moreover, these elements, together with their inverses, give us
all such elements satisfying these conditions. In other words,
\begin{align*}
    &BiGr_{\perp}^\circ (E_6) = \left\{ \nu_k, \nu_k^{-1} : 1 \leq k
    \leq 15 \right\},
    \\
    &BiGr_{\perp}^\circ (E_7) = \left\{ \zeta_k, \zeta_k^{-1} : 1 \leq
    k \leq 76 \right\},
    \\
    &BiGr_{\perp}^\circ (E_8) = \left\{ \eta_k, \eta_k^{-1} : 1 \leq k
    \leq 962 \right\}.
\end{align*}
For organizational purposes, these elements are numbered so that involutions
are first. Specifically,
\begin{align*}
    &\nu_k \in W(E_6)\text{ is an involution if and only if }k \leq 10,
    \\
    &\zeta_k \in W(E_7)\text{ is an involution if and only if }k \leq 39,
    \\
    &\eta_k \in W(E_8)\text{ is an involution if and only if }k \leq 222.
\end{align*}
Furthermore, to avoid listing duplicates, they are arranged in such a manner so
that none of the non-involutions are inverses of each other. That is to say,
for each pair of non-involutions, $\nu_j$ and $\nu_k$ (or $\zeta_j$ and
$\zeta_k$, or $\eta_j$ and $\eta_k$), $\nu_j \neq \nu_k^{-1}$ ($\zeta_j \neq
\zeta_k^{-1}$, $\eta_j \neq \eta_k^{-1}$).

Tables (\ref{table, E6}), (\ref{table, E7}), and (\ref{table, E8}) summarize
the main results regarding nilpotency indices for the Lie types $E_6$, $E_7$,
and $E_8$, respectively.

\subsubsection{Type \texorpdfstring{$E_6$}{E6}}

Consider the reduced expression
\begin{align}
    \label{E6, longest element, reduced expression}
    &w_0(E_6) = (s_2 s_4 s_3 s_5 s_1 s_6) \cdot (s_2 s_4 s_3 s_5 s_1 s_6)
    \cdots (s_2 s_4 s_3 s_5 s_1 s_6)
\end{align}
for the longest element of the Weyl group $W(E_6)$.

\begin{lemma}

    \label{lemma, E6, q-commutator identities}

    Let $x_1, \dots, x_{36} \in {\mathcal U}_q^+[w_0(E_6)]$ be the Lusztig root
    vectors corresponding to the reduced expression (\ref{E6, longest element,
    reduced expression}). Then

    \begin{multicols}{3}

        \begin{enumerate}

            \item $x_2 = [x_1, x_7]$,

            \item $x_{3} = [x_{2}, x_{30}]$,

            \item $x_{4} = [x_{2}, x_{29}]$,

            \item $x_{5} = [x_{3}, x_{36}]$,

            \item $x_{6} = [x_{4}, x_{35}]$,

            \item $x_{13} = [x_{3}, x_{29}]$,

            \item $x_{17} = [x_{6}, x_{30}]$,

            \item $x_{18} = [x_{5}, x_{29}]$.

        \end{enumerate}

    \end{multicols}

\end{lemma}

\begin{proof}

    Applying the algorithm described in the proof of Proposition (\ref{as
    nested E}), we can write each Lusztig root vector, up to a scalar multiple,
    as either a Chevellay generator, or as nested $q$-commutators of Chevellay
    generators. We get $x_1 = E_2$, $x_2 = \mathbf{E}_{2,4}$, $x_3 =
    \mathbf{E}_{2,4,3}$, $x_4 = \mathbf{E}_{2,4,5}$, $x_5 =
    \mathbf{E}_{2,4,3,1}$, $x_6 = \mathbf{E}_{2,4,5,6}$, $x_7 = E_4$, $x_{13} =
    \mathbf{E}_{2,4,3,5}$, $x_{17} = \mathbf{E}_{2,4,3,5,6}$, $x_{18} =
    \mathbf{E}_{2,4,3,1,5}$, $x_{29} = E_5$, $x_{30} = E_3$, $x_{35} = E_6$,
    $x_{36} = E_1$, and consequently the desired $q$-commutator identities
    follow.

\end{proof}

As in the type $F_4$ case, the repeating pattern in the reduced expression
(\ref{E6, longest element, reduced expression}) induces symmetries in the
presentation of the algebra ${\mathcal U}_q^+[w_0(E_6)]$. In particular,
Proposition (\ref{symmetry, general}) can be applied to give the following
result.

\begin{proposition}

    \label{symmetry, E6}

    Let $\mathfrak{g}$ be the complex simple Lie algebra of type $E_6$, and let
    $x_1,\dots, x_{36} \in {\mathcal U}_q^+[w_0(E_6)]$ be as in in Lemma
    (\ref{lemma, E6, q-commutator identities}).  Let $c\in W(E_6)$ be the
    Coxeter element $c = s_2 s_4 s_3 s_5 s_1 s_6$.  The Lusztig symmetry $T_c$
    is an algebra automorphism of ${\mathcal U}_q(\mathfrak{g})$ that sends
    $x_k$ to $x_{k + 6}$ for all $k \in [1, 30]$.

\end{proposition}

Proposition (\ref{symmetry, E6}) implies an explicit presentation of the
algebra ${\mathcal U}_q^+[w_0(E_6)]$ can be found by determining how to write
each of the $q$-commutators $[x_i, x_j]$, with $i < j$ and $i \leq 6$, as a
linear combination of ordered monomials. However, among these ${36 \choose 2} -
{30 \choose 2}$ ($=195$) $q$-commutators, $80$ are trivial (i.e. $[x_i, x_j] =
0$) as a consequence of Proposition (\ref{LS corollary}). As it turns out,
these trivial $q$-commutators, together with the $8$ relations in Lemma
(\ref{lemma, E6, q-commutator identities}), provide enough initial data to
apply the algorithms $\texttt{L(i,j,r,s)}$ and $\texttt{R(i,j,r,s)}$ to write
each of the remaining $195 - 80 - 8$ ($=107$) $q$-commutators $[x_i, x_j]$ as a
linear combination of ordered monomials. That is to say, there is a sequence of
\texttt{L(i,j,k,l)}'s and \texttt{R(i,j,k,l)}'s that, when applied in the
prescribed order, can be used to find all of the remaining $q$-commutator
relations among the Lusztig root vectors $x_1, \cdots, x_{36}$.  These
relations provide a presentation of the quantum Schubert cell algebra
${\mathcal U}_q^+[w_0(E_6)]$.  It is a straightforward process to apply the
algorithms \texttt{L(i,j,r,s)} and \texttt{R(i,j,r,s)}, yet there are $107$
remaining commutation relations to find. In light to this, we used a computer
to automate this process. To give an example, one such sequence (among several
possible sequences) of \texttt{L(i,j,r,s)}'s and \texttt{R(i,j,r,s)}'s begins
\begin{center}
    {\small\tt L(4,30,2,29)},
    {\small\tt L(13,36,3,29)},
    {\small\tt L(13,35,4,30)},
    {\small\tt L(8,29,7,13)},
    {\small\tt L(8,30,7,13)},\ldots
\end{center}

Once a presentation of the algebra ${\mathcal U}_q^+[w_0(E_6)]$ is found, we
can apply commutation relations to write any given element $u \in {\mathcal
U}_q^+[w_0(E_6)]$ explicitly as a linear combination of ordered monomials, and
thus determine whether or not $u$ equals $0$. This may, however, involve
lengthy computations, especially if the commutation relations are complicated.
As it turns out, the commutation relations in ${\mathcal U}_q^+[w_0(E_6)]$ are
relatively simple enough that one could apply these techniques to compute
nilpotency indices, yet it would be tedious to do this all by hand. We used a
computer to automate this process. The following example illustrates how
nilpotency indices can be calculated using these techniques.

\begin{example}

    \normalfont

    We will prove ${\mathcal N} (\llbracket \nu_9 \rrbracket) = 4$. First,
    since ${\mathcal D}_L(\nu_9) = {\mathcal D}_R(\nu_9) = \left\{ 4 \right\}$,
    then, by the definition of nilpotency index, we must show
    \[
        \left(\operatorname{ad}_q(E_4)\right)^4 \left( T_{\nu_9 \cdot s_4}
        (E_4) \right) = 0 \text{ and } \left(\operatorname{ad}_q(E_4)\right)^3
        \left( T_{\nu_9 \cdot s_4} (E_4) \right) \neq 0.
    \]
    Begin by writing $T_{\nu_9 \cdot s_4} (E_4)$ as nested $q$-commutators of
    Chevalley generators. Recall the proof of Proposition (\ref{as nested E})
    provides, in effect, a recursive algorithm. With this, we obtain
    \[
        T_{\nu_9 \cdot s_4} (E_4) = [ \mathbf{E}_{4,3,5,2},
        \mathbf{E}_{4,3,1,5,6}].
    \]
    Next, substitute each Chevalley generator $E_i$ in the expression on the
    right hand side with its corresponding Lusztig root vector. Here, the
    Lusztig root vectors $x_1,\cdots, x_{36}$ are defined from the reduced
    expression (\ref{E6, longest element, reduced expression}) of the longest
    element. We have $E_1 = x_{36}$, $E_2 = x_1$, $E_3 = x_{30}$, $E_4 = x_7$,
    $E_5 = x_{29}$, and $E_6 = x_{35}$. This gives us $T_{\nu_9 \cdot s_4}
    (E_4) = [[[[x_7, x_{30}], x_{29}], x_1], [[[[x_7, x_{30}], x_{36}],
    x_{29}], x_{35}]]$.  Using the $q$-commutation relations in the algebra
    ${\mathcal U}_q^+[w_0]$, write $T_{\nu_9 \cdot s_4} (E_4)$ as a linear
    combination of ordered monomials. A straightforward yet lengthy computation
    gives us $T_{\nu_9 \cdot s_4} (E_4) =  q^2 x_{26} - q \widehat{q} \left(
    x_{17} x_{27} + x_{18} x_{28} \right) + \widehat{q}^{\,2} x_1 x_{27}
    x_{28}$.  Similarly, applying the $q$-commutation relations gives us
    \[
        [x_7, [x_7, [x_7, T_{\nu_9 \cdot s_4} (E_4)]]] = \widehat{q}^{\,4}
        \left([3]_q! \right) \left( \widehat{q} x_{1} x_{7} x_{11} x_{12}
        x_{23} x_{24} - q x_{2} x_{11} x_{12} x_{23} x_{24} \right).
    \]
    This means ${\mathcal N}(\llbracket \nu_9 \rrbracket) > 3$. Finally,
    another computation gives us
    \[
        [x_7, [x_7, [x_7, [x_7, T_{\nu_9 \cdot s_4} (E_4)]]]] = 0.
    \]
    Hence, ${\mathcal N}(\llbracket \nu_9 \rrbracket) = 4$.

\end{example}

In a similar manner, we find the nilpotency index ${\mathcal N}(\llbracket w
\rrbracket )$ associated to each $w \in BiGr_\perp^\circ(E_6)$. As it turns
out, we have the following curious result.

\begin{theorem}

    \label{invariant under dual, E6}

    For all $x \in \Gamma (W(E_6))$, ${\mathcal N}(x) = {\mathcal N}(x^*)$ and
    $\chi(x) = \chi(x^*)$.

\end{theorem}

\begin{proof}

    Let $x \in \Gamma(W(E_6))$.  Theorem (\ref{reduce by L and R}) and
    Proposition (\ref{proposition, orthogonality condition}) together imply
    that $x$ is equivalent (under $\xleftrightarrow{\hspace{4mm}}$) to an
    element of the form $(w, i, j)$ such that either (1) $w = w_0(a, b)$ for
    some $a, b\in \mathbf{I}$, or (2) $w \in BiGr_\perp(E_6)$.  From
    Proposition (\ref{proposition, a}), $x^* \xleftrightarrow{\hspace{4mm}} (w,
    i, j)^*$.  Thus, without loss of generality, we may assume $x$ has the form
    $(w, i, j)$, where either $w \in BiGr_\perp(E_6)$ or $w = w_0(a, b)$ for
    some $a, b, \in \mathbf{I}$.  In the case when $x = (w_0(a, b), i, j)$, we
    have $x^* = (w_0(a, b), j, i)$, and it is a straightforward observation
    that ${\mathcal N}(x) = {\mathcal N}(x^*)$ and $\chi(x) = \chi(x^*)$.

    It remains to consider the case when $x = \llbracket w \rrbracket$ with $w
    \in BiGr_\perp(E_6)$.  Each $w \in BiGr_\perp(E_6)$ can be classified by
    its support (recall Section (\ref{section: rank reduction})).  With this,
    we may view $w$ as belonging to $BiGr_\perp^\circ(X_n)$, where $X_n$ is a
    Lie type with associated Dynkin diagram a subgraph of the $E_6$ Dynkin
    diagram.  That is to say, $X_n$ is a simply laced Lie type with $n \leq 6$.

    Tables (\ref{table, small rank ABCD}) and (\ref{table, general cases ABCD})
    give the nilpotency index ${\mathcal N}(\llbracket w \rrbracket)$ for each
    $w \in BiGr_\perp^\circ(X_n)$ for $X_n$ of type $ABCD$, and upon
    inspection, we observe that ${\mathcal N}(\llbracket w \rrbracket) =
    {\mathcal N}(\llbracket w \rrbracket^*)$ in all cases. In fact, $\llbracket
    w \rrbracket = \llbracket w \rrbracket^*$ whenever the underlying Lie type
    is $A_n$ ($n \neq 2$) or $D_n$. Finally, we obtain the desired result by
    observing ${\mathcal N}(\llbracket w \rrbracket) = {\mathcal
    N}(\llbracket w \rrbracket^*)$ also when $w \in BiGr_\perp^\circ(E_6)$.

    Next, to prove $\chi(x) = \chi(x^*)$, we may again assume $x$ has the form
    $\llbracket w \rrbracket$ for some $w \in BiGr_\perp^\circ(X_n)$, where
    $X_n$ is a simply laced Lie type with $n \leq 6$. This is due to $\chi$
    being invariant under the equivalence relation
    $\xleftrightarrow{\hspace{4mm}}$, together with invariance under
    identifying $w \in BiGr_\perp(E_6)$ as an element of
    $BiGr_\perp^\circ(X_n)$ (recall Section (\ref{section: rank reduction})).
    Tables (\ref{table, small rank ABCD}) and (\ref{table, general cases ABCD})
    give $\chi(\llbracket w \rrbracket)$ when $X_n$ is of type $ABCD$. We see
    $\chi (\llbracket w \rrbracket) = \chi(\llbracket w \rrbracket^*)$ whenever
    $X_n$ is of type $A_n$ or $D_n$.  Lastly, we have $\chi(\llbracket w
    \rrbracket ) = \chi(\llbracket w \rrbracket^*)$ also when $X_n$ is of type
    $E_6$ (for instance, by checking this for each $w \in
    BiGr_\perp^\circ(E_6)$).

\end{proof}

\begin{theorem}

    Table (\ref{table, E6}) gives the nilpotency index ${\mathcal N}(\llbracket
    \nu_k \rrbracket)$ for all $k \in [1, 15]$.

\end{theorem}

\begin{center}
    \begin{table}[h]
\caption{The Lie algebra of Type $E_6$ (data)}
\vspace{-10pt}
\begin{tabular}{lcc}
    \toprule
    $w \in BiGr_{\perp}^{\circ}(E_6)$ & $\chi( \llbracket w
    \rrbracket)$ & ${\mathcal N}(\llbracket w \rrbracket)$
    \\[0.5mm]
    \hline
    $\nu_{10}, \nu_{11}, \nu_{15}$ & $(2, 2, 1)$ & 2
    \\
    $\nu_1$ & $(2, 2, 1)$ & 3
    \\
    \midrule[0.01pt]
    $\nu_{12}$ & $(2, 2, 0)$ & 2
    \\
    $\nu_3, \nu_5, \nu_6, \nu_7, \nu_{13}, \nu_{14}$ & $(2, 2, 0)$ & 3
    \\
    $\nu_8$ & $(2, 2, 0)$ & 4
    \\
    \midrule[0.01pt]
    $\nu_2, \nu_4$ & $(2, 2, -1)$ & 3
    \\
    $\nu_9$ & $(2, 2, -1)$ & 4
    \\
    \bottomrule
\end{tabular}
\label{table, E6}
\end{table}
\end{center}

\noindent Observe that since $\llbracket \nu_k \rrbracket^* = \llbracket
\nu_k^{-1} \rrbracket$ for all $k \in [1, 15]$, then by Theorem (\ref{invariant
under dual, E6}), Table (\ref{table, E6}) also gives the nilpotency index
${\mathcal N}(\llbracket \nu_k^{-1} \rrbracket)$.  Hence, the table can be used
to look up ${\mathcal N}(\llbracket w \rrbracket)$ for every $w \in
BiGr_\perp^\circ(E_6)$.

\noindent
\begin{minipage}{\textwidth}
\begin{theorem}

    \label{equivalence classes, bound, E6}

    Let $\Gamma_6 := \left\{ \llbracket \nu
    \rrbracket \in \Gamma(W(E_6)) : \nu \in BiGr_\perp^\circ(E_6) \right\}$.

    \begin{enumerate}

        \item The $20$ elements of $\Gamma_6$ are contained in at most $15$
            distinct equivalence classes.

        \item \label{equivalence classes, bound, E6, b} At least $6$ of the
            $10$ elements in $\Gamma_6$ of the form $\llbracket w \rrbracket$
            with $w$ a non-involution satisfy $\llbracket w \rrbracket
            \xleftrightarrow{\hspace{4mm}} \llbracket w^{-1} \rrbracket$.

        \item \label{equivalence classes, bound, E6, c} At least $4$ of the
            $20$ elements of $\Gamma_6$, namely $\llbracket \nu_{11}
            \rrbracket$, $\llbracket \nu_{11}^{-1} \rrbracket$, $\llbracket
            \nu_{12} \rrbracket$, and $\llbracket \nu_{12}^{-1} \rrbracket$,
            are equivalent (under the equivalence relation
            $\xleftrightarrow{\hspace{4mm}}$) to an element of the form
            $\llbracket w \rrbracket$ with $w \in BiGr_\perp(E_6) \backslash
            BiGr_\perp^\circ(E_6)$.

    \end{enumerate}

\end{theorem}
\end{minipage}

\begin{proof}

    Define the following elements of $\Gamma(W(E(6))$,
    \[
        \begin{split}
            &X_1 = (s_3 s_4 s_5 s_6 s_5 s_4 s_1 s_3 s_2 s_4 s_5 s_3 s_4 s_1 s_3
            s_2 s_4 s_2, 3, 4),
            \\
            &X_2 = (s_6 s_5 s_4 s_1 s_3 s_2 s_4 s_5 s_6 s_1 s_3 s_2 s_4 s_5 s_3
            s_2 s_4 s_3 s_2, 6, 3),
            \\
            &X_3 = (s_4 s_3 s_2 s_4 s_5 s_6 s_1 s_3 s_4 s_5 s_3 s_4 s_1 s_3 s_2
            s_4 s_2, 4, 4),
            \\
            &X_4 = (s_3 s_4 s_1 s_3 s_2 s_4 s_5 s_6 s_4 s_5 s_3 s_4 s_1 s_3 s_2
            s_4 s_5 s_4 s_2, 3, 5).
        \end{split}
    \]
    We have $X_1 \xlongrightarrow{L} \llbracket \nu_{11} \rrbracket$ and $X_1
    \xlongrightarrow{R} \llbracket s_3 s_4 s_5 s_2 s_4 s_1 s_3 \rrbracket$.
    Thus, from Proposition (\ref{proposition, a}), $X_1^* \xlongrightarrow{R}
    \llbracket \nu_{11}^{-1} \rrbracket$ and $X_1^* \xlongrightarrow{L}
    \llbracket s_3 s_4 s_5 s_2 s_4 s_1 s_3 \rrbracket$. Hence, $\llbracket
    \nu_{11} \rrbracket \xleftrightarrow{\hspace{4mm}} \llbracket \nu_{11}^{-1}
    \rrbracket$. Similarly, $\llbracket \nu_{12} \rrbracket
    \xleftrightarrow{\hspace{4mm}} \llbracket \nu_{12}^{-1} \rrbracket$ from
    $X_2 \xlongrightarrow{L} \llbracket \nu_{12} \rrbracket$ and $X_2
    \xlongrightarrow{R} \llbracket s_6 s_5 s_4 s_3 s_2 s_4 s_5 s_6 \rrbracket$.
    We also have $X_3 \xlongrightarrow{L} \llbracket \nu_{15}^{-1} \rrbracket$
    and $X_3 \xlongrightarrow{R} \llbracket \nu_{15} \rrbracket$, and thus
    $\llbracket \nu_{15} \rrbracket \xleftrightarrow{\hspace{4mm}} \llbracket
    \nu_{15}^{-1} \rrbracket$.  From $X_4 \xlongrightarrow{L} \llbracket
    \nu_{14}^{-1} \rrbracket$ and $X_4 \xlongrightarrow{R} \llbracket \nu_{13}
    \rrbracket$, we get $\llbracket \nu_{14}^{-1} \rrbracket
    \xleftrightarrow{\hspace{4mm}} \llbracket \nu_{13} \rrbracket$. Thus, from
    Proposition (\ref{proposition, a}), $\llbracket \nu_{14} \rrbracket
    \xleftrightarrow{\hspace{4mm}} \llbracket \nu_{13}^{-1} \rrbracket$.

\end{proof}

We remark that the proof of Theorem (\ref{equivalence classes, bound, E6})
gives explicit reductions that one can use to determine specifically which
elements in $\Gamma_6$ satisfy the conditions described in parts
(\ref{equivalence classes, bound, E6, b}) and (\ref{equivalence classes, bound,
E6, c}).  Also, computer experiments suggest that the phrases \textit{at least}
and \textit{at most} in the statement of Theorem (\ref{equivalence classes,
bound, E6}) can be replaced by \textit{exactly}.

Recall, the significance of part (\ref{equivalence classes, bound, E6, c}) of
Theorem (\ref{equivalence classes, bound, E6}) is that nilpotency indices can
sometimes be found using results from a smaller rank setting.  Computer
experiments reveal that over 99.8\% of the elements in $\Gamma(W(E_6))$ are
equivalent (under $\xleftrightarrow{\hspace{4mm}}$) to an element of the form
$(w, i, j)$ with $w$ lacking full support.  Specifically, computer experiments
indicate that there are 24 equivalence classes in $\Gamma(W(E_6))$. The 20
elements in $\Gamma_6$ are contained in 15 equivalence classes. Exactly 13
equivalence classes fail to contain an element of the form $(w, i, j)$ with $w$
lacking full support, and the union of these 13 equivalence classes contains
precisely 760 elements.  As $\Gamma(W(E_6))$ has cardinality 453600 (see
Theorem (\ref{Gamma W, cardinality})), these 760 elements account for
approximately only $0.168$ percent of the total number of elements in
$\Gamma(W(E_6))$.

\subsubsection{Type \texorpdfstring{$E_7$}{E7}}

Nilpotency indices in the type $E_7$ case can be computed using the same
techniques as in the $F_4$ and $E_6$ cases. First, consider the following
reduced expression $w_0(E_7) \in W(E_7)$ for the longest element of the Weyl
group,
\begin{align}
    \label{E7, longest element, reduced expression}
    &w_0(E_7) = (s_1 s_2 s_3 s_4 s_5 s_6 s_7) \cdot (s_1 s_2 s_3 s_4 s_5 s_6
    s_7) \cdots (s_1 s_2 s_3 s_4 s_5 s_6 s_7).
\end{align}

We aim next to find a presentation of the algebra ${\mathcal U}_q^+[w_0(E_7)]$
associated to this particular reduced expression.

\begin{lemma}

    \label{lemma, E7, q-commutator identities}

    Let $y_1, \dots, y_{63} \in {\mathcal U}_q^+[w_0(E_7)]$ be the Lusztig root
    vectors corresponding to the reduced expression (\ref{E7, longest element,
    reduced expression}). Then

    \begin{multicols}{3}

        \begin{enumerate}

            \item $y_{3} = [y_{1}, y_{8}]$,
            \item $y_{4} = [y_{2}, y_{9}]$,
            \item $y_{5} = [y_{4}, y_{49}]$,
            \item $y_{6} = [y_{5}, y_{56}]$,
            \item $y_{7} = [y_{1}, y_{35}]$,
            \item $y_{7} = [y_{6}, y_{63}]$,
            \item $y_{9} = [y_{3}, y_{42}]$,
            \item $y_{11} = [y_{5}, y_{14}]$,
            \item $y_{15} = [y_{2}, y_{42}]$.

        \end{enumerate}

    \end{multicols}

\end{lemma}

\begin{proof}

    Using the algorithm in the proof of Proposition (\ref{as nested E}), we
    begin by writing each of the Lusztig root vectors $y_1$, $y_2$, $y_3$,
    $y_4$, $y_5$, $y_6$, $y_7$, $y_8$, $y_9$, $y_{15}$, $y_{35}$, $y_{42}$,
    $y_{49}$, $y_{56}$, and $y_{63}$ as either a Chevalley generator or as
    nested $q$-commutators of Chevalley generators (up to a scalar multiple).
    We obtain $y_1 = E_1$, $y_2 = E_2$, $y_3 = \mathbf{E}_{1,3}$, $y_4 = [E_2,
    \mathbf{E}_{1,3,4}]$, $y_5 = [[E_2, \mathbf{E}_{1,3,4}], E_5]$, $y_7 = [[[[E_2,
    \mathbf{E}_{1,3,4}], E_5], E_6], E_7]$, $y_6 = [[[E_2, \mathbf{E}_{1,3,4}],
    E_5], E_6]$, $y_8 = E_3$, $y_9 = \mathbf{E}_{1,3,4}$, $y_{15} =
    \mathbf{E}_{2,4}$, $y_{42} = E_4$, $y_{49} = E_5$, $y_{35} = [[[[E_2,
    \mathbf{E}_{3,4}], E_5], E_6], E_7]$, $y_{56} = E_6$, and $y_{63} = E_7$.

    With this, it is a simple observation that $y_3 = [y_1, y_8]$, $y_{4} =
    [y_{2}, y_{9}]$, $y_{5} = [y_{4}, y_{49}]$, $y_{6} = [y_{5}, y_{56}]$,
    $y_{7} = [y_{6}, y_{63}]$, $y_{9} = [y_{3}, y_{42}]$, and $y_{15} = [y_{2},
    y_{42}]$.  It remains to show $y_7 = [y_1, y_{35}]$ and $y_{11} = [y_5,
    y_{14}]$.

    We prove next that $y_7 = [y_1, y_{35}]$.  Since the Chevalley generator
    $E_2$ commutes with $E_1$ and $E_3$, $[E_2, \mathbf{E}_{1,3}] = 0$, and
    thus the $q$-Jacobi identity (\ref{q-Jacobi identity}) gives us $[E_2,
    \mathbf{E}_{1,3,4}] = -q[ \mathbf{E}_{2,4}, \mathbf{E}_{1,3}] +
    \widehat{q}\, \mathbf{E}_{2,4} \mathbf{E}_{1,3}$. However, by the
    definition of $q$-commutators, $-q[ \mathbf{E}_{2,4}, \mathbf{E}_{1,3}] +
    \widehat{q}\, \mathbf{E}_{2,4} \mathbf{E}_{1,3} = [\mathbf{E}_{1,3},
    \mathbf{E}_{2,4}]$.  Hence, we obtain the identity $[E_2,
    \mathbf{E}_{1,3,4}] = [\mathbf{E}_{1,3}, \mathbf{E}_{2,4}]$.  Since $E_1$
    commutes with $\mathbf{E}_{2,4}$, $q$-associativity gives us
    $[\mathbf{E}_{1,3}, \mathbf{E}_{2,4}] = [E_1, [E_3, \mathbf{E}_{2,4}]]$,
    and since $E_2$ and $E_3$ commute with each other, $[E_3, \mathbf{E}_{2,4}]
    = [E_2, \mathbf{E}_{3,4}]$. Thus, we obtain the identity $[E_2,
    \mathbf{E}_{1,3,4}] = [E_1, [E_2, \mathbf{E}_{3,4}]]$. Therefore,
    \[
        y_7 = [y_6, y_{63}] = [[[[E_2, \mathbf{E}_{1,3,4}], E_5], E_6], E_7] =
        [[[[E_1, [E_2, \mathbf{E}_{3,4}]], E_5], E_6], E_7].
    \]
    However, since $E_1$ commutes with $E_5$, $E_6$, and $E_7$, we can apply
    $q$-associativity to move the $q$-brackets to get $y_7 = [E_1, [[[[E_2,
    \mathbf{E}_{3,4}], E_5] , E_6], E_7]] = [y_1, y_{35}]$.

    Lastly, we show $y_{11} = [y_5, y_{14}]$.  For short, define the Weyl group
    elements $u := s_1s_2s_3s_4 \in W(E_7)$ and $v := s_5s_6s_7s_1s_2s_3 \in
    W(E_7)$. From the definition of Lusztig root vectors, $y_{11} =
    T_uT_v(E_4)$, $y_5 = T_u(E_5)$, and $y_{14} = T_uT_{v \cdot
    s_4s_5s_6}(E_7)$. Hence, it suffices to prove the identity $[E_5, T_{v
    \cdot s_4 s_5 s_6} (E_7)] = T_v(E_4)$. Invoking the algorithm in the proof
    of Proposition (\ref{as nested E}) gives us $T_v(E_4) = [E_2,
    [\mathbf{E}_{1,3}, \mathbf{E}_{5,4}]]$ and $T_{v \cdot s_4 s_5 s_6} (E_7) =
    [E_2, \mathbf{E}_{1,3,4}]$. Since the Chevalley generator $E_5$ commutes
    with $E_2$ and $\mathbf{E}_{1,3}$, we have $[E_5, [E_2,
    \mathbf{E}_{1,3,4}]] = [E_2, [E_5, \mathbf{E}_{1,3,4}]] = [E_2,
    [\mathbf{E}_{1,3}, \mathbf{E}_{5,4}]]$, and thus $[E_5, T_{v \cdot
    s_4s_5s_6}(E_7)] = T_v(E_4)$, as desired.

\end{proof}

As in the type $F_4$ and $E_6$ cases, the repeating pattern in the reduced
expression (\ref{E7, longest element, reduced expression}) induces symmetries
in the presentation of the algebra ${\mathcal U}_q^+[w_0(E_7)]$. In particular,
Proposition (\ref{symmetry, general}) can be applied to give the following
result.

\begin{proposition}

    \label{symmetry, E7}

    Let $\mathfrak{g}$ be the complex simple Lie algebra of type $E_7$, and let
    $y_1,\dots, y_{63} \in {\mathcal U}_q^+[w_0(E_7)]$ be as in Lemma
    (\ref{lemma, E7, q-commutator identities}).  Let $c\in W(E_7)$ be the
    Coxeter element $c = s_1 s_2 s_3 s_4 s_5 s_6 s_7$.  The Lusztig symmetry
    $T_c$ is an algebra automorphism of ${\mathcal U}_q(\mathfrak{g})$ that
    sends $y_k$ to $y_{k + 7}$ for all $k \in [1, 56]$.

\end{proposition}

Proposition (\ref{symmetry, E7}) implies an explicit presentation of the
algebra ${\mathcal U}_q^+[w_0(E_7)]$ can be found by determining how to write
each of the $q$-commutators $[y_i, y_j]$, with $i < j$ and $i \leq 7$, as a
linear combination of ordered monomials. However, among these ${63 \choose 2} -
{56 \choose 2}$ ($=413$) $q$-commutators, $145$ are trivial (i.e. $[y_i, y_j] =
0$) as a consequence of Proposition (\ref{LS corollary}). These trivial
$q$-commutators, together with the $9$ relations in Lemma (\ref{lemma, E7,
q-commutator identities}), provide enough initial data to apply the
$\texttt{L(i,j,r,s)}$ and $\texttt{R(i,j,r,s)}$ algorithms to obtain an
explicit presentation of ${\mathcal U}_q^+[w_0(E_7)]$. Since there remains $413
- 145 - 9$ ($=259$) commutation relations to find, we used a computer to
automate this process.  To give an example, one such sequence, among several
possible sequences, of \texttt{L(i,j,r,s)}'s and \texttt{R(i,j,r,s)}'s starts
with
\begin{center}
    {\small\tt L(3,35,1,8)},
    {\small\tt L(4,42,2,9)},
    {\small\tt L(6,14,5,56)},
    {\small\tt L(7,14,6,63)},
    {\small\tt R(3,15,2,42)},\ldots
\end{center}

After an explicit presentation of ${\mathcal U}_q^+[w_0(E_7)]$ is found,
nilpotency indices can be computed in the same manner as in the $F_4$ and $E_6$
cases. With slight modifications to the proof of Theorem (\ref{invariant under
dual, E6}), we obtain the following result.

\begin{theorem}

    \label{invariant under dual, E7}

    For all $x \in \Gamma (W(E_7))$, ${\mathcal N}(x) = {\mathcal N}(x^*)$ and
    $\chi(x) = \chi(x^*)$.

\end{theorem}

\noindent The following theorem is the main result regarding nilpotency indices
in type $E_7$.

\begin{theorem}

    Table (\ref{table, E7}) gives the nilpotency index ${\mathcal N}(\llbracket
    \zeta_k \rrbracket)$ for all $k \in [1, 76]$.

\end{theorem}

\noindent In view of Theorem (\ref{invariant under dual, E7}), Table
(\ref{table, E7}) can be used to look up the nilpotency index ${\mathcal
N}(\llbracket w \rrbracket)$ associated to each $w \in BiGr_\perp^\circ(E_7)$.

\begin{center}
    \begin{table}[H]
\caption{The Lie algebra of Type $E_7$ (data)}
\begin{tabular}{lcc}
    \toprule
    $w \in BiGr_{\perp}^{\circ}(E_7)$ & $\chi( \llbracket w
    \rrbracket)$
    & ${\mathcal N}(\llbracket w \rrbracket)$
    \\[0.5mm]
    \hline
    $\zeta_k$, ($k \in [1, 6]$) \hspace{1mm} & (2, 2, -1) & $3$
    \\
    $\zeta_k$, ($k \in [7, 10] \cup [40, 42]$) & (2, 2, -1) & $4$
    \\
    $\zeta_{11}$, & (2, 2, -1) & $5$
    \\
    \midrule[0.01pt]
    $\zeta_{12}, \zeta_{43}, \zeta_{44}$, & (2, 2, 0) & $2$
    \\
    $\zeta_k$, ($k \in [13, 24] \cup [45, 56]$) & (2, 2, 0) & $3$
    \\
    $\zeta_k$, ($k \in [25, 29] \cup [57, 62]$) & (2, 2, 0) & $4$
    \\
    $\zeta_{30}, \zeta_{31}$, & (2, 2, 0) & $5$
    \\
    \midrule[0.01pt]
    $\zeta_k$, ($k \in [32, 34] \cup [63, 71]$) & (2, 2, 1) & $2$
    \\
    $\zeta_k$, ($k \in [35, 38] \cup [72, 76]$) & (2, 2, 1) & $3$
    \\
    $\zeta_{39}$ & (2, 2, 1) & $4$
    \\
    \bottomrule
\end{tabular}
\label{table, E7}
\end{table}
\end{center}

\begin{theorem}

    \label{equivalence classes, bound, E7}

    Let $\Gamma_7 := \left\{ \llbracket \zeta
    \rrbracket \in \Gamma(W(E_7)) : \zeta \in BiGr_\perp^\circ(E_7) \right\}$.

    \begin{enumerate}

        \item The $113$ elements of $\Gamma_7$ are contained in at most $77$
            distinct equivalence classes.

        \item \label{equivalence classes, bound, E7, b} At least $44$ of the
            $74$ elements in $\Gamma_7$ of the form $\llbracket w \rrbracket$
            with $w$ a non-involution satisfy $\llbracket w \rrbracket
            \xleftrightarrow{\hspace{4mm}} \llbracket w^{-1} \rrbracket$.

        \item \label{equivalence classes, bound, E7, c} At least $25$ of the
            $113$ elements of $\Gamma_7$ are equivalent to an element of the
            form $\llbracket w \rrbracket$ with $w \in BiGr_\perp(E_7)
            \backslash BiGr_\perp^\circ(E_7)$.

    \end{enumerate}

\end{theorem}

\begin{proof}

    Let $\Theta_1,\dots, \Theta_{23}$, $\Lambda_1,\dots, \Lambda_{12}$,
    $\Psi_1,\dots, \Psi_{23} \in \Gamma(W(E_7))$ be as defined in Appendix
    (\ref{appendix B}).  Each of $\Theta_1, \dots, \Theta_{23}$, as well as
    $\Lambda_{12}$ is self-dual (recall, $x \in \Gamma(W)$ is
    \textit{self-dual} if $x = x^*$), and thus Corollary (\ref{corollary, a})
    can be applied. Each of the elements $\Lambda_1, \dots, \Lambda_{12}$ is of
    the form $\llbracket w \rrbracket$ with $w \in BiGr_\perp(E_7) \backslash
    BiGr_\perp^{\circ}(E_7)$.

    We have $\Theta_k \xlongrightarrow{\hspace{2mm} L \hspace{2mm}} \Lambda_k
    \text{ for } k \in [1, 11]$. Furthermore,

    \vspace{-15pt}
    \begin{center}
    \begin{tabular}{p{113pt} p{113pt} p{113pt}}
        $\Theta_{1} \xlongrightarrow{\hspace{1.5mm} L \hspace{1.5mm}} \Psi_{1}
        \xlongrightarrow{\hspace{1.5mm} R \hspace{1.5mm}} \llbracket \zeta_{33}
        \rrbracket$,
        &
        $\Theta_{1} \xlongrightarrow{\hspace{1.5mm} L \hspace{1.5mm}} \Psi_{2}
        \xlongrightarrow{\hspace{1.5mm} R \hspace{1.5mm}} \llbracket \zeta_{70}
        \rrbracket$,
        &
        $\Theta_{2} \xlongrightarrow{\hspace{1.5mm} L \hspace{1.5mm}} \Psi_{3}
        \xlongrightarrow{\hspace{1.5mm} R \hspace{1.5mm}} \llbracket \zeta_{12}
        \rrbracket$,
        \\
        $\Theta_{2} \xlongrightarrow{\hspace{1.5mm} L \hspace{1.5mm}} \Psi_{4}
        \xlongrightarrow{\hspace{1.5mm} R \hspace{1.5mm}} \llbracket \zeta_{44}
        \rrbracket$,
        &
        $\Theta_{3} \xlongrightarrow{\hspace{1.5mm} L \hspace{1.5mm}} \Psi_{5}
        \xlongrightarrow{\hspace{1.5mm} R \hspace{1.5mm}} \llbracket \zeta_{22}
        \rrbracket$,
        &
        $\Theta_{4} \xlongrightarrow{\hspace{1.5mm} L \hspace{1.5mm}} \Psi_{6}
        \xlongrightarrow{\hspace{1.5mm} R \hspace{1.5mm}} \llbracket \zeta_{5}
        \rrbracket$,
        \\
        $\Theta_{5} \xlongrightarrow{\hspace{1.5mm} L \hspace{1.5mm}} \Psi_{7}
        \xlongrightarrow{\hspace{1.5mm} R \hspace{1.5mm}} \llbracket \zeta_{6}
        \rrbracket$,
        &
        $\Theta_{6} \xlongrightarrow{\hspace{1.5mm} R \hspace{1.5mm}} \Psi_{8}
        \xlongrightarrow{\hspace{1.5mm} L \hspace{1.5mm}} \llbracket
        \zeta_{71} \rrbracket$,
        &
        $\Theta_{7} \xlongrightarrow{\hspace{1.5mm} R \hspace{1.5mm}} \Psi_{9}
        \xlongrightarrow{\hspace{1.5mm} L \hspace{1.5mm}} \llbracket
        \zeta_{72} \rrbracket$,
        \\
        $\Theta_{8} \xlongrightarrow{\hspace{1.5mm} R \hspace{1.5mm}} \Psi_{10}
        \hspace{-1pt}
        \xlongrightarrow{\hspace{1mm} L \hspace{1mm}} \llbracket
        \zeta_{69} \rrbracket$,
        &
        $\Theta_{9} \xlongrightarrow{\hspace{1.5mm} L \hspace{1.5mm}} \Psi_{11}
        \hspace{-1pt}
        \xlongrightarrow{\hspace{1mm} R \hspace{1mm}} \llbracket \zeta_{68}
        \rrbracket$,
        &
        $\Theta_{9} \xlongrightarrow{\hspace{1.5mm} R \hspace{1.5mm}} \Psi_{12}
        \hspace{-1pt}
        \xlongrightarrow{\hspace{1mm} L \hspace{1mm}} \llbracket
        \zeta_{67} \rrbracket$,
        \\
        $\Theta_{10}
        \hspace{-1pt}
        \xlongrightarrow{\hspace{1mm} R \hspace{1mm}} \Psi_{13}
        \hspace{-1pt}
        \xlongrightarrow{\hspace{1mm} L \hspace{1mm}} \llbracket
        \zeta_{45} \rrbracket$,
        &
        $\Theta_{11}
        \hspace{-1pt}
        \xlongrightarrow{\hspace{1mm} R \hspace{1mm}} \Psi_{14}
        \hspace{-1pt}
        \xlongrightarrow{\hspace{1mm} L \hspace{1mm}} \llbracket
        \zeta_{43} \rrbracket$,
        &
    \end{tabular}
    \end{center}

    \vspace{-5pt} \noindent and $\Psi_{15} \xlongrightarrow{\hspace{2mm} R
    \hspace{2mm}} \Lambda_{12}$ and $\Psi_{15} \xlongrightarrow{\hspace{2mm} L
    \hspace{2mm}} \llbracket \zeta_{46} \rrbracket$. We also have

    \vspace{-5pt}
    \begin{center}
    \begin{tabular}{p{78pt} p{78pt} p{78pt} p{78pt}}
        $\Theta_{12} \xlongrightarrow{\hspace{2mm} L \hspace{2mm}}
        \llbracket \zeta_{42} \rrbracket$,
        &
        $\Theta_{13} \xlongrightarrow{\hspace{2mm} L \hspace{2mm}}
        \llbracket \zeta_{47} \rrbracket$,
        &
        $\Theta_{14} \xlongrightarrow{\hspace{2mm} R \hspace{2mm}}
        \llbracket \zeta_{48} \rrbracket$,
        &
        $\Theta_{15} \xlongrightarrow{\hspace{2mm} R \hspace{2mm}}
        \llbracket \zeta_{49} \rrbracket$.
        \\
        $\Theta_{16} \xlongrightarrow{\hspace{2mm} L \hspace{2mm}}
        \llbracket \zeta_{50} \rrbracket$,
        &
        $\Theta_{17} \xlongrightarrow{\hspace{2mm} L \hspace{2mm}}
        \llbracket \zeta_{51} \rrbracket$,
        &
        $\Theta_{18} \xlongrightarrow{\hspace{2mm} L \hspace{2mm}}
        \llbracket \zeta_{52} \rrbracket$,
        &
        $\Theta_{19} \xlongrightarrow{\hspace{2mm} R \hspace{2mm}}
        \llbracket \zeta_{63} \rrbracket$,
        \\
        $\Theta_{20} \xlongrightarrow{\hspace{2mm} R \hspace{2mm}}
        \llbracket \zeta_{64} \rrbracket$,
        &
        $\Theta_{21} \xlongrightarrow{\hspace{2mm} R \hspace{2mm}}
        \llbracket \zeta_{65} \rrbracket$,
        &
        $\Theta_{22} \xlongrightarrow{\hspace{2mm} L \hspace{2mm}}
        \llbracket \zeta_{66} \rrbracket$,
        &
        $\Theta_{23} \xlongrightarrow{\hspace{2mm} L \hspace{2mm}}
        \llbracket \zeta_{73} \rrbracket$,
    \end{tabular}
    \end{center}

    \noindent and

    \vspace{-5pt}
    \begin{center}
    \begin{tabular}{p{78pt} p{78pt} p{78pt} p{78pt}}
        $\Psi_{16} \xlongrightarrow{\hspace{2mm} R \hspace{2mm}}
        \llbracket \zeta_{21} \rrbracket$,
        &
        $\Psi_{16} \xlongrightarrow{\hspace{2mm} L \hspace{2mm}}
        \llbracket \zeta_{50} \rrbracket$,
        &
        $\Psi_{17} \xlongrightarrow{\hspace{2mm} R \hspace{2mm}}
        \llbracket \zeta_{23} \rrbracket$,
        &
        $\Psi_{17} \xlongrightarrow{\hspace{2mm} L \hspace{2mm}}
        \llbracket \zeta_{47} \rrbracket$,
        \\
        $\Psi_{18} \xlongrightarrow{\hspace{2mm} L \hspace{2mm}}
        \llbracket \zeta_{34} \rrbracket$,
        &
        $\Psi_{18} \xlongrightarrow{\hspace{2mm} R \hspace{2mm}}
        \llbracket \zeta_{63} \rrbracket$,
        &
        $\Psi_{19} \xlongrightarrow{\hspace{2mm} L \hspace{2mm}}
        \llbracket \zeta_{51} \rrbracket$,
        &
        $\Psi_{19} \xlongrightarrow{\hspace{2mm} R \hspace{2mm}}
        \llbracket \zeta_{52} \rrbracket$,
        \\
        $\Psi_{20} \xlongrightarrow{\hspace{-0.2mm} (3, L) \hspace{-0.2mm}}
        \llbracket \zeta_{53} \rrbracket$,
        &
        $\Psi_{20} \xlongrightarrow{\hspace{-0.2mm} (3, R) \hspace{-0.2mm}}
        \llbracket \zeta_{54} \rrbracket$,
        &
        $\Psi_{21} \xlongrightarrow{\hspace{2mm} R \hspace{2mm}}
        \llbracket \zeta_{55} \rrbracket$,
        &
        $\Psi_{21} \xlongrightarrow{\hspace{2mm} L \hspace{2mm}}
        \llbracket \zeta_{56} \rrbracket$,
        \\
        $\Psi_{22} \xlongrightarrow{\hspace{2mm} R \hspace{2mm}}
        \llbracket \zeta_{65} \rrbracket$,
        &
        $\Psi_{22} \xlongrightarrow{\hspace{2mm} L \hspace{2mm}}
        \llbracket \zeta_{66} \rrbracket$,
        &
        $\Psi_{23} \xlongrightarrow{\hspace{2mm} R \hspace{2mm}}
        \llbracket \zeta_{74} \rrbracket$,
        &
        $\Psi_{23} \xlongrightarrow{\hspace{2mm} L \hspace{2mm}}
        \llbracket \zeta_{75} \rrbracket$.
    \end{tabular}
    \end{center}

    These reductions, as well as those obtained from part (\ref{proposition, a,
    3}) of Proposition (\ref{proposition, a}), prove the desired result.

\end{proof}

As in the type $E_6$ setting, the proof of Theorem (\ref{equivalence classes,
bound, E7}) gives explicit reductions that one can use to determine
specifically which elements in $\Gamma_7$ satisfy the conditions described in
parts (\ref{equivalence classes, bound, E7, b}) and (\ref{equivalence classes,
bound, E7, c}).  Also, computer experiments suggest that the phrases \textit{at
least} and \textit{at most} in the statement of Theorem (\ref{equivalence
classes, bound, E7}) can be replaced by \textit{exactly}.

Recall, the significance of part (\ref{equivalence classes, bound, E7, c}) is
that nilpotency indices can sometimes be found using results from a smaller
rank setting. For instance, as a consequence of the reductions listed in the
proof of Theorem (\ref{equivalence classes, bound, E7}), we have the
equivalence $\Lambda_{11} \xleftrightarrow{\hspace{4mm}} \llbracket
\zeta_{43}\rrbracket$.  Furthermore, $\Lambda_{11} = \llbracket s_1 s_3 s_4 s_2
s_5 s_4 s_3 s_1 \rrbracket$.  The only simple reflections appearing in the
reduced expression $s_1 s_3 s_4 s_2 s_5 s_4 s_3 s_1$ are $s_1,\cdots, s_5$.
These generate the type $D_5$ Weyl group. Hence $\Lambda_{11}$ can be
identified with an element in $\Gamma(W(D_5))$. Here, we identify it with
$\llbracket w_{0,0,1,0,4}^+ \rrbracket$. Since ${\mathcal N}(\llbracket
w_{0,0,1,0,4}^+ \rrbracket = 2$ (from Theorem (\ref{general cases, BCD13})),
then ${\mathcal N}(\llbracket \zeta_{43} \rrbracket) = 2$ also.

Computer experiments reveal that over 99.7\% of the elements in
$\Gamma(W(E_7))$ are equivalent (under $\xleftrightarrow{\hspace{4mm}}$) to an
element of the form $(w, i, j)$ with $w$ lacking full support.  Specifically,
computer experiments show $\Gamma(W(E_7))$ is partitioned into 91 equivalence
classes.  The 113 elements in $\Gamma_7$ belong to 77 equivalence classes.
Exactly 65 equivalence classes fail to contain an element of the form $(w, i,
j)$ with $w$ lacking full support, and the union of these 65 equivalence
classes contains precisely 92064 elements. As $\Gamma(W(E_7))$ has cardinality
34997760 (see Theorem (\ref{Gamma W, cardinality})), these 92064 elements
account for approximately only $0.263$ percent of the total number of elements
in $\Gamma(W(E_7))$.

\subsubsection{Type \texorpdfstring{$E_8$}{E8}}

\label{section, E8}

Nilpotency indices in type $E_8$ are computed employing the same techniques
used in types $F_4$, $E_6$, and $E_7$. Consider first the reduced expression
for the longest element of the Weyl group $W(E_8)$,
\begin{align}
    \label{E8, longest element, reduced expression}
    &w_0(E_8) = (s_1 s_2 s_3 s_4 s_5 s_6 s_7 s_8) (s_1 s_2 s_3 s_4 s_5 s_6 s_7
    s_8) \cdots (s_1 s_2 s_3 s_4 s_5 s_6 s_7 s_8)
\end{align}
Next, we find a presentation of the quantum Schubert cell algebra ${\mathcal
U}_q^+[w_0(E_8)]$ associated to this reduced expression.

\begin{lemma}

    \label{lemma, E8, q-commutator identities}

    Let $z_1, \dots, z_{120} \in {\mathcal U}_q^+[w_0(E_8)]$ be the Lusztig
    root vectors corresponding to the reduced expression (\ref{E8, longest
    element, reduced expression}). Then

    \begin{multicols}{3}

        \begin{enumerate}

            \item $z_{3} = [z_{1}, z_{9}]$,

            \item $z_{4} = [z_{2}, z_{10}]$,

            \item $z_{5} = [z_{4}, z_{96}]$,

            \item $z_{6} = [z_{5}, z_{104}]$,

            \item $z_{7} = [z_{6}, z_{112}]$,

            \item $z_{8} = [z_{1}, z_{80}]$,

            \item $z_{8} = [z_{7}, z_{120}]$,

            \item $z_{10} = [z_{3}, z_{88}]$,

            \item $z_{12} = [z_{5}, z_{16}]$,

            \item $z_{17} = [z_{2}, z_{88}]$.

        \end{enumerate}

    \end{multicols}

\end{lemma}

\begin{proof}

    By applying the algorithm given in the proof of Proposition (\ref{as nested
    E}), we write each Lusztig root vector $z_i$ as either a Chevalley
    generator, or as a scalar multiple of nested $q$-commutators of Chevalley
    generators. The $8$ Lusztig root vectors $z_1$, $z_2$, $z_9$, $z_{88}$,
    $z_{96}$, $z_{104}$, $z_{112}$, and $z_{120}$ are the Chevalley generators.
    In particular, $z_1 = E_1$, $z_2 = E_2$, $z_9 = E_3$,  $z_{88} = E_4$,
    $z_{96} = E_5$, $z_{104} = E_6$, $z_{112} = E_7$, and $z_{120} = E_8$.  The
    remaining Lusztig root vectors can be written as nested $q$-commutators of
    Chevalley generators. We have
        $z_3 = \mathbf{E}_{1,3}$,
        $z_4 = [E_2, \mathbf{E}_{1,3,4}]$,
        $z_5 = [[E_2, \mathbf{E}_{1,3,4}], E_5]$,
        $z_6 = [[[E_2, \mathbf{E}_{1,3,4}], E_5], E_6]$,
        $z_7 = [[[[E_2, \mathbf{E}_{1,3,4}], E_5], E_6], E_7]$,
        $z_8 = [[[[[E_2, \mathbf{E}_{1,3,4}], E_5], E_6], E_7], E_8]$,
        $z_{10} = \mathbf{E}_{1,3,4}$,
        $z_{12} = [\mathbf{E}_{1,3,4}, [[E_2, \mathbf{E}_{3,4}], E_5]]$,
        $z_{16} = \mathbf{E}_{3,4}$,
        $z_{17} = \mathbf{E}_{2,4}$, and
        $z_{80} = [[[[[E_2, \mathbf{E}_{3,4}], E_5], E_6], E_7], E_8]$,

    With these identities, it follows that $z_3 = [z_1, z_9]$, $z_4 = [z_2,
    z_{10}]$, $z_5 = [z_4, z_{96}]$, $z_6 = [z_5, z_{104}]$, $z_7 = [z_6,
    z_{112}]$, $z_8 = [z_7, z_{120}]$, $z_{10} = [z_3, z_{88}]$, and $z_{17} =
    [z_2, z_{88}]$. It remains to show $z_8 = [z_1, z_{80}]$ and $z_{12} =
    [z_5, z_{16}]$.

    To prove $z_8 = [z_1, z_{80}]$, observe first that the Chevalley generator
    $E_1$ commutes with each of the Chevalley generators $E_5$, $E_6$, $E_7$,
    and $E_8$. Hence, by repeatedly applying $q$-associativity, we can
    rearrange the $q$-commutator brackets to obtain
    \[
        [z_1, z_{80}] = [ E_1, [ [[[[E_2, \mathbf{E}_{3,4}], E_5], E_6], E_7] ,
        E_8 ] ] = [ [ [ [ [ E_1 , [E_2, \mathbf{E}_{3,4}] ] , E_5 ] , E_6 ] ,
        E_7 ] , E_8 ].
    \]
    Since the Chevalley generator $E_1$ commutes with each of the Chevalley
    generators $E_2$ and $E_4$, we have $[E_1, [E_2, \mathbf{E}_{3,4}]] = [E_2,
    [E_1, \mathbf{E}_{3,4}]] = [E_2, \mathbf{E}_{1,3,4}]$, and thus we obtain
    $[z_1, z_{80}] = [[[[[E_2, \mathbf{E}_{1,3,4}], E_5], E_6], E_7], E_8]$, as
    desired.

    Next we prove $z_{12} = [z_5, z_{16}]$. For short, define the Weyl group
    elements $u:= s_1 s_2 s_3 s_4$, $v := s_5 s_6 s_7 s_8 s_1 s_2 s_3 s_4 s_5
    s_6 s_7$, and $w:= s_5 s_6 s_7 s_8 s_1 s_2 s_3$. From the definition of
    Lusztig root vectors $z_5 = T_u(E_5)$, $z_{12} = T_{u \cdot w} (E_4)$, and
    $z_{16} = T_{u \cdot v}(E_8)$. However, since $u \leq_R u \cdot v$ and $u
    \leq_R u \cdot w$, we have $T_{u \cdot v} = T_u T_v$ and $T_{u \cdot w} =
    T_u T_{w}$. Therefore, $T_u^{-1}(z_5) = E_5$, $T_u^{-1}(z_{12}) = T_w
    (E_4)$, and $T_u^{-1}(z_{16}) = T_v (E_8)$, and hence it suffices to show
    $[E_5, T_v (E_8)] = T_w(E_4)$. Using the algorithm in the proof of
    Proposition (\ref{as nested E}), we write $T_w(E_4)$ and $T_v(E_8)$ as
    nested $q$-commutators of Chevalley generators. This gives us $T_v(E_8) =
    [E_2, \mathbf{E}_{1,3,4}]$ and $T_w(E_4) = [E_2, [\mathbf{E}_{1,3},
    \mathbf{E}_{5,4}]]$. Finally, since the Chevalley generator $E_5$ commutes
    with each of $E_2$ and $\mathbf{E}_{1,3}$, we have $[E_5, T_v(E_8)] = [E_5,
    [E_2, \mathbf{E}_{1,3,4}]] = [E_2, [E_5, \mathbf{E}_{1,3,4}] = [E_2,
    [\mathbf{E}_{1,3}, \mathbf{E}_{5,4}] = T_w(E_4)$.

\end{proof}

The repeating pattern in the reduced expression (\ref{E8, longest element,
reduced expression}) induces additional symmetries for the presentation of the
algebra ${\mathcal U}_q^+[w_0(E_8)]$. In particular, Proposition
(\ref{symmetry, general}) can be applied to give the following result.

\begin{proposition}

    \label{symmetry, E8}

    Let $\mathfrak{g}$ be the complex simple Lie algebra of type $E_8$, and let
    $z_1,\dots, z_{120} \in {\mathcal U}_q^+[w_0(E_8)]$ be as in Lemma
    (\ref{lemma, E8, q-commutator identities}).  Let $c\in W(E_8)$ be the
    Coxeter element $c = s_1 s_2 s_3 s_4 s_5 s_6 s_7 s_8$.  The Lusztig
    symmetry $T_c$ is an algebra automorphism of ${\mathcal U}_q(\mathfrak{g})$
    that sends $z_k$ to $z_{k + 8}$ for all $k \in [1, 112]$.

\end{proposition}

Proposition (\ref{symmetry, E8}) implies an explicit presentation of the
algebra ${\mathcal U}_q^+[w_0(E_8)]$ can be found by determining how to write
each of the $q$-commutators $[z_i, z_j]$, with $i < j$ and $i \leq 8$, as a
linear combination of ordered monomials. However, among these ${120 \choose 2}
- {112 \choose 2}$ ($=924$) $q$-commutators, $243$ are trivial (i.e. $[z_i,
z_j] = 0$) as a consequence of Proposition (\ref{LS corollary}). These trivial
$q$-commutators, together with the $10$ relations in Lemma (\ref{lemma, E8,
q-commutator identities}), provide enough initial data to apply the
$\texttt{L(i,j,r,s)}$ and $\texttt{R(i,j,r,s)}$ algorithms to write an explicit
presentation of ${\mathcal U}_q^+[w_0(E_8)]$. Some of the commutation relations
in ${\mathcal U}_q^+[w_0(E_8)]$ are very complicated, and as such, it would be
too lengthy to do these computations by hand, yet with the help of a computer
to automate this process, such a sequence of \texttt{L(i,j,r,s)}'s and
\texttt{R(i,j,r,s)}'s can be constructed. For example, one such sequence (among
several possible sequences) begins
\begin{center}
    {\small\tt L(3,80,1,9)},
    {\small\tt L(4,88,2,10)},
    {\small\tt L(6,16,5,104)},
    {\small\tt L(7,16,6,112)},
    {\small\tt R(8,16,7,120)},\ldots
\end{center}

After an explicit presentation of ${\mathcal U}_q^+[w_0(E_8)]$ is found,
nilpotency indices can be computed in the same manner as in the $F_4$, $E_6$,
and $E_7$ cases.  As mentioned, the commutation relations in ${\mathcal
U}_q^+[w_0(E_8)]$ are quite complicated to work with by hand. However, a
computer can easily handle these remaining computations. With slight
modifications to the proof of Theorem (\ref{invariant under dual, E6}), we
obtain the following result.

\begin{theorem}

    \label{invariant under dual, E8}

    For all $x \in \Gamma (W(E_8))$, ${\mathcal N}(x) = {\mathcal N}(x^*)$ and
    $\chi(x) = \chi(x^*)$.

\end{theorem}

\noindent The next theorem is the main result regarding nilpotency indices in
type $E_8$.

\begin{theorem}

    Table (\ref{table, E8}) gives the nilpotency index ${\mathcal N}(\llbracket
    \eta_k \rrbracket)$ for all $k \in [1, 962]$.

\end{theorem}

\begin{center}
    \begin{table}[h]
\caption{The Lie algebra of Type $E_8$ (data)}
\vspace{-10pt}
\begin{tabular}{lcc}
    \toprule
    $w \in BiGr_{\perp}^{\circ}(E_8)$ & $\chi( \llbracket w
    \rrbracket)$
    & ${\mathcal N}(\llbracket w \rrbracket)$
    \\[0.5mm]
    \hline
    $\eta_k$, ($k \in [1, 16] \cup [223, 225]$) \hspace{1mm} & (2, 2, -1) & $3$
    \\
    $\eta_k$, ($k \in [17, 41] \cup [226, 276]$) & (2, 2, -1) & $4$
    \\
    $\eta_k$, ($k \in [42, 57] \cup [277, 311]$) & (2, 2, -1) & $5$
    \\
    $\eta_k$, ($k \in [58, 61] \cup [312, 317]$) & (2, 2, -1) & $6$
    \\
    $\eta_{62}$, & (2, 2, -1) & $7$
    \\
    \midrule[0.01pt]
    $\eta_k$, ($k \in [318, 325]$) & (2, 2, 0) & $2$
    \\
    $\eta_k$, ($k \in [63, 99] \cup [326, 432]$) & (2, 2, 0) & $3$
    \\
    $\eta_k$, ($k \in [100, 134] \cup [433, 614]$) & (2, 2, 0) & $4$
    \\
    $\eta_k$, ($k \in [135, 163] \cup [615, 700]$) & (2, 2, 0) & $5$
    \\
    $\eta_k$, ($k \in [164, 172] \cup [701, 717]$) & (2, 2, 0) & $6$
    \\
    $\eta_k$, ($k \in [173, 176]$) & (2, 2, 0) & $7$
    \\
    \midrule[0.01pt]
    $\eta_k$, ($k \in [177, 181] \cup [718, 779]$) & (2, 2, 1) & $2$
    \\
    $\eta_k$, ($k \in [182, 206] \cup [780, 903]$) & (2, 2, 1) & $3$
    \\
    $\eta_k$, ($k \in [207, 218] \cup [904, 955]$) & (2, 2, 1) & $4$
    \\
    $\eta_k$, ($k \in [219, 221] \cup [956, 962]$) & (2, 2, 1) & $5$
    \\
    $\eta_{222}$ & (2, 2, 1) & $6$
    \\
    \bottomrule
\end{tabular}
\label{table, E8}
\end{table}
\end{center}

\noindent In view of Theorem (\ref{invariant under dual, E8}), Table
(\ref{table, E8}) can be used to look up the nilpotency index ${\mathcal
N}(\llbracket w \rrbracket)$ associated to each $w \in BiGr_\perp^\circ(E_8)$.

\noindent
\begin{minipage}{\textwidth}
\begin{theorem}

    \label{equivalence classes, bound, E8}

    Let $\Gamma_8 := \left\{ \llbracket \eta
    \rrbracket \in \Gamma(W(E_8)) : \eta \in BiGr_\perp^\circ(E_8) \right\}$.

    \begin{enumerate}

        \item The $1702$ elements of $\Gamma_8$ are contained in at most $1118$
            distinct equivalence classes.

        \item At least $400$ of the $1480$ elements in $\Gamma_8$ of the form
            $\llbracket w \rrbracket$ with $w$ a non-involution satisfy
            $\llbracket w \rrbracket \xleftrightarrow{\hspace{4mm}} \llbracket
            w^{-1} \rrbracket$.

        \item At least $230$ of the $1702$ elements of $\Gamma_8$ are
            equivalent to an element of the form $\llbracket w \rrbracket$ with
            $w \in BiGr_\perp(E_8) \backslash BiGr_\perp^\circ(E_8)$.

    \end{enumerate}

\end{theorem}
\end{minipage}

\begin{proof}

    This follows from the reductions listed in Appendix (\ref{appendix B}).

\end{proof}

Analogous to what was observed in the types $E_6$, and $E_7$ cases, computer
experiments reveal that over 99\% of the elements in $\Gamma(W(E_8))$ are
equivalent (under $\xleftrightarrow{\hspace{4mm}}$) to an element of the form
$(w, i, j)$ with $w$ lacking full support.  More specifically, computer
experiments indicate that the 1702 elements in $\Gamma_8$ belong to 1118
equivalence classes. Exactly 1052 equivalence classes fail to contain an
element of the form $(w, i, j)$ with $w$ lacking full support, and the union of
these 1052 equivalence classes contains precisely 80788944 elements.  As
$\Gamma(W(E_8))$ has cardinality 11054776320 (see Theorem (\ref{Gamma W,
cardinality})), these 80788944 elements account for approximately only $0.731$
percent of the total number of elements in $\Gamma(W(E_8))$.

\appendix

\section{The elements of \texorpdfstring{$BiGr_\perp^\circ(E_n)$}{BiGr(En)},
    (\texorpdfstring{$n = 6, 7, 8$}{n = 6, 7, or 8})}

\label{appendix, elements E_n}

In this appendix, we explicitly describe the elements of the sets
$BiGr_\perp^\circ(E_6)$, $BiGr_\perp^\circ(E_7)$, and $BiGr_\perp^\circ(E_8)$
by writing each such element as a composition of reflections. Observe that if a
Weyl group element $w \in W$ is written as a composition of reflections, then
$w^{-1}$ can be written as a composition of the same reflections in the reverse
order. With this in mind, we consider only the involutions, as well as one of
either $w$ or $w^{-1}$, for each non-involution $w \in
BiGr_{\perp}^{\circ}(E_n)$ ($n=6,7,8$).

We order the reflections by placing a certain convex order on the positive
roots, which is constructed from a reduced expression of the longest element of
the Weyl group.  Recall, for a reduced expression $w = s_{i_1} s_{i_2} \cdots
s_{i_N}$ and $k \in [1, N]$, define the positive root
\[
    \beta_k := w_{[1, k - 1]} (\alpha_{i_k}),
\]
where $\alpha_1, \cdots, \alpha_n$ are the simple roots and $w_{[1, k - 1]} :=
s_{i_1} \cdots s_{i_{k-1}}$.  The reflection about the hyperplane orthogonal to
$\beta_k$ will be denoted by $\rho_k$. That is to say,
\[
    \rho_k = w_{[1, k - 1]} \cdot s_{i_k} \cdot \left(w_{[1, k -
    1]}\right)^{-1} \in W.
\]
To write compositions of reflections more compactly, we adopt the notation
\[
    \rho_{t_1, t_2,\ldots, t_m} := \rho_{t_1} \circ \cdots \circ \rho_{t_m} \in
    W.
\]
Using the reduced expressions
\begin{align*}
    \label{En, longest element, reduced expression}
    &w_0(E_6) = (s_2 s_4 s_3 s_5 s_1 s_6) \cdot (s_2 s_4 s_3 s_5 s_1 s_6)
    \cdots (s_2 s_4 s_3 s_5 s_1 s_6),
    \\
    &w_0(E_7) = (s_1 s_2 s_3 s_4 s_5 s_6 s_7) \cdot (s_1 s_2 s_3 s_4 s_5 s_6
    s_7) \cdots (s_1 s_2 s_3 s_4 s_5 s_6 s_7),
    \\
    &w_0(E_8) = (s_1 s_2 s_3 s_4 s_5 s_6 s_7 s_8) (s_1 s_2 s_3 s_4 s_5 s_6 s_7
    s_8) \cdots (s_1 s_2 s_3 s_4 s_5 s_6 s_7 s_8),
\end{align*}
for the longest elements of the Weyl groups $W(E_6)$, $W(E_7)$, $W(E_8)$,
respectively, we label the corresponding reflections
\[
    \rho_1^{[6]}, \cdots, \rho_{36}^{[6]} \in W(E_6),
    \hspace{20pt}
    \rho_1^{[7]}, \cdots, \rho_{63}^{[7]} \in W(E_7),
    \hspace{20pt}
    \rho_1^{[8]}, \cdots, \rho_{120}^{[8]} \in W(E_8),
\]
or, more simply, we will label them $\rho_1, \rho_2, \rho_3,\cdots$ when it is
clear from the context which Weyl group among $W(E_6)$, $W(E_7)$, or $W(E_8)$
is being considered.

Recall, the elements of $BiGr_\perp^\circ(E_n)$ (for $n=6,7,8$) are labelled
$\nu_k^{\pm 1}$, $\zeta_k^{\pm 1}$, $\eta_k^{\pm 1}$,
\begin{align*}
    &BiGr_{\perp}^\circ (E_6) = \left\{ \nu_k, \nu_k^{-1} : 1 \leq k
    \leq 15 \right\},
    \\
    &BiGr_{\perp}^\circ (E_7) = \left\{ \zeta_k, \zeta_k^{-1} : 1 \leq
    k \leq 76 \right\},
    \\
    &BiGr_{\perp}^\circ (E_8) = \left\{ \eta_k, \eta_k^{-1} : 1 \leq k
    \leq 962 \right\}.
\end{align*}

The Weyl group elements $\nu_1,\cdots, \nu_{15}$, $\zeta_1,\cdots,\zeta_{76}$,
$\eta_1,\cdots,\eta_{962}$ are given by the following:

\vspace{5mm}

\begin{flushleft}

\begin{description}

    \item[$\nu_{1} - \nu_{15}$]
        $\rho_{8, 23, 24}$,
        $\rho_{20}$,
        $\rho_{2, 20}$,
        $\rho_{14}$,
        $\rho_{15, 27}$,
        $\rho_{16, 28}$,
        $\rho_{9, 10}$,
        $\rho_{9, 10, 25}$,
        $\rho_{3, 4, 20}$,
        $\rho_{5, 6, 13}$,
        $\rho_{1, 4, 8, 10, 23}$,
        $\rho_{3, 4, 16, 28}$,
        $\rho_{1, 6, 9, 12, 14}$,
        $\rho_{1, 5, 10, 11, 14}$,
        $\rho_{1, 5, 6, 7, 13}$

    \item[$\zeta_{1} - \zeta_{76}$]
        $\rho_{25}$,
        $\rho_{32}$,
        $\rho_{39}$,
        $\rho_{46}$,
        $\rho_{18, 46}$,
        $\rho_{20, 29, 46, 58}$,
        $\rho_{11, 39}$,
        $\rho_{10, 39, 43}$,
        $\rho_{4, 22, 32}$,
        $\rho_{11, 28, 39, 59}$,
        $\rho_{11, 39, 51}$,
        $\rho_{11, 27, 47}$,
        $\rho_{21, 46}$,
        $\rho_{14, 39}$,
        $\rho_{3, 32}$,
        $\rho_{18, 58}$,
        $\rho_{26, 60}$,
        $\rho_{13, 24}$,
        $\rho_{27, 47}$,
        $\rho_{28, 53}$,
        $\rho_{4, 26, 41}$,
        $\rho_{18, 38, 54}$,
        $\rho_{28, 37, 38}$,
        $\rho_{20, 24, 44, 53}$,
        $\rho_{26, 41}$,
        $\rho_{13, 20, 24}$,
        $\rho_{27, 47, 51}$,
        $\rho_{4, 21, 26, 41}$,
        $\rho_{9, 10, 39, 43}$,
        $\rho_{21, 26, 41}$,
        $\rho_{10, 23, 38, 45}$,
        $\rho_{6, 17, 57}$,
        $\rho_{11, 23, 41, 48}$,
        $\rho_{12, 16, 22, 59}$,
        $\rho_{15, 26, 60}$,
        $\rho_{11, 29, 59}$,
        $\rho_{17, 37, 60}$,
        $\rho_{22, 45, 61}$,
        $\rho_{17, 37, 43, 60}$,
        $\rho_{14, 15, 16, 46, 50}$,
        $\rho_{1, 4, 5, 8, 17, 28, 32}$,
        $\rho_{2, 11, 15, 21, 46, 50}$,
        $\rho_{18, 42, 48, 54, 60}$,
        $\rho_{1, 9, 16, 24, 37, 58}$,
        $\rho_{2, 14, 18, 48, 54, 60}$,
        $\rho_{14, 15, 16, 26, 50, 60}$,
        $\rho_{1, 8, 9, 10, 18, 58}$,
        $\rho_{1, 7, 14, 24, 25}$,
        $\rho_{1, 4, 7, 14, 20, 24, 25}$,
        $\rho_{1, 6, 9, 24, 35, 48}$,
        $\rho_{1, 5, 9, 27, 43, 47, 56}$,
        $\rho_{4, 6, 23, 24, 35, 48}$,
        $\rho_{1, 3, 9, 22, 27, 47, 51}$,
        $\rho_{2, 8, 15, 26, 41}$,
        $\rho_{2, 4, 6, 41, 49, 54}$,
        $\rho_{10, 18, 43, 58}$,
        $\rho_{10, 21, 26, 60}$,
        $\rho_{14, 20, 27, 37, 41}$,
        $\rho_{6, 23, 24, 35, 48}$,
        $\rho_{1, 7, 14, 20, 24, 25}$,
        $\rho_{1, 9, 15, 16, 46, 50}$,
        $\rho_{1, 2, 5, 15, 28, 45, 47}$,
        $\rho_{3, 6, 12, 17, 57}$,
        $\rho_{1, 6, 8, 17, 57}$,
        $\rho_{1, 3, 5, 12, 17, 43, 59}$,
        $\rho_{2, 10, 21, 22, 45, 61}$,
        $\rho_{8, 10, 26, 60}$,
        $\rho_{1, 3, 14, 15, 22, 45, 61}$,
        $\rho_{3, 9, 10, 26, 60}$,
        $\rho_{14, 16, 22, 37, 60}$,
        $\rho_{15, 26, 49, 55, 61}$,
        $\rho_{1, 9, 15, 22, 45, 61}$,
        $\rho_{1, 3, 9, 17, 37, 43, 60}$,
        $\rho_{2, 4, 5, 37, 48, 60}$,
        $\rho_{1, 4, 17, 37, 43, 60}$,
        $\rho_{5, 11, 29, 42, 60}$

    \item[$\eta_{1} - \eta_{50}$]
        $\rho_{100}$,
        $\rho_{92}$,
        $\rho_{84}$,
        $\rho_{61}$,
        $\rho_{53}$,
        $\rho_{44}$,
        $\rho_{52}$,
        $\rho_{60}$,
        $\rho_{68}$,
        $\rho_{76}$,
        $\rho_{37, 92}$,
        $\rho_{28, 84}$,
        $\rho_{28, 61}$,
        $\rho_{20, 39, 50, 92}$,
        $\rho_{28, 51, 92}$,
        $\rho_{53, 61}$,
        $\rho_{44, 94, 99}$,
        $\rho_{25, 44, 92, 97}$,
        $\rho_{44, 64, 90, 97}$,
        $\rho_{44, 70}$,
        $\rho_{52, 101, 105, 106}$,
        $\rho_{39, 52, 98, 108}$,
        $\rho_{20, 84}$,
        $\rho_{44, 63}$,
        $\rho_{20, 53}$,
        $\rho_{52, 99}$,
        $\rho_{33, 52, 71, 106}$,
        $\rho_{60, 108}$,
        $\rho_{40, 60, 107, 115}$,
        $\rho_{52, 100}$,
        $\rho_{44, 92}$,
        $\rho_{37, 62, 67, 108}$,
        $\rho_{11, 24, 84}$,
        $\rho_{23, 33, 92, 106}$,
        $\rho_{24, 44, 56}$,
        $\rho_{52, 64, 97}$,
        $\rho_{60, 105, 106}$,
        $\rho_{68, 114, 115}$,
        $\rho_{22, 28, 84}$,
        $\rho_{6, 21, 53, 57}$,
        $\rho_{11, 24, 28, 84}$,
        $\rho_{31, 44, 94, 99}$,
        $\rho_{44, 72, 90, 91}$,
        $\rho_{44, 65, 91, 99}$,
        $\rho_{31, 44, 90}$,
        $\rho_{24, 44, 70}$,
        $\rho_{52, 71, 101}$,
        $\rho_{52, 64, 100}$,
        $\rho_{44, 56, 92}$,
        $\rho_{33, 52, 100, 105}$

    \item[$\eta_{51} - \eta_{100}$]
        $\rho_{12, 61, 69, 97}$,
        $\rho_{20, 84, 105}$,
        $\rho_{25, 44, 63}$,
        $\rho_{13, 21, 31, 84}$,
        $\rho_{32, 52, 99}$,
        $\rho_{60, 72, 108}$,
        $\rho_{22, 28, 64, 84}$,
        $\rho_{44, 70, 94}$,
        $\rho_{31, 44, 90, 97}$,
        $\rho_{52, 71, 101, 106}$,
        $\rho_{21, 51, 69, 99}$,
        $\rho_{44, 64, 70, 94}$,
        $\rho_{71, 101}$,
        $\rho_{64, 100}$,
        $\rho_{55, 66}$,
        $\rho_{56, 92}$,
        $\rho_{22, 35}$,
        $\rho_{10, 84}$,
        $\rho_{17, 84}$,
        $\rho_{4, 61}$,
        $\rho_{47, 93}$,
        $\rho_{11, 53}$,
        $\rho_{44, 95}$,
        $\rho_{52, 103}$,
        $\rho_{60, 111}$,
        $\rho_{68, 119}$,
        $\rho_{37, 45, 80}$,
        $\rho_{70, 117}$,
        $\rho_{59, 116}$,
        $\rho_{37, 115}$,
        $\rho_{28, 113}$,
        $\rho_{19, 64, 84}$,
        $\rho_{66, 78, 116}$,
        $\rho_{27, 72, 92, 115}$,
        $\rho_{20, 55, 66}$,
        $\rho_{19, 32, 84}$,
        $\rho_{38, 91, 108}$,
        $\rho_{25, 37, 92}$,
        $\rho_{26, 37, 83, 101}$,
        $\rho_{17, 28, 84}$,
        $\rho_{55, 61, 66}$,
        $\rho_{16, 61, 69}$,
        $\rho_{72, 108}$,
        $\rho_{72, 90, 91}$,
        $\rho_{20, 70, 94}$,
        $\rho_{14, 51, 90}$,
        $\rho_{37, 62, 109}$,
        $\rho_{28, 51, 106}$,
        $\rho_{28, 51, 83, 102}$,
        $\rho_{50, 62, 105}$

    \item[$\eta_{101} - \eta_{150}$]
        $\rho_{4, 37, 64, 82, 83}$,
        $\rho_{44, 70, 116}$,
        $\rho_{38, 91}$,
        $\rho_{19, 84}$,
        $\rho_{27, 92}$,
        $\rho_{13, 55, 59, 106}$,
        $\rho_{10, 20, 84}$,
        $\rho_{44, 63, 109}$,
        $\rho_{52, 99, 117}$,
        $\rho_{19, 37, 58, 71, 93}$,
        $\rho_{70, 94}$,
        $\rho_{59, 87}$,
        $\rho_{15, 37}$,
        $\rho_{12, 59, 87}$,
        $\rho_{35, 100}$,
        $\rho_{11, 19, 64, 84}$,
        $\rho_{63, 66, 78, 116}$,
        $\rho_{38, 50, 91}$,
        $\rho_{24, 44, 56, 102}$,
        $\rho_{52, 64, 97, 110}$,
        $\rho_{20, 40, 53}$,
        $\rho_{60, 105, 106, 118}$,
        $\rho_{28, 54, 57, 73}$,
        $\rho_{71, 101, 106}$,
        $\rho_{33, 70, 94, 105}$,
        $\rho_{55, 63, 66}$,
        $\rho_{22, 35, 63}$,
        $\rho_{5, 15, 32, 37}$,
        $\rho_{47, 93, 99}$,
        $\rho_{23, 38, 91, 108}$,
        $\rho_{27, 92, 106}$,
        $\rho_{50, 62, 99}$,
        $\rho_{37, 41, 62}$,
        $\rho_{31, 55, 61, 66}$,
        $\rho_{55, 59}$,
        $\rho_{11, 50, 62, 105}$,
        $\rho_{32, 55, 59, 106}$,
        $\rho_{37, 64, 82, 83}$,
        $\rho_{31, 44, 90, 116}$,
        $\rho_{37, 69, 87}$,
        $\rho_{24, 44, 70, 116}$,
        $\rho_{38, 91, 109}$,
        $\rho_{23, 38, 91}$,
        $\rho_{13, 55, 59}$,
        $\rho_{23, 27, 92}$,
        $\rho_{11, 19, 84}$,
        $\rho_{23, 27, 92, 106}$,
        $\rho_{21, 56, 62, 66}$,
        $\rho_{30, 38, 93, 101}$,
        $\rho_{25, 44, 63, 109}$

    \item[$\eta_{151} - \eta_{200}$]
        $\rho_{59, 77, 89}$,
        $\rho_{32, 52, 99, 117}$,
        $\rho_{39, 58, 59, 93}$,
        $\rho_{64, 70, 94}$,
        $\rho_{56, 59, 87}$,
        $\rho_{15, 25, 37}$,
        $\rho_{12, 59, 87, 97}$,
        $\rho_{35, 63, 100}$,
        $\rho_{11, 50, 62, 99}$,
        $\rho_{20, 55, 66, 90}$,
        $\rho_{40, 70, 73, 85}$,
        $\rho_{20, 41, 84, 114}$,
        $\rho_{25, 37, 41, 62}$,
        $\rho_{55, 59, 97}$,
        $\rho_{11, 55, 59}$,
        $\rho_{12, 24, 26, 84}$,
        $\rho_{20, 41, 84}$,
        $\rho_{4, 31, 37, 82, 83}$,
        $\rho_{24, 37, 69, 87}$,
        $\rho_{23, 38, 91, 109}$,
        $\rho_{30, 33, 34, 92}$,
        $\rho_{32, 59, 77, 89}$,
        $\rho_{11, 55, 59, 97}$,
        $\rho_{12, 62, 66, 90}$,
        $\rho_{20, 25, 41, 84}$,
        $\rho_{31, 37, 82, 83}$,
        $\rho_{15, 89, 101}$,
        $\rho_{39, 57, 98, 108}$,
        $\rho_{41, 79, 90, 116}$,
        $\rho_{27, 65, 94, 100}$,
        $\rho_{21, 26, 33, 84}$,
        $\rho_{30, 41, 87}$,
        $\rho_{17, 55, 66}$,
        $\rho_{27, 65, 94}$,
        $\rho_{18, 65, 91, 106}$,
        $\rho_{31, 90, 117}$,
        $\rho_{24, 70, 117}$,
        $\rho_{14, 27, 115}$,
        $\rho_{13, 71, 101}$,
        $\rho_{16, 59, 116}$,
        $\rho_{3, 37, 115}$,
        $\rho_{6, 57, 100}$,
        $\rho_{18, 25, 92}$,
        $\rho_{10, 17, 84}$,
        $\rho_{32, 99, 118}$,
        $\rho_{20, 40, 114}$,
        $\rho_{14, 65, 108}$,
        $\rho_{18, 25, 92, 97}$,
        $\rho_{15, 30, 33, 34}$,
        $\rho_{6, 57, 100, 105}$

    \item[$\eta_{201} - \eta_{250}$]
        $\rho_{16, 41, 62, 90}$,
        $\rho_{10, 37, 62, 109}$,
        $\rho_{32, 50, 57, 116}$,
        $\rho_{20, 50, 87, 102}$,
        $\rho_{26, 55, 70, 81}$,
        $\rho_{13, 18, 55, 59}$,
        $\rho_{24, 30, 41, 87}$,
        $\rho_{26, 50, 87, 97}$,
        $\rho_{10, 59, 87}$,
        $\rho_{17, 55, 63, 66}$,
        $\rho_{30, 81, 116}$,
        $\rho_{27, 65, 94, 102}$,
        $\rho_{5, 70, 94}$,
        $\rho_{27, 65, 94, 99}$,
        $\rho_{31, 90, 97, 117}$,
        $\rho_{14, 23, 27, 115}$,
        $\rho_{26, 39, 101, 109}$,
        $\rho_{18, 65, 91, 99}$,
        $\rho_{10, 56, 59, 87}$,
        $\rho_{24, 30, 81, 116}$,
        $\rho_{5, 64, 70, 94}$,
        $\rho_{26, 50, 63, 87}$,
        $\rho_{28, 43, 51, 67, 75, 84}$,
        $\rho_{12, 26, 38, 81, 100}$,
        $\rho_{24, 40, 56, 57, 100, 114}$,
        $\rho_{1, 7, 12, 17, 29, 39, 69, 108}$,
        $\rho_{5, 10, 21, 62, 77, 109}$,
        $\rho_{5, 11, 19, 32, 77, 85}$,
        $\rho_{1, 7, 9, 12, 26, 68, 101, 113}$,
        $\rho_{5, 12, 17, 26, 69, 77, 105}$,
        $\rho_{3, 17, 21, 26, 33, 53}$,
        $\rho_{3, 4, 6, 25, 40, 60, 109}$,
        $\rho_{1, 3, 10, 22, 61, 70}$,
        $\rho_{10, 11, 28, 51, 83, 110, 118}$,
        $\rho_{1, 3, 10, 20, 25, 61, 100}$,
        $\rho_{1, 9, 10, 13, 34, 51, 81, 92}$,
        $\rho_{2, 12, 17, 53, 92}$,
        $\rho_{1, 9, 18, 25, 35, 79, 82, 92}$,
        $\rho_{2, 6, 12, 29, 40, 52, 88}$,
        $\rho_{1, 2, 3, 10, 44, 63}$,
        $\rho_{10, 11, 24, 27, 44, 102}$,
        $\rho_{1, 3, 10, 25, 52, 64, 97}$,
        $\rho_{5, 16, 26, 52, 81, 103}$,
        $\rho_{4, 20, 25, 32, 72, 84}$,
        $\rho_{1, 8, 23, 26, 27, 35, 45, 84}$,
        $\rho_{1, 3, 10, 25, 32, 60, 105, 106}$,
        $\rho_{4, 25, 28, 61}$,
        $\rho_{10, 12, 26, 60, 81, 110}$,
        $\rho_{12, 39, 52, 64, 109}$,
        $\rho_{1, 18, 26, 52, 81, 97, 101, 110}$

    \item[$\eta_{251} - \eta_{300}$]
        $\rho_{14, 32, 33, 54, 65, 85, 117}$,
        $\rho_{1, 4, 24, 28, 64, 84, 105, 114}$,
        $\rho_{5, 12, 32, 77, 85}$,
        $\rho_{11, 12, 53, 92}$,
        $\rho_{23, 33, 100, 105}$,
        $\rho_{16, 17, 18, 92, 97}$,
        $\rho_{3, 17, 18, 19, 53}$,
        $\rho_{1, 4, 12, 24, 32, 84}$,
        $\rho_{4, 25, 52, 64, 97}$,
        $\rho_{3, 21, 26, 33, 53, 63}$,
        $\rho_{1, 2, 31, 37, 54, 58, 78, 88}$,
        $\rho_{1, 9, 10, 36, 84}$,
        $\rho_{3, 7, 16, 20, 44, 60, 95}$,
        $\rho_{1, 3, 10, 20, 25, 32, 72, 84}$,
        $\rho_{6, 11, 19, 27, 40, 59, 100}$,
        $\rho_{4, 21, 25, 37, 53, 60}$,
        $\rho_{1, 5, 10, 17, 32, 37, 92}$,
        $\rho_{1, 4, 12, 24, 28, 84, 113}$,
        $\rho_{2, 25, 52, 66, 87}$,
        $\rho_{1, 2, 9, 15, 53, 59}$,
        $\rho_{7, 18, 39, 41, 59, 61, 91}$,
        $\rho_{1, 9, 36, 55, 66}$,
        $\rho_{2, 36, 41, 55, 62, 90, 97}$,
        $\rho_{9, 16, 20, 56, 92, 97, 106}$,
        $\rho_{1, 11, 14, 45, 51, 90}$,
        $\rho_{4, 6, 9, 26, 52, 66, 87}$,
        $\rho_{25, 44, 64, 65, 70, 94}$,
        $\rho_{1, 8, 13, 17, 41, 55, 67, 99}$,
        $\rho_{1, 3, 10, 21, 37, 61, 71, 91}$,
        $\rho_{25, 44, 70, 94}$,
        $\rho_{9, 13, 22, 40, 65, 84, 103}$,
        $\rho_{1, 5, 10, 20, 35, 61, 65, 100}$,
        $\rho_{1, 44, 64, 72, 89, 90, 97}$,
        $\rho_{5, 10, 21, 32, 62, 77, 109}$,
        $\rho_{1, 10, 12, 21, 26, 58, 84, 110}$,
        $\rho_{1, 4, 6, 19, 60, 72, 108}$,
        $\rho_{1, 44, 64, 65, 91, 106}$,
        $\rho_{1, 10, 17, 18, 64, 67, 84}$,
        $\rho_{1, 3, 13, 18, 21, 39, 53, 92}$,
        $\rho_{4, 25, 44, 63}$,
        $\rho_{4, 25, 44, 64, 90, 97}$,
        $\rho_{1, 4, 16, 25, 26, 27, 37, 92}$,
        $\rho_{2, 10, 22, 26, 58, 61, 70}$,
        $\rho_{3, 17, 21, 26, 33, 53, 63}$,
        $\rho_{25, 31, 44, 90, 97}$,
        $\rho_{5, 10, 17, 32, 52, 99}$,
        $\rho_{2, 12, 31, 50, 61, 66, 116}$,
        $\rho_{1, 11, 24, 44, 70}$,
        $\rho_{1, 2, 3, 4, 5, 32, 52, 99}$,
        $\rho_{2, 5, 22, 31, 35, 84}$

    \item[$\eta_{301} - \eta_{350}$]
        $\rho_{1, 18, 19, 32, 65, 66, 78, 84}$,
        $\rho_{32, 52, 71, 101, 106}$,
        $\rho_{1, 14, 32, 43, 51, 61, 70}$,
        $\rho_{6, 13, 20, 25, 72, 77, 85}$,
        $\rho_{9, 40, 52, 64, 71, 100}$,
        $\rho_{1, 10, 17, 33, 52, 100, 105}$,
        $\rho_{8, 19, 27, 50, 57, 61, 69}$,
        $\rho_{1, 11, 31, 44, 56, 90}$,
        $\rho_{1, 4, 15, 25, 31, 60, 90}$,
        $\rho_{1, 5, 9, 10, 11, 32, 37, 92}$,
        $\rho_{1, 9, 36, 55, 63, 66}$,
        $\rho_{25, 31, 44, 94, 99}$,
        $\rho_{4, 9, 16, 42, 45, 70, 94}$,
        $\rho_{5, 17, 28, 55, 64, 77, 109}$,
        $\rho_{32, 44, 64, 70, 94, 105}$,
        $\rho_{1, 10, 17, 34, 58, 71, 84}$,
        $\rho_{4, 13, 21, 31, 84}$
        $\rho_{23, 33, 37, 105, 115}$,
        $\rho_{37, 51, 54, 75, 77, 84}$,
        $\rho_{7, 10, 15, 24, 27, 45, 57}$,
        $\rho_{4, 13, 19, 31, 34, 66, 115}$,
        $\rho_{1, 24, 25, 37, 51, 54, 75}$,
        $\rho_{6, 17, 19, 21, 37, 61, 71}$,
        $\rho_{1, 3, 10, 61}$,
        $\rho_{28, 51, 88, 95, 102, 109}$,
        $\rho_{1, 7, 10, 25, 35, 97, 98, 108}$,
        $\rho_{1, 2, 22, 28, 42, 85, 97, 110}$,
        $\rho_{10, 21, 26, 27, 51, 92, 106}$,
        $\rho_{2, 24, 44, 70, 91, 93, 100}$,
        $\rho_{1, 8, 39, 56, 67, 88}$,
        $\rho_{1, 2, 3, 11, 24, 28, 113}$,
        $\rho_{3, 4, 15, 18, 37, 55}$,
        $\rho_{24, 25, 40, 57, 71, 100}$,
        $\rho_{2, 17, 30, 34, 64, 91, 111}$,
        $\rho_{1, 2, 6, 7, 12, 41, 108, 113}$,
        $\rho_{1, 9, 10, 34, 51, 81, 99}$,
        $\rho_{9, 10, 11, 27, 65, 91, 99}$,
        $\rho_{1, 3, 7, 10, 11, 24, 26, 61}$,
        $\rho_{9, 10, 62, 65, 81, 99}$,
        $\rho_{20, 35, 91, 112, 118}$,
        $\rho_{4, 5, 9, 19, 33, 100, 105}$,
        $\rho_{2, 3, 16, 23, 33, 53}$,
        $\rho_{2, 4, 23, 33, 100, 105}$,
        $\rho_{12, 26, 81, 91, 110}$,
        $\rho_{1, 11, 22, 37, 102, 110, 118}$,
        $\rho_{1, 6, 11, 26, 41, 75, 95, 114}$,
        $\rho_{1, 4, 5, 9, 26, 72, 75, 95}$,
        $\rho_{1, 3, 5, 10, 33, 56, 92, 106}$,
        $\rho_{2, 4, 21, 26, 33, 84}$,
        $\rho_{9, 11, 18, 32, 52, 64, 99}$

    \item[$\eta_{351} - \eta_{400}$]
        $\rho_{1, 8, 40, 72, 78, 88, 96, 97}$,
        $\rho_{11, 24, 37, 115}$,
        $\rho_{3, 16, 25, 71, 97, 101}$,
        $\rho_{2, 9, 17, 37, 115}$,
        $\rho_{1, 2, 7, 33, 58, 71, 88, 90}$,
        $\rho_{1, 2, 9, 17, 28, 113}$,
        $\rho_{11, 24, 25, 55, 63, 66}$,
        $\rho_{1, 9, 13, 18, 25, 71, 101}$,
        $\rho_{1, 4, 12, 24, 70, 94}$,
        $\rho_{2, 4, 18, 28, 64, 113}$,
        $\rho_{1, 11, 15, 24, 25, 35, 63}$,
        $\rho_{2, 30, 38, 91}$,
        $\rho_{1, 7, 12, 26, 99, 100}$,
        $\rho_{5, 10, 22, 33, 35, 59}$,
        $\rho_{2, 5, 20, 34, 71, 80, 88}$,
        $\rho_{1, 2, 9, 17, 20, 55, 66}$,
        $\rho_{14, 22, 35, 90, 97}$,
        $\rho_{1, 2, 16, 18, 26, 55, 70, 81}$,
        $\rho_{1, 2, 6, 9, 19, 21, 70, 94}$,
        $\rho_{15, 21, 42, 95, 111, 117}$,
        $\rho_{1, 9, 10, 11, 20, 39, 50, 115}$,
        $\rho_{1, 27, 39, 57, 90, 92}$,
        $\rho_{1, 7, 16, 21, 26, 27, 70, 115}$,
        $\rho_{3, 17, 28, 57, 81, 87, 95}$,
        $\rho_{21, 47, 55, 93, 102}$,
        $\rho_{10, 14, 19, 48, 55, 57, 84}$,
        $\rho_{9, 19, 32, 37, 109, 116}$,
        $\rho_{2, 20, 22, 35, 70, 117}$,
        $\rho_{1, 2, 6, 9, 17, 19, 61}$,
        $\rho_{2, 10, 17, 24, 84}$,
        $\rho_{1, 5, 16, 32, 40, 51, 84}$,
        $\rho_{1, 2, 16, 17, 18, 53}$,
        $\rho_{6, 17, 25, 40, 91, 109}$,
        $\rho_{1, 3, 16, 17, 44, 56, 102}$,
        $\rho_{19, 32, 62, 84, 105}$,
        $\rho_{24, 39, 56, 57, 98, 108}$,
        $\rho_{9, 10, 17, 23, 39, 98, 108}$,
        $\rho_{1, 2, 9, 17, 28, 51, 106}$,
        $\rho_{2, 10, 44, 63}$,
        $\rho_{1, 3, 24, 32, 40, 50, 57, 62}$,
        $\rho_{9, 10, 32, 51, 59, 75, 83}$,
        $\rho_{8, 17, 29, 33, 79, 90, 102}$,
        $\rho_{1, 13, 19, 27, 34, 70, 92, 115}$,
        $\rho_{1, 4, 25, 26, 37, 83, 101}$,
        $\rho_{6, 18, 35, 48, 62, 65, 84}$,
        $\rho_{2, 20, 40, 56, 60, 65, 93}$,
        $\rho_{2, 10, 25, 32, 41, 87, 110}$,
        $\rho_{2, 11, 24, 33, 72, 105, 108}$,
        $\rho_{1, 6, 10, 19, 71, 97, 101, 112}$,
        $\rho_{16, 17, 18, 71, 97, 101}$

    \item[$\eta_{401} - \eta_{450}$]
        $\rho_{23, 33, 59, 106, 116}$,
        $\rho_{3, 4, 15, 18, 66, 95}$,
        $\rho_{23, 33, 70, 94, 105}$,
        $\rho_{3, 4, 18, 22, 35, 59}$,
        $\rho_{1, 10, 18, 27, 57, 59, 106, 116}$,
        $\rho_{1, 2, 12, 17, 24, 71, 97, 101}$,
        $\rho_{5, 20, 56, 65, 81, 97, 116}$,
        $\rho_{3, 23, 32, 33, 55, 59}$,
        $\rho_{4, 20, 25, 70, 94}$,
        $\rho_{5, 9, 13, 19, 38, 41, 113}$,
        $\rho_{1, 4, 25, 37, 51, 54, 75}$,
        $\rho_{1, 3, 10, 20, 25, 70, 94}$,
        $\rho_{2, 12, 26, 70, 81, 116}$,
        $\rho_{5, 26, 27, 35, 50, 87}$,
        $\rho_{1, 2, 16, 21, 26, 27, 51, 106}$,
        $\rho_{1, 2, 16, 26, 27, 37, 65, 94}$,
        $\rho_{2, 3, 16, 20, 39, 50, 51}$,
        $\rho_{3, 5, 26, 28, 33, 55, 81}$,
        $\rho_{27, 30, 59, 106, 117}$,
        $\rho_{1, 3, 6, 23, 27, 65, 66, 84}$,
        $\rho_{1, 16, 19, 26, 31, 37, 72, 116}$,
        $\rho_{1, 7, 11, 26, 56, 75, 94}$,
        $\rho_{1, 2, 16, 44, 95}$,
        $\rho_{7, 23, 26, 61, 66, 116}$,
        $\rho_{3, 15, 21, 26, 51, 66, 95}$,
        $\rho_{2, 27, 65, 78, 88, 94}$,
        $\rho_{4, 18, 25, 37, 41, 55, 62}$,
        $\rho_{1, 8, 22, 28, 56, 57, 91, 109}$,
        $\rho_{10, 18, 28, 64, 81, 84}$,
        $\rho_{1, 6, 9, 21, 53, 57}$,
        $\rho_{9, 21, 25, 26, 27, 84}$,
        $\rho_{11, 19, 22, 35, 38, 84}$,
        $\rho_{2, 22, 35, 38, 59}$,
        $\rho_{11, 24, 55, 66}$,
        $\rho_{2, 7, 22, 38, 41, 57}$,
        $\rho_{15, 35, 36, 44, 102, 110}$,
        $\rho_{9, 10, 15, 32, 34, 51, 66}$,
        $\rho_{2, 14, 21, 64, 82, 96, 109}$,
        $\rho_{11, 21, 55, 59, 110, 118}$,
        $\rho_{2, 20, 40, 71, 81, 97, 100}$,
        $\rho_{17, 18, 20, 37, 65, 91}$,
        $\rho_{20, 41, 50, 87, 91, 110}$,
        $\rho_{1, 5, 20, 56, 69, 81, 97, 115}$,
        $\rho_{1, 9, 28, 33, 54, 73, 82}$,
        $\rho_{1, 6, 8, 21, 29, 85, 97, 110}$,
        $\rho_{1, 3, 10, 22, 61, 70, 117}$,
        $\rho_{2, 37, 45, 80, 87}$,
        $\rho_{1, 11, 48, 56, 59}$,
        $\rho_{2, 14, 22, 36, 50, 64}$,
        $\rho_{1, 2, 8, 11, 38, 45, 56, 102}$

    \item[$\eta_{451} - \eta_{500}$]
        $\rho_{2, 11, 18, 23, 55, 62, 66}$,
        $\rho_{1, 3, 10, 24, 40, 57, 71, 100}$,
        $\rho_{1, 5, 11, 21, 26, 33, 51, 106}$,
        $\rho_{16, 17, 18, 19, 30, 38, 91}$,
        $\rho_{20, 38, 91, 104, 111}$,
        $\rho_{2, 3, 8, 16, 39, 50, 51}$,
        $\rho_{1, 8, 16, 22, 27, 28, 82, 106}$,
        $\rho_{2, 5, 7, 10, 28, 108, 113}$,
        $\rho_{10, 23, 33, 100, 105}$,
        $\rho_{11, 19, 32, 50, 55, 84}$,
        $\rho_{1, 10, 17, 18, 92, 97}$,
        $\rho_{13, 14, 15, 44, 95}$,
        $\rho_{2, 6, 18, 39, 57, 59, 100}$,
        $\rho_{2, 12, 20, 66, 82, 85, 96}$,
        $\rho_{1, 2, 11, 19, 33, 100, 105}$,
        $\rho_{12, 24, 26, 81, 91, 110}$,
        $\rho_{1, 23, 27, 38, 50, 92}$,
        $\rho_{1, 3, 23, 44, 52, 102}$,
        $\rho_{1, 11, 44, 95}$,
        $\rho_{1, 3, 23, 52, 103}$,
        $\rho_{1, 24, 25, 39, 57, 98, 108}$,
        $\rho_{1, 9, 10, 19, 43, 52, 57}$,
        $\rho_{1, 3, 4, 12, 26, 31, 72, 84}$,
        $\rho_{1, 2, 9, 17, 20, 40, 53}$,
        $\rho_{1, 3, 10, 25, 52, 64, 97, 110}$,
        $\rho_{3, 4, 19, 26, 27, 41, 84}$,
        $\rho_{1, 10, 12, 26, 52, 81, 103}$,
        $\rho_{1, 3, 16, 27, 50, 59, 85, 94}$,
        $\rho_{4, 13, 21, 59, 87, 91, 110}$,
        $\rho_{13, 18, 25, 59, 75, 98, 110}$,
        $\rho_{4, 8, 25, 28, 61}$,
        $\rho_{5, 17, 32, 52, 99, 101, 108}$,
        $\rho_{9, 27, 37, 64, 71, 82, 83}$,
        $\rho_{7, 39, 41, 65, 80, 102}$,
        $\rho_{2, 7, 24, 39, 41, 80, 102}$,
        $\rho_{2, 17, 24, 70, 94}$,
        $\rho_{1, 8, 39, 56, 63, 67, 88}$,
        $\rho_{1, 24, 48, 56, 66}$,
        $\rho_{32, 64, 70, 94, 105}$,
        $\rho_{24, 25, 55, 63, 66}$,
        $\rho_{9, 24, 33, 40, 70, 80, 88}$,
        $\rho_{17, 23, 57, 70, 81, 87}$,
        $\rho_{15, 24, 35, 37}$,
        $\rho_{12, 59, 116}$,
        $\rho_{1, 4, 5, 9, 15, 19, 32, 37}$,
        $\rho_{1, 4, 5, 9, 19, 32, 37, 115}$,
        $\rho_{1, 15, 24, 25, 35, 63}$,
        $\rho_{12, 15, 35, 44, 96, 102}$,
        $\rho_{15, 35, 37, 102, 110}$,
        $\rho_{1, 5, 9, 15, 19, 32, 37}$

    \item[$\eta_{501} - \eta_{550}$]
        $\rho_{1, 2, 3, 7, 10, 22, 38, 57}$,
        $\rho_{11, 25, 37, 51, 54, 75}$,
        $\rho_{12, 32, 55, 59, 106}$,
        $\rho_{1, 15, 18, 25, 35, 102, 110}$,
        $\rho_{11, 12, 32, 55, 59, 106}$,
        $\rho_{13, 20, 55, 59}$,
        $\rho_{12, 24, 33, 70, 91, 92}$,
        $\rho_{5, 26, 28, 33, 55, 81}$,
        $\rho_{1, 19, 33, 100, 105}$,
        $\rho_{3, 14, 15, 18, 20, 39, 53}$,
        $\rho_{18, 23, 30, 38, 92, 106}$,
        $\rho_{16, 17, 18, 19, 64, 84}$,
        $\rho_{1, 11, 13, 35, 50, 62}$,
        $\rho_{4, 25, 28, 101, 112, 113}$,
        $\rho_{10, 11, 24, 84}$,
        $\rho_{23, 25, 33, 92, 106}$,
        $\rho_{1, 2, 3, 10, 43, 46, 63, 110}$,
        $\rho_{18, 25, 55, 62, 66, 93}$,
        $\rho_{1, 5, 19, 40, 57, 100, 114}$,
        $\rho_{16, 17, 18, 25, 33, 92, 106}$,
        $\rho_{1, 2, 3, 8, 63, 66, 78}$,
        $\rho_{3, 17, 18, 19, 33, 53}$,
        $\rho_{1, 6, 10, 11, 27, 40, 44, 95}$,
        $\rho_{4, 7, 13, 25, 31, 61}$,
        $\rho_{3, 13, 18, 25, 59, 86, 91}$,
        $\rho_{2, 12, 26, 36, 70, 81, 115}$,
        $\rho_{1, 2, 40, 71, 80, 88, 96, 97}$,
        $\rho_{11, 56, 59, 116}$,
        $\rho_{16, 17, 18, 59, 87, 97}$,
        $\rho_{16, 17, 18, 64, 70, 94}$,
        $\rho_{4, 15, 25, 37}$,
        $\rho_{2, 11, 19, 22, 35, 37, 58}$,
        $\rho_{1, 11, 13, 24, 55, 66}$,
        $\rho_{2, 12, 17, 34, 64, 91, 111}$,
        $\rho_{50, 62, 97, 105}$,
        $\rho_{19, 32, 50, 55, 84}$,
        $\rho_{6, 24, 25, 26, 40, 91, 109}$,
        $\rho_{1, 5, 10, 17, 70, 94}$,
        $\rho_{1, 2, 3, 16, 32, 40, 50, 51}$,
        $\rho_{1, 6, 9, 17, 19, 21, 22, 61}$,
        $\rho_{1, 5, 19, 40, 56, 57, 71, 100}$,
        $\rho_{3, 34, 41, 55, 57, 69}$,
        $\rho_{27, 38, 92, 106}$,
        $\rho_{4, 5, 9, 26, 75, 95}$,
        $\rho_{1, 5, 11, 13, 40, 56, 62, 109}$,
        $\rho_{13, 18, 25, 39, 55, 62, 109}$,
        $\rho_{10, 20, 25, 65, 66, 84}$,
        $\rho_{3, 6, 18, 21, 31, 53, 57}$,
        $\rho_{1, 2, 6, 44, 56, 65, 94, 102}$,
        $\rho_{1, 2, 28, 42, 56, 83, 103}$

    \item[$\eta_{551} - \eta_{600}$]
        $\rho_{2, 31, 44, 56, 94, 97, 99}$,
        $\rho_{1, 16, 18, 22, 28, 92, 97}$,
        $\rho_{9, 24, 44, 70, 94}$,
        $\rho_{2, 3, 12, 71, 97, 101}$,
        $\rho_{4, 5, 9, 19, 37, 115}$,
        $\rho_{1, 2, 3, 6, 19, 70, 94}$,
        $\rho_{1, 2, 9, 15, 17, 25, 37}$,
        $\rho_{5, 10, 17, 20, 70, 94}$,
        $\rho_{15, 24, 35, 62, 74}$,
        $\rho_{3, 4, 15, 18, 32, 37, 55}$,
        $\rho_{1, 2, 6, 19, 70, 94, 105}$,
        $\rho_{1, 18, 50, 57, 62, 70, 105}$,
        $\rho_{4, 13, 25, 59, 87}$,
        $\rho_{2, 21, 30, 33, 62, 109}$,
        $\rho_{1, 6, 10, 20, 56, 65, 91, 99}$,
        $\rho_{1, 2, 4, 22, 33, 38, 62, 81}$,
        $\rho_{2, 14, 42, 56, 66, 77}$,
        $\rho_{1, 9, 10, 34, 51, 81}$,
        $\rho_{16, 17, 18, 50, 62, 97, 99}$,
        $\rho_{2, 4, 33, 34, 55, 81, 102}$,
        $\rho_{9, 11, 26, 55, 70, 81}$,
        $\rho_{1, 8, 12, 38, 45, 95, 97, 109}$,
        $\rho_{2, 3, 16, 39, 50, 51}$,
        $\rho_{1, 2, 3, 6, 14, 55, 57, 59}$,
        $\rho_{10, 11, 50, 62, 97, 105}$,
        $\rho_{2, 5, 23, 40, 65, 92, 97}$,
        $\rho_{2, 12, 13, 34, 50, 83, 103}$,
        $\rho_{1, 6, 10, 11, 41, 75, 95}$,
        $\rho_{2, 13, 14, 15, 18, 44, 95}$,
        $\rho_{1, 3, 10, 20, 56, 92, 97, 106}$,
        $\rho_{2, 10, 44, 63, 109}$,
        $\rho_{1, 2, 7, 11, 19, 38, 57, 100}$,
        $\rho_{1, 17, 23, 28, 57, 81, 84}$,
        $\rho_{1, 2, 3, 7, 28, 56, 62, 81}$,
        $\rho_{1, 2, 31, 44, 56, 90}$,
        $\rho_{10, 37, 64, 82, 83}$,
        $\rho_{2, 4, 21, 26, 31, 33, 84}$,
        $\rho_{16, 17, 23, 32, 59, 77, 89}$,
        $\rho_{9, 16, 38, 55, 66, 85}$,
        $\rho_{6, 14, 24, 27, 48, 59, 84}$,
        $\rho_{6, 19, 35, 56, 78, 82, 109}$,
        $\rho_{5, 20, 50, 62, 71, 77}$,
        $\rho_{2, 14, 22, 35, 59, 91, 110}$,
        $\rho_{11, 23, 33, 71, 101, 106}$,
        $\rho_{2, 12, 17, 24, 64, 70, 117}$,
        $\rho_{1, 4, 12, 24, 64, 70, 94}$,
        $\rho_{2, 12, 17, 24, 64, 70, 94}$,
        $\rho_{4, 8, 9, 38, 45, 56, 102}$,
        $\rho_{1, 6, 19, 39, 101, 109}$,
        $\rho_{1, 11, 13, 20, 25, 59, 87}$

    \item[$\eta_{601} - \eta_{650}$]
        $\rho_{2, 6, 20, 70, 94}$,
        $\rho_{1, 2, 3, 4, 5, 20, 70, 94}$,
        $\rho_{4, 12, 59, 96, 97, 110, 117}$,
        $\rho_{2, 22, 26, 38, 81}$,
        $\rho_{1, 2, 11, 13, 31, 35, 50, 87}$,
        $\rho_{1, 18, 21, 62, 97, 110}$,
        $\rho_{10, 11, 18, 27, 59, 91, 110}$,
        $\rho_{3, 4, 14, 18, 55, 59, 82}$,
        $\rho_{1, 10, 12, 37, 41, 55, 62, 65}$,
        $\rho_{10, 20, 37, 41, 84}$,
        $\rho_{4, 25, 37, 62, 109}$,
        $\rho_{1, 3, 10, 20, 41, 62, 84}$,
        $\rho_{7, 14, 21, 26, 33, 56, 92}$,
        $\rho_{7, 19, 38, 56, 57, 82, 83}$,
        $\rho_{8, 16, 37, 44, 49}$,
        $\rho_{11, 24, 55, 63, 66}$,
        $\rho_{1, 8, 16, 32, 37, 44, 49}$,
        $\rho_{15, 35, 44, 96, 102}$,
        $\rho_{1, 5, 19, 40, 57, 71, 100}$,
        $\rho_{8, 38, 45, 49, 95, 97, 109}$,
        $\rho_{19, 59, 62, 83, 84}$,
        $\rho_{1, 3, 7, 21, 28, 57, 82, 96}$,
        $\rho_{1, 10, 17, 18, 64, 65, 91, 99}$,
        $\rho_{1, 5, 16, 19, 40, 64, 71, 100}$,
        $\rho_{11, 13, 55, 59, 106}$,
        $\rho_{1, 11, 13, 30, 59, 62, 83}$,
        $\rho_{16, 18, 20, 25, 70, 91, 92}$,
        $\rho_{11, 50, 62, 97, 105}$,
        $\rho_{1, 11, 13, 14, 15, 44, 95}$,
        $\rho_{1, 16, 19, 33, 100, 105}$,
        $\rho_{1, 9, 12, 17, 85, 96, 97}$,
        $\rho_{1, 16, 23, 27, 28, 90, 92}$,
        $\rho_{4, 24, 44, 56, 102}$,
        $\rho_{1, 16, 19, 39, 57, 63, 98, 108}$,
        $\rho_{10, 12, 20, 41, 62, 65, 84}$,
        $\rho_{18, 21, 25, 26, 66, 84}$,
        $\rho_{4, 44, 63, 109}$,
        $\rho_{1, 9, 20, 40, 53, 81}$,
        $\rho_{1, 9, 10, 19, 52, 57, 103, 110}$,
        $\rho_{3, 6, 16, 28, 89, 99, 118}$,
        $\rho_{1, 7, 17, 21, 28, 63, 93, 101}$,
        $\rho_{1, 2, 4, 20, 41, 65, 84, 114}$,
        $\rho_{1, 2, 19, 32, 38, 73, 85, 106}$,
        $\rho_{1, 3, 10, 25, 59, 77, 89}$,
        $\rho_{1, 9, 28, 33, 54, 63, 73, 82}$,
        $\rho_{4, 25, 31, 37, 82, 83}$,
        $\rho_{1, 11, 25, 26, 38, 74, 81, 91}$,
        $\rho_{18, 31, 37, 55, 83, 85}$,
        $\rho_{5, 10, 17, 32, 37, 49, 92}$,
        $\rho_{1, 5, 10, 17, 64, 70, 94}$

    \item[$\eta_{651} - \eta_{700}$]
        $\rho_{11, 12, 55, 59}$,
        $\rho_{8, 9, 16, 37, 55, 73}$,
        $\rho_{1, 4, 8, 16, 37, 44, 49}$,
        $\rho_{3, 10, 23, 25, 38, 49, 102}$,
        $\rho_{1, 4, 8, 16, 32, 37, 44, 49}$,
        $\rho_{1, 11, 13, 20, 55, 59}$,
        $\rho_{16, 18, 19, 33, 70, 91, 92}$,
        $\rho_{3, 13, 17, 41, 55, 90, 103}$,
        $\rho_{10, 11, 24, 59, 62, 83, 84}$,
        $\rho_{2, 11, 31, 50, 87, 97, 99}$,
        $\rho_{1, 10, 17, 18, 32, 65, 94, 99}$,
        $\rho_{4, 7, 20, 90, 93, 104}$,
        $\rho_{3, 24, 34, 41, 55, 57, 69}$,
        $\rho_{1, 6, 9, 18, 25, 57, 100, 105}$,
        $\rho_{1, 9, 10, 11, 18, 65, 91, 99}$,
        $\rho_{1, 3, 23, 38, 50, 91}$,
        $\rho_{23, 27, 38, 92, 106}$,
        $\rho_{1, 5, 32, 56, 62, 81, 92}$,
        $\rho_{2, 13, 18, 35, 50, 62}$,
        $\rho_{4, 21, 48, 56, 62, 66}$,
        $\rho_{9, 10, 26, 37, 62, 96, 110}$,
        $\rho_{4, 12, 30, 62, 85, 90}$,
        $\rho_{1, 2, 5, 9, 17, 39, 50, 51}$,
        $\rho_{1, 10, 12, 37, 41, 62, 65}$,
        $\rho_{1, 6, 10, 11, 20, 56, 92, 106}$,
        $\rho_{20, 41, 62, 84}$,
        $\rho_{10, 20, 84, 105}$,
        $\rho_{2, 12, 33, 52, 102, 110}$,
        $\rho_{1, 9, 14, 25, 37, 78, 102, 116}$,
        $\rho_{10, 11, 19, 27, 34, 37, 92}$,
        $\rho_{1, 18, 23, 32, 33, 59, 77, 89}$,
        $\rho_{17, 20, 21, 25, 35, 84}$,
        $\rho_{1, 7, 13, 17, 28, 56, 63, 92}$,
        $\rho_{1, 2, 8, 26, 30, 52, 81}$,
        $\rho_{13, 18, 23, 36, 50, 60, 115}$,
        $\rho_{25, 31, 66, 78, 82, 84}$,
        $\rho_{1, 10, 15, 26, 30, 37, 43, 77}$,
        $\rho_{2, 9, 35, 38, 87, 95, 97}$,
        $\rho_{12, 55, 59, 97}$,
        $\rho_{1, 3, 6, 16, 17, 18, 55, 59}$,
        $\rho_{1, 11, 13, 25, 59, 87}$,
        $\rho_{1, 4, 8, 16, 19, 37, 49, 73}$,
        $\rho_{16, 17, 18, 50, 62, 99}$,
        $\rho_{1, 3, 16, 17, 25, 37, 41, 62}$,
        $\rho_{9, 23, 27, 28, 90, 92}$,
        $\rho_{1, 2, 9, 25, 39, 58, 88, 90}$,
        $\rho_{1, 3, 32, 44, 63, 65, 94}$,
        $\rho_{10, 19, 38, 57, 91, 93}$,
        $\rho_{10, 25, 31, 37, 82, 83}$,
        $\rho_{1, 9, 10, 11, 32, 59, 77, 89}$

    \item[$\eta_{701} - \eta_{750}$]
        $\rho_{7, 21, 57, 71, 82, 96}$,
        $\rho_{1, 11, 13, 55, 59}$,
        $\rho_{11, 19, 59, 62, 83, 84}$,
        $\rho_{1, 10, 17, 18, 27, 65, 91, 99}$,
        $\rho_{1, 10, 17, 18, 65, 91, 99}$,
        $\rho_{4, 21, 56, 62, 66}$,
        $\rho_{10, 13, 37, 62, 96, 110}$,
        $\rho_{1, 6, 10, 20, 56, 92, 106}$,
        $\rho_{1, 10, 19, 27, 34, 37, 92}$,
        $\rho_{4, 25, 44, 63, 109}$,
        $\rho_{1, 7, 17, 28, 63, 93, 101}$,
        $\rho_{1, 25, 26, 38, 74, 81, 91}$,
        $\rho_{25, 31, 37, 82, 83}$,
        $\rho_{5, 10, 17, 32, 52, 99, 117}$,
        $\rho_{11, 25, 59, 87, 97}$,
        $\rho_{12, 20, 62, 66, 90}$,
        $\rho_{20, 25, 41, 62, 84}$,
        $\rho_{1, 10, 23, 32, 33, 71, 101}$,
        $\rho_{1, 9, 19, 27, 65, 94, 118}$,
        $\rho_{1, 3, 23, 27, 65, 94}$,
        $\rho_{13, 31, 71, 101, 119}$,
        $\rho_{7, 10, 19, 26, 34, 108, 113}$,
        $\rho_{1, 4, 9, 16, 33, 46, 56}$,
        $\rho_{1, 4, 5, 19, 33, 100, 105}$,
        $\rho_{9, 10, 11, 26, 72, 75, 95}$,
        $\rho_{2, 5, 9, 64, 77, 88}$,
        $\rho_{18, 25, 33, 92, 106}$,
        $\rho_{4, 7, 26, 66, 75, 96, 116}$,
        $\rho_{1, 9, 10, 24, 31, 40, 57, 114}$,
        $\rho_{5, 6, 11, 26, 71, 97, 101}$,
        $\rho_{1, 5, 7, 16, 35, 87, 103, 109}$,
        $\rho_{5, 16, 27, 64, 70, 94}$,
        $\rho_{1, 2, 7, 16, 19, 32, 59, 87}$,
        $\rho_{1, 5, 11, 37, 64, 105, 114}$,
        $\rho_{1, 2, 5, 11, 37, 64, 115}$,
        $\rho_{1, 2, 3, 12, 40, 43, 97}$,
        $\rho_{3, 7, 12, 26, 99, 100}$,
        $\rho_{12, 30, 33, 35, 90, 114}$,
        $\rho_{11, 12, 26, 55, 72, 90, 103}$,
        $\rho_{10, 14, 23, 83, 96, 103, 115}$,
        $\rho_{1, 9, 13, 31, 39, 90, 92}$,
        $\rho_{1, 2, 4, 23, 47, 50, 94, 102}$,
        $\rho_{1, 2, 4, 30, 62, 65, 81, 99}$,
        $\rho_{16, 17, 27, 65, 91, 99}$,
        $\rho_{2, 15, 22, 47, 50, 93, 102}$,
        $\rho_{1, 10, 18, 59, 97, 110, 117}$,
        $\rho_{1, 4, 16, 23, 33, 72, 101, 105}$,
        $\rho_{2, 3, 11, 56, 92}$,
        $\rho_{10, 17, 24, 84}$,
        $\rho_{1, 2, 6, 11, 12, 40, 62, 109}$

    \item[$\eta_{751} - \eta_{800}$]
        $\rho_{1, 5, 9, 32, 77, 88}$,
        $\rho_{18, 25, 33, 57, 71, 94, 106}$,
        $\rho_{2, 6, 17, 25, 40, 91, 109}$,
        $\rho_{16, 17, 35, 82, 96, 108}$,
        $\rho_{2, 4, 37, 62, 109}$,
        $\rho_{16, 23, 32, 33, 71, 101}$,
        $\rho_{2, 12, 21, 33, 39, 115}$,
        $\rho_{13, 15, 31, 101, 112}$,
        $\rho_{3, 4, 8, 19, 55, 57, 109}$,
        $\rho_{2, 24, 26, 30, 90, 103, 117}$,
        $\rho_{2, 16, 59, 87}$,
        $\rho_{3, 4, 5, 16, 31, 99, 118}$,
        $\rho_{16, 59, 96, 103, 110, 117}$,
        $\rho_{1, 3, 16, 18, 31, 55, 57, 66}$,
        $\rho_{2, 9, 10, 48, 56, 59}$,
        $\rho_{1, 3, 6, 30, 66, 90, 117}$,
        $\rho_{3, 10, 37, 51, 54, 75, 77}$,
        $\rho_{1, 3, 13, 18, 23, 38, 81, 115}$,
        $\rho_{19, 27, 59, 94, 110}$,
        $\rho_{2, 5, 10, 13, 30, 59, 87}$,
        $\rho_{1, 4, 5, 13, 31, 35, 50, 87}$,
        $\rho_{9, 13, 41, 50, 63, 110, 118}$,
        $\rho_{1, 2, 3, 16, 32, 34, 51, 66}$,
        $\rho_{1, 16, 41, 62, 90, 97}$,
        $\rho_{9, 10, 14, 23, 26, 59, 84}$,
        $\rho_{18, 23, 25, 33, 59, 106, 117}$,
        $\rho_{1, 3, 17, 84}$,
        $\rho_{1, 5, 7, 11, 13, 62, 80, 109}$,
        $\rho_{10, 25, 37, 62, 83, 85}$,
        $\rho_{13, 30, 41, 87, 103}$,
        $\rho_{7, 27, 28, 80, 95}$,
        $\rho_{4, 9, 13, 41, 50, 87, 110}$,
        $\rho_{3, 4, 16, 25, 71, 97, 101}$,
        $\rho_{1, 2, 3, 4, 25, 70, 94}$,
        $\rho_{2, 20, 40, 81, 114}$,
        $\rho_{2, 9, 14, 17, 27, 115}$,
        $\rho_{12, 27, 65, 94, 111}$,
        $\rho_{12, 27, 65, 94, 104}$,
        $\rho_{2, 4, 5, 32, 66, 101, 110}$,
        $\rho_{1, 3, 32, 33, 34, 58, 88, 105}$,
        $\rho_{6, 17, 23, 31, 40, 114}$,
        $\rho_{2, 7, 14, 27, 80, 95}$,
        $\rho_{12, 13, 65, 94, 111}$,
        $\rho_{23, 27, 65, 94, 99}$,
        $\rho_{3, 11, 13, 25, 41, 87, 110}$,
        $\rho_{3, 5, 12, 32, 41, 87, 110}$,
        $\rho_{1, 8, 16, 23, 32, 33, 34, 36}$,
        $\rho_{4, 19, 32, 66, 101, 110}$,
        $\rho_{2, 7, 15, 22, 32, 35, 56}$,
        $\rho_{17, 55, 66, 70}$

    \item[$\eta_{801} - \eta_{850}$]
        $\rho_{17, 55, 57, 66, 70}$,
        $\rho_{2, 5, 17, 33, 34, 91, 109}$,
        $\rho_{1, 16, 31, 41, 87, 90, 97}$,
        $\rho_{1, 3, 6, 10, 38, 94, 97, 109}$,
        $\rho_{9, 13, 41, 50, 101, 110}$,
        $\rho_{1, 7, 11, 27, 28, 80, 95}$,
        $\rho_{1, 3, 10, 12, 32, 65, 94, 111}$,
        $\rho_{1, 2, 9, 13, 21, 56, 64, 114}$,
        $\rho_{1, 2, 4, 9, 37, 115}$,
        $\rho_{4, 13, 20, 50, 87, 102}$,
        $\rho_{10, 23, 33, 71, 101, 106}$,
        $\rho_{4, 6, 15, 22, 34, 50, 97}$,
        $\rho_{1, 12, 26, 55, 72, 90, 103}$,
        $\rho_{1, 2, 9, 15, 35, 36, 102, 110}$,
        $\rho_{2, 8, 16, 32, 34, 51, 66}$,
        $\rho_{14, 21, 39, 55, 102}$,
        $\rho_{18, 64, 65, 91, 99}$,
        $\rho_{1, 2, 3, 4, 72, 90, 91}$,
        $\rho_{1, 9, 12, 34, 64, 91, 111}$,
        $\rho_{1, 2, 9, 34, 58, 88, 97, 105}$,
        $\rho_{16, 41, 57, 62, 71, 90}$,
        $\rho_{4, 6, 9, 21, 55, 102, 106}$,
        $\rho_{1, 3, 5, 31, 90, 91, 100}$,
        $\rho_{18, 25, 33, 71, 94, 106}$,
        $\rho_{1, 6, 10, 27, 40, 95, 117}$,
        $\rho_{1, 3, 10, 30, 81, 116}$,
        $\rho_{1, 7, 32, 33, 38, 56, 88, 116}$,
        $\rho_{1, 10, 17, 20, 40, 114}$,
        $\rho_{1, 4, 30, 81, 116}$,
        $\rho_{1, 2, 3, 10, 20, 40, 114}$,
        $\rho_{17, 18, 31, 65, 70, 94}$,
        $\rho_{1, 3, 16, 23, 32, 33, 71, 101}$,
        $\rho_{2, 5, 20, 81, 95, 102, 116}$,
        $\rho_{1, 4, 16, 23, 31, 57, 81, 116}$,
        $\rho_{3, 5, 11, 37, 115}$,
        $\rho_{3, 5, 11, 21, 37, 64, 114}$,
        $\rho_{6, 13, 20, 56, 65, 94, 102}$,
        $\rho_{5, 10, 23, 31, 55, 57, 66}$,
        $\rho_{1, 3, 16, 17, 35, 82, 96}$,
        $\rho_{2, 9, 12, 21, 35, 63, 114}$,
        $\rho_{9, 11, 26, 31, 33, 50, 87}$,
        $\rho_{1, 2, 5, 16, 19, 55, 59}$,
        $\rho_{4, 18, 19, 37, 115}$,
        $\rho_{4, 12, 39, 64, 101, 109}$,
        $\rho_{1, 6, 39, 57, 100, 105}$,
        $\rho_{1, 10, 11, 18, 59, 116}$,
        $\rho_{1, 2, 4, 8, 19, 38, 56, 102}$,
        $\rho_{9, 12, 17, 24, 59, 97, 116}$,
        $\rho_{2, 4, 56, 59, 87}$,
        $\rho_{9, 10, 13, 17, 25, 59, 87}$

    \item[$\eta_{851} - \eta_{900}$]
        $\rho_{3, 17, 20, 25, 50, 66, 87}$,
        $\rho_{18, 27, 30, 65, 91, 99}$,
        $\rho_{1, 2, 5, 16, 59, 62, 83, 84}$,
        $\rho_{10, 11, 18, 59, 97, 116}$,
        $\rho_{1, 4, 32, 48, 50, 57, 63, 105}$,
        $\rho_{2, 5, 22, 33, 35, 59}$,
        $\rho_{4, 6, 19, 30, 37, 55, 109}$,
        $\rho_{4, 7, 14, 39, 42, 56, 66}$,
        $\rho_{1, 3, 6, 16, 17, 85, 96}$,
        $\rho_{9, 14, 17, 19, 37, 55, 102}$,
        $\rho_{1, 2, 5, 15, 33, 66, 95}$,
        $\rho_{1, 9, 15, 35, 37, 102, 110}$,
        $\rho_{2, 15, 42, 66, 79, 95}$,
        $\rho_{1, 9, 11, 12, 13, 14, 15, 53}$,
        $\rho_{1, 3, 17, 23, 57, 70, 81, 87}$,
        $\rho_{6, 13, 34, 94, 102, 105}$,
        $\rho_{1, 13, 39, 81, 85, 96}$,
        $\rho_{3, 6, 11, 75, 88, 95}$,
        $\rho_{1, 2, 3, 32, 42, 78, 81, 95}$,
        $\rho_{1, 10, 17, 33, 70, 94, 105}$,
        $\rho_{1, 10, 55, 62, 65, 81, 97}$,
        $\rho_{2, 9, 11, 27, 65, 91, 99}$,
        $\rho_{1, 6, 16, 17, 85, 96, 106}$,
        $\rho_{1, 4, 16, 31, 41, 87, 90, 97}$,
        $\rho_{2, 9, 11, 65, 91, 99}$,
        $\rho_{1, 2, 5, 9, 77, 88}$,
        $\rho_{5, 12, 17, 32, 65, 94, 99}$,
        $\rho_{2, 10, 25, 26, 81, 103, 118}$,
        $\rho_{1, 5, 10, 31, 56, 90, 104, 118}$,
        $\rho_{1, 4, 6, 26, 27, 40, 95, 117}$,
        $\rho_{1, 3, 4, 20, 25, 40, 114}$,
        $\rho_{1, 2, 24, 25, 26, 63, 81, 116}$,
        $\rho_{1, 18, 38, 80, 88, 97}$,
        $\rho_{9, 16, 24, 31, 33, 70, 94}$,
        $\rho_{10, 12, 17, 31, 57, 81, 116}$,
        $\rho_{1, 3, 16, 17, 31, 90, 97, 117}$,
        $\rho_{1, 13, 24, 25, 40, 88, 102, 116}$,
        $\rho_{16, 17, 56, 59, 87}$,
        $\rho_{1, 3, 24, 31, 50, 57, 87}$,
        $\rho_{1, 10, 23, 31, 90, 99, 117}$,
        $\rho_{3, 4, 6, 13, 18, 37, 115}$,
        $\rho_{1, 18, 19, 32, 64, 90, 117}$,
        $\rho_{13, 18, 23, 50, 102, 116}$,
        $\rho_{1, 2, 11, 19, 32, 96, 99, 118}$,
        $\rho_{3, 4, 18, 21, 79, 95, 104}$,
        $\rho_{2, 7, 10, 24, 56, 83, 103}$,
        $\rho_{1, 2, 9, 10, 23, 41, 54, 102}$,
        $\rho_{1, 2, 3, 10, 20, 50, 87, 102}$,
        $\rho_{3, 12, 26, 72, 96, 99, 118}$,
        $\rho_{1, 11, 12, 33, 57, 59, 116}$

    \item[$\eta_{901} - \eta_{950}$]
        $\rho_{3, 4, 12, 24, 64, 70, 94}$,
        $\rho_{1, 2, 3, 10, 33, 70, 94, 105}$,
        $\rho_{1, 3, 5, 23, 64, 70, 81, 87}$,
        $\rho_{9, 13, 41, 50, 87, 103}$,
        $\rho_{2, 13, 20, 41, 87, 96, 102}$,
        $\rho_{4, 16, 26, 56, 65, 94, 103}$,
        $\rho_{10, 24, 30, 50, 66, 87}$,
        $\rho_{1, 3, 10, 11, 18, 71, 97, 101}$,
        $\rho_{2, 3, 4, 24, 70, 94}$,
        $\rho_{10, 11, 23, 33, 71, 101, 106}$,
        $\rho_{3, 5, 10, 33, 65, 94, 102}$,
        $\rho_{1, 16, 24, 48, 56, 66}$,
        $\rho_{1, 2, 3, 4, 25, 64, 70, 94}$,
        $\rho_{10, 24, 25, 55, 63, 66}$,
        $\rho_{1, 2, 6, 17, 34, 63, 91, 109}$,
        $\rho_{1, 32, 48, 50, 57, 63, 105}$,
        $\rho_{6, 26, 27, 40, 95, 117}$,
        $\rho_{1, 7, 16, 23, 27, 28, 115}$,
        $\rho_{1, 11, 24, 30, 41, 87}$,
        $\rho_{16, 23, 31, 57, 81, 116}$,
        $\rho_{2, 6, 9, 14, 20, 56, 115}$,
        $\rho_{4, 5, 7, 39, 41, 80, 102}$,
        $\rho_{5, 64, 65, 70, 94}$,
        $\rho_{1, 6, 13, 17, 20, 72, 99, 101}$,
        $\rho_{1, 10, 17, 18, 64, 70, 117}$,
        $\rho_{10, 11, 56, 59, 87}$,
        $\rho_{1, 10, 17, 18, 64, 70, 94}$,
        $\rho_{1, 18, 25, 64, 70, 91, 93}$,
        $\rho_{1, 2, 5, 17, 32, 63, 90, 117}$,
        $\rho_{1, 8, 16, 30, 33, 34, 36}$,
        $\rho_{1, 4, 5, 19, 39, 64, 101, 109}$,
        $\rho_{1, 10, 59, 87, 94}$,
        $\rho_{1, 4, 26, 50, 94, 103, 105}$,
        $\rho_{10, 59, 87, 94, 97}$,
        $\rho_{1, 14, 18, 20, 50, 97, 115}$,
        $\rho_{11, 24, 26, 50, 97, 116}$,
        $\rho_{4, 20, 50, 87, 102}$,
        $\rho_{1, 2, 9, 38, 49, 88}$,
        $\rho_{1, 2, 3, 5, 70, 94, 101}$,
        $\rho_{2, 3, 8, 16, 32, 37, 49}$,
        $\rho_{1, 21, 25, 38, 87, 102}$,
        $\rho_{18, 27, 65, 91, 99}$,
        $\rho_{1, 11, 21, 33, 55, 102, 106}$,
        $\rho_{1, 2, 9, 32, 38, 49, 88, 106}$,
        $\rho_{1, 10, 17, 31, 90, 97, 117}$,
        $\rho_{11, 14, 21, 27, 57, 115}$,
        $\rho_{1, 5, 10, 17, 27, 90, 102, 117}$,
        $\rho_{1, 3, 5, 50, 57, 70, 116}$,
        $\rho_{1, 3, 8, 16, 17, 35, 82, 95}$,
        $\rho_{9, 12, 17, 50, 63, 66, 116}$

    \item[$\eta_{951} - \eta_{962}$]
        $\rho_{1, 2, 5, 16, 19, 55, 59, 97}$,
        $\rho_{1, 3, 16, 17, 35, 82, 96, 109}$,
        $\rho_{1, 3, 56, 59, 87}$,
        $\rho_{1, 3, 11, 32, 39, 42, 81, 95}$,
        $\rho_{10, 11, 12, 55, 59, 97}$,
        $\rho_{1, 2, 9, 11, 32, 63, 65, 94}$,
        $\rho_{24, 26, 50, 63, 87}$,
        $\rho_{1, 4, 9, 13, 24, 50, 94, 103}$,
        $\rho_{17, 25, 50, 63, 66, 87}$,
        $\rho_{1, 2, 3, 4, 18, 64, 70, 94}$,
        $\rho_{11, 26, 50, 63, 87}$,
        $\rho_{1, 2, 9, 38, 49, 88, 97}$

\end{description}
\end{flushleft}

\section{Proofs of Theorems (\ref{equivalence classes, bound, E7}) and
(\ref{equivalence classes, bound, E8})}

\label{appendix B}

Recall, an element $(w, i, j) \in \Gamma(W)$ is \textit{self-dual} if $w =
w^{-1}$ and $i = j$. We first define specific elements
\[
    \Lambda_1,\dots, \Lambda_{12}, \Theta_1,\dots, \Theta_{23}, \Psi_1,\dots,
    \Psi_{23} \in \Gamma(W(E_7))
\]
and
\[
    \Xi_1,\dots, \Xi_{56}, \Omega_1,\dots, \Omega_{119}, \Upsilon_1,\dots,
    \Upsilon_{308} \in \Gamma(W(E_8)).
\]

They satisfy the following properties:
\begin{enumerate}

    \item $\Lambda_{1},\dots, \Lambda_{12}$ are of the form $\llbracket
        w\rrbracket$ with $w \in BiGr_\perp(E_7) \backslash
        BiGr_\perp^\circ(E_7)$,

    \item $\Xi_{1},\dots, \Xi_{56}$ are of the form $\llbracket w\rrbracket$
        with $w \in BiGr_\perp(E_8) \backslash BiGr_\perp^\circ(E_8)$,

    \item $\Theta_{1},\dots, \Theta_{23}$ are self-dual elements in
        $\Gamma(W(E_7))$,

    \item $\Omega_{1},\dots, \Omega_{119}$ are self-dual elements in
        $\Gamma(W(E_8))$.

\end{enumerate}
After defining these elements, we list some reductions involving them. These
reductions will show specific equivalences of the forms $\llbracket \zeta_i
\rrbracket \xleftrightarrow{\hspace{4mm}} \llbracket \zeta_j \rrbracket$,
$\llbracket \zeta_i \rrbracket \xleftrightarrow{\hspace{4mm}} \llbracket
\zeta_i^{-1} \rrbracket$, $\llbracket \eta_i \rrbracket
\xleftrightarrow{\hspace{4mm}} \llbracket \eta_j \rrbracket$, or $\llbracket
\eta_i \rrbracket \xleftrightarrow{\hspace{4mm}} \llbracket \eta_i^{-1}
\rrbracket$.  The equivalences that follow from these reductions prove Theorem
(\ref{equivalence classes, bound, E8}). For example, the reductions
\[
    \Upsilon_{165} \xlongrightarrow{L} \Xi_{29}, \hspace{8pt} \Upsilon_{165}
    \xlongrightarrow{R} \llbracket \eta_{428} \rrbracket,
    \hspace{8pt}
    \Upsilon_{166} \xlongrightarrow{L} \Xi_{29}, \hspace{8pt} \Upsilon_{166}
    \xlongrightarrow{R} \llbracket \eta_{420} \rrbracket, \hspace{8pt}
    \Omega_{69} \xlongrightarrow{L} \Xi_{29},
\]
prove $\Xi_{29} \xleftrightarrow{\hspace{4mm}} \Omega_{69}
\xleftrightarrow{\hspace{4mm}} \llbracket \eta_{420} \rrbracket
\xleftrightarrow{\hspace{4mm}} \llbracket \eta_{428} \rrbracket$. Since
$\Omega_{69}$ is self-dual, Corollary (\ref{corollary, a}) can be applied to
give us
\[
    \Xi_{29} \xleftrightarrow{\hspace{4mm}} \Xi_{29}^*
    \xleftrightarrow{\hspace{4mm}} \llbracket \eta_{420}
    \rrbracket\xleftrightarrow{\hspace{4mm}} \llbracket \eta_{420}^{-1}
    \rrbracket \xleftrightarrow{\hspace{4mm}} \llbracket \eta_{428} \rrbracket
    \xleftrightarrow{\hspace{4mm}} \llbracket \eta_{428}^{-1} \rrbracket.
\]
We also point out that since $\Xi_{29} \in \Gamma(W(E_8))$ is of the form
$\llbracket w \rrbracket$ with $w$ bigrassmannian and lacking full support, we
can naturally view $\Xi_{29}$ as belonging to $\Gamma (W(X_n))$ for some Lie
type $X_n$ associated to a proper subgraph of the type $E_8$ Dynkin diagram
(i.e. $X_n$ is a simply-laced Lie type with rank $n \leq 7$).  In this
situation, we have
\[
    \Xi_{29} = \llbracket s_4 s_3 s_2 s_4 s_5 s_6 s_1 s_3 s_4 s_5 s_3 s_2 s_4
    s_1 s_3 \rrbracket \in \Gamma(W(E_8)).
\]
Observe that only the simple reflections $s_1, \dots, s_6$ are involved in the
reduced expression above. These generate the type $E_6$ Weyl group, and hence
we can identify $\Xi_{29}$ with an element of $\Gamma(W(E_6))$. In this
particular case, we identify it with $\llbracket \nu_{13}^{-1} \rrbracket \in
\Gamma(W(E_6))$. From Table (\ref{table, E6}), ${\mathcal N}(\llbracket
\nu_{13}^{-1} \rrbracket) = 3$, and thus ${\mathcal N}(\Xi_{29}) = 3$ also.
Therefore,
\[
    {\mathcal N}(\llbracket \eta_{420} \rrbracket) = {\mathcal N}(\llbracket
    \eta_{420}^{-1} \rrbracket) = {\mathcal N}(\llbracket \eta_{428}
    \rrbracket) = {\mathcal N}(\llbracket \eta_{428}^{-1} \rrbracket) = 3.
\]
The main point we make with this example is that nilpotency indices can be
sometimes be found using results from a smaller rank setting.

The elements $\Lambda_1,\dots, \Lambda_{12}$, $\Theta_1,\dots, \Theta_{23}$,
$\Psi_1,\dots, \Psi_{23}$, $\Xi_1,\dots, \Xi_{56}$, $\Omega_1,\dots,
\Omega_{119}$, and $\Upsilon_1,\dots, \Upsilon_{308}$ are given by the
following:

\begin{flushleft}
\begin{description}

    \item[$\Lambda_{1} - \Lambda_{12}$]
        $\llbracket \rho_{50, 51, 62} \rrbracket$,
        $\llbracket \rho_{58, 59} \rrbracket$,
        $\llbracket \rho_{45, 61} \rrbracket$,
        $\llbracket \rho_{53} \rrbracket$,
        $\llbracket \rho_{26} \rrbracket$,
        $\llbracket \rho_{1, 2, 4, 20, 54} \rrbracket$,
        $\llbracket \rho_{9, 20, 54} \rrbracket$,
        $\llbracket \rho_{4, 16, 50} \rrbracket$,
        $\llbracket \rho_{2, 9, 15, 20, 54} \rrbracket$,
        $\llbracket \rho_{2, 8, 11, 15, 51} \rrbracket$,
        $\llbracket \rho_{4, 22} \rrbracket$,
        $\llbracket \rho_{22, 45} \rrbracket$,

    \item[$\Psi_{1} - \Psi_{23}$]
        $(\rho_{1, 6, 11, 15, 51, 52}, \text{\small 4}, \text{\small 6})$,
        $(\rho_{3, 4, 12, 16, 26, 52}, \text{\small 5}, \text{\small 6})$,
        $(\rho_{1, 4, 14, 20, 24, 44, 62}, \text{\small 5}, \text{\small 7})$,
        $(\rho_{1, 5, 12, 36, 38, 39, 61}, \text{\small 6}, \text{\small 7})$,
        $(\rho_{1, 3, 4, 5, 6, 18, 46}, \text{\small 4}, \text{\small 5})$,
        $(\rho_{5, 13, 23, 32, 55, 60}, \text{\small 5}, \text{\small 6})$,
        $(\rho_{2, 9, 11, 23, 41, 46}, \text{\small 5}, \text{\small 4})$,
        $(\rho_{15, 20, 37, 41, 46}, \text{\small 4}, \text{\small 5})$,
        $(\rho_{1, 22, 36, 45, 59}, \text{\small 4}, \text{\small 5})$,
        $(\rho_{14, 15, 26, 32, 57}, \text{\small 2}, \text{\small 4})$,
        $(\rho_{1, 12, 15, 16, 17, 60}, \text{\small 3}, \text{\small 4})$,
        $(\rho_{8, 24, 32, 42, 57}, \text{\small 4}, \text{\small 4})$,
        $(\rho_{12, 19, 42, 48, 50, 53}, \text{\small 3}, \text{\small 4})$,
        $(\rho_{1, 4, 6, 8, 12, 35, 39}, \text{\small 1}, \text{\small 4})$,
        $(\rho_{1, 4, 7, 22, 23, 41, 45}, \text{\small 5}, \text{\small 4})$,
        $(\rho_{1, 6, 12, 15, 43, 44, 54}, \text{\small 4}, \text{\small 5})$,
        $(\rho_{2, 3, 19, 23, 29}, \text{\small 4}, \text{\small 2})$,
        $(\rho_{3, 11, 24, 26, 34}, \text{\small 4}, \text{\small 3})$,
        $(\rho_{2, 5, 12, 17, 41, 55}, \text{\small 4}, \text{\small 6})$,
        $(\rho_{3, 4, 22, 24, 47, 54}, \text{\small 2}, \text{\small 5})$,
        $(\rho_{7, 16, 20, 33, 37}, \text{\small 6}, \text{\small 2})$,
        $(\rho_{4, 6, 12, 16, 27, 47}, \text{\small 4}, \text{\small 5})$,
        $(\rho_{2, 14, 26, 41, 47}, \text{\small 5}, \text{\small 4})$,

    \item[$\Theta_{1} - \Theta_{23}$]
        $(\rho_{25, 50, 51, 62}, \text{\small 6}, \text{\small 6})$,
        $(\rho_{25, 58, 59}, \text{\small 7}, \text{\small 7})$,
        $(\rho_{25, 45, 61}, \text{\small 5}, \text{\small 5})$,
        $(\rho_{25, 53}, \text{\small 6}, \text{\small 6})$,
        $(\rho_{13, 26, 44, 53}, \text{\small 4}, \text{\small 4})$,
        $(\rho_{14, 35, 39}, \text{\small 4}, \text{\small 4})$,
        $(\rho_{20, 35, 40, 44}, \text{\small 4}, \text{\small 4})$,
        $(\rho_{16, 31, 44, 59}, \text{\small 2}, \text{\small 2})$,
        $(\rho_{2, 13, 39, 61}, \text{\small 4}, \text{\small 4})$,
        $(\rho_{4, 32, 57, 58}, \text{\small 3}, \text{\small 3})$,
        $(\rho_{3, 17, 33, 40, 63}, \text{\small 1}, \text{\small 1})$,
        $(\rho_{18, 46, 58}, \text{\small 5}, \text{\small 5})$,
        $(\rho_{1, 13, 18, 58}, \text{\small 2}, \text{\small 2})$,
        $(\rho_{4, 28, 37, 38}, \text{\small 4}, \text{\small 4})$,
        $(\rho_{2, 11, 39, 63}, \text{\small 4}, \text{\small 4})$,
        $(\rho_{11, 21, 46, 63}, \text{\small 5}, \text{\small 5})$,
        $(\rho_{7, 18, 28, 53}, \text{\small 6}, \text{\small 6})$,
        $(\rho_{11, 28, 44, 45}, \text{\small 5}, \text{\small 5})$,
        $(\rho_{11, 15, 40, 48, 63}, \text{\small 4}, \text{\small 4})$,
        $(\rho_{1, 13, 24}, \text{\small 3}, \text{\small 3})$,
        $(\rho_{3, 10, 26, 36, 56}, \text{\small 4}, \text{\small 4})$,
        $(\rho_{10, 17, 33, 43}, \text{\small 5}, \text{\small 5})$,
        $(\rho_{8, 26, 41, 56}, \text{\small 4}, \text{\small 4})$,

    \item[$\Omega_{1} - \Omega_{50}$]
        $(\rho_{1, 53}, \text{\small 3}, \text{\small 3})$,
        $(\rho_{3, 37, 69}, \text{\small 4}, \text{\small 4})$,
        $(\rho_{3, 16, 20, 47, 77}, \text{\small 3}, \text{\small 3})$,
        $(\rho_{79, 83, 120}, \text{\small 7}, \text{\small 7})$,
        $(\rho_{17, 30, 86, 92}, \text{\small 4}, \text{\small 4})$,
        $(\rho_{16, 34, 61, 73}, \text{\small 3}, \text{\small 3})$,
        $(\rho_{27, 54, 57, 92}, \text{\small 5}, \text{\small 5})$,
        $(\rho_{56, 88, 92}, \text{\small 5}, \text{\small 5})$,
        $(\rho_{9, 84}, \text{\small 4}, \text{\small 4})$,
        $(\rho_{2, 3, 44, 71, 89}, \text{\small 4}, \text{\small 4})$,
        $(\rho_{71, 77, 89, 118}, \text{\small 5}, \text{\small 5})$,
        $(\rho_{10, 20, 47, 67}, \text{\small 2}, \text{\small 2})$,
        $(\rho_{13, 52, 79, 103}, \text{\small 5}, \text{\small 5})$,
        $(\rho_{6, 79, 85, 97}, \text{\small 6}, \text{\small 6})$,
        $(\rho_{5, 23, 47, 60}, \text{\small 6}, \text{\small 6})$,
        $(\rho_{2, 37, 41, 62, 108}, \text{\small 4}, \text{\small 4})$,
        $(\rho_{9, 30, 61, 86}, \text{\small 4}, \text{\small 4})$,
        $(\rho_{11, 53, 88}, \text{\small 3}, \text{\small 3})$,
        $(\rho_{25, 37, 88, 92}, \text{\small 5}, \text{\small 5})$,
        $(\rho_{4, 23, 25, 38}, \text{\small 5}, \text{\small 5})$,
        $(\rho_{10, 31, 34, 84}, \text{\small 5}, \text{\small 5})$,
        $(\rho_{53, 92}, \text{\small 5}, \text{\small 5})$,
        $(\rho_{28, 61, 100}, \text{\small 6}, \text{\small 6})$,
        $(\rho_{61, 100}, \text{\small 6}, \text{\small 6})$,
        $(\rho_{20, 53, 92}, \text{\small 5}, \text{\small 5})$,
        $(\rho_{21, 29, 72, 84}, \text{\small 4}, \text{\small 4})$,
        $(\rho_{20, 47, 53}, \text{\small 3}, \text{\small 3})$,
        $(\rho_{4,  53, 61}, \text{\small 3}, \text{\small 3})$,
        $(\rho_{37, 92, 115}, \text{\small 5}, \text{\small 5})$,
        $(\rho_{28, 84, 113}, \text{\small 4}, \text{\small 4})$,
        $(\rho_{28, 48, 49, 61}, \text{\small 2}, \text{\small 2})$,
        $(\rho_{14, 29, 65, 84}, \text{\small 4}, \text{\small 4})$,
        $(\rho_{36, 59, 100, 114}, \text{\small 6}, \text{\small 6})$,
        $(\rho_{22, 35, 53, 63}, \text{\small 3}, \text{\small 3})$,
        $(\rho_{38, 72, 91, 108}, \text{\small 5}, \text{\small 5})$,
        $(\rho_{3, 44, 65, 98}, \text{\small 4}, \text{\small 4})$,
        $(\rho_{2, 44, 94, 99}, \text{\small 4}, \text{\small 4})$,
        $(\rho_{27, 79, 92}, \text{\small 5}, \text{\small 5})$,
        $(\rho_{2, 44, 92}, \text{\small 4}, \text{\small 4})$,
        $(\rho_{20, 53, 83, 102}, \text{\small 4}, \text{\small 4})$,
        $(\rho_{11, 53, 94, 99}, \text{\small 3}, \text{\small 3})$,
        $(\rho_{4, 32, 52, 64, 107}, \text{\small 5}, \text{\small 5})$,
        $(\rho_{8, 39, 85, 115}, \text{\small 6}, \text{\small 6})$,
        $(\rho_{11, 53, 63}, \text{\small 3}, \text{\small 3})$,
        $(\rho_{4, 23, 25, 66}, \text{\small 4}, \text{\small 4})$,
        $(\rho_{9, 20, 84}, \text{\small 4}, \text{\small 4})$,
        $(\rho_{12, 60, 79}, \text{\small 6}, \text{\small 6})$,
        $(\rho_{1, 20, 53}, \text{\small 3}, \text{\small 3})$,
        $(\rho_{4, 30, 61, 67}, \text{\small 5}, \text{\small 5})$,
        $(\rho_{1, 20, 47, 53}, \text{\small 3}, \text{\small 3})$,

    \item[$\Omega_{51} - \Omega_{100}$]
        $(\rho_{4, 31, 61}, \text{\small 2}, \text{\small 2})$,
        $(\rho_{22, 35, 92}, \text{\small 2}, \text{\small 2})$,
        $(\rho_{13, 38, 79, 85, 91}, \text{\small 6}, \text{\small 6})$,
        $(\rho_{21, 25, 46, 58}, \text{\small 1}, \text{\small 1})$,
        $(\rho_{37, 45, 91, 110}, \text{\small 5}, \text{\small 5})$,
        $(\rho_{72, 108}, \text{\small 7}, \text{\small 7})$,
        $(\rho_{12, 24, 70}, \text{\small 5}, \text{\small 5})$,
        $(\rho_{32, 78, 84}, \text{\small 6}, \text{\small 6})$,
        $(\rho_{20, 54, 97}, \text{\small 3}, \text{\small 3})$,
        $(\rho_{20, 32, 43}, \text{\small 1}, \text{\small 1})$,
        $(\rho_{61, 70, 117}, \text{\small 5}, \text{\small 5})$,
        $(\rho_{2, 12, 59}, \text{\small 4}, \text{\small 4})$,
        $(\rho_{31, 61, 93, 97, 120}, \text{\small 5}, \text{\small 5})$,
        $(\rho_{4, 32, 52, 64}, \text{\small 5}, \text{\small 5})$,
        $(\rho_{1, 53, 59, 116}, \text{\small 4}, \text{\small 4})$,
        $(\rho_{37, 71, 101}, \text{\small 6}, \text{\small 6})$,
        $(\rho_{7, 20, 32, 99}, \text{\small 6}, \text{\small 6})$,
        $(\rho_{31, 51, 90}, \text{\small 5}, \text{\small 5})$,
        $(\rho_{4, 53, 61, 107}, \text{\small 3}, \text{\small 3})$,
        $(\rho_{19, 32, 64, 84}, \text{\small 4}, \text{\small 4})$,
        $(\rho_{27, 40, 72, 92}, \text{\small 5}, \text{\small 5})$,
        $(\rho_{2, 31, 44, 94, 99}, \text{\small 4}, \text{\small 4})$,
        $(\rho_{3, 31, 44, 93, 97}, \text{\small 4}, \text{\small 4})$,
        $(\rho_{13, 25, 54, 63}, \text{\small 3}, \text{\small 3})$,
        $(\rho_{2, 44, 63, 65}, \text{\small 4}, \text{\small 4})$,
        $(\rho_{25, 37, 92, 115}, \text{\small 5}, \text{\small 5})$,
        $(\rho_{52, 64, 88, 100}, \text{\small 5}, \text{\small 5})$,
        $(\rho_{20, 70, 94, 105}, \text{\small 5}, \text{\small 5})$,
        $(\rho_{22, 31, 35, 90}, \text{\small 2}, \text{\small 2})$,
        $(\rho_{3, 31, 44, 56}, \text{\small 4}, \text{\small 4})$,
        $(\rho_{14, 51, 70}, \text{\small 3}, \text{\small 3})$,
        $(\rho_{19, 71, 84}, \text{\small 4}, \text{\small 4})$,
        $(\rho_{30, 76, 81}, \text{\small 4}, \text{\small 4})$,
        $(\rho_{31, 76, 90, 97}, \text{\small 5}, \text{\small 5})$,
        $(\rho_{13, 23, 34, 76}, \text{\small 3}, \text{\small 3})$,
        $(\rho_{32, 50, 62, 99}, \text{\small 4}, \text{\small 4})$,
        $(\rho_{9, 44, 70, 94, 112}, \text{\small 4}, \text{\small 4})$,
        $(\rho_{5, 50, 62, 71, 77}, \text{\small 5}, \text{\small 5})$,
        $(\rho_{24, 30, 76, 81}, \text{\small 4}, \text{\small 4})$,
        $(\rho_{19, 71, 84, 96}, \text{\small 4}, \text{\small 4})$,
        $(\rho_{10, 17, 31, 84}, \text{\small 4}, \text{\small 4})$,
        $(\rho_{38, 91, 104}, \text{\small 5}, \text{\small 5})$,
        $(\rho_{5, 39, 56, 92}, \text{\small 5}, \text{\small 5})$,
        $(\rho_{9, 10, 20, 84}, \text{\small 4}, \text{\small 4})$,
        $(\rho_{27, 92, 104}, \text{\small 5}, \text{\small 5})$,
        $(\rho_{19, 84, 96}, \text{\small 4}, \text{\small 4})$,
        $(\rho_{2, 59, 87}, \text{\small 4}, \text{\small 4})$,
        $(\rho_{16, 20, 83, 102}, \text{\small 4}, \text{\small 4})$,
        $(\rho_{2, 12, 59, 87}, \text{\small 4}, \text{\small 4})$,
        $(\rho_{35, 100, 112}, \text{\small 6}, \text{\small 6})$,

    \item[$\Omega_{101} - \Omega_{119}$]
        $(\rho_{3, 20, 54, 97}, \text{\small 3}, \text{\small 3})$,
        $(\rho_{30, 48, 61, 109}, \text{\small 4}, \text{\small 4})$,
        $(\rho_{38, 50, 57, 91}, \text{\small 5}, \text{\small 5})$,
        $(\rho_{18, 55, 59, 75}, \text{\small 4}, \text{\small 4})$,
        $(\rho_{30, 64, 83, 96}, \text{\small 4}, \text{\small 4})$,
        $(\rho_{26, 38, 83, 91}, \text{\small 5}, \text{\small 5})$,
        $(\rho_{27, 65, 84, 117}, \text{\small 5}, \text{\small 5})$,
        $(\rho_{2, 37, 41, 62}, \text{\small 4}, \text{\small 4})$,
        $(\rho_{25, 38, 84, 117}, \text{\small 5}, \text{\small 5})$,
        $(\rho_{57, 71, 77, 118}, \text{\small 6}, \text{\small 6})$,
        $(\rho_{64, 70, 88, 94}, \text{\small 5}, \text{\small 5})$,
        $(\rho_{38, 91, 104, 109}, \text{\small 5}, \text{\small 5})$,
        $(\rho_{55, 59, 96}, \text{\small 4}, \text{\small 4})$,
        $(\rho_{17, 31, 55, 66, 90}, \text{\small 4}, \text{\small 4})$,
        $(\rho_{10, 62, 70, 81}, \text{\small 4}, \text{\small 4})$,
        $(\rho_{16, 27, 59, 93, 101}, \text{\small 5}, \text{\small 5})$,
        $(\rho_{8, 19, 51, 85, 93}, \text{\small 4}, \text{\small 4})$,
        $(\rho_{13, 18, 19, 84}, \text{\small 4}, \text{\small 4})$,
        $(\rho_{55, 59, 96, 97}, \text{\small 4}, \text{\small 4})$,

    \item[$\Xi_{1} - \Xi_{56}$]
        $\llbracket \rho_{28} \rrbracket$, $\llbracket \rho_{37} \rrbracket$,
        $\llbracket \rho_{59} \rrbracket$, $\llbracket \rho_{108} \rrbracket$,
        $\llbracket \rho_{70} \rrbracket$, $\llbracket \rho_{30} \rrbracket$,
        $\llbracket \rho_{16, 17, 18, 64, 70} \rrbracket$, $\llbracket
        \rho_{11, 56, 59} \rrbracket$, $\llbracket \rho_{114, 115} \rrbracket$,
        $\llbracket \rho_{20, 21, 35, 38, 59} \rrbracket$, $\llbracket \rho_{1,
        5, 16, 26, 81, 103} \rrbracket$, $\llbracket \rho_{1, 7, 16, 27, 28}
        \rrbracket$, $\llbracket \rho_{31, 90} \rrbracket$, $\llbracket
        \rho_{16, 59} \rrbracket$, $\llbracket \rho_{14, 27} \rrbracket$,
        $\llbracket \rho_{3, 37} \rrbracket$, $\llbracket \rho_{99, 118}
        \rrbracket$, $\llbracket \rho_{1, 7, 32, 33, 38, 56, 88} \rrbracket$,
        $\llbracket \rho_{1, 6, 10, 27, 40, 95} \rrbracket$, $\llbracket
        \rho_{24, 70} \rrbracket$, $\llbracket \rho_{30, 109} \rrbracket$,
        $\llbracket \rho_{1, 4, 7, 16, 23, 27, 28} \rrbracket$, $\llbracket
        \rho_{20, 105} \rrbracket$, $\llbracket \rho_{32, 50, 57} \rrbracket$,
        $\llbracket \rho_{32, 99} \rrbracket$, $\llbracket \rho_{11, 31, 56,
        90} \rrbracket$, $\llbracket \rho_{1, 5, 10, 31, 56, 90, 104}
        \rrbracket$, $\llbracket \rho_{16, 17, 18, 31, 90, 97} \rrbracket$,
        $\llbracket \rho_{2, 9, 12, 17, 97} \rrbracket$, $\llbracket \rho_{11,
        24, 30, 109} \rrbracket$, $\llbracket \rho_{6, 26, 27, 40, 95}
        \rrbracket$, $\llbracket \rho_{16, 23, 31, 57, 81} \rrbracket$,
        $\llbracket \rho_{1, 7, 16, 23, 27, 28} \rrbracket$, $\llbracket
        \rho_{1, 2, 5, 17, 32, 63, 90} \rrbracket$, $\llbracket \rho_{30, 81}
        \rrbracket$, $\llbracket \rho_{31, 90, 97} \rrbracket$, $\llbracket
        \rho_{14, 23, 27} \rrbracket$, $\llbracket \rho_{24, 30, 81}
        \rrbracket$, $\llbracket \rho_{13, 18, 25, 106} \rrbracket$,
        $\llbracket \rho_{105, 106, 119} \rrbracket$, $\llbracket \rho_{1, 3,
        12, 33, 106} \rrbracket$, $\llbracket \rho_{2, 5, 16, 33, 63, 106}
        \rrbracket$, $\llbracket \rho_{4, 18, 64} \rrbracket$, $\llbracket
        \rho_{6, 19, 72} \rrbracket$, $\llbracket \rho_{1, 6, 10, 11, 14, 20,
        56} \rrbracket$, $\llbracket \rho_{64, 97, 111} \rrbracket$,
        $\llbracket \rho_{3, 11, 56} \rrbracket$, $\llbracket \rho_{10, 23,
        102} \rrbracket$, $\llbracket \rho_{19, 57, 109} \rrbracket$,
        $\llbracket \rho_{1, 6, 16, 18, 20, 25, 106} \rrbracket$, $\llbracket
        \rho_{17, 30, 109} \rrbracket$, $\llbracket \rho_{25, 63, 110}
        \rrbracket$, $\llbracket \rho_{12, 33, 106} \rrbracket$, $\llbracket
        \rho_{5, 16, 19, 33, 106} \rrbracket$, $\llbracket \rho_{2, 10, 19, 57,
        96, 110} \rrbracket$, $\llbracket \rho_{19, 56, 57, 109} \rrbracket$,

    \item[$\Upsilon_{1} - \Upsilon_{50}$]
        $(\rho_{25, 61}, \text{\small 5}, \text{\small 2})$,
        $(\rho_{1, 4, 24, 53}, \text{\small 5}, \text{\small 3})$,
        $(\rho_{6, 19, 53, 68}, \text{\small 7}, \text{\small 3})$,
        $(\rho_{5, 10, 53, 60}, \text{\small 6}, \text{\small 3})$,
        $(\rho_{53, 59, 85, 92}, \text{\small 5}, \text{\small 4})$,
        $(\rho_{6, 17, 57, 78, 86, 95}, \text{\small 6}, \text{\small 4})$,
        $(\rho_{47, 50, 93, 102}, \text{\small 7}, \text{\small 4})$,
        $(\rho_{16, 19, 79, 98, 99}, \text{\small 4}, \text{\small 6})$,
        $(\rho_{10, 17, 70, 78, 94}, \text{\small 4}, \text{\small 5})$,
        $(\rho_{19, 40, 71, 85, 97}, \text{\small 4}, \text{\small 6})$,
        $(\rho_{1, 17, 23, 53, 94, 99}, \text{\small 4}, \text{\small 3})$,
        $(\rho_{1, 11, 55, 59, 96, 100}, \text{\small 4}, \text{\small 4})$,
        $(\rho_{50, 70, 74, 118}, \text{\small 5}, \text{\small 7})$,
        $(\rho_{2, 3, 30, 86, 92}, \text{\small 2}, \text{\small 4})$,
        $(\rho_{4, 63, 74, 100}, \text{\small 2}, \text{\small 5})$,
        $(\rho_{5, 33, 55, 66, 77}, \text{\small 5}, \text{\small 3})$,
        $(\rho_{5, 9, 64, 71, 84, 89}, \text{\small 5}, \text{\small 4})$,
        $(\rho_{4, 9, 37, 48, 52, 71}, \text{\small 5}, \text{\small 4})$,
        $(\rho_{6, 16, 25, 39, 49, 61}, \text{\small 5}, \text{\small 2})$,
        $(\rho_{5, 9, 32, 71, 84, 89}, \text{\small 5}, \text{\small 4})$,
        $(\rho_{33, 34, 74, 75}, \text{\small 7}, \text{\small 3})$,
        $(\rho_{3, 18, 27, 74, 75}, \text{\small 6}, \text{\small 3})$,
        $(\rho_{21, 35, 93, 118}, \text{\small 6}, \text{\small 7})$,
        $(\rho_{4, 7, 35, 63, 91}, \text{\small 6}, \text{\small 5})$,
        $(\rho_{13, 17, 40, 70, 92}, \text{\small 6}, \text{\small 5})$,
        $(\rho_{13, 17, 40, 92, 115}, \text{\small 6}, \text{\small 5})$,
        $(\rho_{4, 7, 42, 54, 63, 68}, \text{\small 6}, \text{\small 2})$,
        $(\rho_{2, 12, 13, 83, 91, 99}, \text{\small 6}, \text{\small 4})$,
        $(\rho_{9, 14, 32, 51, 70}, \text{\small 4}, \text{\small 3})$,
        $(\rho_{23, 27, 65, 84, 117}, \text{\small 5}, \text{\small 5})$,
        $(\rho_{2, 14, 32, 51, 70, 88}, \text{\small 3}, \text{\small 3})$,
        $(\rho_{13, 19, 28, 81, 84}, \text{\small 5}, \text{\small 4})$,
        $(\rho_{4, 33, 44, 90, 102}, \text{\small 4}, \text{\small 4})$,
        $(\rho_{11, 30, 48, 81, 92}, \text{\small 4}, \text{\small 4})$,
        $(\rho_{2, 16, 55, 64, 69, 90}, \text{\small 2}, \text{\small 4})$,
        $(\rho_{11, 22, 34, 61, 97}, \text{\small 5}, \text{\small 3})$,
        $(\rho_{10, 17, 24, 62, 71, 84}, \text{\small 4}, \text{\small 4})$,
        $(\rho_{31, 35, 44, 70, 73}, \text{\small 3}, \text{\small 6})$,
        $(\rho_{5, 9, 64, 84}, \text{\small 5}, \text{\small 4})$,
        $(\rho_{1, 17, 25, 44}, \text{\small 3}, \text{\small 4})$,
        $(\rho_{6, 14, 18, 60}, \text{\small 4}, \text{\small 6})$,
        $(\rho_{4, 9, 19, 30, 38, 94, 109}, \text{\small 2}, \text{\small 5})$,
        $(\rho_{16, 21, 30, 85, 96}, \text{\small 4}, \text{\small 5})$,
        $(\rho_{7, 10, 28, 64, 101, 109}, \text{\small 5}, \text{\small 6})$,
        $(\rho_{5, 7, 28, 90, 94, 102}, \text{\small 4}, \text{\small 5})$,
        $(\rho_{2, 6, 13, 25, 52, 102, 105}, \text{\small 4}, \text{\small 5})$,
        $(\rho_{3, 11, 28, 94, 99, 100}, \text{\small 4}, \text{\small 5})$,
        $(\rho_{13, 18, 35, 77, 87}, \text{\small 6}, \text{\small 3})$,
        $(\rho_{11, 34, 47, 51, 99}, \text{\small 5}, \text{\small 3})$,
        $(\rho_{5, 16, 17, 37, 92}, \text{\small 5}, \text{\small 5})$,

    \item[$\Upsilon_{51} - \Upsilon_{100}$]
        $(\rho_{10, 24, 25, 37, 62, 85}, \text{\small 5}, \text{\small 4})$,
        $(\rho_{7, 23, 26, 38, 42, 84}, \text{\small 6}, \text{\small 4})$,
        $(\rho_{25, 38, 84, 104, 118}, \text{\small 5}, \text{\small 6})$,
        $(\rho_{8, 11, 53, 63, 88}, \text{\small 3}, \text{\small 3})$,
        $(\rho_{7, 12, 84, 103, 114}, \text{\small 7}, \text{\small 4})$,
        $(\rho_{4, 9, 52, 71, 78, 88}, \text{\small 5}, \text{\small 5})$,
        $(\rho_{8, 12, 34, 59, 75, 94}, \text{\small 2}, \text{\small 4})$,
        $(\rho_{24, 59, 69, 77, 87}, \text{\small 5}, \text{\small 5})$,
        $(\rho_{4, 6, 30, 53, 61}, \text{\small 3}, \text{\small 4})$,
        $(\rho_{40, 61, 115}, \text{\small 8}, \text{\small 3})$,
        $(\rho_{37, 61}, \text{\small 2}, \text{\small 3})$,
        $(\rho_{29, 37, 114}, \text{\small 3}, \text{\small 2})$,
        $(\rho_{3, 8, 15, 18, 53, 78}, \text{\small 5}, \text{\small 3})$,
        $(\rho_{6, 40, 66, 89}, \text{\small 6}, \text{\small 1})$,
        $(\rho_{3, 21, 31, 66, 73}, \text{\small 4}, \text{\small 6})$,
        $(\rho_{3, 37, 69, 75, 77}, \text{\small 5}, \text{\small 4})$,
        $(\rho_{4, 23, 34, 60, 89, 111}, \text{\small 5}, \text{\small 6})$,
        $(\rho_{11, 59, 77, 104, 118}, \text{\small 5}, \text{\small 6})$,
        $(\rho_{3, 8, 53, 59}, \text{\small 4}, \text{\small 4})$,
        $(\rho_{4, 53, 58, 59, 78, 83}, \text{\small 7}, \text{\small 5})$,
        $(\rho_{7, 15, 33, 56, 59, 67, 75}, \text{\small 4}, \text{\small 5})$,
        $(\rho_{10, 20, 22, 25, 60, 108}, \text{\small 5}, \text{\small 6})$,
        $(\rho_{1, 37, 45, 53}, \text{\small 3}, \text{\small 2})$,
        $(\rho_{1, 36, 69, 114}, \text{\small 3}, \text{\small 1})$,
        $(\rho_{11, 13, 15, 28, 52, 75}, \text{\small 5}, \text{\small 3})$,
        $(\rho_{12, 17, 29, 38, 54, 66}, \text{\small 6}, \text{\small 3})$,
        $(\rho_{4, 12, 15, 52, 75}, \text{\small 2}, \text{\small 3})$,
        $(\rho_{4, 18, 28, 64, 84, 117}, \text{\small 4}, \text{\small 5})$,
        $(\rho_{3, 8, 11, 24, 52, 66}, \text{\small 4}, \text{\small 4})$,
        $(\rho_{13, 19, 61, 100}, \text{\small 4}, \text{\small 6})$,
        $(\rho_{3, 4, 5, 52, 100}, \text{\small 4}, \text{\small 5})$,
        $(\rho_{11, 20, 25, 54, 97}, \text{\small 5}, \text{\small 3})$,
        $(\rho_{2, 3, 14, 29, 33, 65}, \text{\small 4}, \text{\small 1})$,
        $(\rho_{14, 32, 51, 90}, \text{\small 3}, \text{\small 3})$,
        $(\rho_{2, 4, 5, 19, 61, 71}, \text{\small 4}, \text{\small 2})$,
        $(\rho_{3, 5, 12, 33, 92, 106}, \text{\small 3}, \text{\small 5})$,
        $(\rho_{19, 32, 71, 84}, \text{\small 4}, \text{\small 4})$,
        $(\rho_{18, 34, 54, 57, 77}, \text{\small 6}, \text{\small 3})$,
        $(\rho_{16, 20, 32, 56, 100}, \text{\small 2}, \text{\small 6})$,
        $(\rho_{20, 25, 70, 93}, \text{\small 2}, \text{\small 5})$,
        $(\rho_{11, 24, 35, 56, 75}, \text{\small 5}, \text{\small 3})$,
        $(\rho_{3, 18, 19, 28, 46}, \text{\small 2}, \text{\small 1})$,
        $(\rho_{1, 4, 22, 32, 34, 45}, \text{\small 5}, \text{\small 1})$,
        $(\rho_{3, 15, 31, 36, 41}, \text{\small 2}, \text{\small 1})$,
        $(\rho_{5, 12, 29, 38, 46, 69}, \text{\small 4}, \text{\small 1})$,
        $(\rho_{1, 5, 9, 17, 23, 35, 53}, \text{\small 4}, \text{\small 3})$,
        $(\rho_{4, 15, 24, 35, 53}, \text{\small 5}, \text{\small 3})$,
        $(\rho_{1, 11, 24, 46, 56}, \text{\small 2}, \text{\small 3})$,
        $(\rho_{6, 71, 77, 88, 89, 118}, \text{\small 5}, \text{\small 4})$,
        $(\rho_{6, 14, 79, 85, 96, 97}, \text{\small 6}, \text{\small 5})$,

    \item[$\Upsilon_{101} - \Upsilon_{150}$]
        $(\rho_{2, 5, 9, 22, 35, 84}, \text{\small 4}, \text{\small 4})$,
        $(\rho_{9, 17, 25, 57, 86, 92}, \text{\small 3}, \text{\small 4})$,
        $(\rho_{6, 13, 18, 65, 94, 100}, \text{\small 3}, \text{\small 5})$,
        $(\rho_{4, 6, 19, 70, 77, 105}, \text{\small 5}, \text{\small 5})$,
        $(\rho_{30, 34, 89, 99, 118}, \text{\small 5}, \text{\small 6})$,
        $(\rho_{1, 9, 18, 53, 57, 103}, \text{\small 4}, \text{\small 3})$,
        $(\rho_{23, 25, 61, 66, 76}, \text{\small 4}, \text{\small 4})$,
        $(\rho_{3, 10, 33, 34, 63, 66, 98}, \text{\small 5}, \text{\small 6})$,
        $(\rho_{6, 9, 37, 72, 98}, \text{\small 4}, \text{\small 6})$,
        $(\rho_{9, 13, 14, 18, 71, 75, 91}, \text{\small 5}, \text{\small 6})$,
        $(\rho_{6, 11, 21, 42, 58, 84}, \text{\small 5}, \text{\small 4})$,
        $(\rho_{3, 30, 61, 86}, \text{\small 3}, \text{\small 4})$,
        $(\rho_{2, 5, 26, 61, 103}, \text{\small 4}, \text{\small 2})$,
        $(\rho_{4, 16, 23, 79, 98, 99}, \text{\small 2}, \text{\small 6})$,
        $(\rho_{3, 4, 63, 78, 93}, \text{\small 2}, \text{\small 5})$,
        $(\rho_{1, 5, 13, 41, 42, 67}, \text{\small 6}, \text{\small 1})$,
        $(\rho_{4, 12, 26, 47, 67}, \text{\small 5}, \text{\small 1})$,
        $(\rho_{4, 16, 33, 48, 52, 66}, \text{\small 5}, \text{\small 4})$,
        $(\rho_{3, 17, 30, 37, 41, 42, 52}, \text{\small 5}, \text{\small 4})$,
        $(\rho_{3, 8, 15, 53, 59}, \text{\small 4}, \text{\small 4})$,
        $(\rho_{1, 3, 21, 42, 74, 84}, \text{\small 3}, \text{\small 4})$,
        $(\rho_{9, 10, 59, 61, 76}, \text{\small 3}, \text{\small 4})$,
        $(\rho_{4, 23, 47, 60, 72, 74}, \text{\small 5}, \text{\small 6})$,
        $(\rho_{22, 51, 66, 74, 116}, \text{\small 4}, \text{\small 7})$,
        $(\rho_{2, 14, 30, 74, 85}, \text{\small 4}, \text{\small 7})$,
        $(\rho_{14, 21, 39, 60, 109}, \text{\small 4}, \text{\small 6})$,
        $(\rho_{6, 30, 50, 71, 84}, \text{\small 6}, \text{\small 4})$,
        $(\rho_{1, 4, 13, 24, 35, 53}, \text{\small 5}, \text{\small 3})$,
        $(\rho_{11, 12, 31, 35, 42, 61}, \text{\small 5}, \text{\small 2})$,
        $(\rho_{11, 12, 37, 92}, \text{\small 5}, \text{\small 5})$,
        $(\rho_{13, 15, 21, 42, 53, 63}, \text{\small 5}, \text{\small 3})$,
        $(\rho_{5, 9, 10, 11, 28, 84}, \text{\small 5}, \text{\small 4})$,
        $(\rho_{9, 13, 21, 25, 36, 64, 84}, \text{\small 5}, \text{\small 4})$,
        $(\rho_{1, 3, 4, 6, 8, 44, 63}, \text{\small 3}, \text{\small 4})$,
        $(\rho_{16, 25, 32, 59, 77, 83}, \text{\small 6}, \text{\small 5})$,
        $(\rho_{16, 25, 64, 69, 77, 119}, \text{\small 6}, \text{\small 5})$,
        $(\rho_{1, 11, 24, 32, 52}, \text{\small 3}, \text{\small 5})$,
        $(\rho_{5, 9, 11, 44, 95}, \text{\small 3}, \text{\small 4})$,
        $(\rho_{6, 19, 53, 80, 105}, \text{\small 3}, \text{\small 7})$,
        $(\rho_{28, 49, 81, 82, 103}, \text{\small 3}, \text{\small 6})$,
        $(\rho_{1, 4, 31, 93, 97}, \text{\small 3}, \text{\small 5})$,
        $(\rho_{1, 24, 27, 54, 57}, \text{\small 1}, \text{\small 3})$,
        $(\rho_{2, 16, 20, 35, 38, 84}, \text{\small 5}, \text{\small 4})$,
        $(\rho_{3, 30, 36, 66, 70, 116}, \text{\small 5}, \text{\small 4})$,
        $(\rho_{1, 17, 21, 28, 91, 99, 108}, \text{\small 4}, \text{\small 5})$,
        $(\rho_{9, 19, 28, 36, 84}, \text{\small 4}, \text{\small 4})$,
        $(\rho_{1, 5, 31, 35, 86, 90}, \text{\small 4}, \text{\small 7})$,
        $(\rho_{22, 28, 51, 78, 88, 96}, \text{\small 6}, \text{\small 5})$,
        $(\rho_{1, 24, 28, 63, 90, 94}, \text{\small 3}, \text{\small 5})$,
        $(\rho_{2, 3, 32, 36, 72, 88, 96}, \text{\small 4}, \text{\small 5})$,

    \item[$\Upsilon_{151} - \Upsilon_{200}$]
        $(\rho_{3, 10, 20, 72, 84}, \text{\small 3}, \text{\small 4})$,
        $(\rho_{29, 37, 61, 72}, \text{\small 2}, \text{\small 3})$,
        $(\rho_{1, 4, 15, 37, 46, 52}, \text{\small 5}, \text{\small 3})$,
        $(\rho_{3, 14, 51, 67, 88, 90}, \text{\small 5}, \text{\small 4})$,
        $(\rho_{1, 2, 8, 19, 56, 61, 66}, \text{\small 4}, \text{\small 4})$,
        $(\rho_{11, 30, 38, 50, 66, 94}, \text{\small 6}, \text{\small 5})$,
        $(\rho_{4, 28, 63, 94}, \text{\small 5}, \text{\small 5})$,
        $(\rho_{5, 13, 52, 88, 103}, \text{\small 4}, \text{\small 5})$,
        $(\rho_{4, 21, 31, 37, 91}, \text{\small 3}, \text{\small 5})$,
        $(\rho_{12, 15, 17, 43, 51, 59}, \text{\small 5}, \text{\small 4})$,
        $(\rho_{23, 40, 64, 80, 84}, \text{\small 6}, \text{\small 4})$,
        $(\rho_{13, 15, 25, 52, 64}, \text{\small 2}, \text{\small 5})$,
        $(\rho_{9, 17, 18, 80, 84}, \text{\small 4}, \text{\small 4})$,
        $(\rho_{5, 10, 17, 32, 37, 52}, \text{\small 4}, \text{\small 3})$,
        $(\rho_{10, 20, 79, 84}, \text{\small 4}, \text{\small 3})$,
        $(\rho_{6, 11, 19, 37, 54, 76, 111}, \text{\small 5}, \text{\small 3})$,
        $(\rho_{5, 6, 17, 20, 50, 91, 110}, \text{\small 4}, \text{\small 5})$,
        $(\rho_{13, 18, 25, 37, 62, 110}, \text{\small 4}, \text{\small 5})$,
        $(\rho_{9, 10, 34, 51, 70, 81}, \text{\small 4}, \text{\small 3})$,
        $(\rho_{2, 4, 22, 38, 62, 81}, \text{\small 4}, \text{\small 2})$,
        $(\rho_{1, 9, 25, 44, 94, 99}, \text{\small 4}, \text{\small 4})$,
        $(\rho_{3, 6, 19, 28, 63, 91, 93}, \text{\small 5}, \text{\small 5})$,
        $(\rho_{6, 17, 19, 22, 61, 65}, \text{\small 5}, \text{\small 2})$,
        $(\rho_{2, 25, 31, 44, 94, 97, 99}, \text{\small 4}, \text{\small 4})$,
        $(\rho_{1, 2, 18, 44, 71, 80, 95}, \text{\small 4}, \text{\small 4})$,
        $(\rho_{3, 7, 24, 36, 55, 59, 103}, \text{\small 4}, \text{\small 5})$,
        $(\rho_{14, 39, 47, 54, 60, 109}, \text{\small 4}, \text{\small 6})$,
        $(\rho_{26, 34, 55, 61, 103, 109}, \text{\small 4}, \text{\small 5})$,
        $(\rho_{2, 10, 18, 43, 44, 98, 99}, \text{\small 4}, \text{\small 6})$,
        $(\rho_{5, 10, 21, 32, 51, 92, 106}, \text{\small 4}, \text{\small 5})$,
        $(\rho_{1, 12, 24, 32, 37, 54}, \text{\small 5}, \text{\small 3})$,
        $(\rho_{2, 11, 25, 26, 38, 94, 109}, \text{\small 4}, \text{\small 5})$,
        $(\rho_{1, 3, 11, 31, 44, 56}, \text{\small 4}, \text{\small 4})$,
        $(\rho_{2, 5, 10, 33, 44, 90, 102}, \text{\small 4}, \text{\small 4})$,
        $(\rho_{14, 22, 25, 51, 64, 90}, \text{\small 3}, \text{\small 2})$,
        $(\rho_{4, 20, 22, 43, 56, 97}, \text{\small 5}, \text{\small 3})$,
        $(\rho_{18, 19, 32, 65, 66, 84}, \text{\small 5}, \text{\small 4})$,
        $(\rho_{2, 5, 33, 44, 95, 99}, \text{\small 4}, \text{\small 4})$,
        $(\rho_{2, 13, 27, 30, 37, 54, 61}, \text{\small 4}, \text{\small 3})$,
        $(\rho_{4, 16, 18, 27, 28, 57, 92}, \text{\small 4}, \text{\small 5})$,
        $(\rho_{6, 17, 24, 39, 52, 102}, \text{\small 4}, \text{\small 5})$,
        $(\rho_{3, 7, 20, 25, 61, 100}, \text{\small 5}, \text{\small 6})$,
        $(\rho_{12, 71, 88, 95, 101, 106}, \text{\small 3}, \text{\small 6})$,
        $(\rho_{13, 14, 54, 95, 103, 105}, \text{\small 3}, \text{\small 5})$,
        $(\rho_{1, 10, 11, 20, 25, 92, 97, 106}, \text{\small 4}, \text{\small 5})$,
        $(\rho_{1, 2, 6, 9, 19, 37, 116}, \text{\small 4}, \text{\small 4})$,
        $(\rho_{5, 7, 29, 50, 57, 61}, \text{\small 4}, \text{\small 4})$,
        $(\rho_{3, 5, 7, 8, 15, 53, 59}, \text{\small 4}, \text{\small 4})$,
        $(\rho_{24, 32, 50, 57, 61, 76}, \text{\small 4}, \text{\small 4})$,
        $(\rho_{5, 7, 16, 29, 50, 59, 61}, \text{\small 4}, \text{\small 4})$,

    \item[$\Upsilon_{201} - \Upsilon_{250}$]
        $(\rho_{23, 62, 69, 99}, \text{\small 4}, \text{\small 4})$,
        $(\rho_{5, 26, 51, 52, 103, 106}, \text{\small 4}, \text{\small 3})$,
        $(\rho_{5, 6, 17, 38, 50, 91}, \text{\small 4}, \text{\small 5})$,
        $(\rho_{23, 30, 39, 98, 109}, \text{\small 4}, \text{\small 6})$,
        $(\rho_{6, 10, 11, 60, 89, 90}, \text{\small 4}, \text{\small 5})$,
        $(\rho_{15, 35, 44, 102, 110}, \text{\small 4}, \text{\small 6})$,
        $(\rho_{42, 62, 66, 110}, \text{\small 4}, \text{\small 6})$,
        $(\rho_{27, 38, 92, 104, 106}, \text{\small 5}, \text{\small 5})$,
        $(\rho_{1, 3, 5, 17, 55, 66}, \text{\small 4}, \text{\small 4})$,
        $(\rho_{12, 30, 55, 83, 103}, \text{\small 4}, \text{\small 5})$,
        $(\rho_{3, 5, 17, 20, 55, 66}, \text{\small 4}, \text{\small 4})$,
        $(\rho_{2, 3, 4, 12, 70, 94}, \text{\small 4}, \text{\small 5})$,
        $(\rho_{2, 3, 22, 25, 35}, \text{\small 4}, \text{\small 2})$,
        $(\rho_{21, 33, 42, 51, 59}, \text{\small 4}, \text{\small 2})$,
        $(\rho_{1, 2, 5, 10, 59, 87}, \text{\small 4}, \text{\small 4})$,
        $(\rho_{50, 62, 96, 97, 105}, \text{\small 4}, \text{\small 4})$,
        $(\rho_{11, 34, 38, 51, 81}, \text{\small 5}, \text{\small 3})$,
        $(\rho_{9, 11, 13, 15, 23, 53}, \text{\small 4}, \text{\small 3})$,
        $(\rho_{2, 10, 24, 28, 55, 81, 82}, \text{\small 4}, \text{\small 4})$,
        $(\rho_{5, 9, 21, 37, 72, 91}, \text{\small 4}, \text{\small 5})$,
        $(\rho_{10, 55, 62, 70, 81}, \text{\small 4}, \text{\small 4})$,
        $(\rho_{2, 5, 9, 22, 35, 51, 90}, \text{\small 4}, \text{\small 2})$,
        $(\rho_{2, 7, 10, 28, 62, 81}, \text{\small 4}, \text{\small 4})$,
        $(\rho_{2, 3, 4, 40, 71, 90, 91}, \text{\small 4}, \text{\small 5})$,
        $(\rho_{10, 30, 62, 81, 92}, \text{\small 4}, \text{\small 4})$,
        $(\rho_{2, 9, 14, 27, 51, 61}, \text{\small 4}, \text{\small 3})$,
        $(\rho_{6, 16, 29, 41, 63, 110}, \text{\small 4}, \text{\small 5})$,
        $(\rho_{2, 11, 34, 63, 71, 91}, \text{\small 3}, \text{\small 5})$,
        $(\rho_{20, 38, 91, 111}, \text{\small 5}, \text{\small 6})$,
        $(\rho_{2, 9, 12, 24, 37, 75}, \text{\small 4}, \text{\small 3})$,
        $(\rho_{1, 4, 6, 24, 26, 50, 82}, \text{\small 5}, \text{\small 4})$,
        $(\rho_{1, 13, 17, 18, 25, 61}, \text{\small 4}, \text{\small 2})$,
        $(\rho_{1, 2, 8, 28, 62, 81}, \text{\small 4}, \text{\small 4})$,
        $(\rho_{1, 10, 17, 33, 56, 92}, \text{\small 4}, \text{\small 5})$,
        $(\rho_{11, 19, 32, 84, 96}, \text{\small 4}, \text{\small 4})$,
        $(\rho_{5, 11, 32, 52, 64, 109, 114}, \text{\small 4}, \text{\small 5})$,
        $(\rho_{3, 4, 22, 29, 35, 38}, \text{\small 5}, \text{\small 2})$,
        $(\rho_{4, 13, 18, 22, 34, 45}, \text{\small 5}, \text{\small 2})$,
        $(\rho_{5, 22, 33, 34, 45}, \text{\small 6}, \text{\small 2})$,
        $(\rho_{9, 17, 59, 87, 94}, \text{\small 5}, \text{\small 4})$,
        $(\rho_{15, 18, 66, 103, 105}, \text{\small 3}, \text{\small 6})$,
        $(\rho_{5, 9, 17, 52, 72, 97}, \text{\small 4}, \text{\small 5})$,
        $(\rho_{1, 5, 11, 32, 44, 95}, \text{\small 4}, \text{\small 4})$,
        $(\rho_{3, 4, 23, 53}, \text{\small 4}, \text{\small 3})$,
        $(\rho_{9, 10, 11, 24, 84}, \text{\small 4}, \text{\small 4})$,
        $(\rho_{6, 10, 12, 60, 80, 105}, \text{\small 4}, \text{\small 6})$,
        $(\rho_{1, 13, 24, 28, 81, 84, 114}, \text{\small 5}, \text{\small 4})$,
        $(\rho_{2, 16, 33, 41, 62, 90}, \text{\small 4}, \text{\small 4})$,
        $(\rho_{13, 17, 41, 56, 64, 83}, \text{\small 3}, \text{\small 4})$,
        $(\rho_{14, 18, 19, 64, 83, 84}, \text{\small 4}, \text{\small 4})$,

    \item[$\Upsilon_{251} - \Upsilon_{300}$]
        $(\rho_{3, 5, 13, 50, 80, 101}, \text{\small 4}, \text{\small 3})$,
        $(\rho_{3, 6, 49, 50, 98, 99}, \text{\small 5}, \text{\small 6})$,
        $(\rho_{1, 5, 15, 44, 103, 109}, \text{\small 4}, \text{\small 6})$,
        $(\rho_{9, 18, 27, 50, 64, 100, 114}, \text{\small 5}, \text{\small 6})$,
        $(\rho_{5, 7, 11, 28, 87, 103, 109}, \text{\small 4}, \text{\small 6})$,
        $(\rho_{12, 71, 95, 101, 106}, \text{\small 3}, \text{\small 6})$,
        $(\rho_{1, 6, 19, 39, 96, 101, 109}, \text{\small 4}, \text{\small 6})$,
        $(\rho_{20, 50, 102, 117}, \text{\small 4}, \text{\small 5})$,
        $(\rho_{2, 11, 23, 33, 71, 101}, \text{\small 4}, \text{\small 6})$,
        $(\rho_{2, 16, 26, 71, 75, 95}, \text{\small 6}, \text{\small 4})$,
        $(\rho_{5, 9, 31, 77, 88}, \text{\small 5}, \text{\small 4})$,
        $(\rho_{3, 4, 53, 63, 81}, \text{\small 4}, \text{\small 3})$,
        $(\rho_{1, 9, 14, 51, 70}, \text{\small 4}, \text{\small 3})$,
        $(\rho_{13, 17, 39, 85, 96}, \text{\small 6}, \text{\small 5})$,
        $(\rho_{2, 5, 35, 87, 95}, \text{\small 4}, \text{\small 4})$,
        $(\rho_{1, 10, 25, 64, 70, 94}, \text{\small 4}, \text{\small 5})$,
        $(\rho_{3, 5, 23, 59, 83, 84}, \text{\small 5}, \text{\small 4})$,
        $(\rho_{6, 16, 20, 25, 84}, \text{\small 5}, \text{\small 4})$,
        $(\rho_{2, 4, 9, 20, 54, 97}, \text{\small 4}, \text{\small 3})$,
        $(\rho_{3, 5, 11, 32, 39, 50, 51}, \text{\small 5}, \text{\small 3})$,
        $(\rho_{1, 12, 14, 28, 57, 84}, \text{\small 5}, \text{\small 4})$,
        $(\rho_{12, 14, 15, 55, 59, 96}, \text{\small 4}, \text{\small 4})$,
        $(\rho_{18, 25, 33, 57, 92, 106}, \text{\small 5}, \text{\small 5})$,
        $(\rho_{2, 4, 30, 40, 62, 109}, \text{\small 4}, \text{\small 4})$,
        $(\rho_{4, 8, 41, 42, 55, 96, 99}, \text{\small 4}, \text{\small 5})$,
        $(\rho_{1, 16, 20, 25, 35, 84}, \text{\small 5}, \text{\small 4})$,
        $(\rho_{6, 8, 13, 17, 50, 63, 69}, \text{\small 4}, \text{\small 4})$,
        $(\rho_{9, 23, 27, 34, 59, 77}, \text{\small 6}, \text{\small 5})$,
        $(\rho_{1, 4, 5, 23, 37, 41, 62}, \text{\small 6}, \text{\small 4})$,
        $(\rho_{24, 30, 37, 73, 92, 119}, \text{\small 6}, \text{\small 5})$,
        $(\rho_{3, 7, 12, 62, 81, 99}, \text{\small 2}, \text{\small 4})$,
        $(\rho_{3, 6, 10, 27, 62, 81, 92}, \text{\small 4}, \text{\small 4})$,
        $(\rho_{10, 12, 26, 50, 70, 80}, \text{\small 5}, \text{\small 3})$,
        $(\rho_{2, 5, 23, 38, 57, 100}, \text{\small 5}, \text{\small 6})$,
        $(\rho_{10, 17, 25, 37, 41, 55, 62}, \text{\small 5}, \text{\small 4})$,
        $(\rho_{16, 26, 30, 38, 86, 94}, \text{\small 5}, \text{\small 4})$,
        $(\rho_{9, 14, 17, 37, 55, 102}, \text{\small 4}, \text{\small 4})$,
        $(\rho_{21, 55, 59, 103, 106}, \text{\small 4}, \text{\small 5})$,
        $(\rho_{2, 9, 15, 22, 37}, \text{\small 4}, \text{\small 3})$,
        $(\rho_{19, 30, 55, 84, 96}, \text{\small 4}, \text{\small 4})$,
        $(\rho_{1, 4, 6, 55, 57, 59, 106}, \text{\small 5}, \text{\small 4})$,
        $(\rho_{9, 15, 32, 37, 55}, \text{\small 4}, \text{\small 3})$,
        $(\rho_{1, 2, 3, 4, 19, 64, 70, 94}, \text{\small 4}, \text{\small 5})$,
        $(\rho_{21, 38, 55, 102}, \text{\small 5}, \text{\small 4})$,
        $(\rho_{11, 15, 37, 42, 49}, \text{\small 5}, \text{\small 3})$,
        $(\rho_{10, 17, 32, 39, 52, 101, 109}, \text{\small 4}, \text{\small 5})$,
        $(\rho_{1, 5, 11, 23, 37, 69, 115}, \text{\small 3}, \text{\small 4})$,
        $(\rho_{12, 60, 79, 86, 93, 105}, \text{\small 5}, \text{\small 6})$,
        $(\rho_{1, 7, 16, 34, 41, 51, 70}, \text{\small 4}, \text{\small 3})$,
        $(\rho_{6, 9, 24, 25, 71, 77, 88}, \text{\small 5}, \text{\small 4})$,

    \item[$\Upsilon_{301} - \Upsilon_{308}$]
        $(\rho_{1, 7, 28, 62, 81, 115}, \text{\small 4}, \text{\small 4})$,
        $(\rho_{4, 25, 37, 62, 109, 117}, \text{\small 4}, \text{\small 5})$,
        $(\rho_{10, 48, 49, 52, 57, 79}, \text{\small 4}, \text{\small 3})$,
        $(\rho_{1, 7, 24, 28, 81, 82, 94}, \text{\small 4}, \text{\small 4})$,
        $(\rho_{2, 9, 22, 40, 50, 52, 102}, \text{\small 5}, \text{\small 4})$,
        $(\rho_{2, 7, 9, 21, 73, 95, 99}, \text{\small 4}, \text{\small 5})$,
        $(\rho_{9, 11, 24, 37, 48, 91, 108}, \text{\small 5}, \text{\small 4})$,
        $(\rho_{3, 12, 17, 30, 50, 55, 84}, \text{\small 3}, \text{\small 4})$.

\end{description}
\end{flushleft}

\noindent Finally, we list some reductions involving these elements.

\vspace{5mm}

\noindent 1) We have $\Omega_i \xlongrightarrow{L} \llbracket \eta_j
\rrbracket$ for $(i, j)$ among the following:

\vspace{2pt}

\begin{adjustwidth}{15pt}{0pt}

    \begin{flushleft}
        \noindent
        $(\text{\small 1}, \text{\small 321})$, $(\text{\small 2}, \text{\small
        325})$, $(\text{\small 3}, \text{\small 757})$, $(\text{\small 4},
        \text{\small 758})$, $(\text{\small 5}, \text{\small 759})$,
        $(\text{\small 6}, \text{\small 768})$, $(\text{\small 7}, \text{\small
        776})$, $(\text{\small 8}, \text{\small 763})$, $(\text{\small 23},
        \text{\small 224})$, $(\text{\small 24}, \text{\small 225})$,
        $(\text{\small 25}, \text{\small 271})$, $(\text{\small 26},
        \text{\small 269})$, $(\text{\small 27}, \text{\small 272})$,
        $(\text{\small 28}, \text{\small 262})$, $(\text{\small 29},
        \text{\small 274})$, $(\text{\small 30}, \text{\small 264})$,
        $(\text{\small 31}, \text{\small 275})$, $(\text{\small 32},
        \text{\small 268})$, $(\text{\small 33}, \text{\small 265})$,
        $(\text{\small 34}, \text{\small 311})$, $(\text{\small 36},
        \text{\small 365})$, $(\text{\small 37}, \text{\small 376})$,
        $(\text{\small 38}, \text{\small 378})$, $(\text{\small 39},
        \text{\small 395})$, $(\text{\small 44}, \text{\small 342})$,
        $(\text{\small 47}, \text{\small 358})$, $(\text{\small 51},
        \text{\small 379})$, $(\text{\small 54}, \text{\small 410})$,
        $(\text{\small 64}, \text{\small 416})$, $(\text{\small 65},
        \text{\small 413})$, $(\text{\small 70}, \text{\small 544})$,
        $(\text{\small 71}, \text{\small 577})$, $(\text{\small 72},
        \text{\small 614})$, $(\text{\small 73}, \text{\small 589})$,
        $(\text{\small 74}, \text{\small 568})$, $(\text{\small 75},
        \text{\small 608})$, $(\text{\small 76}, \text{\small 580})$,
        $(\text{\small 77}, \text{\small 593})$, $(\text{\small 78},
        \text{\small 597})$, $(\text{\small 79}, \text{\small 604})$,
        $(\text{\small 80}, \text{\small 605})$, $(\text{\small 81},
        \text{\small 573})$, $(\text{\small 82}, \text{\small 578})$,
        $(\text{\small 86}, \text{\small 696})$, $(\text{\small 87},
        \text{\small 699})$, $(\text{\small 88}, \text{\small 700})$,
        $(\text{\small 90}, \text{\small 897})$, $(\text{\small 91},
        \text{\small 868})$, $(\text{\small 92}, \text{\small 839})$,
        $(\text{\small 93}, \text{\small 896})$, $(\text{\small 94},
        \text{\small 848})$, $(\text{\small 95}, \text{\small 859})$,
        $(\text{\small 96}, \text{\small 876})$, $(\text{\small 97},
        \text{\small 882})$, $(\text{\small 98}, \text{\small 887})$,
        $(\text{\small 99}, \text{\small 889})$, $(\text{\small 100},
        \text{\small 895})$, $(\text{\small 101}, \text{\small 843})$,
        $(\text{\small 102}, \text{\small 900})$, $(\text{\small 111},
        \text{\small 945})$, $(\text{\small 112}, \text{\small 952})$,
        $(\text{\small 113}, \text{\small 938})$, $(\text{\small 114},
        \text{\small 944})$, $(\text{\small 115}, \text{\small 954})$,
        $(\text{\small 118}, \text{\small 950})$, $(\text{\small 119},
        \text{\small 962})$.

    \end{flushleft}

\end{adjustwidth}

\vspace{5mm}

\noindent 2) We have $\Omega_i \xlongrightarrow{L} \Upsilon_j
\xlongrightarrow{R} \llbracket \eta_k \rrbracket$ for $(i, j, k)$ among the
following:

\vspace{2pt}

\begin{adjustwidth}{15pt}{0pt}

    \begin{flushleft}
        \noindent
        $(\text{\small 25}, \text{\small 296}, \text{\small 39})$,
        $(\text{\small 32}, \text{\small 297}, \text{\small 40})$,
        $(\text{\small 33}, \text{\small 298}, \text{\small 276})$,
        $(\text{\small 81}, \text{\small 299}, \text{\small 132})$,
        $(\text{\small 82}, \text{\small 300}, \text{\small 131})$,
        $(\text{\small 83}, \text{\small 301}, \text{\small 133})$,
        $(\text{\small 84}, \text{\small 302}, \text{\small 607})$,
        $(\text{\small 85}, \text{\small 303}, \text{\small 134})$,
        $(\text{\small 89}, \text{\small 304}, \text{\small 163})$,
        $(\text{\small 115}, \text{\small 305}, \text{\small 218})$,
        $(\text{\small 116}, \text{\small 306}, \text{\small 955})$,
        $(\text{\small 117}, \text{\small 307}, \text{\small 951})$,
        $(\text{\small 118}, \text{\small 308}, \text{\small 953})$.

    \end{flushleft}

\end{adjustwidth}

\vspace{5mm}

\noindent 3) We have $\Omega_i \xlongrightarrow{R} \llbracket \eta_j
\rrbracket$ for $(i, j)$ among the following:

\vspace{2pt}

\begin{adjustwidth}{15pt}{0pt}

    \begin{flushleft}
        \noindent
        $(\text{\small 9}, \text{\small 319})$, $(\text{\small 10},
        \text{\small 718})$, $(\text{\small 11}, \text{\small 719})$,
        $(\text{\small 12}, \text{\small 730})$, $(\text{\small 13},
        \text{\small 720})$, $(\text{\small 14}, \text{\small 721})$,
        $(\text{\small 15}, \text{\small 729})$, $(\text{\small 16},
        \text{\small 744})$, $(\text{\small 17}, \text{\small 737})$,
        $(\text{\small 18}, \text{\small 777})$, $(\text{\small 19},
        \text{\small 779})$, $(\text{\small 22}, \text{\small 223})$,
        $(\text{\small 35}, \text{\small 326})$, $(\text{\small 40},
        \text{\small 334})$, $(\text{\small 41}, \text{\small 335})$,
        $(\text{\small 42}, \text{\small 333})$, $(\text{\small 43},
        \text{\small 340})$, $(\text{\small 46}, \text{\small 357})$,
        $(\text{\small 48}, \text{\small 361})$, $(\text{\small 49},
        \text{\small 362})$, $(\text{\small 50}, \text{\small 363})$,
        $(\text{\small 52}, \text{\small 370})$, $(\text{\small 53},
        \text{\small 396})$, $(\text{\small 55}, \text{\small 429})$,
        $(\text{\small 56}, \text{\small 93})$, $(\text{\small 58},
        \text{\small 391})$, $(\text{\small 59}, \text{\small 402})$,
        $(\text{\small 60}, \text{\small 404})$, $(\text{\small 61},
        \text{\small 403})$, $(\text{\small 63}, \text{\small 432})$,
        $(\text{\small 66}, \text{\small 401})$, $(\text{\small 103},
        \text{\small 903})$, $(\text{\small 104}, \text{\small 811})$,
        $(\text{\small 105}, \text{\small 818})$, $(\text{\small 106},
        \text{\small 838})$, $(\text{\small 107}, \text{\small 842})$,
        $(\text{\small 108}, \text{\small 877})$, $(\text{\small 109},
        \text{\small 857})$, $(\text{\small 110}, \text{\small 883})$,
        $(\text{\small 116}, \text{\small 951})$.

    \end{flushleft}

\end{adjustwidth}

\vspace{5mm}

\noindent 4) We have $\Upsilon_i \xlongrightarrow{L} \llbracket \eta_j
\rrbracket$ and $\Upsilon_i \xlongrightarrow{R} \llbracket \eta_k \rrbracket$
for $(i, j, k)$ among the following:

\vspace{2pt}

\begin{adjustwidth}{15pt}{0pt}

    \begin{flushleft}
        \noindent
        $(\text{\small 1}, \text{\small 324}, \text{\small 322})$,
        $(\text{\small 2}, \text{\small 321}, \text{\small 322})$,
        $(\text{\small 3}, \text{\small 321}, \text{\small 318})$,
        $(\text{\small 4}, \text{\small 321}, \text{\small 320})$,
        $(\text{\small 5}, \text{\small 325}, \text{\small 323})$,
        $(\text{\small 6}, \text{\small 759}, \text{\small 178})$,
        $(\text{\small 7}, \text{\small 759}, \text{\small 740})$,
        $(\text{\small 8}, \text{\small 724}, \text{\small 718})$,
        $(\text{\small 9}, \text{\small 727}, \text{\small 718})$,
        $(\text{\small 10}, \text{\small 731}, \text{\small 718})$,
        $(\text{\small 11}, \text{\small 736}, \text{\small 718})$,
        $(\text{\small 12}, \text{\small 725}, \text{\small 718})$,
        $(\text{\small 13}, \text{\small 722}, \text{\small 719})$,
        $(\text{\small 14}, \text{\small 745}, \text{\small 730})$,
        $(\text{\small 15}, \text{\small 754}, \text{\small 730})$,
        $(\text{\small 16}, \text{\small 723}, \text{\small 720})$,
        $(\text{\small 17}, \text{\small 726}, \text{\small 720})$,
        $(\text{\small 18}, \text{\small 733}, \text{\small 720})$,
        $(\text{\small 19}, \text{\small 734}, \text{\small 720})$,
        $(\text{\small 20}, \text{\small 751}, \text{\small 720})$,
        $(\text{\small 21}, \text{\small 757}, \text{\small 742})$,
        $(\text{\small 22}, \text{\small 757}, \text{\small 728})$,
        $(\text{\small 23}, \text{\small 179}, \text{\small 721})$,
        $(\text{\small 24}, \text{\small 732}, \text{\small 729})$,
        $(\text{\small 25}, \text{\small 739}, \text{\small 729})$,
        $(\text{\small 26}, \text{\small 748}, \text{\small 729})$,
        $(\text{\small 27}, \text{\small 738}, \text{\small 729})$,
        $(\text{\small 28}, \text{\small 749}, \text{\small 729})$,
        $(\text{\small 29}, \text{\small 773}, \text{\small 774})$,
        $(\text{\small 30}, \text{\small 766}, \text{\small 773})$,
        $(\text{\small 31}, \text{\small 773}, \text{\small 765})$,
        $(\text{\small 32}, \text{\small 775}, \text{\small 773})$,
        $(\text{\small 33}, \text{\small 775}, \text{\small 744})$,
        $(\text{\small 34}, \text{\small 750}, \text{\small 737})$,
        $(\text{\small 35}, \text{\small 750}, \text{\small 747})$,
        $(\text{\small 36}, \text{\small 735}, \text{\small 750})$,
        $(\text{\small 37}, \text{\small 750}, \text{\small 753})$,
        $(\text{\small 38}, \text{\small 756}, \text{\small 777})$,
        $(\text{\small 39}, \text{\small 777}, \text{\small 764})$,
        $(\text{\small 40}, \text{\small 771}, \text{\small 777})$,
        $(\text{\small 41}, \text{\small 756}, \text{\small 767})$,
        $(\text{\small 42}, \text{\small 769}, \text{\small 760})$,
        $(\text{\small 43}, \text{\small 769}, \text{\small 761})$,
        $(\text{\small 44}, \text{\small 762}, \text{\small 769})$,
        $(\text{\small 45}, \text{\small 769}, \text{\small 770})$,
        $(\text{\small 46}, \text{\small 769}, \text{\small 743})$,
        $(\text{\small 47}, \text{\small 776}, \text{\small 181})$,
        $(\text{\small 48}, \text{\small 768}, \text{\small 778})$,
        $(\text{\small 49}, \text{\small 768}, \text{\small 180})$,
        $(\text{\small 50}, \text{\small 746}, \text{\small 752})$,
        $(\text{\small 51}, \text{\small 755}, \text{\small 752})$,
        $(\text{\small 52}, \text{\small 755}, \text{\small 741})$,
        $(\text{\small 53}, \text{\small 772}, \text{\small 779})$,
        $(\text{\small 66}, \text{\small 12}, \text{\small 223})$,
        $(\text{\small 67}, \text{\small 224}, \text{\small 15})$,
        $(\text{\small 68}, \text{\small 225}, \text{\small 11})$,
        $(\text{\small 75}, \text{\small 260}, \text{\small 261})$,
        $(\text{\small 76}, \text{\small 266}, \text{\small 263})$,
        $(\text{\small 77}, \text{\small 273}, \text{\small 270})$,
        $(\text{\small 80}, \text{\small 386}, \text{\small 328})$,
        $(\text{\small 81}, \text{\small 393}, \text{\small 329})$,
        $(\text{\small 82}, \text{\small 332}, \text{\small 330})$,
        $(\text{\small 83}, \text{\small 331}, \text{\small 336})$,
        $(\text{\small 84}, \text{\small 336}, \text{\small 339})$,
        $(\text{\small 85}, \text{\small 338}, \text{\small 348})$,
        $(\text{\small 86}, \text{\small 348}, \text{\small 341})$,
        $(\text{\small 87}, \text{\small 347}, \text{\small 348})$,
        $(\text{\small 88}, \text{\small 352}, \text{\small 351})$,
        $(\text{\small 89}, \text{\small 353}, \text{\small 366})$,
        $(\text{\small 90}, \text{\small 359}, \text{\small 366})$,
        $(\text{\small 91}, \text{\small 354}, \text{\small 355})$,
        $(\text{\small 92}, \text{\small 356}, \text{\small 367})$,
        $(\text{\small 93}, \text{\small 360}, \text{\small 381})$,
        $(\text{\small 94}, \text{\small 360}, \text{\small 383})$,
        $(\text{\small 95}, \text{\small 360}, \text{\small 385})$,
        $(\text{\small 96}, \text{\small 371}, \text{\small 344})$,
        $(\text{\small 97}, \text{\small 371}, \text{\small 345})$,
        $(\text{\small 98}, \text{\small 382}, \text{\small 380})$,
        $(\text{\small 99}, \text{\small 375}, \text{\small 326})$,
        $(\text{\small 100}, \text{\small 378}, \text{\small 83})$,
        $(\text{\small 101}, \text{\small 89}, \text{\small 334})$,
        $(\text{\small 102}, \text{\small 346}, \text{\small 335})$,
        $(\text{\small 103}, \text{\small 84}, \text{\small 335})$,
        $(\text{\small 104}, \text{\small 337}, \text{\small 333})$,
        $(\text{\small 105}, \text{\small 343}, \text{\small 333})$,
        $(\text{\small 106}, \text{\small 342}, \text{\small 82})$,
        $(\text{\small 108}, \text{\small 358}, \text{\small 369})$,
        $(\text{\small 109}, \text{\small 358}, \text{\small 374})$,
        $(\text{\small 110}, \text{\small 358}, \text{\small 384})$,
        $(\text{\small 111}, \text{\small 394}, \text{\small 362})$,
        $(\text{\small 112}, \text{\small 424}, \text{\small 363})$,
        $(\text{\small 113}, \text{\small 379}, \text{\small 86})$,
        $(\text{\small 114}, \text{\small 387}, \text{\small 370})$,
        $(\text{\small 115}, \text{\small 87}, \text{\small 370})$,
        $(\text{\small 116}, \text{\small 410}, \text{\small 422})$,
        $(\text{\small 117}, \text{\small 410}, \text{\small 426})$,
        $(\text{\small 118}, \text{\small 99}, \text{\small 429})$,
        $(\text{\small 124}, \text{\small 93}, \text{\small 327})$,
        $(\text{\small 125}, \text{\small 397}, \text{\small 350})$,
        $(\text{\small 126}, \text{\small 400}, \text{\small 350})$,
        $(\text{\small 127}, \text{\small 368}, \text{\small 372})$,
        $(\text{\small 128}, \text{\small 373}, \text{\small 368})$,
        $(\text{\small 129}, \text{\small 364}, \text{\small 368})$,
        $(\text{\small 130}, \text{\small 392}, \text{\small 390})$,
        $(\text{\small 131}, \text{\small 388}, \text{\small 390})$,
        $(\text{\small 132}, \text{\small 377}, \text{\small 390})$,
        $(\text{\small 135}, \text{\small 88}, \text{\small 391})$,
        $(\text{\small 137}, \text{\small 409}, \text{\small 411})$,
        $(\text{\small 138}, \text{\small 418}, \text{\small 411})$,
        $(\text{\small 139}, \text{\small 398}, \text{\small 411})$,
        $(\text{\small 140}, \text{\small 406}, \text{\small 411})$,
        $(\text{\small 141}, \text{\small 94}, \text{\small 402})$,
        $(\text{\small 142}, \text{\small 96}, \text{\small 404})$,
        $(\text{\small 143}, \text{\small 97}, \text{\small 403})$,
        $(\text{\small 145}, \text{\small 419}, \text{\small 431})$,
        $(\text{\small 147}, \text{\small 399}, \text{\small 416})$,
        $(\text{\small 148}, \text{\small 416}, \text{\small 408})$,
        $(\text{\small 149}, \text{\small 416}, \text{\small 414})$,
        $(\text{\small 150}, \text{\small 416}, \text{\small 423})$,
        $(\text{\small 151}, \text{\small 413}, \text{\small 98})$,
        $(\text{\small 153}, \text{\small 430}, \text{\small 432})$,
        $(\text{\small 154}, \text{\small 421}, \text{\small 432})$,
        $(\text{\small 156}, \text{\small 95}, \text{\small 401})$,
        $(\text{\small 157}, \text{\small 412}, \text{\small 405})$,
        $(\text{\small 158}, \text{\small 412}, \text{\small 415})$,
        $(\text{\small 159}, \text{\small 412}, \text{\small 425})$,
        $(\text{\small 160}, \text{\small 427}, \text{\small 417})$,
        $(\text{\small 167}, \text{\small 538}, \text{\small 533})$,
        $(\text{\small 168}, \text{\small 543}, \text{\small 534})$,
        $(\text{\small 169}, \text{\small 539}, \text{\small 535})$,
        $(\text{\small 170}, \text{\small 540}, \text{\small 536})$,
        $(\text{\small 171}, \text{\small 586}, \text{\small 541})$,
        $(\text{\small 172}, \text{\small 556}, \text{\small 542})$,
        $(\text{\small 173}, \text{\small 559}, \text{\small 546})$,
        $(\text{\small 174}, \text{\small 612}, \text{\small 547})$,
        $(\text{\small 175}, \text{\small 590}, \text{\small 548})$,
        $(\text{\small 176}, \text{\small 601}, \text{\small 549})$,
        $(\text{\small 177}, \text{\small 594}, \text{\small 550})$,
        $(\text{\small 178}, \text{\small 613}, \text{\small 551})$,
        $(\text{\small 179}, \text{\small 582}, \text{\small 552})$,
        $(\text{\small 180}, \text{\small 588}, \text{\small 553})$,
        $(\text{\small 181}, \text{\small 557}, \text{\small 562})$,
        $(\text{\small 182}, \text{\small 561}, \text{\small 563})$,
        $(\text{\small 183}, \text{\small 600}, \text{\small 564})$,
        $(\text{\small 184}, \text{\small 610}, \text{\small 565})$,
        $(\text{\small 185}, \text{\small 566}, \text{\small 570})$,
        $(\text{\small 186}, \text{\small 560}, \text{\small 569})$,
        $(\text{\small 187}, \text{\small 571}, \text{\small 576})$,
        $(\text{\small 188}, \text{\small 581}, \text{\small 572})$,
        $(\text{\small 189}, \text{\small 555}, \text{\small 575})$,
        $(\text{\small 190}, \text{\small 595}, \text{\small 584})$,
        $(\text{\small 191}, \text{\small 602}, \text{\small 585})$,
        $(\text{\small 192}, \text{\small 591}, \text{\small 592})$,
        $(\text{\small 193}, \text{\small 599}, \text{\small 596})$,
        $(\text{\small 194}, \text{\small 606}, \text{\small 598})$,
        $(\text{\small 195}, \text{\small 603}, \text{\small 611})$,
        $(\text{\small 201}, \text{\small 545}, \text{\small 537})$,
        $(\text{\small 202}, \text{\small 567}, \text{\small 537})$,
        $(\text{\small 203}, \text{\small 558}, \text{\small 574})$,
        $(\text{\small 204}, \text{\small 554}, \text{\small 558})$,
        $(\text{\small 205}, \text{\small 698}, \text{\small 697})$,
        $(\text{\small 206}, \text{\small 814}, \text{\small 795})$,
        $(\text{\small 207}, \text{\small 819}, \text{\small 796})$,
        $(\text{\small 208}, \text{\small 867}, \text{\small 819})$,
        $(\text{\small 209}, \text{\small 826}, \text{\small 797})$,
        $(\text{\small 210}, \text{\small 817}, \text{\small 798})$,
        $(\text{\small 211}, \text{\small 810}, \text{\small 799})$,
        $(\text{\small 212}, \text{\small 825}, \text{\small 810})$,
        $(\text{\small 213}, \text{\small 828}, \text{\small 800})$,
        $(\text{\small 214}, \text{\small 808}, \text{\small 801})$,
        $(\text{\small 215}, \text{\small 834}, \text{\small 808})$,
        $(\text{\small 216}, \text{\small 820}, \text{\small 802})$,
        $(\text{\small 217}, \text{\small 835}, \text{\small 820})$,
        $(\text{\small 218}, \text{\small 815}, \text{\small 803})$,
        $(\text{\small 219}, \text{\small 841}, \text{\small 803})$,
        $(\text{\small 220}, \text{\small 880}, \text{\small 803})$,
        $(\text{\small 221}, \text{\small 822}, \text{\small 804})$,
        $(\text{\small 222}, \text{\small 836}, \text{\small 804})$,
        $(\text{\small 223}, \text{\small 885}, \text{\small 804})$,
        $(\text{\small 224}, \text{\small 823}, \text{\small 805})$,
        $(\text{\small 225}, \text{\small 869}, \text{\small 805})$,
        $(\text{\small 226}, \text{\small 891}, \text{\small 805})$,
        $(\text{\small 227}, \text{\small 872}, \text{\small 806})$,
        $(\text{\small 228}, \text{\small 872}, \text{\small 844})$,
        $(\text{\small 229}, \text{\small 845}, \text{\small 807})$,
        $(\text{\small 230}, \text{\small 809}, \text{\small 821})$,
        $(\text{\small 231}, \text{\small 829}, \text{\small 812})$,
        $(\text{\small 232}, \text{\small 840}, \text{\small 813})$,
        $(\text{\small 233}, \text{\small 850}, \text{\small 813})$,
        $(\text{\small 234}, \text{\small 866}, \text{\small 813})$,
        $(\text{\small 235}, \text{\small 871}, \text{\small 813})$,
        $(\text{\small 236}, \text{\small 892}, \text{\small 813})$,
        $(\text{\small 237}, \text{\small 830}, \text{\small 816})$,
        $(\text{\small 238}, \text{\small 881}, \text{\small 824})$,
        $(\text{\small 239}, \text{\small 881}, \text{\small 863})$,
        $(\text{\small 240}, \text{\small 833}, \text{\small 827})$,
        $(\text{\small 241}, \text{\small 862}, \text{\small 832})$,
        $(\text{\small 242}, \text{\small 837}, \text{\small 855})$,
        $(\text{\small 243}, \text{\small 851}, \text{\small 855})$,
        $(\text{\small 244}, \text{\small 864}, \text{\small 855})$,
        $(\text{\small 245}, \text{\small 865}, \text{\small 855})$,
        $(\text{\small 246}, \text{\small 879}, \text{\small 855})$,
        $(\text{\small 247}, \text{\small 846}, \text{\small 847})$,
        $(\text{\small 248}, \text{\small 846}, \text{\small 856})$,
        $(\text{\small 249}, \text{\small 846}, \text{\small 861})$,
        $(\text{\small 250}, \text{\small 846}, \text{\small 870})$,
        $(\text{\small 252}, \text{\small 894}, \text{\small 853})$,
        $(\text{\small 253}, \text{\small 894}, \text{\small 874})$,
        $(\text{\small 254}, \text{\small 894}, \text{\small 888})$,
        $(\text{\small 255}, \text{\small 894}, \text{\small 890})$,
        $(\text{\small 256}, \text{\small 894}, \text{\small 898})$,
        $(\text{\small 257}, \text{\small 894}, \text{\small 901})$,
        $(\text{\small 259}, \text{\small 878}, \text{\small 831})$,
        $(\text{\small 260}, \text{\small 897}, \text{\small 873})$,
        $(\text{\small 261}, \text{\small 868}, \text{\small 198})$,
        $(\text{\small 262}, \text{\small 858}, \text{\small 811})$,
        $(\text{\small 263}, \text{\small 858}, \text{\small 875})$,
        $(\text{\small 264}, \text{\small 896}, \text{\small 200})$,
        $(\text{\small 265}, \text{\small 887}, \text{\small 203})$,
        $(\text{\small 266}, \text{\small 886}, \text{\small 889})$,
        $(\text{\small 267}, \text{\small 849}, \text{\small 838})$,
        $(\text{\small 268}, \text{\small 893}, \text{\small 838})$,
        $(\text{\small 269}, \text{\small 843}, \text{\small 201})$,
        $(\text{\small 270}, \text{\small 199}, \text{\small 842})$,
        $(\text{\small 271}, \text{\small 860}, \text{\small 842})$,
        $(\text{\small 273}, \text{\small 877}, \text{\small 852})$,
        $(\text{\small 274}, \text{\small 854}, \text{\small 877})$,
        $(\text{\small 276}, \text{\small 202}, \text{\small 857})$,
        $(\text{\small 278}, \text{\small 884}, \text{\small 883})$,
        $(\text{\small 279}, \text{\small 204}, \text{\small 883})$,
        $(\text{\small 281}, \text{\small 900}, \text{\small 902})$,
        $(\text{\small 282}, \text{\small 900}, \text{\small 205})$,
        $(\text{\small 284}, \text{\small 899}, \text{\small 903})$,
        $(\text{\small 285}, \text{\small 206}, \text{\small 903})$,
        $(\text{\small 287}, \text{\small 936}, \text{\small 930})$,
        $(\text{\small 288}, \text{\small 942}, \text{\small 931})$,
        $(\text{\small 289}, \text{\small 935}, \text{\small 932})$,
        $(\text{\small 290}, \text{\small 943}, \text{\small 933})$,
        $(\text{\small 291}, \text{\small 948}, \text{\small 940})$,
        $(\text{\small 292}, \text{\small 940}, \text{\small 934})$,
        $(\text{\small 293}, \text{\small 947}, \text{\small 937})$,
        $(\text{\small 294}, \text{\small 949}, \text{\small 939})$,
        $(\text{\small 295}, \text{\small 946}, \text{\small 941})$.

    \end{flushleft}

\end{adjustwidth}

\vspace{5mm}

\noindent 5) We have $\Upsilon_i \xlongrightarrow{L} \llbracket \eta_j
\rrbracket$ and $\Upsilon_i \xlongrightarrow{R} \Xi_k$ for $(i, j, k)$ among
the following:

\vspace{2pt}

\begin{adjustwidth}{15pt}{0pt}

    \begin{flushleft}
        \noindent
        $(\text{\small 54}, \text{\small 773}, \text{\small 39})$,
        $(\text{\small 55}, \text{\small 777}, \text{\small 40})$,
        $(\text{\small 56}, \text{\small 769}, \text{\small 42})$,
        $(\text{\small 57}, \text{\small 181}, \text{\small 43})$,
        $(\text{\small 58}, \text{\small 746}, \text{\small 45})$,
        $(\text{\small 59}, \text{\small 779}, \text{\small 47})$,
        $(\text{\small 60}, \text{\small 318}, \text{\small 9})$,
        $(\text{\small 61}, \text{\small 319}, \text{\small 11})$,
        $(\text{\small 69}, \text{\small 12}, \text{\small 3})$, $(\text{\small
        70}, \text{\small 15}, \text{\small 4})$, $(\text{\small 71},
        \text{\small 14}, \text{\small 6})$, $(\text{\small 72}, \text{\small
        225}, \text{\small 5})$, $(\text{\small 73}, \text{\small 13},
        \text{\small 2})$, $(\text{\small 79}, \text{\small 41}, \text{\small
        8})$, $(\text{\small 107}, \text{\small 349}, \text{\small 12})$,
        $(\text{\small 119}, \text{\small 85}, \text{\small 13})$,
        $(\text{\small 120}, \text{\small 90}, \text{\small 14})$,
        $(\text{\small 121}, \text{\small 91}, \text{\small 15})$,
        $(\text{\small 122}, \text{\small 92}, \text{\small 16})$,
        $(\text{\small 133}, \text{\small 390}, \text{\small 18})$,
        $(\text{\small 134}, \text{\small 389}, \text{\small 19})$,
        $(\text{\small 146}, \text{\small 431}, \text{\small 22})$,
        $(\text{\small 152}, \text{\small 98}, \text{\small 23})$,
        $(\text{\small 155}, \text{\small 432}, \text{\small 24})$,
        $(\text{\small 161}, \text{\small 401}, \text{\small 25})$,
        $(\text{\small 162}, \text{\small 412}, \text{\small 26})$,
        $(\text{\small 163}, \text{\small 407}, \text{\small 27})$,
        $(\text{\small 164}, \text{\small 417}, \text{\small 28})$,
        $(\text{\small 197}, \text{\small 579}, \text{\small 31})$,
        $(\text{\small 198}, \text{\small 583}, \text{\small 32})$,
        $(\text{\small 199}, \text{\small 587}, \text{\small 33})$,
        $(\text{\small 200}, \text{\small 609}, \text{\small 34})$,
        $(\text{\small 272}, \text{\small 842}, \text{\small 49})$,
        $(\text{\small 275}, \text{\small 877}, \text{\small 50})$,
        $(\text{\small 277}, \text{\small 857}, \text{\small 51})$.

    \end{flushleft}

\end{adjustwidth}

\vspace{5mm}

\noindent 6) We have $\Upsilon_i \xlongrightarrow{L} \Xi_j$ and $\Upsilon_i
\xlongrightarrow{R} \llbracket \eta_k \rrbracket$ for $(i, j, k)$ among the
following:

\vspace{2pt}

\begin{adjustwidth}{15pt}{0pt}

    \begin{flushleft}
        \noindent
        $(\text{\small 62}, \text{\small 10}, \text{\small 325})$,
        $(\text{\small 63}, \text{\small 41}, \text{\small 750})$,
        $(\text{\small 64}, \text{\small 44}, \text{\small 768})$,
        $(\text{\small 65}, \text{\small 46}, \text{\small 763})$,
        $(\text{\small 74}, \text{\small 1}, \text{\small 16})$, $(\text{\small
        78}, \text{\small 7}, \text{\small 267})$, $(\text{\small 123},
        \text{\small 17}, \text{\small 429})$, $(\text{\small 136},
        \text{\small 20}, \text{\small 391})$, $(\text{\small 144},
        \text{\small 21}, \text{\small 403})$, $(\text{\small 165},
        \text{\small 29}, \text{\small 428})$, $(\text{\small 166},
        \text{\small 29}, \text{\small 420})$, $(\text{\small 196},
        \text{\small 30}, \text{\small 603})$, $(\text{\small 251},
        \text{\small 54}, \text{\small 846})$, $(\text{\small 258},
        \text{\small 55}, \text{\small 894})$, $(\text{\small 280},
        \text{\small 52}, \text{\small 883})$, $(\text{\small 283},
        \text{\small 53}, \text{\small 900})$, $(\text{\small 286},
        \text{\small 48}, \text{\small 903})$.

    \end{flushleft}

\end{adjustwidth}

\vspace{5mm}

\noindent 7) We have $\Omega_i \xlongrightarrow{L} \Xi_j$ for $(i, j)$ among
the following:

\vspace{2pt}

\begin{adjustwidth}{15pt}{0pt}

    \begin{flushleft}
        \noindent
        $(\text{\small 57}, \text{\small 19})$, $(\text{\small 67},
        \text{\small 27})$, $(\text{\small 68}, \text{\small 28})$,
        $(\text{\small 69}, \text{\small 29})$, $(\text{\small 83},
        \text{\small 35})$, $(\text{\small 84}, \text{\small 36})$,
        $(\text{\small 85}, \text{\small 37})$, $(\text{\small 89},
        \text{\small 38})$, $(\text{\small 117}, \text{\small 56})$.

    \end{flushleft}

\end{adjustwidth}

\vspace{5mm}

\noindent 8) We have $\Omega_i \xlongrightarrow{R} \Xi_j$ for $(i, j)$ among
the following:

\vspace{2pt}

\begin{adjustwidth}{15pt}{0pt}

    \begin{flushleft}
        \noindent
        $(\text{\small 20}, \text{\small 42})$, $(\text{\small 21},
        \text{\small 45})$, $(\text{\small 45}, \text{\small 12})$,
        $(\text{\small 62}, \text{\small 22})$.

    \end{flushleft}

\end{adjustwidth}

\end{document}